\documentclass[a4paper,11pt,reqno]{amsart}
\usepackage{hyperref}
\usepackage{nicematrix}
\usepackage{xifthen}
\usepackage{xspace}
\usepackage[new]{old-arrows}
\usepackage{latexsym}
\usepackage{amssymb,amsmath,wasysym}
\usepackage{mathtools, tensor}
\usepackage{tikz}
\usepackage{multicol, pdflscape}
\usetikzlibrary{patterns}
\usepackage{mathdots}
\usepackage{graphics}
\usepackage{url}
\usepackage[enableskew,vcentermath]{youngtab}
  
\usepackage{booktabs} 
\usepackage{xcolor} 
\usepackage[width=4.5in,font=small]{caption}
\usepackage{chngpage}
\usepackage{calc}

\newcommand{\nocontentsline}[3]{}
\newcommand{\tocless}[2]{\bgroup\let\addcontentsline=\nocontentsline#1{#2}\egroup}

\allowdisplaybreaks

\newtheorem{theorem}{Theorem}[section]
\newtheorem{theoremBoth}{TheoremBoth}[section]

\newtheorem*{theorem*}{Theorem}
\newtheorem{corollary}[theorem]{Corollary}
\newtheorem{corollaryBoth}[theoremBoth]{Corollary}

\newtheorem{proposition}[theorem]{Proposition}
\newtheorem{conjecture}[theorem]{Conjecture}

\newtheorem{lemma}[theorem]{Lemma}

\theoremstyle{definition}
\newtheorem{definition}[theorem]{Definition}
\newtheorem{definitionBoth}[theoremBoth]{Definition} 
\newtheorem{notation}[theorem]{Definition} 
\newtheorem{remark}[theorem]{Remark}
\newtheorem{example}[theorem]{Example}

\counterwithin{figure}{section}

\newcommand{\N}{\mathbb{N}}
\newcommand{\Z}{\mathbb{Z}}

\newcommand{\C}{\mathbb{C}}

\renewcommand{\epsilon}{\varepsilon}

\renewcommand{\theta}{\vartheta}

\DeclareMathOperator{\SL}{SL}

\DeclareMathOperator{\sgn}{sgn}

\DeclareMathOperator{\wt}{wt}

\newcommand{\mbf}[1]{\mathbf{#1}}

\newcommand{\Par}{\mathrm{Par}}

\newcommand{\ts}{\hskip0.8pt}

\newcommand{\mfrac}[2]{{\textstyle\frac{#1}{#2}}}

\newcounter{thmlistcnt}
\newenvironment{thmlist}%
	{\setcounter{thmlistcnt}{0}%
	\begin{list}{\emph{(\roman{thmlistcnt})}}{%
		\usecounter{thmlistcnt}%
		\setlength{\topsep}{0pt}%
		\setlength{\leftmargin}{0pt}%
		\setlength{\itemsep}{0pt}%
		\setlength{\labelwidth}{17pt}
		\setlength{\itemindent}{30pt}}%
	}%
	{\end{list}}%

\newenvironment{defnlistS}%
	{\setcounter{defnlistcnt}{0}%
	\begin{list}{(\alph{defnlistcnt})}{%
		\usecounter{defnlistcnt}%
		\setlength{\topsep}{0pt}%
		\setlength{\leftmargin}{24pt}%
		\setlength{\itemsep}{0pt}%
		\setlength{\labelwidth}{24pt}
		\setlength{\itemindent}{0pt}}%
	}%
	{\end{list}}%
	
\newenvironment{bulletlist}%
	{\begin{list}{$\bullet$}{%
		\setlength{\topsep}{0pt}%
		\setlength{\leftmargin}{24pt}%
		\setlength{\itemsep}{0pt}%
		\setlength{\labelwidth}{24pt}
		\setlength{\itemindent}{0pt}}%
	}%
	{\end{list}}%
	
\newcounter{defnlistcnt}
\newenvironment{defnlist}%
	{\setcounter{defnlistcnt}{0}%
	\begin{list}{(\alph{defnlistcnt})}{%
		\usecounter{defnlistcnt}%
		\setlength{\topsep}{0pt}%
		\setlength{\leftmargin}{32pt}%
		\setlength{\itemsep}{0pt}%
		\setlength{\labelwidth}{32pt}
		\setlength{\itemindent}{0pt}}%
	}%
	{\end{list}}%
	
\newenvironment{defnlistE}%
	{\setcounter{defnlistcnt}{0}%
	\begin{list}{(\alph{defnlistcnt})}{%
		\usecounter{defnlistcnt}%
		\setlength{\topsep}{3pt}%
		\setlength{\leftmargin}{40pt}%
		\setlength{\itemsep}{0pt}%
		\setlength{\labelwidth}{60pt}
		\setlength{\itemindent}{0pt}}%
	}%
	{\end{list}}%

\newcounter{examplelistcnt}
\newenvironment{examplelist}%
	{\setcounter{examplelistcnt}{0}%
	\begin{list}{(\alph{examplelistcnt})}{%
		\usecounter{examplelistcnt}%
		\setlength{\topsep}{3pt}%
		\setlength{\leftmargin}{24pt}%
		\setlength{\itemsep}{3pt}%
		\setlength{\labelwidth}{24pt}
		\setlength{\itemindent}{0pt}}%
	}%
	{\end{list}}%
	
\newcommand{\mus}{{\mu_\star}}

\newcommand{\sigmas}{{\sigma_\star}}
\newcommand{\nus}{{\nu_\star}}
\newcommand{\nud}{{\nu^\dagger}}
\newcommand{\taus}{{\tau_\star}}
\newcommand{\betas}{{\beta_\star}}

\newcommand{\muS}{{\mu/\mus}}	
\newcommand{\sigmaS}{{\sigma/\sigmas}}
\newcommand{\tauS}{{\tau/\taus}}

\newcommand{\AmuSm}{A^\-(\muS)}
\newcommand{\AmuSp}{A^\+(\muS)}

\newcommand{\oM}{\negt{1}} 
\newcommand{\tM}{\negt{2}} 
\newcommand{\dM}{\negt{3}} 
\newcommand{\fM}{\negt{4}} 

\newcommand{\yraised}{\raisebox{1pt}{$y$}}

\newcommand{\negt}[1]{\,\raisebox{1pt}{\scalebox{1}{-}\hskip-0.1pt}#1}
\renewcommand{\negt}[1]{\raisebox{0.2pt}{$\hskip0.5pt\mbf{#1}$}}

\DeclareMathOperator{\mSSYT}{\hskip-1ptSSYT}
\DeclareMathOperator{\SSYT}{SSYT}
\DeclareMathOperator{\YT}{YT}
\DeclareMathOperator{\PYT}{PYT}
\DeclareMathOperator{\mPSSYT}{\hskip-1ptPSSYT}
\DeclareMathOperator{\PSSYT}{PSSYT}
\newcommand{\sSSYT}{\mathrm{SSYT}^\pm}	
\newcommand{\sPSSYT}{\mathrm{PSSYT}^\pm}
\newcommand{\srPSSYT}{\mathrm{PSSYT}^\mp}

\newcommand{\PSSYTw}[4]{%
	\PSSYT\bigl(#1,#2\bigr)_{\hskip-0.5pt\raisebox{1pt}{$\scriptstyle(#3,#4)$}}}
\newcommand{\PSSYTwAdapted}[5]{%
	\PSSYT_{#5}\bigl(#1,#2\bigr)_{\hskip-0.5pt\raisebox{1pt}{$\scriptstyle(#3,#4)$}}}

\newcommand{\PSSYTwk}[5]{%
	\PSSYT_{#1}\bigl(#2,#3\bigr)_{\hskip-0.5pt\raisebox{1pt}{$\scriptstyle(#4,#5)$}}}
\newcommand{\srPSSYTw}[4]{%
	\PSSYT^\mp\bigl(#1,#2\bigr)_{\hskip-0.5pt\raisebox{1pt}{$\scriptstyle(#3,#4)$}}}	

\newcommand{\SSYTw}[3]{%
	\SSYT\bigl(#1\bigr)_{\hskip-0.5pt\raisebox{1pt}{$\scriptstyle(#2,#3)$}}}

\newcommand{\opluss}{\hspace*{1pt}\oplus\hspace*{1pt}}
\newcommand{\sqcups}{\hspace*{1pt}\sqcup\hspace*{1pt}}

\newcounter{i}
\newcounter{j}

\newcommand{\tableauLineWidth}{0.5pt}

\newcommand{\pyoung}[3]{
	\begin{tikzpicture}[line width=0.6pt, x=#1,y=-#2]
		\pyoungInner{#3}
	\end{tikzpicture}
}

\newcommand{\pyoungAnnotated}[4]{
	\begin{tikzpicture}[line width=0.6pt, x=#1,y=-#2]
		\pyoungInnerAnnotated{#3}{#4}
	\end{tikzpicture}
}

\newcommand{\pyoungScaled}[4]{
	\scalebox{#3}{\pyoung{#1}{#2}{#4}}
}

\newcommand{\pyoungInner}[1]{
	\setcounter{i}{0}
	\setcounter{j}{0}
	\foreach \tRow in #1 {
		\addtocounter{i}{1}
		\setcounter{j}{0}
		\foreach \tEntry in \tRow {
			\addtocounter{j}{1}
			\node at (\value{j}+0.5, \value{i}+0.5) {\tEntry};
			\draw[line width = \tableauLineWidth] (\value{j},\value{i})--(\value{j},\value{i}+1);
		}
		\draw[line width = \tableauLineWidth] (\value{j}+1,\value{i})--(\value{j}+1,\value{i}+1);
		\draw[line width = \tableauLineWidth] (1,\value{i})--(\value{j}+1,\value{i});
		\draw[line width = \tableauLineWidth] (1,\value{i}+1)--(\value{j}+1,\value{i}+1);
	}
}

\newcommand{\pyoungInnerAnnotated}[2]{
	\setcounter{i}{0}
	\setcounter{j}{0}
	\foreach \tRow in #1 {
		\addtocounter{i}{1}
		\setcounter{j}{0}
		\foreach \tEntry in \tRow {
			\addtocounter{j}{1}
			\node at (\value{j}+0.5, \value{i}+0.5) {\tEntry};
			\draw[line width = \tableauLineWidth] (\value{j},\value{i})--(\value{j},\value{i}+1);
		}
		\draw[line width = \tableauLineWidth] (\value{j}+1,\value{i})--(\value{j}+1,\value{i}+1);
		\draw[line width = \tableauLineWidth] (1,\value{i})--(\value{j}+1,\value{i});
		\draw[line width = \tableauLineWidth] (1,\value{i}+1)--(\value{j}+1,\value{i}+1);
	}
	\setcounter{j}{0}	
	\foreach \annotation in #2 {
		\addtocounter{j}{1}
		\node at (\value{j}+0.5, 0.5) {\annotation};
	}	
}

\newcommand{\youngBox}[2]{%
	\begin{tikzpicture}[x=#1cm,y=#1cm]%
		\fill[color = #2] (0,0)--(1,0)--(1,1)--(0,1)--(0,0); 
	;\end{tikzpicture}%
}

\newcommand{\tableauBox}[3]{%
	\draw[line width=\tableauLineWidth] (#1,#2)--(#1,#2+1)--(#1-1,#2+1)--(#1-1,#2)--(#1,#2);
	\node at (#1-0.5, #2+0.5) {#3};
}

\newcommand{\tableauBoxDoubleFilled}[4]{%
	\draw[line width=\tableauLineWidth] (#1,#2)--(#1,#2+1)--(#1-2,#2+1)--(#1-2,#2)--(#1,#2);
	\fill[color=#4] (#1,#2)--(#1,#2+1)--(#1-2,#2+1)--(#1-2,#2)--(#1,#2);
	\node at (#1-1, #2+0.5) {#3};
}

\newcommand{\tB}[3]{\tableauBox{#2+1}{#1+1}{#3}}

\newcommand{\tBDF}[4]{\tableauBoxDoubleFilled{#2+1}{#1+1}{#3}{#4}}

\newcommand{\tableauBoxFilled}[4]{%
	\fill[color = #4] (#1,#2)--(#1,#2+1)--(#1-1,#2+1)--(#1-1,#2)--(#1,#2);
	\draw[line width=\tableauLineWidth] (#1,#2)--(#1,#2+1)--(#1-1,#2+1)--(#1-1,#2)--(#1,#2);	
	\node at (#1-0.5, #2+0.5) {#3};
}

\newcommand{\tableauBoxFilledNoOutline}[4]{%
	\fill[color = #4] (#1-0.1,#2+0.1)--(#1-0.1,#2+0.9)--(#1-0.9,#2+0.9)--(#1-0.9,#2+0.1)--(#1-0.1,#2+0.1);
	\node at (#1-0.5, #2+0.5) {#3};
}

\newcommand{\tBF}[4]{\tableauBoxFilled{#2+1}{#1+1}{#3}{#4}}
\newcommand{\tBFN}[4]{\tableauBoxFilledNoOutline{#2+1}{#1+1}{#3}{#4}}

\newcommand{\smallBoxWidth}{0.65cm}

\definecolor{darkgrey}{rgb}{0.6,0.6,0.6}
\definecolor{optdarkgrey}{rgb}{1,1,1}
\definecolor{grey}{rgb}{0.7,0.7,0.7}
\definecolor{lightgrey}{rgb}{0.8,0.8,0.8}
\definecolor{verylightgrey}{rgb}{0.9,0.9,0.9}

\newcommand{\spy}[2]{\, \raisebox{#1}{#2}\ \hskip0.5pt }

\DeclareMathOperator{\supp}{supp}

\newcommand{\youngL}[1]{\protect{\young#1}}

\newcommand{\meet}{\cap} 

\newcommand{\Mn}{M'}

\newcommand{\nuSeq}[1]{{\nu^{(#1)}}}

\newcommand{\muSSeq}[1]{{\muS^{\raisebox{-1pt}{$\!\scriptstyle(#1)$}}}}
\newcommand{\muSSeqs}[1]{{\raisebox{1pt}{$\scriptstyle{\muS^{\raisebox{-2pt}{$\scriptscriptstyle(#1)$}}}$}}}
\newcommand{\PSeq}[1]{{\mathcal{P}^{\raisebox{-1pt}{$\scriptstyle(#1)$}}}}
\newcommand{\RSeq}[1]{{\mathcal{R}^{\raisebox{-1pt}{$\scriptstyle(#1)$}}}}
\newcommand{\PSeqs}[1]{{\mathcal{P}^{\raisebox{-1pt}{$\scriptscriptstyle(#1)$}}}}
\newcommand{\QSeq}[1]{{\mathcal{Q}^{\raisebox{-1pt}{$\scriptstyle(#1)$}}}}

\newcommand{\SM}{\mathcal{M}}
\renewcommand{\SS}{\mathcal{S}}
\newcommand{\SMw}[1]{\mathcal{M}_\swtp{#1}}

\newcommand{\TMw}[1]{T_\swtp{#1}}

\newcommand{\T}{\mathcal{T}}

\newcommand{\A}{\mathcal{A}}

\newcommand{\W}{\mathcal{W}}

\newcommand{\V}{\mathcal{V}}

\let\opthyphen\-

\renewcommand{\-}{{\scriptscriptstyle-}}
\newcommand{\minus}{{\scriptscriptstyle-}}
\newcommand{\+}{{\scriptscriptstyle+}}

\newcommand{\swtp}[1]{{(#1^\-,#1^\+)}}
\newcommand{\swte}[2]{\bigl(#1,#2\bigr)}
\DeclareMathOperator{\swt}{swt}

\newcommand{\unRHDT}{{\unrhd\raisebox{0.25ex}{\hspace*{-8.5pt}$\RHD$}}}
\newcommand{\unLHDT}{{\unlhd\raisebox{0.25ex}{\hspace*{-8.5pt}$\LHD$}}}
\newcommand{\unLHDS}{{\scalebox{0.8}{$\unLHDT$}}}
\newcommand{\unLHD}{{\,\unLHDT\hspace*{4pt}}}
\newcommand{\unRHD}{{\,\unRHDT\hspace*{4pt}}}

\newcommand{\dec}[2]{\bigl\langle#1, #2\bigr\rangle}
\newcommand{\decs}[2]{\langle#1, #2\rangle}
\newcommand{\decss}[2]{\scriptstyle \langle #1, #2\rangle}

\newcommand{\ellp}{p}
\newcommand{\ellpb}{{(\ellp)}}
\newcommand{\ellmb}{{\scriptscriptstyle (\ell^\-)}}

\newcommand\unlhddotT{\unlhd\hspace*{-4.75pt}\raisebox{1pt}{$\cdot$}}
\newcommand{\unlhddot}{{\,\unlhddotT\hspace*{4pt}}}
\newcommand{\unlhddots}{{\,\,\unlhddotT\hspace*{4pt}\,}}

\newcommand\unrhddotT{\unrhd\hspace*{-7pt}\raisebox{1pt}{$\cdot$}}
\newcommand{\unrhddot}{{\,\unrhddotT\hspace*{6pt}}}
\newcommand{\unrhddots}{{\,\,\unrhddotT\hspace*{6pt}\,}}

\newcommand{\unlhddotS}{{\scalebox{0.8}{$\unlhddotT$}\,}}
\newcommand{\notunlhddot}{\not\hskip-7pt\unlhddot}

\newcommand{\notunlhd}{\hbox{$\,\not\hskip-4pt\unlhd\;$}}

\newcommand\ledotT{\le\hspace*{-6.5pt}\raisebox{1pt}{$\cdot$}}
\newcommand\ltdotT{<\hspace*{-7pt}\raisebox{0.1pt}{$\cdot$}}
\newcommand{\ledot}{{\,\ledotT\,}}
\newcommand{\ltdot}{{\,\ltdotT\,}}

\renewcommand{\ss}{\scriptstyle}

\newcommand{\rb}[1]{\rotatebox{90}{$\ss #1$}}

\newcommand{\minusprime}{{-\raisebox{-2pt}{$\scriptstyle '$}}}

\newcommand{\rM}{{r^{\scriptscriptstyle -}}}
\newcommand{\rP}{{r^{\scriptscriptstyle +}}}

\newcommand{\DM}{D^\-}

\newcommand{\decMap}{\!\leftrightarrow\!}

\DeclareMathOperator{\LBound}{L}
\DeclareMathOperator{\LPBound}{LP}
\DeclareMathOperator{\LZBound}{LZ}

\renewcommand{\ts}{\hskip0.5pt}

\newcommand{\pmap}{F}         
\newcommand{\pf}[1]{p^{(#1)}} 
\newcommand{\KM}{K}           

\newcommand{\rx}{\raisebox{1pt}{$x$}}
\newcommand{\ry}{\raisebox{1pt}{$y$}}

\newcommand{\subsubsubsection}[1]{\medskip\noindent\textsc{\small#1}.}

\newcommand{\Shift}{S} 

\newcommand{\emphd}[1]{#1}

\newcommand{\preceqtab}{\, \raisebox{0.5pt}{$\scriptstyle\preceq$}}

\newcommand{\nSmall}{d}

\newcommand{\Fmap}{\mathcal{F}}
\newcommand{\Gmap}{\mathcal{G}}
\newcommand{\Hmap}{\mathcal{H}}

\DeclareRobustCommand\verylongrightarrow
     {\relbar\joinrel\relbar\joinrel\relbar\joinrel\rightarrow}

\newcommand{\names}[1]{\textbf{#1}. }

\newcommand{\KK}{C}

\newcommand{\FC}{\KK} 
\newcommand{\GD}{W}

\newcommand{\rk}{k}

\newcommand{\nurho}{\rho}
\newcommand{\nurhod}{\rho^\dagger}

\newcommand{\KC}{C}

\makeatletter
\newcommand\xlabel[2][]{\phantomsection\def\@currentlabelname{#1}\label{#2}}
\makeatother

\begin{document}
\numberwithin{equation}{section}

\title[]{Two stability theorems on plethysms of Schur functions}
\author{Rowena Paget and Mark Wildon}
\date{\today}

\subjclass[2010]{Primary: 05E05, Secondary: 05E10,  20G05}

\maketitle
\thispagestyle{empty}

\begin{abstract}
The plethysm product of Schur functions corresponds to composing  polynomial representations of infinite general linear groups. Finding the plethysm coefficients $\langle s_\nu \circ s_\mu, s_\lambda\rangle$ that express an arbitrary plethysm $s_\nu \circ s_\mu$ as a sum $\sum_\lambda  \langle s_\nu \circ s_\mu, s_\lambda \rangle s_\lambda$ of Schur functions is a fundamental open problem in algebraic combinatorics. We prove two stability theorems for plethysm coefficients under the operations of adding and/or joining an arbitrary partition to either $\mu$ or $\nu$. In both theorems $\mu$ may be replaced with an arbitrary skew partition. As special cases we obtain all stability results on the plethysm product of two Schur functions in the literature to date. The proofs are entirely combinatorial using plethystic semistandard tableaux with positive and negative entries.
\end{abstract}

\setcounter{tocdepth}{1}
\tableofcontents

\vspace*{-0.35in}

\addtocontents{toc}{\vspace*{3pt}}
\addtocontents{toc}{\textbf{Introduction and overview}}
\section{Introduction}\label{sec:introduction}

\subsection{Background}
Determining the decomposition of an arbitrary plethysm product $s_\nu \circ s_\mu$
into Schur functions was identified
by Richard Stanley in \cite{StanleyPositivity} as a central open problem in algebraic
combinatorics. It is equivalent to decomposing a polynomial representation of 
$\mathrm{GL}_n(\mathbb{C})$
defined 
by a composition of Schur functors into a direct sum of irreducible representations
of $\mathrm{GL}_n(\mathbb{C})$
and to decomposing a representation of a symmetric
group induced from an arbitrary irreducible representation of a
wreath product subgroup into a direct sum of irreducible representations.
The \emph{plethysm coefficients} are the multiplicities in these decompositions.
We refer the reader to the introduction to \cite{PagetWildonGeneralized}
for a full account of these connections.

\subsection{Stability}
In this paper we prove two theorems showing that certain sequences of
plethysm coefficients are ultimately constant, with \emph{explicit} bounds for when stability occurs.
We also give practical sufficient conditions for the stable value to be zero.
These theorems include as special cases all stability results on the plethysm product
of two Schur functions
in the current literature, sometimes with new bounds when none were proved originally.
For instance  a special case of
Theorem~\ref{thm:muStable}, first proved in
\cite[p354]{BrionStability},
is that 
\smash{$\langle s_\nu \circ s_{\mu + M\kappa}, s_{\lambda + M|\nu|\kappa} \rangle$}
is ultimately constant for large $M$, while a special case of
Theorem~\ref{thm:nuStable} is a key motivating result,
proved in \cite{LawOkitaniStability} without an explicit bound, that if~$d$ is even then
\smash{$\bigl\langle s_{\nu + (M)} \circ s_{(m)}, s_{\lambda \,\sqcup\, (d^M) + M(m-d) } \bigr\rangle$}
is ultimately constant for large $M$. Our proofs are entirely
combinatorial, using the plethystic semistandard signed
tableaux defined in Definition~\ref{defn:plethysticSemistandardSignedTableau}~below.

\subsection{Main results}
In both our main theorems,
$\nu$ is a partition, $\muS$ is a skew partition and~$\lambda$ 
is a partition of $|\nu| |\muS|$.
We define
$\muS \opluss (\gamma,\delta) = \bigl( (\mu \sqcups \gamma') \,+\, \delta \bigr) / \mu^\star$
and $\muS \opluss M(\gamma, \delta) = \muS \opluss (M\gamma, M\delta)$,
where $\sqcup$ is the join of partitions, defined formally before~\eqref{eq:oplus}.
The order $\unlhd$ on pairs of partitions is defined in Definition~\ref{defn:ellSignedDominanceOrder}
by reading the pair as a composition
and then applying the dominance order. 
Example~\ref{ex:GFlip} motivates
the conjugation seen when $|\kappa^\-|$ is~odd.

\begin{theorem}[Signed inner stability]\label{thm:muStable}
Let $\nu$ be a partition of $n$
and let $\muS$ be a skew partition.
Let $\kappa^\-$ and $\kappa^\+$ be partitions. 
If $|\kappa^\-|$ is even then set $\nuSeq{M} = \nu$ for all $M$;
if $|\kappa^\-|$ is odd 
then set $\nuSeq{M} = \nu$ if $M$ is even and $\nuSeq{M} = \nu'$ if $M$ is odd.
Then
\[ \bigl\langle s_\nuSeq{M} \circ s_{\muS \opluss M(\kappa^\-, \kappa^\+)}, 
s_{\lambda \opluss nM(\kappa^\-, \kappa^\+)} \bigr\rangle \]
is constant for $M$ at least the explicit bound in Theorem~\ref{thm:muStableSharp}.
Moreover if~$\eta^\-$ and $\eta^\+$ are partitions 
with $\ell(\eta^\-) \le \ell(\kappa^\-)$ and $|\eta^\-| + |\eta^\+| = |\kappa^\-| + |\kappa^\+|$
and $\swtp{\eta} \notunlhd \swtp{\kappa}$
then 
\[ \bigl\langle s_\nuSeq{M} \circ s_{\muS \opluss M(\kappa^\-, \kappa^\+)}, 
s_{\lambda \opluss nM(\eta^\-, \eta^\+)} \bigr\rangle \]
is zero for $M$ greater than the explicit bound in 
Proposition~\ref{prop:muStableZero}. 
\end{theorem}



Our second main theorem requires the strongly maximal signed weights 
defined in Definition~\ref{defn:stronglyMaximalSignedWeight} and first exemplified
in Example~\ref{ex:stronglyMaximalSemistandardSignedTableauFamilies}.
To orient the reader, we remark that, by 
Lemma~\ref{lemma:maximalAndStronglyMaximalSingletonSemistandardSignedTableauFamilies},
$(\varnothing, \mu)$ and  $(\mu', \varnothing)$ 
are strongly maximal signed weights of shape $\mu$ and size~$1$; their signs are 
$1$ and $(-1)^{|\mu|}$, respectively.
The strongly maximal signed weight relevant to the stability
of \smash{$\bigl\langle s_{\nu + (M)} \circ s_{(m)}, s_{\lambda  \,\sqcup\, (d^M) + M(m-d)} \bigr\rangle$}
for even~$d$ is $\bigl( (1^d), (m-d) \bigr)$.
This signed weight has shape $(m)$, size $1$ and sign $(-1)^d$: see Example~\ref{ex:LawOkitaniSignedWeightsAreStronglyMaximal}(i).

\begin{theorem}[Signed outer stability]\label{thm:nuStable}
Let $\nu$ be a partition 
and let $\muS$ be a skew partition.
Let $\swtp{\kappa}$ be a strongly maximal signed weight of shape $\muS$
and size $R \in \N$.
Set $\nu^{(M)} = \nu + (M^R)$
if $\swtp{\kappa}$ has sign $+1$ and $\nu^{(M)} = \nu \sqcups (R^M)$ if 
$\swtp{\kappa}$ has sign $-1$. Then
\[ \bigl\langle s_{\nu^{(M)}} \circ s_\muS, 
s_{\lambda \opluss M\swtp{\kappa}} \bigr \rangle \]
is constant for $M$ at least the explicit bound in Theorem~\ref{thm:nuStableSharp}.
 Moreover if~$\eta^\-$ and $\eta^\+$ are partitions with $\ell(\eta^\-) \le \ell(\kappa^\-)$ and 
 $\swtp{\kappa} \lhd \swtp{\eta}$ then
\[ \bigl\langle s_{\nu^{(M)}} \circ s_\muS, 
s_{\lambda \opluss M\swtp{\eta}} \bigr \rangle \]
is zero for $M$ greater than the explicit bound in Proposition~\ref{prop:nuStableZero}.
\end{theorem}

The full versions of both theorems give practical sufficient conditions
for the constant multiplicity in their first parts to be zero.
For instance, as we explain after Example~\ref{ex:dBPW},
the final part of Theorem~\ref{thm:muStableSharp} implies that 
unless 
$(\lambda^\-, \lambda^\+) \unlhd n(\mu^\-, \mu^\+)$, the plethysm coefficient
 $\langle s_\nu \circ s_{\mu \,\oplus\, M(\kappa^\-,\kappa^\+)}$, $s_{\lambda  \,\oplus\, nM(\kappa^\-,\kappa^\+)}
\rangle$ is zero for $M$ sufficiently large. Here
$(\lambda^\-, \lambda^\+)$ and $(\mu^\-, \mu^\+)$ are the $\ell(\kappa^\-)$-decompositions
of $\lambda$ and $\mu$, as defined in Definition~\ref{defn:ellDecomposition}, 
and~$\unlhd$ is the \hbox{$\ell(\kappa^\-)$-signed} dominance order in Definition~\ref{defn:ellSignedDominanceOrder}.


%

In 
Corollaries~\ref{cor:muStableSharpPositive},~\ref{cor:muStableSharpPositiveNonSkew}
and~\ref{cor:nuStableSharpPositiveNonSkew},
we give the corollaries of our two main theorems for the special cases where 
$\kappa^\- = \varnothing$ and $\mus= \varnothing$, showing how the 
explicit bounds and conditions in Theorems~\ref{thm:muStableSharp} and~\ref{thm:nuStableSharp} 
simplify. 
Corollary~\ref{cor:nuStableEmpty} 
and Corollary~\ref{cor:nuStableSingleton} are the cases $\nu = \varnothing$
and $R=1$, respectively,
of Theorem~\ref{thm:nuStable} and are also of significant interest in their own right.


\subsection{Strongly maximal signed weights}
An important motivation for strongly
maximal signed weights is
that if $\muS$ is a skew partition
and $\kappa$ is the lexicographically
maximal partition labelling a Schur function summand of $s_{(1^R)} \circ s_\muS$ then 
$(\varnothing, \kappa)$ is a strongly maximal signed weight of a $\muS$-tableau family of size~$R$. 
We plan to prove this result in a separate paper on signed maximal constituents of plethysms.
Many further examples of strongly maximal signed
weights, with full proofs, are given
in \S\ref{subsec:stronglyMaximalSignedWeightExamples}.
In particular we draw attention to Lemma~\ref{lemma:allTableauFamily}.
This  was motivated by (9) in~\cite{BriandOrellanaRosas} by Briand, Orellana
and Rosas, as we discuss in \S\ref{subsec:earlierWork}.

\subsection{Skew partitions}
It is worth noting that the results on 
plethysms $s_\nu \circ s_{\mu/\mu^\star}$ where
$\mu/\mu^\star$ is a skew partition with $\mu^\star \not= \varnothing$ are entirely novel to this paper:
it is a feature of our method 
that this extension from partitions to skew partitions is mostly routine.
See Examples~\ref{ex:needMaximalGood}
and~\ref{ex:nuStableSingletonSkewPartition}
for examples exploiting this generality.
Remark~\ref{remark:nuSkew} explains why the further extension replacing~$\nu$
with a skew partition is a straightforward corollary of our main theorems.

\subsection{A stronger conjecture}
Theorem~\ref{thm:nuStable} was motivated by Proposition 5.3 in \cite{LawOkitani}, which
in turn was motivated by a conjecture of Bessenrodt, Bowman and Paget \cite[Conjecture 1.2]{BessenrodtBowmanPaget},
that the plethysm coefficients $\langle s_{\nu \,\sqcup\, (1^M)} \circ s_{(2)}, s_{\lambda \oplus
M((1), (1))} \rangle$ are non-decreasing with $M$. A proof of this conjecture appears to require
fundamentally different methods to those used in this paper: we believe it is true and that a proof
will be of wide interest. More generally, we make the following conjecture, which includes
the BBP-conjecture as a special case.

\begin{conjecture}\label{conj:generalizedBBP}
The sequences of plethysm coefficients in Theorems~\ref{thm:muStable} and~\ref{thm:nuStable}
are non-decreasing with respect to $M$.
\end{conjecture}

\subsection{Earlier work}\label{subsec:earlierWork} 
We believe the two main theorems in this paper imply
all the stability results on Schur functions published in the literature. These include the stable version of Foulkes' Conjecture. In this subsection we
survey~\cite{BowmanPaget} by Bowman and Paget,
\cite{BriandOrellanaRosas} by Briand, Orellana and Rosas,
\cite{BrionStability} by Brion, \cite{CarreThibon} by Carr{\'e} and Thibon,
\cite{ColmenarejoStability} by Colmenarejo,~\cite{deBoeckPagetWildon}
by de Boeck, Paget and Wildon, 
 \cite{LawOkitaniStability, LawOkitani} by Law and Okitani,
 \cite{Manivel} by Manivel and \cite{Weintraub} by Weintraub. 
 (Except in the case of~\cite{CarreThibon}, we silently change the notation used by these authors to be consistent, as far as possible,
with this paper.)

\subsubsubsection{\names{Bowman--Paget}Theorem A of \normalsize\cite{BowmanPaget}}
This states
that the plethysm coefficients $\langle s_{(n + N)} \circ s_{(m + M)},$  $s_{\lambda + (mN + nM + MN)}\rangle$
are ultimately constant. 
For $M$ varying this is the special case of Theorem~\ref{thm:muStable}
for $\nu = (n+N)$, $\mu = (m)$
taking $\swtp{\kappa} = \bigl(\varnothing, (1)\bigr)$. The bound from Corollary~\ref{cor:muStableSharpPositiveNonSkew}, applied replacing~$\lambda$ 
with $\lambda + (mN)$,
is $M \ge (n+N-1)m - (\lambda_1 + mN) = (n-1)m - \lambda_1$ which improves
on $M \ge |\lambda| = mn$ in \cite{BowmanPaget}.
For~$N$ varying this is the special
case of Theorem~\ref{thm:nuStable} for $\nu = (n)$, $\mu = (m+M)$,
again with the same choice
of $\swtp{\kappa}$; by 
Lemma~\ref{lemma:maximalAndStronglyMaximalSingletonSemistandardSignedTableauFamilies},
$\bigl( \varnothing, (1) \bigr)$ is a strongly $1$-maximal signed weight.
The bound from Corollary~\ref{cor:nuStableSharpPositiveNonSkew} is
in general worse than $N \ge |\lambda|$  in \cite{BowmanPaget}.
A corollary 
(see \cite[Corollary 9.4]{BowmanPaget})
is that the stable version of Foulkes' Conjecture \cite{Foulkes} holds with equality.
We emphasise that the main contribution of~\cite{BowmanPaget} is to prove
the result using Schur--Weyl duality with the partition algebra, thereby giving an
explicit and \emph{clearly positive} formula for the multiplicities. 
This goes significantly
beyond the results obtainable by the general methods in this paper.

\subsubsubsection{\names{Bowman--Paget--Wildon}Theorem A of \normalsize\cite{BowmanPagetWildon}}
This generalises the Bowman--Paget result 
replacing the one-row partition $(n+N)$ with $\nu + (N)$
for an arbitrary partition $\nu$; very similar remarks apply.

\subsubsubsection{\names{Briand--Orellana--Rosas}Result (7) in \normalsize\cite{BriandOrellanaRosas}}
This states that $\langle s_\nu \circ s_\mu, s_\lambda \rangle 
= \langle s_\nu \circ s_{\mu+(M^\ell)}, s_{\lambda + n(M^\ell)} \rangle$
provided that $\ell(\nu) \le \ell$. This is a weaker version of Theorem~1.2 in \cite{deBoeckPagetWildon} by de Boeck, Paget and Wildon, discussed below.

\subsubsubsection{Result (9) in \normalsize\cite{BriandOrellanaRosas}}
This  states that 
\begin{equation}
\label{eq:BOR}
\langle s_{\nu + M(1^R)} \circ s_\mu, s_{\lambda + M(q^\ell)} \rangle
\end{equation} is constant,
where $R$ is the number of semistandard tableaux of shape $\mu$ with entries from $\{1,\ldots, \ell\}$
and $q = R|\mu|/\ell$.
By Theorem~\ref{thm:nuStable} applied with the strongly $\ell$-maximal
signed weight $\bigl(\varnothing, (q^\ell) \bigr)$ (see Lemma~\ref{lemma:allTableauFamily}) the plethysm coefficient is ultimately constant. In fact Theorem~\ref{thm:nuStable}
 implies the more general
result in which $\mu$ is replaced with an arbitrary skew partition.
The relevant strongly maximal semistandard signed tableau family is, as one would expect from the
statement of~(9), all semistandard tableaux of shape~$\mu$ with entries from $\{1,\ldots, \ell\}$.
Corollary~\ref{cor:nuStableSharpPositiveNonSkew}
can be used to give explicit stability bounds for~\eqref{eq:BOR};
Proposition~\ref{prop:BOR} shows that in many cases of interest, 
stability is immediate.

\subsubsubsection{\names{Brion}Theorem \normalsize\cite[\S 2.1]{BrionStability}} This states
that $\langle s_\nu \circ s_{\mu + M\kappa}, s_{\lambda + nM\kappa}\rangle$ is ultimately
constant. There is a bound implicitly defined using the root system of type~A. 
This is the special case of Theorem~\ref{thm:muStable} taking $\swtp{\kappa} = (\varnothing,
\kappa)$. The bound from Theorem~\ref{thm:muStableSharp} is the same.

\subsubsubsection{Theorem \normalsize\cite[\S 3.1]{BrionStability}} This states
that $\langle s_{\nu + (M)} \circ s_\mu, s_{\lambda + M \mu}\rangle$  
is ultimately constant
with an explicit bound. This is the special case of Theorem~\ref{thm:nuStable}
taking $\swtp{\kappa} = (\varnothing, \mu)$; by 
Lemma~\ref{lemma:maximalAndStronglyMaximalSingletonSemistandardSignedTableauFamilies}
this is a strongly $\ell(\mu)$-maximal signed weight.
Brion's bound improves on the bound from Theorem~\ref{thm:nuStableSharp}
or Corollary~\ref{cor:nuStableSharpPositiveNonSkew}
by using orthogonality in the type A root system.

\subsubsubsection{\names{Carr{\'e}--Thibon}\!\!\!\!\!\!\!} 
We first note that in \cite{CarreThibon} $J$ is a partition
and $Jp$ is, in our notation
$(p) \,\sqcup\, J$. If $J$ has first part $a$
and $p \ge a$ then $(p) \sqcup J = \bigl(J \sqcup (a)\bigr) + (p-a)$, and so, by taking
$p$ sufficiently large, we can
interpret $Jp$ as an addition of $(p-a)$ to the partition $J \sqcup (a)$.

\subsubsubsection{Theorem 4.1 in \normalsize\cite{CarreThibon}} The special case 
(see
the remark after the proof in \cite{CarreThibon})
relevant to plethysm coefficients
is equivalent, by the previous notational remark,
to the theorem in \S 2.1 of Brion \cite{BrionStability}, discussed above.

\subsubsubsection{Theorem 4.2 in \normalsize\cite{CarreThibon}} 
It follows very similarly that the special case relevant to plethysm coefficients
is that $\langle s_{\nu + (M)} \circ s_\mu, s_{\lambda + (M|\mu|)} \rangle$
is ultimately constant. When $\mu = (m)$ this is  a special case
of the theorem in \S 3.1 of Brion~\cite{BrionStability} 
discussed above. When $\mu \not= (m)$ we have $\mu \lhd (m)$ and so
the stable multiplicity is zero by 
the `moreover' part of Theorem~\ref{thm:nuStable} applied
with the strongly maximal signed weight $(\varnothing, \mu)$.
(By Lemma~\ref{lemma:maximalAndStronglyMaximalSingletonSemistandardSignedTableauFamilies} this is a strongly $\ell(\mu)$-maximal signed weight.)

\smallskip
\noindent We remark that \cite{CarreThibon} precedes \cite{BrionStability} and the method of
vertex operators used in \cite{CarreThibon} is completely different to Brion's
geometric arguments. 

\subsubsubsection{\names{Colmenarejo}Theorem 1.1 in \normalsize\cite{ColmenarejoStability}} This states  four stability results.
The first is the special case of the second taking, in the notation of \cite{ColmenarejoStability},
$\pi = (1)$. The remaining three are:
\begin{bulletlist}
\item $\langle s_\nu \circ s_{\mu + M\kappa}, s_{\lambda + nM\kappa} \rangle$ is ultimately
constant. As just seen, this is the special case of Theorem~\ref{thm:muStable} taking $\swtp{\kappa} = \bigl(\varnothing,\kappa\bigr)$.
\item $\langle s_{\nu + (M)} \circ s_\mu, s_{\lambda + M\mu} \rangle $ is ultimately constant.
This is the special case of Theorem~\ref{thm:nuStable} taking $\swtp{\kappa} = (\varnothing, \mu)$;
by Lemma~\ref{lemma:maximalAndStronglyMaximalSingletonSemistandardSignedTableauFamilies}
this is a strongly $\ell(\mu)$-maximal signed weight.
\item $\langle s_{\nu + (M)} \circ s_\mu, s_{\lambda + M(|\mu|)} \rangle$ is ultimately constant.
This is the same as Theorem 4.2 in Carr{\'e} and Thibon \cite{CarreThibon} already discussed.


\end{bulletlist}

\subsubsubsection{\names{deBoeck--Paget--Wildon}Theorem~1.1 in \normalsize\cite{deBoeckPagetWildon}}
This states the equality $\langle s_\nu \circ s_{(M)\,\sqcup\, \mu}, s_{(nM)
\sqcup \lambda} \rangle = \langle s_\nu \circ s_\mu, s_\lambda
\rangle$ provided $M$ is at least the greatest part of~$\mu$. Applying the $\omega$-involution
(see \cite[page 21]{MacDonald} 
or \cite[\S 7.6]{StanleyII})  this becomes
\[ \langle s_\nud \circ s_{\mu' + (1^M)}, s_{\lambda' + (1^{nM})} \rangle = \langle s_\nu \circ s_\mu, s_\lambda\rangle, \]
provided $M \ge \ell(\mu')$, where $\nud = \nu$ if $M$ is even and  $\nud = \nu'$ if $M$ is odd. 
Observe that when $M \ge \ell(\mu')$ we have $\mu' + (1^{M+1}) = \bigl(\mu' \sqcup (1) \bigr) + (1^{M})$ and when $nM \ge \ell(\lambda')$ we have $\lambda' + (1^{n(M+1)}) = \bigl(\lambda' \sqcup (1^M)\bigr) + (1^{nM})$. The plethysm coefficient above is therefore
\begin{equation} \label{eq:dBPW}
\langle s_\nud \circ s_{\mu' \sqcups (1^M) + (1^{\ell(\mu')}) },
s_{\lambda' \sqcups (1^{nM}) + (1^{n\ell(\mu')})} \rangle. \end{equation}
That it is ultimately constant now follows from 
Theorem~\ref{thm:muStable}, taking $\kappa^\- = (1)$, and $\kappa^\+ = \varnothing$ and replacing $\mu$ with 
$\mu' + (1^{\ell(\mu')})$ and $\lambda$ with $\lambda' + (1^{n\ell(\mu')})$.
As we show in Example~\ref{ex:dBPW}, 
 the explicit bounds in Theorem~\ref{thm:muStableSharp}
show that in fact the plethysm coefficient, as stated 
in~\eqref{eq:dBPW},
is immediately constant provided $\ell(\lambda') \le n\ell(\mu')$.

\subsubsubsection{Theorem 1.2 in \normalsize\cite{deBoeckPagetWildon}}
This states that
$\langle s_\nu \circ s_{\mu + M(1^r)}, s_{\lambda  + M(n^r)}\rangle$
is constant for~$M$ greater than an explicit bound. By Theorem~\ref{thm:muStable},
applied with $\kappa^\- = \varnothing$ and $\kappa^\+ = (1^r)$, the plethysm
coefficient is ultimately constant. The 
bound from Theorem~\ref{thm:muStableSharp} is the same, as we show at the end
of \S\ref{sec:muStable}.

\subsubsubsection{\names{Law--Okitani}Proposition 5.3 in \normalsize\cite{LawOkitani}} 
This states that 
$\langle s_{\nu \sqcups (1^M)} \circ s_{(2)},$ $s_{\lambda \oplus M((1),(1))} \rangle$ is ultimately
constant. This is the special case of Theorem~\ref{thm:nuStable} taking $\mu = (2)$ and
$\swtp{\kappa} = \bigl( (1), (1) \bigr)$; by Lemma~\ref{lemma:maximalAndStronglyMaximalSingletonSemistandardSignedTableauFamilies}
$\bigl((1),(1)\bigr)$
 is a strongly $1$-maximal signed weight. 

\subsubsubsection{Theorem 1 in \normalsize\cite{LawOkitaniStability}}
This paper followed \cite{LawOkitani}. An equivalent statement of Theorem~1 is that when $d$ is even
\begin{equation}
\label{eq:LOeven} 
\langle s_{\nu + (M)} \circ s_{(m)}, s_{\lambda \opluss M((1^d),(m-d))} \rangle \end{equation}
is ultimately
constant
and when $d$ is odd
\begin{equation}
\label{eq:LOodd} \langle s_{\nu \,\sqcup\, (1^M)} \circ s_{(m)}, s_{\lambda \opluss M((1^d),(m-d))} \rangle \end{equation}
is ultimately constant. 
This result was briefly known between March 2022
and September 2022 as Wildon's Conjecture: it was an important motivation
for Theorem~\ref{thm:nuStable}, 
and is exemplified in~\S\ref{subsec:cutUpsetForLawOkitani}.
No bounds on $M$ were proved in \cite{LawOkitaniStability}.
These results are unified as the special case of
Theorem~\ref{thm:nuStable} taking $\mu = (m)$ and $\swtp{\kappa} = \bigl((1^d), (m-d)\bigr)$;
by Example~\ref{ex:LawOkitaniSignedWeightsAreStronglyMaximal}(i) this is a strongly $1$-maximal
signed weight.

\subsubsubsection{\names{Manivel}Main result and Theorem 4.3.1 in \normalsize\cite{Manivel}}
This is the same result as Theorem~A in~\cite{BowmanPaget} by Bowman and Paget,
already discussed. We emphasise that the proof
in \cite{Manivel} is by novel geometric methods.

\subsubsubsection{\names{Weintraub}Theorem 0.1 in \normalsize\cite{Weintraub}}
This states that
$\langle s_{\nu + (M)} \circ s_\mu, s_{\lambda + M(|\mu|)} \rangle$ is ultimately constant.
It is the same as Theorem 4.2 in \cite{CarreThibon} and the final result of Colmenarejo \cite{ColmenarejoStability} both discussed above; 
we mention that
Weintraub's proof precedes both these papers and the methods used are different from either.

\subsection{Outline}
This paper is split into the five parts indicated in the table of contents.
Each section is written to be read independently as far as possible.

\subsubsection*{Introduction and overview \emph{(\S 1--2)}}
In \S\ref{sec:overview} we give an overview of the proof: we hope this will
persuade the reader that while the proof is lengthy, because of many minor technical difficulties,
the overall concept of finding stable bijections between certain semistandard signed tableaux 
and between certain plethystic semistandard
signed tableaux is quite simple.

\subsubsection*{Preliminaries \emph{(\S \ref{sec:preliminaryDefinitions}--\ref{sec:twisted})}}
In \S\ref{sec:preliminaryDefinitions} we give basic definitions.
In particular we define plethystic semistandard signed tableaux in 
Definition~\ref{defn:plethysticSemistandardSignedTableau}. The reader should be able to skip this section
and then use it as a reference.
In \S\ref{sec:maximalSignedWeights} we define the strongly
maximal signed weights in  Theorem~\ref{thm:nuStable}.
In \S\ref{sec:symmetricFunctions}
we give background
results on plethysms of symmetric functions.
Finally in \S\ref{sec:twisted} we define the $\ell^\-$-twisted dominance
order in Definition~\ref{defn:ellTwistedDominanceOrder} 
and generalize classical results on Kostka numbers to the twisted case.
This is a key definition novel to this paper.

\vspace*{-3pt}

\subsubsection*{Signed Weight Lemma and stable partition systems \emph{(\S\ref{sec:SWL}--\S\ref{sec:stablePartitionSystems})}} 
In~\S\ref{sec:SWL} we prove the critical Signed Weight Lemma (Lemma~\ref{lemma:SWL}): 
this lemma specifies the overall strategy of the proof of the main theorems and  is motivated
by~\S\ref{sec:overview}. 
To apply the lemma we require the idea of a stable partition system, as defined in 
\S\ref{subsec:stablePartitionSystem}.
We give two motivating examples of stable partition systems 
in~\S\ref{subsec:444to822partitionSystem}
and~\S\ref{subsec:stablePartitionSystemsAsIntervals} and then in~\S\ref{sec:stablePartitionSystems} 
we construct
the stable partition systems used to prove Theorems~\ref{thm:muStable} and~\ref{thm:nuStable}. 
Also in \S\ref{subsec:cutUpsetForLawOkitani}
we show some of the main ideas in the proofs
of Theorems~\ref{thm:muStable} and Theorem~\ref{thm:nuStable} by examples using the three
key results proved by the end of \S\ref{sec:stablePartitionSystems}, namely
\begin{itemize}
\item  Proposition~\ref{prop:plethysticSignedKostkaNumbers}
on plethystic signed Kostka numbers,
stating that $\langle s_\nu \circ s_\muS, e_{\alpha^\-}h_{\alpha^\+}\rangle = |\mPSSYT(\nu,\muS)_{(\alpha^\-,\alpha^\+)}|$;
\item Lemma~\ref{lemma:SWL}, the Signed Weight Lemma;
\item Corollary~\ref{cor:signedIntervalStable}, that intervals for the $\ell^\-$-twisted 
dominance order define
stable partition systems.
\end{itemize}

\vspace*{-3pt}

\subsubsection*{Proof of Theorem~\ref{thm:muStable} \emph{(\S\ref{sec:twistedWeightBoundInner}--\S\ref{sec:muStableSharpPositive})}}
In~\S\ref{sec:twistedWeightBoundInner} we prove Proposition~\ref{prop:twistedWeightBoundInner} giving an upper bound
in the $\ell^\-$-twisted
dominance order on the constituents of an arbitrary plethysm $s_\nu \circ s_\muS$.
This is the final technical preliminary needed to apply Corollary~\ref{cor:signedIntervalStable},
and hence the Signed Weight Lemma (Lemma~\ref{lemma:SWL}), to 
prove Theorem~\ref{thm:muStable} in~\S\ref{sec:muStable}.
We give the important special case of this theorem when all tableaux
have only positive entries in \S\ref{sec:muStableSharpPositive}.

\subsubsection*{Proof of Theorem~\ref{thm:nuStable} 
\emph{(\S\ref{sec:twistedWeightBoundForStronglyMaximalWeight}--\S\ref{sec:nuStableSharpApplications})}}
In \S\ref{sec:twistedWeightBoundForStronglyMaximalWeight} we prove the analogous
upper bound in Corollary~\ref{cor:signedWeightBoundForStronglyMaximalSignedWeight}
on the constituents in the plethysms in Theorem~\ref{thm:nuStable},
and in~\S\ref{sec:nuStable} we prove Theorem~\ref{thm:nuStable}.
In \S\ref{sec:nuStableSharpApplications} we give
many applications of this theorem, including its important
special case when all tableaux have partition shape
and only positive entries.

\subsection{Computer software} {\sc Magma} 
\cite{Magma} code that can be used to verify all of our examples and compute
with the $\ell^\-$-twisted dominance order in Definition~\ref{defn:ellTwistedDominanceOrder} may be
downloaded as part of the arXiv submission of this paper. Example~\ref{ex:17family} is most easily checked
using the second author's Haskell~\cite{Haskell98} code \cite{PZYT}.
Computer algebra is not essential to any of our proofs or  examples.

\section{Overview of proof}\label{sec:overview}
The original Law--Okitani stability result \cite[Proposition 5.3]{LawOkitani},
later generalized in the main theorem of~\cite{LawOkitaniStability},
is that, for any partition $\nu$
and any partition $\lambda$ of $2|\nu|$, the sequence of plethysm coefficients
\begin{equation}\label{eq:LawOkitaniOriginal} 
\langle s_{\nu \sqcups (1^M)} \circ s_{(2)}, s_{\lambda \sqcups (1^M) + (M) } \rangle \end{equation}
is ultimately constant. This is the special case of Theorem~\ref{thm:nuStable} 
for the strongly maximal signed weight $\bigl( (1), (1)\bigr)$ of shape $(2)$, size $1$
and sign $-1$. (This weight is strongly maximal
by Example~\ref{ex:LawOkitaniSignedWeightsAreStronglyMaximal}(i); see \S\ref{subsec:needStronglyMaximal} for motivation for strongly maximal weights.)
Here we use the special case $\nu = (3,1)$ and $\lambda = (6,2)$ of~\eqref{eq:LawOkitaniOriginal} 
that 
\[ \langle s_{(3,1,1^M)} \circ s_{(2)}, s_{(6+M,2,1^M)} \rangle \]
is ultimately constant
to
sketch the overall strategy of the proofs of the two main results in this paper,
indicating why certain steps cannot, we believe, be simplified.
In particular, in \S\ref{subsec:plethysticBijectionOuter} 
we give the bijection on plethystic semistandard signed tableaux (see
Definition~\ref{defn:plethysticSemistandardSignedTableau}) used 
to prove this stability result; it is generalized in Theorem~\ref{thm:nuStableSharp}.
 
\subsection{Elementary-homogeneous products}\label{subsec:elementaryHomogeneousProducts}
The first key idea is \emph{to approximate Schur functions as products of elementary and homogeneous
symmetric functions}. For~\eqref{eq:LawOkitaniOriginal}, we
set $\alpha = \lambda - (1^{\ell(\lambda)})$ and decompose the partition $\lambda \sqcup (1^M) + (M) $
as $(1^{\ell(\lambda) + M}) + \bigl( \alpha \,+\, (M) \bigr)$. It then follows
from Young's rule (see the start of \S\ref{sec:symmetricFunctions})
 that $s_{\lambda \sqcups (1^M) + (M) }$ is a summand of
$e_{(\ell(\lambda)+M)} h_{\alpha + (M)}$. In our specific example, $\lambda = (6,2)$, $\alpha = (5,1)$, and so, when $M=0$, 
the product is
\begin{align}\label{eq:e2h51} 
e_{(2)}h_{(5,1)} &= s_{(6,2)} + s_{(7,1)} + 2s_{(6,1,1)} + s_{(5,2,1)} + s_{(5,1,1,1)}. \\
\intertext{As expected, this has $s_{(6,2)}$ as a summand, but also, of course, some Schur functions labelled by extra partitions.
For general $M \in \N$,~\eqref{eq:e2h51} becomes}
 \label{eq:e2h51M} 
e_{(2+M)}h_{(5+M,1)} &= s_{(6+M,2,1^M)} + s_{(7+M,1,1^M)} + 2s_{(6+M,1,1,1^M)}\nonumber \\
& \hspace*{1in} + s_{(5+M,2,1,1^M)}  + s_{(5+M,1,1,1,1^M)}. \end{align}
Note that the summands in~\eqref{eq:e2h51M}
are in bijection with the summands in~\eqref{eq:e2h51} and the coefficients are independent of $M$.
This points to a potential inductive proof, provided all the partitions in~\eqref{eq:e2h51}
are `inductively smaller' than $(6,2)$ in some sense. However, we 
must consider not just the partitions appearing in~\eqref{eq:e2h51}, but
the new partitions that arise when we apply this `approximation' strategy to them.
For instance, $s_{(5,1,1,1)}$ appears in~\eqref{eq:e2h51}
and $(5,1,1,1) = (1,1,1,1) + (4)$, so we must also consider the product 
\begin{equation}\label{eq:e4h4} e_{(4)}h_{(4)} = s_{(5,1,1,1)} + s_{(4,1,1,1,1)},\end{equation} 
where we see $s_{(4,1,1,1,1)}$ for the first time.
It therefore appears we need an order in which $(6,2)$, $(5,1,1,1)$ and $(4,1,1,1,1)$ 
form a chain. This is a property of the $1$-twisted dominance order,
defined by taking $\ell^\-=1$ 
in Definition~\ref{defn:ellTwistedDominanceOrder}.
The up-set of $(6,2)$ (as defined in \S\ref{subsec:upsetsAndTwistedIntervals}), 
in this order is
\begin{equation}\label{eq:62upset} \begin{split}
 (6,2)^\unlhddotS = &\{(6,2), (5,2,1),  (4,2,1,1),  (3,2,1^3), (2,2,1^4) \} \\
  & \quad \cup \{(7,1), (6,1,1), (5,1,1,1), (4,1^4), (3,1^5), (2,1^6), (1^8) \}. \end{split} \end{equation}
Note that both $(5,1,1,1)$ and $(4,1,1,1,1)$ appear
and that each of the two subsets in the decomposition above is a chain, increasing when read left to right. (Thus `inductively smaller' means `bigger in the $1$-twisted
dominance order'.)
See Figure~\ref{fig:62upset} for the Hasse diagram of the order. 
By Lemma~\ref{lemma:twistedKostkaMatrix}, for every $\sigma \in (6,2)^\unlhddotS$, the summands of 
\smash{$e_{(\ell(\sigma))}h_{\sigma - (1^{\ell(\sigma)})}$} are in $\sigma^\unlhddotS$ and
so in $(6,2)^\unlhddotS$;
for example, 
this is clear for $\sigma = (6,2)$ and $\sigma = (5,1,1,1)$ 
from the products $e_{(2)}h_{(5,1)}$ and $e_{(4)}h_{(4)}$ given 
in~\eqref{eq:e2h51} and~\eqref{eq:e4h4} above.

\begin{remark}
Many other strategies for `approximating' $s_{(6+M,2,1^M)}$ by a product of more tractable
symmetric functions, for example
any strategy using homogeneous symmetric functions alone,
would fail at the point of~\eqref{eq:e2h51M} by giving an expansion with a growing
number of Schur functions, or with non-constant coefficients.
\end{remark}

\subsection{Rough inductive hypothesis}\label{subsec:roughInductiveHypothesis}
We now suppose inductively --- but see \S\ref{subsec:cutUpsetsAreStable} below for a difficulty 
here --- 
that $\langle s_{(3,1,1^M)} \circ s_{(2)}, s_{\sigma  \sqcups (1^M) + (M)} \rangle$ is ultimately constant
for each of the partitions \smash{$\sigma \in (6,2)^\unlhddotS$} except perhaps for $(6,2)$.
 Since stability is known inductively for each summand
of $e_{(2+M)}h_{(5+M,1)}$,
except $s_{(6+M,2,1^M)}$,
to show that the plethysm coefficients
$\langle s_{(3,1,1^M)} \circ s_{(2)}, s_{(6+M,2,1^M)}\rangle$ are ultimately
constant, it suffices to show that
\begin{equation}\label{eq:62eh} 
\langle s_{(3,1,1^M)} \circ s_{(2)}, e_{(2+M)}h_{(5+M,1)} \rangle\end{equation}
is ultimately constant.

\subsection{Plethystic semistandard signed tableaux}\label{subsec:plethysticSemistandardSignedTableaux}
To show that the pleth\opthyphen{}ysm coefficients
$\langle s_{(3,1,1^M)} \circ s_{(2)}, e_{(2+M)}h_{(5+M,1)} \rangle$ in~\eqref{eq:62eh}
are ultimately constant
we need the second key idea: \emph{$s_\nu \circ s_\muS$ is 
the generating function enumerating the
plethystic semistandard signed
tableaux defined in Definition~\ref{defn:plethysticSemistandardSignedTableau}.}
Moreover, by Proposition~\ref{prop:plethysticSignedKostkaNumbers},
 the inner product of $s_\nu \circ s_\muS$ with $e_{\pi^\-}h_{\pi^\+}$ is
the number of plethystic semistandard signed tableaux of signed weight $(\pi^\-,\pi^\+)$,
in the sense of Definition~\ref{defn:signedWeightPlethystic}.
For instance,
\begin{equation}\langle s_{(3,1)} \circ s_{(2)}, e_{(2)}h_{(5,1)} \rangle = 
\bigl|\PSSYT\bigl( (3,1), (2) \bigr)_{((2),(5,1))}\bigr|\label{eq:s31weight}\end{equation}
is the number of plethystic semistandard signed tableaux of shape 
$\bigl((3,1), (2) \bigr)$ and signed weight $\bigl( (2), (5,1)\bigr)$. 
The three such plethystic semistandard signed tableaux are:
\smallskip
\[ \pyoung{1.2cm}{0.7cm}{ {{\young(\oM1), \young(\oM2), \young(11)}, {\young(11)}} }  \spy{18pt}{,}
  \pyoung{1.2cm}{0.7cm}{ {{\young(\oM1), \young(11), \young(11)}, {\young(\oM2)}} }  \spy{18pt}{,}
  \pyoung{1.2cm}{0.7cm}{ {{\young(\oM1), \young(11), \young(12)}, {\young(\oM1)}} }  \spy{18pt}{,}
 \]
where $\mbf{1}$ stands for the negative entry $-1$.
More generally,
\[
\langle s_{(3,1,1^M)} \circ s_{(2)}, e_{(2+M)}h_{(5+M,1)} \rangle = 
\bigl|\PSSYT\bigl (3,1,1^M), (2) \bigr)_{((2+M),(5+M,1)}\bigr|. \]
Thus $\langle s_{(3,1,1^M)} \circ s_{(2)}, e_{(2+M)}h_{(5+M,1)} \rangle$ is ultimately constant if and only if
\[ \bigl| \PSSYT\bigl( (3,1,1^M), (2) \bigr)_{((2+M),(5+M,1))} \bigr| \] 
is ultimately constant. Hence
proving the stability of the plethysm coefficient
 $\langle s_{(3,1,1^M)} \circ s_{(2)}, s_{(6+M,2,1^M)} \rangle$
reduces to the combinatorial
problem of enumerating certain plethystic semistandard signed tableaux. We solve this problem
in  \S\ref{subsec:plethysticBijectionOuter} below
by exhibiting explicit bijections between the sets 
\smash{$\mathrm{PSSYT}\bigl((3,1,1^M), (2)\bigr)_{((2+M),(5+M,1))}$}
for $M$ sufficiently large. (The proof of Theorem~\ref{thm:nuStableSharp}
has the general argument.)
In our specific example, $M=0$ is already sufficiently large and the constant multiplicity is~$3$.

\subsection{Why the inductive step as described fails in general}
This is an honest sketch of the proof, except for one problem.
We saw in \S\ref{subsec:elementaryHomogeneousProducts} 
that we have to consider all the partitions in the up-set $(6+M,2,1^M)^\unlhddotS$, not just
those in the support (see Definition~\ref{defn:supp})  of $e_{(2+M)}h_{(5+M,1)}$.
If all these partitions were of the form $\sigma \sqcup (1^M) + (M) $
for $\sigma \in (6,2)^\unlhddotS$, then nothing new would be needed, and the inductive step would go through.
The problem is that
this is not the case: for instance
\[ (7,2,1)^\unlhddotS = \{ \sigma \sqcup (1) + (1) : \sigma \in (6,2)^\unlhddotS \} \cup \{ (2,2,1^6), (1^{10}) \}\]
where the union is disjoint,
and there is no way to deduce from the inductive assumptions for partitions in $(6,2)^\unlhddotS$
that $\langle s_{(3,1,1^{M+1})} \circ s_{(2)}, s_{(2+M,2,1^6,1^M)} \rangle$
is ultimately constant, as required in the inductive step.

\subsection{Cut up-sets}\label{subsec:cutUpsetsAreStable}
We get around this obstacle to the inductive strategy as presented in 
\S\ref{subsec:roughInductiveHypothesis}
by the third key idea: \emph{we do not need
to consider every partition appearing in the up-set $(6+M,2,1^M)^\unlhddotS$,
only those that appear in the plethysm
\smash{$s_{(3,1,1^M)} \circ s_{(2)}$}}. It follows from the Littlewood--Richardson rule
that  only partitions with at most $4+M$
parts appear in this plethysm, so rather than work with $(6+M,2,1^M)^\unlhddotS$,
we can instead take the `cut' up-set
\begin{align*} \PSeq{M} &= \bigl\{ \sigma \in \Par(8+2M) : \sigma \unrhddot (6+M,2,1^M),\, \ell(\sigma) \le 4+M \bigr\}. 
\intertext{Thus $\PSeq{0} = \{ (6,2), (5,2,1), (4,2,1,1), (7,1), (6,1,1), (5,1,1,1) \}$ and
in general we have}
 \PSeq{M} &= \bigl\{ (6+M,2,1^M), (5+M,2,1,1^M), (4+M,2,1,1,1^M),  \\
  &\qquad\qquad (7+M,1,1^M), (6+M,1,1,1^M), (5+M,1,1,1,1^M) \bigr\}\end{align*}
for each $M \in \N_0$. When $M=1$ the `cut' removes the two partitions $(2,2,1^6)$ and $(1^{10})$
blocking the inductive argument, and in general,
every partition in $\PSeq{M}$ is of the form $\sigma \sqcup (1^M) + (M) $
for $\sigma \in \PSeq{0}$. Note however that $\PSeq{0}$ 
contains $(4,2,1,1)$ and so $\PSeq{0}$ is not contained in the support
of
$e_{(2)}h_{(5,1)}$. Thus we must still consider more partitions than are immediately
required by~\eqref{eq:e2h51}.  


\subsection{Signed Weight Lemma}\label{subsec:overviewSWL}
As we show by proving the Signed Weight Lemma (Lemma~\ref{lemma:SWL}),
after this refinement, the inductive step goes through.
Because
of our use of this critical lemma,
our proofs are not explicitly inductive. Instead, each proof
specifies the relevant way to apply the Signed Weight Lemma, and  verifies its
hypotheses: the most technical part of the argument is captured in the notion of a stable
partition system, as defined in Definition~\ref{defn:stablePartitionSystem}.

\subsection{Twisted dominance order}
The definition of a stable partition system is deliberately quite general. This generality is needed
for other applications of the Signed Weight Lemma (Lemma~\ref{lemma:SWL}) beyond the scope of this paper, and, in any case,
seems to us to be the clearest way to present the proof.
 In practice, the stable partition systems we use
are certain families of intervals for the twisted dominance order on partitions
(see Definition~\ref{defn:ellTwistedDominanceOrder}). For instance $\PSeq{M}$ above
is the interval $[(6,2) \oplus M(1,1), (5,1,1,1) \oplus M(1,1)]_\unlhddotS$ for the $1$-twisted dominance order.
Definition~\ref{defn:ellTwistedDominanceOrder}
is a key definition in this paper; more broadly, the attractive interplay between
the $\ell$-decomposition $\decs{\pi^\-}{\pi^\+}$ defined in Definition~\ref{defn:ellDecomposition}, the partition $\pi$, and the symmetric function
 $e_{\pi^\-}h_{\pi^\+}$ is seen in many results and proofs below, notably
 Lemma~\ref{lemma:twistedKostkaMatrix} and Proposition~\ref{prop:twistedWeightBoundInner}.

\subsection{Counting plethystic tableaux}
The plethystic semistandard signed tableaux 
in \smash{$\PSSYT\bigl( (3), (2) \bigr)_{((1),(4,1))}$} are
\vspace*{2pt}
\begin{equation}\label{eq:PSSYTsEx} \pyoung{1.2cm}{0.7cm}{ {{\young(\oM1), \young(11), \young(12)}}  }  \spy{3pt}{,}
   \pyoung{1.2cm}{0.7cm}{ {{\young(\oM2), \young(11), \young(11)}}  }
\end{equation}
By
Proposition~\ref{prop:plethysticSignedKostkaNumbers} we have
\[ |\PSSYT\bigl( (3), (2) \bigr)_{((1),(4,1))}|= \langle
s_3 \circ s_2 , e_1 h_{(4,1)} \rangle.\]
This is the second key idea  seen in~\S\ref{subsec:plethysticSemistandardSignedTableaux},
showing the tight connection between plethystic semistandard signed
tableaux and symmetric functions.
For instance, since by Young's rule
$e_1 h_{(4,1)} = s_{(6)} + 2s_{(5,1)}
+ s_{(4,2)} + s_{(4,1,1)}$,
we may count these plethystic semistandard signed tableaux algebraically 
by evaluating the inner product displayed above using the
known plethysm $s_3 \circ s_2 = s_{(6)} + s_{(4,2)} + s_{(2,2,2)}$ 
(see for instance \cite[I.8 Example 6(a)]{MacDonald}).

 \subsection{Bijections between plethystic tableaux}\label{subsec:plethysticBijectionOuter}
In \S\ref{subsec:plethysticSemistandardSignedTableaux} we 
claimed that $| \PSSYT\bigl( (3,1,1^M), (2) \bigr)_{((2+M),(5+M,1))} | = 3$
for all $M \in \N_0$ and showed the three tableaux when $M = 0$.
Two of these plethystic semistandard signed tableaux are obtained
from the tableaux in~\eqref{eq:PSSYTsEx}
by inserting $\young(\oM1)$ as a new entry in position $(1,1)$, moving the existing entry down 
to row~$2$. But the plethystic semistandard signed tableau shown in the margin
is not obtained in this way, because by the row semistandard condition
in Definition~\ref{defn:plethysticSemistandardSignedTableau}, $\young(11)$
cannot appear left of $\young(\oM2)$\,.  Generally this insertion map defines
an injection between the sets for $M$ and $M+1$ and,
by the part of the proof of Theorem~\ref{thm:nuStableSharp} dealing with condition (ii)
in  the Signed Weight Lemma (Lemma~\ref{lemma:SWL}),
 it is surjective for $M \ge 1$,\marginpar{\scalebox{0.9}{$\pyoung{1.2cm}{0.7cm}{ {{\young(\oM1), \young(\oM2), \young(11)}, {\young(11)}} }$}}
proving the claimed stability result. For a larger example see Example~\ref{ex:omegaBound411}.




\addtocontents{toc}{\smallskip}
\addtocontents{toc}{\textbf{Preliminaries}}

\section{Partitions, tableaux and plethystic tableaux}\label{sec:preliminaryDefinitions}
In this section we give numbered definitions for the key terms novel to this paper. Other than these, we believe
our notation is standard; 
we hope that the reader will be able to skim this section and 
then treat it as a reference. For
essential preliminaries on symmetric functions see  the
start of~\S\ref{sec:symmetricFunctions}.

\subsubsection*{Weights, compositions and partitions}
A \emph{weight}, also sometimes called a \emph{composition}, is an infinite 
sequence of non-negative integers with finite sum, called its \emph{size}. 
The \emph{length} of a 
\emph{weight} $\alpha$, denoted~$\ell(\alpha)$, is the maximum $\ell$ such that $\alpha_\ell \not=0$.
(We set $\ell(\varnothing) = 0$.) Dually, we often write $a(\alpha)$ for~$\alpha_1$.
A weight is a \emph{partition} if $\alpha$ is non-increasing.
The terms in a weight or partition are called \emph{parts}. 
We always omit the infinite tail of zero parts when writing weights or partitions. 
Let~$\W$ be the set of weights, 
let $\Par$ be the set of partitions, and let $\Par(n)$ be the set of partitions of~$n$.

\subsubsection*{Young diagrams and skew partitions}
We write $[\lambda]$ for the \emph{Young diagram} of a partition $\lambda$, defined by
\[ [\lambda] = \bigl\{ (i, j) : 1 \le i \le \ell(\lambda), 1 \le j \le \lambda_i \bigr\}. \]
The elements of $[\lambda]$ are called \emph{boxes}.
A \emph{skew partition} is a pair of partitions, denoted $\lambda/\lambda^\star$,
such that 
$[\lambda^\star] \subseteq [\lambda]$. The \emph{size} of a skew partition $\lambda/\lambda^\star$,
denoted $|\lambda/\lambda^\star|$,
is $|\lambda| - |\lambda^\star|$. 
We extend the definition of Young diagrams
to skew partitions in the obvious
way, by setting $[\lambda/\lambda^\star] = [\lambda] \backslash
[\lambda^\star]$. We draw Young diagrams in the `English' convention with box $(1,1)$
in the top-left of the page. 
The \emph{conjugate partition} to $\lambda$, denoted $\lambda'$, is the unique partition with Young diagram
$\big\{ (j, i) : (i, j) \in [\lambda] \bigr\}$. For example $(3,2)' = (2,2,1)$.
The conjugate of a skew partition $\mu/\mus$ is $\mu'/\mu_\star'$.

\subsubsection*{Operations on partitions}
The sum and difference
of partitions is defined componentwise by $(\alpha + \beta)_i = \alpha_i + \beta_i$,
and $(\alpha - \beta)_i = \alpha_i - \beta_i$ 
when~$\beta$ is a subpartition of $\alpha$. 
Let $\alpha \sqcup \beta$ be the partition whose multiset of non-zero parts is the disjoint union
of the multisets of non-zero parts of 
$\alpha$ and $\beta$; equivalently $(\alpha \sqcup \beta)' = \alpha' + \beta'$.
We say that $\alpha \sqcup \beta$ is the \emph{join} of $\alpha$ and $\beta$.
For instance,
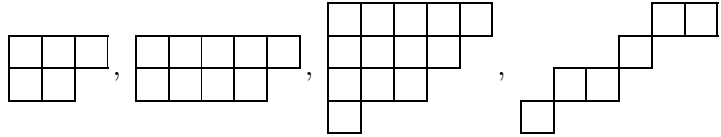
\[ \yng(3,2)\,,\ \yng(3,3,2,1)\, , \ \yng(5,5,2,1)\, , \ \young(::::\ \ ,:::\ ,:\ \ ,\ ) \]
are the Young diagrams of $(3,2)$, $(3,2) \sqcup (3,1)$, $(3,2) \sqcup (3,1) + (2,2)
$ 
and $(6,4,3,1) \backslash (4,3,1)$, respectively.
As already seen in the statements
of the two main theorems,
we define, for partitions $\gamma$ and $\delta$,
\begin{equation}\label{eq:oplus}
\muS \oplus (\gamma,\delta) = \bigl( (\mu \sqcup \gamma')  +\delta\bigr) / \mus
\end{equation}
with the special case that $\lambda \oplus (\gamma,\delta)
= (\lambda \sqcup \gamma')  + \delta $. 
Note the conjugation of~$\gamma$. (In examples we often omit the parentheses,
writing instead $\lambda \,\sqcup\, \gamma'+ \delta $.) We suggest `$\oplus$' be read as `adjoin'.
For example, $(3,2) \oplus \bigl((2,1,1), (2,2)\bigr) = \bigl((3,2) \sqcup (3,1)\bigr)
+ (2,2) = (3, 3,2,1) + (2,2) = (5,5,2,1)$ was seen above,
and $(3,2) \opluss 2\bigl((2,1,1), (2,2)\bigr) = (7,7,3,2,1,1)$.
Note this agrees with
$(3,2) \oplus \bigl((2,1,1), (2,2) \bigr) \oplus \bigl((2,1,1), (2,2) \bigr)$.
There is one annoyingly technical point, 
seen by comparing
$\varnothing \,\oplus\, \bigl((1),(2) \bigr) = \varnothing \,\sqcup\, (1) \,+\,(2) =(1) \,+\,(2) = (3)$
with $\varnothing + (2) \sqcup (1) = (2) \sqcup (1) = (2,1)$, which we address in the following
definition. 

\begin{definition}\label{defn:large}
Let $\muS$ be a skew partition.
Given $\ell^\-$ and $\ell^\+ \in \N_0$, we say that $\muS$ is
$(\ell^\-, \ell^\+)$-\emph{large} if either $\ell^\- = 0$ or $\ell^\+ = 0$ or
$\mu_{\ell^\+} \ge \ell^\-$.
\end{definition}

It is deliberate that $\mus$ does not enter in the body of Definition~\ref{defn:large}.
Equivalently, $\muS$ is $(\ell^\-, \ell^\+)$-\emph{large} if $(\ell^\+,\ell^\-)$ either
has a zero coordinate or is  a box of $[\mu]$: see Figure~\ref{fig:large} for an example.

\begin{figure}[h!t]
\begin{center}
\vspace*{-15pt}
\begin{tikzpicture}[x=0.6cm, y=-0.6cm]
\fill[pattern = north west lines] (3,2)--(4,2)--(4,3)--(3,3)--(3,2);

\fill[color=lightgray] (0,0)--(3,0)--(3,1)--(1,1)--(1,2)--(0,2)--(0,0);

\draw (0,0)--(6,0); \draw (0,1)--(6,1); \draw (0,2)--(5,2); \draw (0,3)--(5,3);
\draw(0,4)--(2,4);
\draw (0,0)--(0,4); \draw (1,0)--(1,4); \draw (2,0)--(2,4); \draw (3,0)--(3,3);
\draw (4,0)--(4,3); \draw (5,0)--(5,3); \draw (6,0)--(6,1);

\node[left] at (-0.125,2.5) {$\scriptstyle \ell^{\scriptstyle +}=\hskip1.5pt3$};
\node[above] at (3.5,-0.125) {$\scriptstyle \ell^{\scriptstyle -}=\hskip1.5pt4$};
\end{tikzpicture}
\end{center}
\caption{The skew partition $(6,5,5,2)/(3,1)$ above is $(4,3)$-large
in the sense of Definition~\ref{defn:large}
because $(3,4) \in [(6,5,5,2)]$. It is also $(5,3)$-large,
but not $(5,4)$-large.}\label{fig:large}
\end{figure}

By Lemma~\ref{lemma:adjoinToLarge},
joining $\gamma'$ and adding~$\delta$ are commuting operations when applied
to a partition that is $\bigl(\ell(\gamma), \ell(\delta)\bigr)$-large.
For instance $(3,2)$  \emph{is not} $(3,2)$-large, and
$(3,2) + (2,2) \sqcup (3,1) = (5,4,3,1)$ is not equal to $(3,2) \sqcup (3,1) + (2,2) = (5,5,2,1)$. But $(3,3)$ \emph{is} $(3,2)$-large, and applied to this partition the operations
commute.
See Remark~\ref{remark:ellDecompositionLarge} for why
we choose to join first.

\begin{remark}\label{remark:becomesLarge}
Fix partitions $\kappa^\-$ and $\kappa^\+$ and let $\ell^\- = \ell(\kappa^\-)$,
$\ell^\+ = \ell(\kappa^\+)$. 
For any partition $\mu$, the adjoining
map $\mu \mapsto \mu \opluss (\kappa^\-, \kappa^\+)$
increases~$\mu_{\ell^\+}$ by at least $\kappa^\+_{\ell^\+}$
and increases~$\mu_{\ell^\-}'$ by at least $\kappa^\-_{\ell^\-}$.
(The example $\varnothing \oplus \bigl( (1), (2) \bigr)
= (3)$ above shows that `at least' cannot be replaced with `exactly'.)
If $\kappa^\- = \varnothing$ then $\mu$ is already 
$(\ell^\-,\ell^\+)$-large;
otherwise $\mu$  becomes 
$(\ell^\-,\ell^\+)$-large after at most
$\lceil \ell^\+/\kappa^\-_{\ell^\-}\rceil$ 
adjoinings. 
The dual result  holds for $\kappa^\+$, now with 
$\lceil \ell^\-/\kappa^\+_{\ell^\+}\rceil$  adjoinings.
Thus there exists~$\KK$ such that 
$\muS \,\oplus\, \KK\swtp{\kappa}$ is $(\ell^\-,\ell^\+)$-large.
For later use, for instance in the context of Lemma~\ref{lemma:twistedDownsetLarge},
we remark that one further application of the adjoining map gives an
$(\ell^\-+1,\ell^\+)$-large partition.
Setting $\sigmaS = \muS \oplus \KK\swtp{\kappa}$, it follows
from Lemma~\ref{lemma:adjoinToLarge} that for all $M \ge \KK$,
\begin{align*}
\qquad \muS \oplus M\swtp{\kappa} &= \sigmaS \oplus (M-\KK)\swtp{\kappa}  \\
&= \sigmaS \oplus \swtp{\kappa} \opluss \stackrel{M-\KK}{\cdots} \opluss \swtp{\kappa} \end{align*}
\end{remark}

By this remark, 
there is no loss of generality in
assuming in our main theorems that all the partitions involved are large,
and so, in practice, there is rarely any need to worry about whether to add $\kappa^\+$ 
or join $\kappa^{\-\prime}$ 
first in the map
$\lambda \mapsto \lambda \oplus \swtp{\kappa}$. (By~\eqref{eq:oplus}, joining
first is our definition.)
In the important special
case where $\kappa^\- = \varnothing$, this technicality does not arise.

\subsubsection*{Dominance order}
We partially order partitions of the same size
by the \emph{dominance order},
defined as usual by $\kappa \unlhd \lambda$ if and only if $\kappa_1 +\cdots + \kappa_i \le
\lambda_1 + \cdots + \lambda_i$ for all $i$. We  use the obvious extension of the
dominance order to compositions and to 
partitions of different size: in the latter case,
to indicate that the partitions may have
different sizes, we write $\unLHD\!$ rather~than~$\unlhd$.

\subsubsection*{Signed tableaux and signed weights}
We work throughout with tableaux having entries from $\Z \backslash \{0\}$.

\begin{definition}[Signed tableau]\label{defn:signedTableau}
Let $\muS$ be a skew partition.
A \emph{signed tableau} of \emph{shape} $\muS$ is 
a function $t : [\muS] \rightarrow \Z \backslash \{0\}$. If $t(i,j) = x$ then we 
say that $t$ has \emph{entry} $x$ in box $(i,j)$. 
\end{definition}

We write $\YT(\muS)$ for the set of signed tableaux of shape $\muS$.
Recall that $\W$ is the set of weights.

\begin{definition}[Signed weight]\label{defn:signedWeight}
A signed weight is an element of $\W \times \W$. 
\end{definition}

\begin{definition}[Signed weight of a signed tableau]\label{defn:signedWeightTableau}
The \emph{signed weight} of a signed tableau~$t$ 
is the pair $\swtp{\alpha} \in \W \times \W$
where, for each $i \in \N$, 
$\alpha^\-_i$ is the 
number of entries of $t$ equal to $-i$, and
$\alpha^\+_i$ is the number of entries of~$t$ equal to $i$.  
\end{definition}

If a tableau $t$ has only positive entries then its signed
weight is $(\varnothing,\alpha)$ for some weight $\alpha$, and
in this case we say, as usual, that $\alpha$ is the \emph{weight} of $t$
and write $\alpha = \wt(t)$.

\begin{definition}[Sign of a signed tableau]\label{defn:signTableau}
The \emph{sign} of a signed tableau $t$, denoted $\sgn(t)$, 
is $-1$ if $t$ has an odd number of negative
entries and $+1$ if~$t$ has an even number of negative entries.
A signed tableau is \emph{negative} if its sign is $-1$ and \emph{positive}
if its sign is $+1$.
\end{definition}

Equivalently, the sign of a signed tableau of weight $\swtp{\alpha}$ is
$(-1)^{|\alpha^\-|}$.

\subsubsection*{Semistandard signed tableaux}
\marginpar{\qquad\young(::::\ \ ,:::\ ,:\ \ ,\ )}
\marginpar{\ \qquad\young(:::\ ,::\ ,::\ ,:\ ,\ ,\ )}
Recall that a \emph{horizontal strip} is a skew partition whose Young diagram
has at most one box in each column and a \emph{vertical strip} is a skew partition
whose Young diagram has at most one box in each row.
For instance the skew partition $(6,4,3,1) / (4,3,1)$ seen earlier in this section is a horizontal
strip but not a vertical strip, and its conjugate $(4,3,3,2,1,1)/(3,2,2,1)$ is a vertical strip 
but not a horizontal strip. (The diagrams are shown in the margin.)

\begin{definition}[Semistandard signed tableau]\label{defn:semistandardSignedTableau}
Let $t$ be a signed tableau. We say $t$
is \emph{semistandard} if
\begin{defnlist}
\item equal positive entries of $T$
lie in horizontal strips
\item equal negative entries of $T$ lie in vertical strips, 
\item all entries are weakly increasing 
when rows are read left-to-right and columns are read
top-to-bottom 
with respect to  the total order on $\Z \backslash \{0\}$ defined by

\vspace*{-15pt}
\[\qquad -1 \prec -2 \prec \ldots \prec 1 \prec 2 \ldots. \]
\end{defnlist}
\end{definition}

Note that $-1$ is the least element in this order. 
We write $\sSSYT(\muS)$ for the set of all semistandard signed $\muS$-tableaux
and  \smash{$\SSYT(\muS)_{\swtp{\alpha}}$} for the subset of those signed $\muS$-tableaux
of signed weight~$\swtp{\alpha}$.
We omit $\pm$ in the second case since it is clear from the signed weight that signed tableaux
are required. As already seen, we adopt the convention that negative entries are shown in tableaux by bold numbers.
For example, \smash{$\SSYTw{(5,4,3,1)}{(2,2)}{(5,3,1)}$} has size two, containing the two
semistandard signed tableaux 
\[ \young(\oM\tM 111,\oM\tM 22,113,2)\, , \quad
   \young(\oM\tM 111,\oM122,\tM 23,1) \vspace*{2pt} \]
and
\smash{$\SSYTw{(5,4,3,1)}{(3,1)}{(5,3,1)}$}
has a unique semistandard
signed tableau, obtained from the second semistandard
signed tableau above by changing 
the entry of $-2$ in box $(3,1)$ to~$-1$.

\begin{definition}[Signed colexicographic order]
\label{defn:signedColexicographicOrder}
Let $s$ and $t$ be distinct semistandard signed tableaux of the same shape. 
We set $s < t$ if and only if \emph{either} 
\begin{defnlist}
\item[(i)] $\sgn(s) = -1$ and $\sgn(t) = 1$ \emph{or} 
\item[(ii)] $\sgn(s) = \sgn(t)$
and considering the \emphd{largest} entry, $m$ say, that appears in a different position in $s$ and $t$,
in the \emphd{rightmost} column in which the multiplicity of $m$ differs between $s$ and $t$,
the multiplicity is \emphd{less} in $s$ than in $t$.
\end{defnlist}
This is the \emph{signed colexicographic order}.
The \emph{sign-reversed colexicographic order} is defined identically except that if $\sgn(s) = -1$
and $\sgn(t) = 1$ then now $s > t$.
\end{definition}

We emphasise that here `largest entry' is with respect to the order in Definition~\ref{defn:semistandardSignedTableau} in which
$-1 \prec -2 \prec \ldots \prec 1 \prec 2 < \ldots$.
For example, the signed colexicographic order restricted to
semistandard signed tableaux of shape $(2,1)$ having entries from $\{ 1, 2, 3\}$ is
\[ \young(11,2)\spy{0pt}{\,$<$\hspace*{1pt}} \young(11,3)\spy{0pt}{\,$<$\hspace*{1pt}} 
\young(12,2)\spy{0pt}{\,$<$\hspace*{1pt}} \young(12,3)\spy{0pt}{\,$<$\hspace*{1pt}}
\young(22,3)\spy{0pt}{\,$<$\hspace*{1pt}} \young(13,2)\spy{0pt}{\,$<$\hspace*{1pt}}
\young(13,3)\spy{0pt}{\,$<$\hspace*{1pt}} \young(23,3) \spy{0pt}{\, }\]
and the total order on 
$\sSSYT\bigl( (1^2) \bigr)$ is
\[ \scalebox{0.9}{$\begin{aligned}
&\young(\oM,1) < \young(\tM,1) < \young(\dM,1) < \young(\fM,1) < 
\ldots < \young(\oM,2) < \young(\tM,2) <
\young(\dM,2) < \young(\fM,2) < \ldots \\[1pt] &\ldots < \young(\oM,\oM) < \young(\oM,\tM) 
< \young(\tM,\tM) < \young(\oM,\dM) < \young(\tM,\dM) < \young(\dM,\dM) 
< \young(\oM,\fM) < \ldots < \young(1,2) < \young(1,3) < \young(2,3) < \young(1,4) < \ldots .
\end{aligned}$} \]
Changing the order to the sign-reversed colexicographic order, the positive tableaux seen
in the bottom row instead come first,
and the order within each line is unchanged. In either order
we have $\scalebox{0.9}{$\young(\oM,\oM)$} <
\scalebox{0.9}{$\young(\oM,\tM)$}\,$; the greatest entry that has different multiplicity is~$-2$
and it appears only in the tableau that is larger.
More generally, the signed colexicographic order on $(1^m)$-tableaux
with only positive entries agrees with the colexicographic order on $m$-subsets of $\N$, 
whence its name. 
It is notable that the signed colexicographic order could be replaced
with any other total order on semistandard tableaux in which negative tableaux 
(in the sense of Definition~\ref{defn:signTableau}) precede positive tableaux
without changing any of our results: we explain this in Remark~\ref{remark:otherOrder} below
and use this freedom in the proof of Theorem~\ref{thm:nuStable}
(see Definition~\ref{defn:adaptedSignedColexicographicOrder}).

\subsubsection*{Plethystic semistandard signed tableaux}
We can now define our key combinatorial objects.

\begin{definition}[Plethystic signed tableau]\label{defn:plethysticSignedTableau}
A \emph{plethystic signed tableau~$T$} of \emph{outer shape} $\nu$
and \emph{inner shape} $\muS$ is a function $T : [\nu] \rightarrow 
\YT(\muS)$. 
If $T(i,j) = t$ then we 
say that $T$ has \emph{entry} $t$ in box $(i,j)$. 
We call the entries of $T$ \emph{inner tableaux}.
\end{definition}

Let $\PYT(\nu, \muS)$
denote the set of plethystic signed tableaux of outer shape $\nu$ and inner shape $\muS$. For example three elements of $\PYT\bigl((3,2), (2) \bigr)$
are shown below
\[ \pyoung{1.15cm}{0.7cm}{{{\young(1\oM), 
 \young(\oM 2), \young(22)}, {\young(\oM\tM), \young(13)}}}\spy{18pt}{,}
\pyoung{1.15cm}{0.7cm}{{{\young(\oM1), \young(22), \young(12)}, 
{\young(\oM\tM), \young(13)}}}\spy{18pt}{,}
\pyoung{1.15cm}{0.7cm}{{{\young(\oM1), \young(22), \young(22)}, {\young(\oM\tM), \young(13)}}}\spy{18pt}{.}\]
Note that each inner $\muS$-tableau 
in a plethystic signed tableau has a sign defined by Definition~\ref{defn:signTableau}.
For example, the first plethystic signed tableau above has
$\young(1\oM)$\hskip0.5pt, $\young(\oM2)$ as its two negative entries, and so has sign $(-1)^2$
and therefore is positive, in the sense of this definition.
Moreover,
if these inner $\muS$-tableaux are semistandard in the sense 
of Definition~\ref{defn:semistandardSignedTableau}, as in the
second and third examples above (but not the first, because of $\young(1\oM)$\hskip0.5pt),
then they
are totally ordered by the signed and sign-reversed colexicographic orders 
in Definition~\ref{defn:signedColexicographicOrder}. 
We use this 
to lift Definition~\ref{defn:semistandardSignedTableau} almost
verbatim to the plethystic setting.

\begin{definition}[Plethystic semistandard signed tableau]
\label{defn:plethysticSemistandardSignedTableau}
Let $T$ be a pleth\opthyphen{}ystic signed tableau with
semistandard inner tableau entries. We say that $T$
is \emph{semistandard}
if 
\begin{defnlist}
\item equal positive entries of $T$
lie in horizontal strips
\item equal negative entries of $T$ lie in vertical strips, 
\item all entries are weakly increasing 
when rows are read left-to-right and columns are read
top-to-bottom with respect to 
the signed colexicographic order.
\end{defnlist}
 We say that~$T$ is \emph{sign-reversed semistandard}
if the same holds with respect to the sign-reversed colexicographic order.
\end{definition}

Let $\sPSSYT(\nu, \muS)$ and  $\srPSSYT(\nu, \muS)$ denote the sets of all plethystic semistandard
signed tableaux and sign-reversed plethystic semistandard signed tableaux
of outer shape $\nu$ and inner shape $\muS$.
Thus the first two tableaux displayed above are in $\PYT\bigl((3,2), (2)\bigr)$
but not in either of these subsets, because in the first
$\young(1\oM)$ is not
semistandard and in the second $\young(22) > \young(12)$ violates
condition (c) above.
The third tableau is in
$\sPSSYT\bigl((3,2), (2)\bigr)$ but not in $\srPSSYT\bigl((3,2), (2)\bigr)$.

\begin{definition}[Signed weight of a plethystic signed tableau]\label{defn:signedWeightPlethystic}
The \emph{signed weight} of a 
plethystic signed tableau $T$ 
is the sum of the signed weights of its inner tableaux.
\end{definition}

We denote by $\PSSYT(\nu, \muS)_{\swtp{\alpha}}$ 
and $\srPSSYT(\nu, \muS)_{\swtp{\alpha}}$ the subsets of those 
plethystic semistandard signed tableaux
of signed weight $\swtp{\alpha}$.
For instance the signed weights of the 
elements of $\PYT\bigl((3,2), (2)\bigr)$ shown above are
$\bigl( (3,1), (2,3,1) \bigr)$, $\bigl( (2,1), (3,3,1) \bigr)$
and $\bigl( (2,1), (2,4,1) \bigr)$.

The definition of the signed colexicographic
order (Definition~\ref{defn:signedColexicographicOrder}) 
applies to both these subsets, since the inner $\muS$-tableau entries
are totally ordered.
The three elements of  $\PSSYTw{(2,2)}{(3)}{(3)}{(7,2)}$ are, ordered by 
the signed colexicographic order,
\[ \pyoung{1.6cm}{0.75cm}{{{\young(\oM 11), \young(\oM 12)}, {\young(\oM 12), \young(111)}}}
\spy{20pt}{,} \pyoung{1.6cm}{0.75cm}{{{\young(\oM 11), \young(\oM 22)}, {\young(\oM 11), \young(111)}}}
\spy{20pt}{,} \pyoung{1.6cm}{0.75cm}{{{\young(\oM 11), \young(\oM 12)}, {\young(\oM 11), \young(112)}}}
\spy{20pt}{.} \]
For instance, the third plethystic semistandard signed tableau is greater
than the second because the greatest $(3)$-tableau entry of the third,
namely $\young(112)$\hskip1pt, is not in the second.
To explain one feature that may at first seem surprising, note that since
$\young(\oM 11)$\, has negative sign, it may appear multiple times
in the same column of a plethystic semistandard tableau, but it cannot be 
repeated within the same row. See before Remark~\ref{remark:otherOrder} for the analogous example
using sign-reversed plethystic semistandard signed tableaux and also~Example~\ref{ex:GFlip}
for another example showing repeated inner tableaux.

\section{Maximal and strongly maximal signed weights}\label{sec:maximalSignedWeights}
The results and definitions in~\S\ref{subsec:maximalSignedWeights} are needed throughout; the remainder
of this section has the definitions needed in Theorem~\ref{thm:nuStable}.
In \S\ref{subsec:stronglyMaximalSignedWeightExamples},
\S\ref{subsec:stronglyMaximalSignedWeightsTable} and \S\ref{subsec:needStronglyMaximal}
we give motivating examples: these final three subsections are not logically essential.

\subsection{Greatest signed weights}\label{subsec:maximalSignedWeights}
We begin with a partial order on signed weights.
Let $\W_{\ell^\-}$ be the set of weights of length at most $\ell^\- \in \N_0$.

\begin{definition}[$\ell^\-$-Signed dominance order]\label{defn:ellSignedDominanceOrder}
Let $\ell^\- \in \N_0$.
The $\ell^\-$-signed dominance order is the partial order on $\W_{\ell^\-} \!\!\times \W$ defined by
$(\alpha^\-, \alpha^\+) \unlhd (\beta^\-, \beta^\+)$ if 
\[ (\alpha^\-_1, \ldots, \alpha^\-_{\ell^\-}, \alpha^\+_1, \alpha^\+_2, \ldots) \unlhd
   (\beta^\-_1, \ldots, \beta^\-_{\ell^\-}, \beta^\+_1, \beta^\+_2, \ldots). 
\]
\end{definition}

For example we have
$\bigl( (1,1,1), (2,1) \bigr) \unlhd \bigl( (2,1), (3) \bigr)$
in the $3$-signed dominance order
because $(1,1,1,2,1) \,\unlhd\, (2,1,0,3,0)$, whereas
$\bigl( (3), (2,1) \bigr)$ and $\bigl( (2,1), (3) \bigr)$ are incomparable in the $2$-signed dominance
order
since  $(3,0,2,1)$ and $(2,1,3,0)$ are incomparable in the dominance order.
This example should make it clear that no ambiguity arises from using the same symbol $\unlhd$ for
both the dominance and $\ell^\-$-signed dominance order. 
The value of~$\ell^\-$ will always be clear from context;
in all our main theorems, $\ell^\-$ is the length of the partition $\kappa^\-$.

\begin{definition}\label{defn:greatestSignedTableau}
Let $\ell^\- \in \N_0$. Given a skew partition $\tauS$,
let~$t^\-$ be the semistandard signed tableau
with only negative entries defined by putting $\min(\ell^\-, \tau_i - \taus_i)$
entries from $-1,\ldots, -\ell^\-$ into row $i$ of $[\tauS]$. Supposing that $t^\-$ has shape $\sigma/\taus$,
let $t_{\ell^-}(\tauS)$ 
be the semistandard signed tableau of shape $\tauS$ obtained from $t^\-$ by putting
$\tau'_j - \sigma'_j$ entries from $1, 2, \ldots $ into column $j$.
\end{definition}

\begin{definition}[Greatest signed weight]\label{defn:greatestSignedWeight}
Let $\ell^\- \in \N_0$. Given a skew partition $\tauS$
we define the \emph{$\ell^\-$-greatest signed weight} of shape $\tauS$,
denoted $\omega_{\ell^-}(\tauS)$, to be the signed weight of $t_{\ell^\-}(\tauS)$.
\end{definition}

We immediately justify calling $\omega_{\ell^\-}(\tauS)$ `greatest'.
Examples of both definitions are given following this lemma.

\begin{lemma}\label{lemma:greatestSignedWeight}
Let $\tauS$ be a skew partition and let $\ell^\- \in \N_0$.
The tableau $t_{\ell^\-}(\tauS)$ is the greatest signed tableau of shape $\tauS$
when signed weights
are ordered by the $\ell^\-$-signed dominance order. Moreover, 
given any $\tauS$-tableau $t$ with negative entries from $\{-1,\ldots, -\ell^\-\}$
we have, writing $\swt(t)$ for the signed weight of $t$,
\[ (\swt(t)^\-, \swt(t)^\+) \unlhd \bigl(\omega_{\ell^\-}(\tauS)^\-, \omega_{\ell^\-}(\tauS)^\+ \bigr) 
\]
where both $\omega_{\ell^\-}(\tauS)^\-$ and $\omega_{\ell^\-}(\tauS)^\+$ are partitions.
\end{lemma}

\begin{proof}
It is clear from the construction of $t_{\ell^\-}(\tauS)$ 
that, amongst all semistandard signed $\tauS$-tableaux,
$t_{\ell^\-}(\tauS)$ greedily maximizes first the number of
$-1$s, then the number of $-2$s, and so on, until all the negative entries
in $\{-1,\ldots, -\ell^\-\}$ are 
placed, and then the number of $1$s, then the number of $2$s, and so on,
until all positive entries are placed. The displayed inequality
is therefore obvious from the definition of the $\ell^\-$-signed dominance order in Definition~\ref{defn:ellSignedDominanceOrder}. By the construction of $t_{\ell^\-}(\tauS)$,
each entry $-k$ for $k \ge 2$ has an entry $-(k-1)$ to its left, and each entry $k$
for $k \ge 2$ has an entry $k-1$ above it. Therefore $\omega_{\ell^\-}(\tauS)^\-$ and $\omega_{\ell^\-}(\tauS)^\+$
 are partitions.
\end{proof}

See Lemma~\ref{lemma:ellDecompositionGreatestSignedWeight} for a strengthening of the final
part of this lemma.

\begin{example}\label{ex:2greatest}
The $2$-greatest tableaux $t_{2}\bigl((6,4,4,1)/\taus\bigr)$ for the four choices 
$\varnothing$, $(1,1)$, $(2,1)$, $(3,3)$ of~$\taus$
are shown below
\[
\young(\oM\tM1111,\oM\tM22,\oM\tM33,\oM) \spy{0pt}{,\ }
\young(:\oM\tM111,:\oM\tM2,\oM\tM13,\oM) \spy{0pt}{,\ } 
\young(::\oM\tM11,:\oM\tM1,\oM\tM12,\oM) \spy{0pt}{,\ }
\young(:::\oM\tM1,:::\oM,\oM\tM11,\oM) \spy{0pt}{.} 
\]
Their greatest signed weights $\omega_{2}\bigl( (6,4,4,1) / \taus \bigr)$
are
$\bigl( (4,3), (4,2,2) \bigr)$, $\bigl( (4,3),$ $(4,1,1) \bigr)$, $\bigl( (4,3), (4,1) \bigr)$ and
$\bigl( (4,2), (3) \bigr)$.  We continue this example in Example~\ref{ex:2greatestWeight}.
\end{example}

In general it is quite fiddly to specify $\omega_{\ell^\-}(\tauS)$ except by the algorithmic
construction above. In the partition case however there is a simple formula, which the reader
will easily guess from the previous example.
We postpone it to~\eqref{eq:greatestSignedWeightPartitionCase} since it is an example of the $\ell^\-$-decomposition
of partitions in \S\ref{sec:twisted}.

The final remark below is not logically essential, but 
will help orient the reader, while addressing one potential confusion.

\begin{remark}\label{remark:greatestSignedTableauIsLeastInSignedColexicographicOrder}
Let $\ell^\- \in \N_0$ and let $\tauS$ be a skew partition.
Recall from Definition~\ref{defn:signedColexicographicOrder} that negative tableaux
precede positive tableaux in the signed colexicographic  order and \emph{vice versa}
in the sign-reversed colexicographic order.
It follows, by a similar argument to Lemma~\ref{lemma:greatestSignedWeight},
that $t_{\ell^\-}(\tauS)$ is 
the least tableau in the signed colexicographic order if
$|\omega_{\ell^\-}(\tauS)^\-|$ is odd, and in the sign-reversed colexicographic order
if $|\omega_{\ell^\-}(\tauS)^\-|$ is even. More generally signed tableaux of \emph{large} signed weight
(in the $\ell^\-$-signed dominance order) are \emph{small} in the sign and sign-reversed colexicographic orders.
\end{remark}

\subsection{Semistandard signed tableau families}\label{subsec:semistandardSignedTableauFamilies}
For Theorem~\ref{thm:nuStable} we must 
extend these ideas to families of semistandard signed tableaux.

\begin{definition}[Semistandard signed tableau families]\label{defn:semistandardSignedTableauFamily}
Let $\tauS$ be a skew partition and let $R \in \N$. 
\begin{defnlistS} 
\item A \emph{row-type semistandard signed tableau family}
of shape $\tauS$ and size~$R$ is the multiset of entries in a plethystic
semistandard signed tableau of outer shape $(R)$ and inner shape $\tauS$.
\item A \emph{column-type semistandard signed tableau family}
of shape $\tauS$ and size $R$ is the multiset of entries in a plethystic
semistandard signed tableau of outer shape $(1^R)$ and inner shape $\tauS$.
\end{defnlistS}
The \emph{signed weight} of a semistandard signed tableau family is the
sum of the signed weights of its $\tauS$-tableau elements.
A tableau family is \emph{singleton} if it has a single element.
\end{definition}

Signed weights are ordered by the $\ell$-signed dominance order
in Definition~\ref{defn:ellSignedDominanceOrder}.

\begin{definition}[Maximal signed weights]\label{defn:maximalSignedWeight}
A semistandard signed tableau family of signed weight $\swtp{\kappa}$
is \emph{maximal} if its signed weight
is maximal in the $\ell(\kappa^\-)$-signed dominance order amongst all semistandard signed tableau families
of its type, shape and size, \emph{considering only} those families whose negative
entries come from $\{-1,\ldots, -\ell(\kappa^\-)\}$.
A \emph{maximal signed weight} is the signed weight of a maximal semistandard signed tableau family.
\end{definition}

For example, the maximal singleton semistandard signed tableau families of shape $(2,2)$ have as their unique elements
the tableaux $t_{\ell^\-}\bigl((2,2)\bigr)$ for $\ell^\- = 2$, $1$ and $0$, shown below:
\begin{equation} \young(\oM\tM,\oM\tM) \spy{0pt}{,\quad}
   \young(\oM1,\oM2) \spy{0pt}{,\quad}
   \young(11,22)\spy{0pt}{.} \label{eq:singletonMaximals22} \end{equation}
Their maximal signed weights are
$\bigl( (2,2), \varnothing \bigr)$, $\bigl( (2), (1,1) \bigr)$
and $\bigl( \varnothing, (2,2) \bigr)$, respectively. 
This shows that `maximal' must be interpreted using the appropriate value of $\ell^\-$:
for instance, while $\bigl( (2), (1,1) \bigr) \unrhd 
\bigl( \varnothing, (2,2) \bigr)$ in the $1$-signed dominance order of Definition~\ref{defn:ellSignedDominanceOrder},
the signed weight $\bigl( \varnothing, (2,2) \bigr)$ is still maximal according
to Definition~\ref{defn:maximalSignedWeight},
because it is compared only with other signed weights of the form $\bigl( \varnothing, \tau^\+)$
using the $0$-signed dominance order.
Note also that the tableau shown in the margin of signed weight $\bigl( (1), (2,1) \bigr)$
\marginpar{\raisebox{0pt}{$\qquad\qquad\ \young(\oM1,12)$}}
is not maximal, because $\bigl( (1), (2,1) \bigr) \lhd \bigl( (2), (1,1) \bigr)$ in the $1$-signed
dominance order; this illustrates that Definition~\ref{defn:maximalSignedWeight}
requires a comparison with tableaux of \emph{both possible signs}.

More generally
Lemma~\ref{lemma:maximalAndStronglyMaximalSingletonSemistandardSignedTableauFamilies}
classifies all singleton maximal semistandard signed tableau families.
In these singleton examples, 
the row/column-type is irrelevant.
We now give an example showing all features of Definition~\ref{defn:maximalSignedWeight}.

\begin{example}\label{ex:maximalSemistandardSignedTableauFamiliesMixedSign}
The five maximal row-type semistandard signed tableau families of shape $(2)$ and size $3$ are
%
%
%
%
\begin{align*}
   & \left\{ \hskip1pt\young(\oM\tM)\spy{0pt}{,\ts} \young(\oM\tM)\spy{0pt}{,\ts} \young(\oM\tM)\hskip1pt \right\}\!, \; 
     \left\{ \hskip1pt\young(\oM1)\spy{0pt}{,\ts} \young(\oM2)\spy{0pt}{,\ts} \young(\oM3)\hskip1pt \right\}\!,  \;
     \left\{ \hskip1pt\young(\oM1)\spy{0pt}{,\ts} \young(\oM2)\spy{0pt}{,\ts} \young(11)\hskip1pt \right\}\!,  \\
   &  \left\{ \hskip1pt\young(\oM1)\spy{0pt}{,\ts} \young(11)\spy{0pt}{,\ts} \young(11)\hskip1pt \right\}\!,  \;
     \left\{ \hskip1pt\young(11)\spy{0pt}{,\ts} \young(11)\spy{0pt}{,\ts} \young(11)\hskip1pt \right\} 
\end{align*}
of signed weights $\bigl(\hskip-0.5pt  (3,3), \varnothing \hskip-0.5pt \bigr)$,\! $\bigl(\hskip-0.5pt  (3), (1,1,1)\hskip-0.5pt  \bigr)$,\! 
$\bigl(\hskip-0.5pt  (2), (3,1) \hskip-0.5pt \bigr)$,\! $\bigl( \hskip-0.5pt (1), (5) \hskip-0.5pt \bigr)$ and \hbox{$\bigl(\hskip-0.5pt \varnothing, (6) \hskip-0.5pt\bigr)$\hskip-1pt,} respectively. 
Note that two of the families have tableaux of both signs and three have a repeated positive tableau.
The seven maximal column-type semistandard signed tableau families of shape $(2)$ and size $3$ are
%
%
%
%
%
%
\begin{align*}
   & \left\{ \hskip1pt\young(\oM\tM)\spy{0pt}{,\ts} \young(\oM\dM)\spy{0pt}{,\ts} \young(\oM\fM)\hskip1pt \right\}\!, \;
      \left\{ \hskip1pt\young(\oM\tM)\spy{0pt}{,\ts} \young(\oM\dM)\spy{0pt}{,\ts} \young(\tM\dM)\hskip1pt \right\}\!, \;    \\
   & \left\{ \hskip1pt\young(\oM1)\spy{0pt}{,\ts} \young(\oM\tM)\spy{0pt}{,\ts} \young(\oM\dM)\hskip1pt \right\}\!, \; 
     \left\{ \hskip1pt\young(\oM1)\spy{0pt}{,\ts} \young(\oM1)\spy{0pt}{,\ts} \young(\oM\tM)\hskip1pt \right\}\!, \; 
     \left\{ \hskip1pt\young(\oM1)\spy{0pt}{,\ts} \young(\oM1)\spy{0pt}{,\ts} \young(\oM1)\hskip1pt \right\}\!,  \; \\
   &   \left\{ \hskip1pt\young(11)\spy{0pt}{,\ts} \young(12)\spy{0pt}{,\ts} \young(22)\hskip1pt \right\} \!, \;
     \left\{ \hskip1pt\young(11)\spy{0pt}{,\ts} \young(12)\spy{0pt}{,\ts} \young(13)\hskip1pt \right\} 
\end{align*}
of signed weights $\bigl( (3,1,1,1), \varnothing \bigr)$, $\bigl( (2,2,2), \varnothing \bigr)$,
$\bigl( (3,1,1), (1) \bigr)$, 
$\bigl( (3,1), (2) \bigr)$, $\bigl( (3), (3) \bigr)$, $\bigl( \varnothing, (3,3) \bigr)$
and $\bigl( \varnothing, (4,1,1) \bigr)$, respectively. Again note that two families have a repeated
negative tableau.
We continue this example in Example~\ref{ex:stronglyMaximalSemistandardSignedTableauFamilies}.
\end{example}


We use Definition~\ref{defn:maximalSignedWeight} at a critical point in
the proof of Lemma~\ref{lemma:columnIsExceptionalOrBounded};
it is also needed in Definition~\ref{defn:stronglyMaximalSignedWeight}
shortly below.

\subsection{Strongly maximal signed weights}\label{subsec:stronglyMaximalSignedWeights}
We define the maximal semistandard signed tableau families in the statement of Theorem~\ref{thm:nuStable}
as follows. Say that a $\tauS$-tableau is \emph{$\ell^\-$-negative greatest}
if it agrees with $t_{\ell^\-}(\tauS)$ in its negative entries.
Let $\max \SM$ denote the maximum integer entry of all the tableaux in a semistandard signed tableau family $\SM$.

\begin{definition}[Strongly maximal]\label{defn:stronglyMaximalSignedWeight}
Let $\tauS$ be a non-empty
skew partition. Let $c^\+ \in \N_0$.
Let $(\kappa^\-, \kappa^\+)$ be a signed weight and set $\ell^- = \ell(\kappa^\-)$.
Let $\epsilon \in \{-1,+1\}$ be the sign of $t_{\ell^\-}(\tauS)$. 
A semistandard signed tableau family $\SM$ of shape $\tauS$
and signed weight $(\kappa^\-,\kappa^\+)$ 
is \emph{strongly $c^\+$-maximal} if 
\begin{defnlist}
\item each $t \in \SM$ is $\ell(\kappa^\-)$-negative greatest;
\item if $\epsilon = +1$ then $\SM$ has column-type; if $\epsilon =-1$ then $\SM$ has row-type;
\item if $\swtp{\phi}$ is the signed weight of a maximal semistandard signed tableau family $\T$
of the same shape, size and type as $\SM$, such that 
each member of $\T$ is $\ell(\kappa^\-)$-negative greatest and $\max \T \le \max \SM$,
then \smash{$\sum_{i=1}^{c^+} \phi^\+_i \le \sum_{i=1}^{c^+} \kappa^\+_i$} 
with equality if and only if $\T = \SM$.
\end{defnlist}
\vspace*{2pt}\noindent
The \emph{sign} of $\SM$ is $\epsilon$.
A signed weight is \emph{strongly $c^\+$-maximal} if it is the signed weight of 
a strongly $c^\+$-maximal semistandard signed tableau family; its \emph{shape} and \emph{sign}
are the common shape and sign of the tableaux in the family and its \emph{size} is the size 
of the family.
\end{definition}

See \S\ref{subsec:needStronglyMaximal} for motivation
for this definition and
Remark~\ref{remark:boundedExceptionalColumns} for how strongly maximal
tableau families are used to define a bijection on plethystic
semistandard signed tableaux.
Later in \S\ref{sec:twistedWeightBoundForStronglyMaximalWeight} we give a running example using the 
strongly $1$-maximal signed weight $(\varnothing, (4,1,1))$ 
of the tableau family $\bigl\{\young(11)\,, \young(12)\,, \young(13)\bigr\}$ found
in Example~\ref{ex:stronglyMaximalSemistandardSignedTableauFamilies}(0) using Example~\ref{ex:maximalSemistandardSignedTableauFamiliesMixedSign};
this running example illustrates the significance of condition (c).

As an immediate example of Definition~\ref{defn:stronglyMaximalSignedWeight},
it is routine to check
that the three singleton tableau families of shape $(2,2)$ in~\eqref{eq:singletonMaximals22} 
are strongly  $0$-, $2$- and $1$-maximal respectively.
The relevant values of $\ell(\kappa^\-)$, specifying the least negative entry,
are $2$, $1$ and $0$ respectively. The final tableau family
is also strongly $2$-maximal.

\begin{lemma}\label{lemma:stronglyMaximalSemistandardSignedTableauFamilyIsUnique}
If $\swtp{\kappa}$ is a strongly maximal signed weight of shape $\muS$ then there is a unique 
semistandard signed tableau family $\SM$ of shape $\muS$ and the same size and type as $\swtp{\kappa}$.
The $\muS$-tableau entries of $\SM$ are distinct and agree in their negative entries.
\end{lemma}

\begin{proof}
By (a) in Definition~\ref{defn:stronglyMaximalSignedWeight},
the tableaux in $\SM$ are equal in their negative entries. If the
sign is $+1$ then by (b), $\SM$ has column-type, and since the inner tableaux of sign $+1$
in a plethystic semistandard tableau of shape $(1^R)$ are distinct,
the tableaux in $\SM$ are distinct. The proof is similar if the sign is $-1$.
The uniqueness of $\SM$ is obvious from~(c).
\end{proof}


\begin{example}\label{ex:stronglyMaximalSemistandardSignedTableauFamilies}
Using Lemma~\ref{lemma:stronglyMaximalSemistandardSignedTableauFamilyIsUnique}
we find all strongly maximal signed tableau families and signed weights of shape $(2)$ and size $3$,
considering each possibility for $\ell^\-$, the length of the negative part of the signed weight, in turn.
\begin{itemize}

\item[(0)] Take $\ell^\- = 0$. Then all integer entries are positive and, by (b) the family has column-type.
The two relevant maximal signed weights of shape $(2)$ and size $3$ 
seen in Example~\ref{ex:maximalSemistandardSignedTableauFamiliesMixedSign} are
$\bigl( \varnothing, (3,3) \bigr)$ and $\bigl( \varnothing, (4,1,1) \bigr)$.
Now $\bigl(\varnothing, (4,1,1) \bigr)$ is strongly $1$-maximal,
satisfying (c) because when compared to $\bigl(\varnothing, (3,3)\bigr)$,
we have $(4,1,1)_1 > (3,3)_1$.
Similarly $\bigl(\varnothing, (3,3)\bigr)$ is strongly $1$- and $2$-maximal.
(This is relevant to the end of Example~\ref{ex:stronglyMaximalSemistandardSignedTableauFamilies} below.)
Note that to verify (c),
$\bigl(\varnothing, (3,3) \bigr)$ should not be
compared to $\bigl( \varnothing, (4,1,1)\bigr)$ because $(4,1,1)$ has strictly greater length, corresponding to a larger maximum integer entry.\smallskip

\item[(1)]
Take $\ell^\- = 1$. 
By Lemma~\ref{lemma:stronglyMaximalSemistandardSignedTableauFamilyIsUnique},  a strongly maximal semistandard signed tableau family
of shape $(2)$ and size $3$ has the form
\[     \qquad \left\{ \hskip1pt\young(\oM x)\spy{0pt}{,\ts} \young(\oM \yraised)\spy{0pt}{,\ts} \young(\oM z)\hskip1pt \right\} \]
where, since by (b) the family has row-type, $x < y < z$. Taking $x = 1$, $y=2$ and $z=3$
we obtain $\left\{ \hskip1pt\young(\oM1)\spy{0pt}{,\ts} \young(\oM2)\spy{0pt}{,\ts} \young(\oM3)\hskip1pt \right\}$.
None of the four other maximal row-type weights
$(\phi^\-,\phi^\+)$
of shape $(2)$ and size $3$
seen in Example~\ref{ex:maximalSemistandardSignedTableauFamiliesMixedSign}
are of a family all of whose members are $1$-greatest. Therefore (c) holds
and so the tableau family above is strongly $3$-maximal, of strongly maximal signed 
weight $\bigl((3), (1,1,1)\bigr)$.
It is clear
that this choice of $x$, $y$ and $z$ defines
the unique strongly $3$-maximal signed weight of shape $(2)$ and size $3$.\smallskip

\item[(2)]
There is no  strongly maximal semistandard signed tableau family with $\ell^\- = 2$
because by (a) 
each $(2)$-tableau element is \raisebox{0.5pt}{$\young(\oM\tM)$}\,, but, as observed above, by (b) the three
$(2)$-tableaux in the family
are distinct. In particular, while we
saw in Example~\ref{ex:maximalSemistandardSignedTableauFamiliesMixedSign} that
 $\bigl\{ \young(\oM\tM)\hskip1pt, \hskip1pt\young(\oM\dM)\hskip1pt,
\hskip1pt\young(\tM\dM) \bigr\}$ is a maximal semistandard signed tableau family of shape $(2)$
and size $3$ in the $3$-signed dominance order, it is \emph{not} strongly maximal
in the $3$-signed dominance order.
(If instead $R=1$ then $\bigl\{ \hskip1pt\young(\oM\tM)\hskip1pt \bigr\}$ is strongly $0$-maximal in the $2$-signed dominance order.)

\end{itemize}
\end{example}

Note we do \emph{not} assume in Definition~\ref{defn:stronglyMaximalSignedWeight} that $\SM$
is maximal; instead, as we now show,
 this follows from the three hypotheses, as the reader may have guessed from the previous
example. 
       
\begin{lemma}\label{lemma:stronglyMaximalImpliesMaximal}
Let $\SM$ be a strongly $c^\+$-maximal semistandard signed tableau family
of signed weight $\swtp{\kappa}$. 
Let $\swtp{\psi}$ be the signed weight of a maximal semistandard signed 
tableau family $\SS$ of the same shape, size and type as $\SM$ 
with negative entries from $\bigl\{ -1,\ldots, -\ell(\kappa^\-) \bigr\}$ and such that $\SS \not=\SM$. 
Then 
\begin{align}
|\psi^\-| &\le |\kappa^\-|
\label{eq:stronglyMaximalM}\\
\intertext{with equality if and only if \emph{either} $\max \SS > \max \SM$ \emph{or}}
\textstyle |\psi^\-| + \sum_{i=1}^{c^\+} \psi^\+_i &< \textstyle  |\kappa^\-| + \sum_{i=1}^{c^\+} \kappa^\+_i 
\label{eq:stronglyMaximalP}.
\end{align}
Moreover $\swtp{\kappa}$ is a maximal signed weight in the sense of Definition~\ref{defn:maximalSignedWeight}.
\end{lemma}

\begin{proof}
Set $\ell^\- = \ell(\kappa^\-)$.
Let $\tauS$ be the shape of $\mathcal{S}$ and $\mathcal{M}$.
Suppose there exists $t \in \SS$ such that $t$ is not $\ell^\-$-negative greatest.
By Lemma~\ref{lemma:greatestSignedWeight}, if $t$ has signed weight
$\swtp{\alpha}$ then $|\alpha^\-|  < |\omega_{\ell^\-}(\tauS)^\-|$
and, by summing over all $t \in \SS$, we see that~\eqref{eq:stronglyMaximalM} holds.
Moreover $\psi^\- \lhd \kappa^\-$ and so $\swtp{\kappa} \notunlhd \swtp{\psi}$,
as required in the final claim.

In the remaining case every element of $\SS$ is $\ell^\-$-negative greatest.
Hence $\psi^\- = \kappa^\-$.
If $\max \SS > \max \SM$ then we need only verify the final claim.
Since $\max \SS = \ell(\psi^\+)$
and $\max \SM = \ell(\kappa^\+)$, we have
\[ |\psi^\-| + \sum_{i=1}^{\ell(\kappa^\+)} \psi^\+_i 
< |\psi^\-| + |\psi^\+| = |\kappa^\-| + |\kappa^\+| = |\kappa^\-| + \sum_{i=1}^{\ell(\kappa^\+)} \kappa^\+_i. \]
Hence, by definition of the $\ell^\-$-signed
dominance order in Definition~\ref{defn:ellSignedDominanceOrder},
we have $\swtp{\kappa} \notunlhd \swtp{\psi}$, as required.
We have now reduced further to the case where $\max \SS \le \max \SM$. 
By (c) in Definition~\ref{defn:stronglyMaximalSignedWeight}, noting that
$|\psi^\-| = |\kappa^\-|$,
we now have~\eqref{eq:stronglyMaximalP}. 
It now follows from the 
definition of the $\ell^\-$-signed
dominance order in Definition~\ref{defn:ellSignedDominanceOrder}, as in the previous
paragraph, that $\swtp{\kappa} \notunlhd \swtp{\psi}$. This completes the proof.
\end{proof}

\begin{remark}\label{remark:stronglyMaximal}
We remark that the converse to this lemma also holds: if~$\SM$ is a maximal semistandard
signed tableau family such that either~\eqref{eq:stronglyMaximalM} or~\eqref{eq:stronglyMaximalP} holds
when $\SM$ is compared with a maximal semistandard signed tableau family $\SS$ 
then all tableaux in~$\SM$ have the same sign and are $\ell^\-$-negative greatest;
given this, \emph{provided}~$\SM$
has the type specified by (b), we have (b) and~\eqref{eq:stronglyMaximalP} implies that~(c)
holds. This gives an equivalent definition of `strongly maximal'; in  this paper we prefer  
Definition~\ref{defn:stronglyMaximalSignedWeight} since, while it has a technical flavour,
examples can easily be given
straight from the definition,
rather than via the argument of Lemma~\ref{lemma:stronglyMaximalImpliesMaximal} and this remark.
\end{remark}

If $\swtp{\kappa}$ is a strongly
maximal signed weight then $\kappa^\-$ and $\kappa^\+$ are partitions.
This is implicitly assumed in the statement of Theorem~\ref{thm:nuStable} 
 because we have only defined the adjoining operation $\oplus$ for partitions.
To prove this fact we use
the Bender--Knuth involution
on semistandard tableaux, of general skew shape, but having only positive entries. 
(For a textbook presentation of the method see
the proof of Theorem 7.10.2 in \cite{StanleyII}.)
We remark that the proof of the following lemma
generalizes --- but with much more work --- to show that \emph{any} maximal signed weight
is a pair of partitions.

\begin{lemma}\label{lemma:stronglyMaximalSignedWeightsArePairsOfPartitions}
If $\swtp{\kappa}$ is a strongly maximal signed weight then $\kappa^\-$ and~$\kappa^\+$ are partitions.
\end{lemma}

\begin{proof}
Set $\ell^- = \ell(\kappa^\-)$.
Suppose that $\swtp{\kappa}$ has shape $\muS$ and size $R$.
Let $\SM$ be the unique strongly maximal semistandard signed tableau family of signed weight $\swtp{\kappa}$.
By Lemma~\ref{lemma:greatestSignedWeight}, $\omega_{\ell^\-}(\muS)^\-$ is a partition.
By (a) in Definition~\ref{defn:stronglyMaximalSignedWeight} we have
$\kappa^\- = R\omega_{\ell^\-}(\muS)^\-$, and so $\kappa^\-$ is a partition.

Fix $i < \ell(\kappa^\+)$.
Let \smash{$t^\+_{(1)} \ldots, t^\+_{(R)}$} be the subtableaux of skew shape defined by taking
the positive 
entries of each tableau in $\SM$. By (a) and (b) 
in Definition~\ref{defn:stronglyMaximalSignedWeight}, \smash{$t^\+_{(1)}, \ldots, t^\+_{(R)}$ }are distinct
semistandard  tableaux of the same shape.
Applying the Bender--Knuth involution swapping $i$ and $i+1$ to each $t^\+_{(k)}$ gives
distinct semistandard tableaux \smash{$u^\+_{(1)}, \ldots, u^\+_{(R)}$}. Let $\mathcal{U}$
be the semistandard signed tableau family obtained
by replacing the subtableau~\smash{$t^\+_{(k)}$} with \smash{$u^\+_{(k)}$}
in each inner tableau in $\SM$. 
Observe that $\mathcal{U}$ has signed weight $(\kappa^\-, \lambda^\+)$ where
\[ \lambda^\+_k = \begin{cases} \kappa^\+_{i+1} & \text{if $k = i$} 
	\\ \kappa^\+_i & \text{if $k=i+1$} 
	\\ \kappa^\+_k & \text{if $k \not= i, i+1$.}\end{cases} \]
By the `moreover' part of Lemma~\ref{lemma:stronglyMaximalImpliesMaximal},
$(\kappa^\-,\kappa^\+)$ is a maximal signed weight in the $\ell^\-$-signed dominance
order. Comparing it with $(\kappa^\-,\lambda^\+)$, we see that \emph{either} $\kappa^\+_i = \kappa^\+_{i+1}$
and so $\kappa^\+_i = \lambda^\+_i = \kappa^\+_{i+1}$, 
\emph{or} $\kappa^\+ \notunlhd \lambda^\+$ and so $\kappa^\+_i > \kappa^\+_{i+1}$.
This completes the proof.
\end{proof}

We use this lemma later to prove Proposition~\ref{prop:stronglyMaximalSignedWeightsAreEllDecompositions}
and then in the proof of Corollary~\ref{cor:signedWeightBoundForStronglyMaximalSignedWeight}
and in Lemma~\ref{lemma:stablePartitionSystemForNuVarying}.

\begin{remark}\label{remark:asymmetry}
Definition~\ref{defn:stronglyMaximalSignedWeight} 
is deliberately asymmetric with respect to positive and negative entries. The effect of this is seen
most obviously in Example~\ref{ex:stronglyMaximalSemistandardSignedTableauFamilies}(2)
and in Definition~\ref{defn:ellTwistedDominanceOrder} below. This
asymmetry ultimately
reflects our decision in Definition~\ref{defn:ellSignedDominanceOrder} to order 
the negative part of signed weights first.
For this reason, while applying the $\omega$-involution to Theorem~\ref{thm:muStable} 
gives no new results, as we show
in Example~\ref{ex:omegaTwistGivesMore}, this is not the case for Theorem~\ref{thm:nuStable}.
\end{remark}

\subsection{Further 
examples of strongly maximal signed weights}\label{subsec:stronglyMaximalSignedWeightExamples}
This section is not logically essential: it is included  to show that 
Definition~\ref{defn:stronglyMaximalSignedWeight} is not overly restrictive, and so there
is a rich supply of strongly maximal signed weights to which Theorem~\ref{thm:nuStable} may be applied.
(Many more illustrative examples are shown in Table~4.23 in \S\ref{subsec:stronglyMaximalSignedWeightsTable}.)
We begin with singleton semistandard signed tableau families, generalizing the small example
immediately after Definition~\ref{defn:maximalSignedWeight}.

\begin{lemma}\label{lemma:maximalAndStronglyMaximalSingletonSemistandardSignedTableauFamilies}
Let $\tauS$ be a skew partition. 
The maximal singleton semistandard signed tableau families
are precisely $\{ t_{\ell^\-}(\tauS) \}$ for $0 \le \ell^\- \le \max \{ \tau_i - \taus_i : 1 \le i \le
\ell(\tau) \}$. If $c^\+$ is the greatest
positive entry of $t_{\ell^\-}(\tauS)$  then 
$\{ t_{\ell^\-}(\tauS) \}$ is strongly $c^\+$-maximal.
\end{lemma}

\begin{proof}
It is obvious that (a) holds in Definition~\ref{defn:stronglyMaximalSignedWeight} and
we stipulate that the singleton family has row-type or column-type according
to the sign of $t_{\ell^\-}(\tauS)$ so that (b) holds.
Finally, by Lemma~\ref{lemma:greatestSignedWeight}, if $t$ is a
$\tauS$-tableau
of signed weight $\swt(t)$ with negative entries from $\{-1,\ldots,-\ell^\-\}$ then
\[ \bigl|\swt(t)^\-\bigr| + \sum_{i=1}^{c^\+} \swt(t)^\+_i \le \bigl|\omega_{\ell^\-}(\tauS)\bigr| 
+ \sum_{i=1}^{c^\+}\omega_{\ell^\-}(\tauS)_i \]
so we have (c).
\end{proof}

In particular, when $\taus = \varnothing$, by taking $\ell^\- = 0$
we find that $(\varnothing, \tau)$ is strongly $\ell(\tau)$-maximal
and by taking $\ell^\- = a(\tau)$ that 
$(\tau', \varnothing)$ is strongly $0$-maximal.
This gives the strongly maximal signed weights mentioned in the introduction.

%


\begin{example}\label{ex:LawOkitaniSignedWeightsAreStronglyMaximal}
Let $m \in \N$ and let $0 \le d \le m$.
The greatest tableau $t_d\bigl((m)\bigr)$
is \[
\raisebox{-3pt}{\begin{tikzpicture}[x=0.5cm,y=-0.5cm, line width=0.5pt] 
\draw(0,0)--(9,0)--(9,1)--(0,1)--(0,0); \draw (1,0)--(1,1); \draw (2,0)--(2,1); \draw (4,0)--(4,1); \draw (5,0) --(5,1);
\draw (6,0)--(6,1); \draw (8,0)--(8,1); \node at (0.5,0.5) {$\oM$}; \node at (1.5,0.5) {$\tM$}; \node at (3,0.5) 
{$\ldots$}; \node at (4.5,0.5) {$\mbf{d}$}; \node at (5.5,0.5) {$1$};
\node at (7,0.5) {$\ldots$}; \node at (8.5,0.5) {$1$}; \end{tikzpicture}\,\raisebox{3pt}{.}}\]
It has signed weight
$\omega_d\bigl( (m) \bigr) = \bigl( (1^d), (m-d) \bigr)$. 

\begin{examplelist}
\item[(i)] By Lemma~\ref{lemma:maximalAndStronglyMaximalSingletonSemistandardSignedTableauFamilies},
$\bigl\{\hskip-0.25pt t_d\bigl((m)\bigr) \hskip-0.7pt\bigr\}$ is a strongly $1$-maximal
semistandard signed tableau family. (To satisfy (b) in Definition~\ref{defn:stronglyMaximalSignedWeight}
we stipulate that it has row-type if $d$ is odd and column-type if $d$ is even.)
This can also be seen directly from Definition~\ref{defn:stronglyMaximalSignedWeight}:
since $t_d\bigl((m)\bigr)$ has leftmost $d$ boxes
\smash{$\raisebox{-4pt}{\begin{tikzpicture}[x=0.5cm,y=-0.5cm, line width=0.5pt] 
\draw(0,0)--(5,0)--(5,1)--(0,1)--(0,0); \draw (1,0)--(1,1); \draw (2,0)--(2,1); \draw (4,0)--(4,1); \draw (5,0) --(5,1); \node at (0.5,0.5) {$\oM$}; \node at (1.5,0.5) {$\tM$}; \node at (3,0.5) 
{$\ldots$}; \node at (4.5,0.5) {$\mbf{d}$};\end{tikzpicture}}$}
it is $d$-negative greatest, and clearly it has the greatest possible
number of $1$s of all such tableaux. Therefore $\bigl( (1^d), (m-d) \bigr)$ is 
a strongly $1$-maximal signed weight. Note this holds even when $d=0$.
By Theorem~\ref{thm:nuStable}, if $\nu$ and $\lambda$ are any partitions, then
$\langle s_{\nu^{(M)}} \circ s_{(m)}, s_{\lambda \sqcups
(d^M) + M(m-d) }\rangle$ is ultimately constant
where $\nu^{(M)} = \nu + (M)$ if $d$ is even and $\nu^{(M)} = \nu \sqcup (1^M)$ if $d$ is odd, proving~\eqref{eq:LOeven} and~\eqref{eq:LOodd} in
\S\ref{subsec:earlierWork};
as discussed earlier,
these results were first proved in \cite{LawOkitaniStability}.
See Proposition~\ref{prop:LawOkitaniBound} 
for explicit bounds deduced from our Theorem~\ref{thm:nuStableSharp}, together
with a sufficient condition for the constant value of the plethysm coefficient to be zero.

\item[(ii)]
Suppose that $d < m$. For $h \in \N$, let 
$u^{(h)}$ be the $(m)$-tableau obtained from $t_d\bigl((m)\bigr)$ 
by changing the rightmost $1$ to $h$.
Thus $t = u^{(1)}$. Fix $R \in \N$. We claim that the tableau family
\[ \qquad
\quad \{ u^{(1)}, \ldots, u^{(R)} \} \]
of shape $(m)$ and size $R$ is strongly $1$-maximal,
of row-type if $d$ is odd and column-type if $d$ is even.
Clearly it satisfies conditions (a) and (b) in Definition~\ref{defn:stronglyMaximalSignedWeight}.
For (c), we observe that the family has the maximum possible
number of entries of $1$ of all families of size~$R$ formed from $d$-negative greatest tableaux.
The corresponding strongly $1$-maximal signed weight of shape $(m)$ and size $R$
is 
\[ \qquad\quad\bigl( (R^d), ((m-d)R - (R-1), 1^{R-1}) \bigr). \]
By Theorem~\ref{thm:nuStable}, if $\nu$ and $\lambda$ are any partitions then
\[ \qquad\quad\langle s_{\nu^{(M)}} \circ s_{(m)}, s_{\lambda \,\oplus\, M((R^d), ((m-d)R- (R-1),1^{R-1}) )} \rangle \]
is ultimately constant,
where $\nu^{(M)} = \nu + (M^R)$ if $d$ is even and $\nu^{(M)} = \nu \sqcup (R^M)$ if $d$ is odd. 
\end{examplelist}
\end{example}

After Corollary~\ref{cor:nuStableEmpty} we give some
notable special cases of (i) and (ii) above in which the plethysm coefficient is constant
for all $M \in \N_0$. See also Example~\ref{ex:LawOkitaniExplicitBounds}
for some explicit stability bounds obtained 
using Proposition~\ref{prop:LawOkitaniBound}.


%
%
%
%
%
%
%
%
%
%
%
%
%

\begin{example}\label{ex:signedTableauFamily}
The plethystic semistandard signed tableau shown in the margin
\marginpar{\raisebox{-1.1in}{$\pyoung{1.25cm}{1.15cm}{ {{\youngL{(\oM1,\oM2)}}, {\youngL{(\oM 1,\oM3)}}, {\youngL{(\oM 1,\oM4)}}} }$}}
of shape $\bigl( (1,1,1), (2,2) \bigr)$ has entries from the semistandard signed $(2,2)$-tableau family 
\[ \SM = \left\{ \, \young(\oM 1,\oM2)\spy{6pt}{,} \young(\oM 1,\oM3)\spy{6pt}{,} \young(\oM 1,\oM4) \,\right\} \]
of size $3$, sign $+1$ and signed weight
 $\bigl( (6), (3,1,1,1) \bigr)$.
Suppose that $\mathcal{T}$ is a column-type
semistandard signed tableau family of size $3$, shape $(2,2)$
in which each member of $\mathcal{T}$ is $1$-negative greatest.
Then each $(2,2)$-tableau in $T$ has two entries of $-1$ in its first column
and since box $(2,2)$ cannot contain either $-1$ or $1$,
the inequality
$\tau^\+_1 \le \kappa_1^\+ = 3$
required by Definition~\ref{defn:stronglyMaximalSignedWeight}(c)
when $c^\+ = 1$ holds.  
Moreover, we have equality if and only if every $(2,2)$-tableau in $\mathcal{T}$ has the
form shown in the margin, and in this case it is easy to see that $\mathcal{T}$ is the family~$\SM$. 
\marginpar{
$\young(\oM1,\oM\star)$}
Therefore $\SM$ is strongly $1$-maximal
and $\bigl( (6), (3,1,1,1) \bigr)$ is a strongly maximal signed weight
of shape $(2,2)$, size~$3$ and sign $+1$.
By Theorem~\ref{thm:nuStable}, $\langle s_{\nu + (M, M, M)} \circ s_{(2,2)}, s_{\lambda \oplus
M((6), (3,1,1,1))} \rangle$ is ultimately constant for any partitions $\nu$ and $\lambda$.
\end{example}

The special case $\mus = \varnothing$ of
the following
lemma was used in \S\ref{subsec:earlierWork} to show that~\eqref{eq:BOR},
taken from \cite[(9)]{BriandOrellanaRosas}, is a special case of Theorem~\ref{thm:nuStable}.

\begin{lemma}\label{lemma:allTableauFamily}
Let $\muS$ be a skew partition. Fix $\ell \in \N$ and 
let $\SM$ be the set of all semistandard $\muS$-tableaux having
entries from $\{1,\ldots, \ell\}$.
The signed weight of $\SM$ is $\bigl(\varnothing, (q^\ell)\bigr)$
where $q = |\SM||\muS|/\ell$; it
is a strongly $c^\+$-maximal signed weight of  sign $+1$
for all $c^\+ \in \{1,\ldots, \ell\}$.
\end{lemma}

\begin{proof}
Clearly each $t \in \SM$ is $0$-negative greatest and $\SM$ has column-type
since its entries are distinct.
Hence (a) and (b) in Definition~\ref{defn:stronglyMaximalSignedWeight} hold.
Let $R = |\SM|$. If $\T$ is another
tableau family of shape $\muS$ and size $R$ then $\T$ contains
a $\muS$-tableau with maximum entry strictly greater than $\ell$.
Hence $\max \SM < \max \T$, and condition (c) holds vacuously for any permitted $c^\+$. Since the skew Schur function $s_\muS$ is symmetric,
each element of $\{1,\ldots, \ell\}$ appears equally often as
an entry in a tableau $t \in \SM$, and so the signed
weight of $\T$ is $\bigl( \varnothing, (q^\ell) \bigr)$ for some
$q \in \N$. Since each tableau has $|\muS|$
entries, the common multiplicity $q$ is $|\SM||\muS|/\ell$, as claimed.
\end{proof}

The following example shows the usefulness of skew partitions in Theorem~\ref{thm:nuStable}.

\begin{example}\label{ex:needMaximalGood}
Take $\muS = (2,1) / (1)$.
By Definition~\ref{defn:stronglyMaximalSignedWeight},
or alternatively by Lemma~\ref{lemma:maximalAndStronglyMaximalSingletonSemistandardSignedTableauFamilies},
the signed weight $\bigl( \varnothing, (2) \bigr)$ 
of the tableau shown in the margin is strongly $1$-maximal. Since it is defined by a single
semistandard signed tableau of sign $+1$, the size is $1$ and the sign is $+1$.
\marginpar{$\qquad\qquad\quad \young(:1,1)$}
It therefore follows from Theorem~\ref{thm:nuStable} that
$\langle s_{\nu+(M)} \circ s_{(2,1)/(1)}, s_{\lambda + (2M)} \rangle$ 
is ultimately constant, for all
partitions $\nu$ and $\lambda$ such that $2|\nu| = |\lambda|$;
using $s_{(2,1)/(1)} = s_{(2)} + s_{(1,1)} = s_{(1)}^2$ an equivalent formulation
is that $\langle s_{\nu + (M)} \circ s_{(1)}^2, s_{\lambda + (2M)}\rangle$ is ultimately
constant. Since the plethysm product is not distributive over addition in its second component,
this result is already non-trivial to prove by other methods.
\end{example}

The final example in this subsection is included
to give an idea of the rich behaviour of maximal signed weights of large size.
It is instructive but not logically essential, and so we omit far more details than usual.

\enlargethispage{8pt}

\begin{example}\label{ex:maximalSemistandardSignedTableauFamiliesShape2}
As seen earlier in 
Example~\ref{ex:maximalSemistandardSignedTableauFamiliesMixedSign},
the column-type semistandard signed tableau families of shape~$(2)$ and size $3$, namely
\[ \left\{ \hskip1pt\young(11)\spy{0pt}{,\,} \young(12)\spy{0pt}{,\ts} \young(13)\hskip1pt \right\}\!,  \;
   \left\{ \hskip1pt\young(11)\spy{0pt}{,\ts} \young(12)\spy{0pt}{,\ts} \young(22)\hskip1pt \right\}\!,
 \]
are maximal, of signed weights $\bigl( \varnothing, (4,1,1)\bigr)$ and  $\bigl( \varnothing, (3,3) \bigr)$ respectively. Correspondingly $s_{(1^3)} \circ s_{(2)}$ has maximal constituents
$s_{(4,1,1)}$ and $s_{(3,3)}$, and in fact $s_{(1^3)} \circ s_{(2)} = s_{(4,1,1)} + s_{(3,3)}$.
More generally, one can show (see for instance \cite[page 138, Example 6]{MacDonald}
or \cite[\S 8.5]{PagetWildonTwisted}) that
$s_{(1^n)} \circ s_{(2)} = \sum_{\lambda} s_{2[\lambda]}$ where the sum
is over all partitions of $n$ having distinct parts and 
$2[\lambda]$ is the partition whose main diagonal hook lengths are $2\lambda_1, \ldots, 2\lambda_{\ell(\lambda)}$ 
and such that $2[\lambda]_i = \lambda_i + i$ for $1 \le i \le \ell(\lambda)$.
For each such partition $2[\lambda]$
there is a unique maximal column-type semistandard tableau family
of shape $(2)$ and size $R$ and signed weight $(\varnothing, 2[\lambda])$. 
In particular, any two partitions $2[\lambda]$ are either equal or incomparable in the dominance order,
and so every constituent of the plethysm $s_{(1^n)} \circ s_{(2)}$ is \emph{both} maximal and minimal.
Two examples
have already been given and the 
tableau in the margin indicates how
\[ \bigl\{ \hskip1pt \young(11)\spy{0pt}{,} \young(12)\spy{0pt}{,} \young(13)\hskip1pt  \bigr\}
\cup \bigl\{ \hskip1pt  \young(22)\spy{0pt}{,} \young(23)\hskip1pt  \bigr\} \]
is constructed 
from $2[(3,2)]$ (shown below in the margin) by forming the three $(2)$-tableaux in the first set  of
total signed weight $\bigl( \varnothing, (4,1,1) \bigr)$ from the entries
$\{1,1,1,1,2,3\}$
in the hook on the box $(1,1)$ 
and two tableaux in the second set
of total signed  weight $\bigl( \varnothing, (0,3,1) \bigr)$
from the entries $\{2,2,2,3\}$ in the hook on the box $(2,2)$ on the main diagonal boxes.
Summing $(4,1,1)$ and $(0,3,1)$ we obtain
a maximal semistandard tableau family
of shape $(2)$ and size $5$ and signed weight $\bigl( \varnothing, (4,4,2) \bigr)$,
corresponding to the partition~$2[(3,2)]$.%
\marginpar{\raisebox{18pt}{ $\young(1111,2222,33)$}}%
\marginpar{\raisebox{12pt}{ $\young(6\:\:\:,\:4\:\:,\:\:)$}}%
\ We invite
the reader to check that 
if $\alpha$ is a partition of $n$ having distinct parts then the maximal signed weight $(\varnothing, 2[\alpha])$ of
shape $(2)$, size $n$ and column-type is strongly $1$-maximal if and only if
$\alpha \in \{(n), (n-1,1), (n-2,2), (3,2,1)\}$ and strongly $2$-maximal if and only if
$\alpha = (k+1, k-1)$ or $\alpha = (k+1,k)$ where $k = \lfloor n/2 \rfloor$, according to
the parity of $n$. If $\alpha$ is the least distinct parts partition in the lexicographic order 
on partitions of $n$ then $\alpha = (\ell, \ell-1, \ldots, b+2, b,\ldots, 1)$
for some $b$ and one can show that 
the corresponding maximal semistandard
tableau family has strongly $(\ell-1)$-maximal signed weight $(\varnothing, 2[\alpha])$.
(It may also be strongly $c^\+$-maximal for other $c^\+$: for instance
if
$b = \ell$ so that $\alpha$ is $(\ell,\ell-1,\ldots, 1)$ then $2[\alpha] = (\ell+1, \stackrel{\ell}{\ldots},
\ell+1)$ and the maximal semistandard tableau family is the initial segment
of the colexicographic order ending at $\young(\ell\ell)\hspace*{0.5pt}$ and
Lemma~\ref{lemma:allTableauFamily} applies.)
These remarks imply that if $n \le 7$ then all maximal signed weights of shape $(2)$, size $n$
and sign~$+1$
are strongly maximal. When $n=8$, we have $2[5,2,1] = (6,4,4,1,1)$
and $\bigl( \varnothing, (6,4,4,1,1) \bigr) $
is maximal, but comparison with the signed weights from $2[5,3] = (6,5,2,2,1)$
and $2[4,3,1] = (5,5,4,2)$ show that it is not strongly $c^\+$-maximal
for any value of $c^\+$.
\end{example}

\subsection{Tables of strongly maximal signed weights}
\label{subsec:stronglyMaximalSignedWeightsTable}

Table~4.23 overleaf shows column-type strongly maximal signed weights
of shape $\muS$  and size~$R$ with $2 \le R \le 5$. 
(Singleton strongly maximal signed weights of size $1$ are classified in 
Lemma~\ref{lemma:maximalAndStronglyMaximalSingletonSemistandardSignedTableauFamilies}.)
The entries in the column~$c^\+$ show all 
the values for which the weight is strongly $c^\+$-maximal.
The `unsigned' weights with $\ell^\- = 0$
may be used in Corollary~\ref{cor:nuStableSharpPositiveNonSkew}
as well as Theorem~\ref{thm:nuStable}, or its sharp version
Theorem~\ref{thm:nuStableSharp}.

Many further strongly maximal signed weights, including those of row-type,
can be found using the Haskell software \cite{PZYT} mentioned in the introduction:
see the module \texttt{MaximalTableauFamily.hs} for instructions.

\subsection{Why maximal weights are not sufficient}\label{subsec:needStronglyMaximal}
In this subsection we show that, while
Theorem~\ref{thm:nuStable} certainly requires maximal signed weights,
this is not a sufficient hypothesis for this theorem, and so some stronger notion,
such\hskip4pt as\hskip4pt the\hskip4pt strongly\hskip4pt maximal\hskip4pt signed\hskip4pt weights\hskip4pt in\hskip4pt Definition~\ref{defn:stronglyMaximalSignedWeight}
\hbox{is\hskip4pt certainly}

\bigskip

\hspace*{-0.3in}\scalebox{0.8}{\begin{tabular}{cccll}
\toprule 
$\muS$ & $\ell^\-$ & $R$ & $(\kappa^\-, \kappa^\+)$ & $c^\+$\rule[-5pt]{0pt}{11pt} \\  \toprule
(3)    & 0 & 2   & $\swte{\varnothing}{(5,1)}$ & $1$ \\ \midrule
       &   & 3   & $\swte{\varnothing}{(6,3)}$ & $1,2$ \\ 
       &   &     & $\swte{\varnothing}{(7,1,1)}$ & $1$ \\ \midrule
       &   & 4   & $\swte{\varnothing}{(6,6)}$ & $1,2$ \\ 
       &   &     & $\swte{\varnothing}{(8,3,1)}$ & $1$ \\ 
       &   &     & $\swte{\varnothing}{(9,1,1,1)}$ & $1$ \\ \midrule 
       &   & 5   & $\swte{\varnothing}{(8,6,1)}$ & $2$ \\ 
       &   &     & $\swte{\varnothing}{(9,4,2)}$ & $1$ \\ 
       &   &     & $\swte{\varnothing}{(10,3,1,1)}$ & $1$ \\ 
       &   &     & $\swte{\varnothing}{(11,1,1,1,1)}$ & $1$\rule[-5pt]{0pt}{11pt} \\  \toprule
       & 2 & $R$ & $\swte{(2R)}{(1^R)}$ & $1,\ldots, R$  \rule[-5pt]{0pt}{11pt} \\  \toprule
(2,1)  & 0 & 2   & $\swte{\varnothing}{(3,3)}$ & $1,2$ \\
       &   &     & $\swte{\varnothing}{(4,1,1)}$ & $1$ \\ \midrule
       &   & 3   & $\swte{\varnothing}{(5,3,1)}$ & $1$ \\ 
       &   &     & $\swte{\varnothing}{(6,1,1,1)}$ & $1$ \\ \midrule
       &   & 4   & $\swte{\varnothing}{(7,3,1,1)}$ & $1$ \\ 
       &   &     & $\swte{\varnothing}{(8,1,1,1,1)}$ & $1$ \\ \midrule      
       &   & 5   & $\swte{\varnothing}{(7,5,3)}$ & $1$ \\ 
       &   &     & $\swte{\varnothing}{(9,3,1,1,1)}$ & $1$ \\ \toprule
       & 1 & $R$ & $\swte{(2R)}{(1^R)}$  & $1, \ldots, R$\rule[-5pt]{0pt}{11pt} \\  \bottomrule
\end{tabular}}      \qquad
\scalebox{0.8}{
\begin{tabular}{cccll}
\toprule 
$\muS$ & $\ell^\-$ & $R$ & $(\kappa^\-, \kappa^\+)$ & $c^\+$\rule[-5pt]{0pt}{11pt} \\  \toprule  
(4)    & 0 & 2   & $\swte{\varnothing}{(7,1)}$    & $1,2$ \\ \midrule
       &   & 3   & $\swte{\varnothing}{(9,3)}$    & $1,2$ \\ 
       &   &     & $\swte{\varnothing}{(10,1,1)}$ & $1$  \\ \midrule
       &   & 4   & $\swte{\varnothing}{(10,6)}$   & $1,2$ \\ 
       &   &     & $\swte{\varnothing}{(12,3,1)}$ & $1$ \\
       &   &     & $\swte{\varnothing}{(13,1,1,1)}$ & $1$ \\ \midrule
       &   & 5   & $\swte{\varnothing}{(10,10)}$    & $1,2$ \\ 
       &   &     & $\swte{\varnothing}{(14,4,2)}$   & $1$ \\
       &   &     & $\swte{\varnothing}{(15,3,1,1)}$ & $1$ \\ 
       &   &     & $\swte{\varnothing}{(16,1,1,1,1)}$ & $1$ \rule[-5pt]{0pt}{11pt}\\ \toprule  
       & 2 & 2   & $\swte{(2,2)}{(3,1)}$ & $1,2$ \\ \midrule
       &   & 3   & $\swte{(3,3)}{(3,3)}$ & $1,2$ \\ 
       &   &     & $\swte{(3,3)}{(4,1,1)}$ & $1$ \\ \midrule
       &   & 4   & $\swte{(4,4)}{(4,3,1)}$ & $1,2,3$ \\ 
       &   &     & $\swte{(4,4)}{(5,1,1,1)}$ & $1$ \\ \midrule
       &   & 5   & $\swte{(5,5)}{(4,4,2)}$  & $1,2,3$ \\ 
       &   &     & $\swte{(5,5)}{(5,3,1,1)}$ & $1$ \\ 
       &   &     & $\swte{(5,5)}{(6,1,1,1,1)}$ & $1$ \\ \bottomrule
\end{tabular}}

\smallskip
\hspace*{-0.3in}\scalebox{0.8}{\begin{tabular}{cccll}
\toprule 
$\muS$ & $\ell^\-$ & $R$ & $(\kappa^\-, \kappa^\+)$ & $c^\+$\rule[-5pt]{0pt}{11pt} \\  \toprule       
(3,1)  & 0 & 2   & $\swte{\varnothing}{(5,3)}$ & $1,2$ \\ 
       &   &     & $\swte{\varnothing}{(6,1,1)}$ & $1,2$ \\ \midrule
       &   & 3   & $\swte{\varnothing}{(8,3,1)}$ & $1$ \\ 
	   &   &     & $\swte{\varnothing}{(9,1,1,1)}$ & $1$ \\ \midrule
	   &   & 4   & $\swte{\varnothing}{(11,3,1,1)}$ & $1,2$ \\
	   &   &     & $\swte{\varnothing}{(12,1,1,1,1)}$ & $1,2,3$ \\ \midrule
	   &   & 5   & $\swte{\varnothing}{(12,5,3)}$ & $1$ \\ 
	   &   &     & $\swte{\varnothing}{(14,3,1,1,1)}$ & $1$\rule[-5pt]{0pt}{11pt}  \\ \toprule 
       & 1 & 2   & $\swte{(4)}{(3,1)}$ &  $1,2$ \\ \midrule
       &   & 3   & $\swte{(6)}{(3,3)}$ & $1,2,3$ \\ 
       &   &     & $\swte{(6)}{(4,1,1)}$ & $1$ \\ \midrule
       &   & 4   & $\swte{(8)}{(4,3,1)}$ & $1,2,3$ \\ 
       &   &     & $\swte{(8)}{(5,1,1,1)}$ & $1$ \\ \midrule
       &   & 5   & $\swte{(10)}{(4,4,2)}$ & $1,2,3,4$ \\ 
       &   &     & $\swte{(10)}{(5,3,1,1)}$ & $1$ \\
       &   &     & $\swte{(10)}{(6,1,1,1,1)}$ & $1$ \rule[-5pt]{0pt}{11pt} \\ \toprule
(2,2)  & 0 & 2   & $\swte{\varnothing}{(4,3,1)}$ & $1,2,3$ \\ \midrule
       &   & 3   & $\swte{\varnothing}{(6,3,3)}$ & $1$ \\ 
	   &   &     & $\swte{\varnothing}{(5,5,2)}$ & $1$   \\ \midrule
 &   & 4   & $\swte{\varnothing}{(7,5,4)}$ & $1,2$ \\ 
	   &   &     & $\swte{\varnothing}{(8,4,3,1)}$ & $1$ \\ 
	   &   &     & $\swte{\varnothing}{(7,6,2,1)}$ & $2$ \\ 
	   &   &     & $\swte{\varnothing}{(7,5,3)}$ & $3$             
\\ \bottomrule
\end{tabular}}\qquad
\scalebox{0.8}{\begin{tabular}{cccll}
\toprule 
$\muS$ & $\ell^\-$ & $R$ & $(\kappa^\-, \kappa^\+)$ & $c^\+$\rule[-5pt]{0pt}{11pt} \\  \toprule 
(2,2)& 0   & 5   & $\swte{\varnothing}{(8,6,6)}$ & $1,2,3$ \\
	   &   &     & $\swte{\varnothing}{(10,4,4,2)}$ & $1$ \\ 
	   &   &     & $\swte{\varnothing}{(8,8,2,2)}$ & $2$\rule[-5pt]{0pt}{11pt} \\ \toprule        
(2,2)  & 1 & 2   & $\swte{(4)}{(2,1,1)}$ &  $1,2,3$ \\ \midrule
       &   & 3   & $\swte{(6)}{(2,2,2)}$ & $1,2,3$ \\ 
       &   &     & $\swte{(6)}{(3,1,1,1)}$ & $1$ \\ \midrule
       &   & 4   & $\swte{(8)}{(3,2,2,1)}$ & $1,2,3,4$ \\ 
       &   &     & $\swte{(8)}{(4,1,1,1,1)}$ & $1$ \\
       &   & 5   & $\swte{(10)}{(3,3,2,2)}$ & $1,2,3,4$ \\ 
       &   &     & $\swte{(10)}{(4,2,2,1,1)}$ & $1$\rule[-5pt]{0pt}{11pt}  \\ \toprule        
(3,2)/(1)  & 0 & 2   & $\swte{\varnothing}{(6,1,1)}$ & $1$ \\ \midrule
       &   & 3   & $\swte{\varnothing}{(7,5)}$ & $1,2$ \\ 
	   &   &     & $\swte{\varnothing}{(9,1,1,1)}$ & $1$ \\ \midrule
	   &   & 4   & $\swte{\varnothing}{(8,8)}$ & $1,2$ \\
	   &   &     & $\swte{\varnothing}{(12,1,1,1,1)}$ & $1$ \\ \midrule
	   &   & 5   & none & \rule[-5pt]{0pt}{11pt}  \\ \toprule
	   & 1 & 2   & none & \\ \midrule
	   &   & 3   & $\swte{(6)}{(4,2)}$ & $1,2$ \\ \midrule 
	   &   & 4   & $\swte{(8)}{(4,4)}$ & $1,2$ \\ \midrule 
	   &   & 5   & $\swte{(10)}{(6,2,2)}$ & $1$ 
\\ \bottomrule
\end{tabular}}

\smallskip
\begin{center}
\begin{minipage}{4.5in}
\small {\sc Table 4.23}
Column-type strongly $c^\+$-maximal signed weights
for certain shapes $\muS$  and sizes~$R$ with $2 \le R \le 5$.
\end{minipage}
\end{center}
\medskip

\noindent 
required. Again this section is not logically necessary, but we believe it is important to explain
what we \emph{cannot} hope to prove. In the following example we shall need to use
Proposition~\ref{prop:plethysticSignedKostkaNumbers}.

\setcounter{theorem}{23}
\begin{example}\label{ex:needMaximalBad}
Taking $\muS = (2,1) / (1)$ as in Example~\ref{ex:needMaximalGood},
 suppose instead 
we take the non-maximal
signed weight $\bigl( (1), (1) \bigr)$, dominated in the $1$-signed dominance order (see Definition~\ref{defn:ellSignedDominanceOrder}) by $\bigl( (2), \varnothing\bigr)$,
 and, to give the simplest possible example, $\nu = (1)$ and $\lambda = (2)$.
Since this signed weight has sign $-1$ and $(2) \hskip0.5pt\oplus\hskip0.5pt (N-1)\bigl( (1), (1) \bigr) = 
\bigl((2) \hskip0.5pt \sqcup (1^{N-1})\hskip0.5pt  \bigr)+ (N-1)  = (N+1,1^{N-1})$,
the prediction of Theorem~\ref{thm:nuStable} --- wrongly
applied with a weight that is not even maximal --- is that 
$\langle s_{(1^N)} \circ s_{(2,1)/(1)}, s_{(N+1,1^{N-1})}\rangle$
is ultimately constant. 
To see this is false,
let 
$t_{+-}, t_{-+}$ and $t_{++}$ be the three 
 semistandard signed tableaux of shape $(2,1)/(1)$ shown below
\[ \young(:1,\oM)\, , \quad \young(:\oM,1)\, , \quad \young(:1,1)\, .\] 
(Again $\mbf{1}$ stands for $-1$.)
For each $N \in \N_0$ and $L \in \{0,\ldots, N-1\}$
there is a unique 
plethystic semistandard signed tableau of outer shape~$(1^N)$ and inner shape $(2,1)/(1)$
which has $L$ inner tableaux $t_{\+\-}$, $N-1-L$ inner tableaux $t_{-+}$
and a final inner tableau $t_{\+\+}$. (Note that only inner tableaux of negative sign are repeated.)
These are all the plethystic semistandard signed tableaux of signed weight $\bigl( (N+1), (N-1) \bigr)$
and  so $\bigl|\PSSYT\bigl( (1^N), (2,1)/(1) \bigr)_{(N-1),(N+1)}\bigr| = N$. 
By Proposition~\ref{prop:plethysticSignedKostkaNumbers} (Plethystic Signed Kostka Numbers)
it follows that 
\[ \langle s_{(1^N)} \circ s_{(2,1)/(1)}, e_{(N-1)}h_{(N+1)}\rangle = N.\] 
Thus 
condition (ii) in the Signed Weight Lemma (Lemma~\ref{lemma:SWL}) does not hold when the lemma
is applied (as is usual in this paper) with the twisted
symmetric functions defined in Definition~\ref{defn:ellTwistedSymmetricFunction};
this is the first point where the proof can be seen to fail.
Moreover, by a very similar enumeration of plethystic tableaux one can show that
$\langle s_{(1^N)} \circ s_{(2,1)/(1)}, e_{(N-d)}h_{(N+d}\rangle = 0$
for each $d > 1$ and $N \ge d$;
it now follows from the identity 
\[ s_{(N+1,1^{N-1})} = e_{(N-1)}h_{(N+1)} - e_{(N-2)}h_{(N+2)} + \cdots + (-1)^{N-1} h_{(2N)} \]
that 
$\langle s_{(1^N)} \circ s_{(2,1)/(1)}, s_{(N+1,1^{N-1})} \rangle = N$
for all $N \in \N_0$, showing that in fact the plethysm coefficient is not stable.
\end{example}

The non-uniqueness seen in Example~\ref{ex:needMaximalBad}
is completely typical of the non-maximal case, and always leads
to a similar obstruction to our proof strategy; indeed in most such cases, there is no stability result
to be proved. But as the following example shows, mere maximality is not enough.

\begin{example}\label{ex:whyStronglyMaximalSignedWeights}
The two tableau families of shape $(2,1)$, size $4$ and
signed weight $\bigl( \varnothing, (6,4,2) \bigr)$ are
\begin{align*}
\mathcal{S} &= \left\{ \,\young(11,2)\spy{0pt}{,} \young(11,3)\spy{0pt}{,} 
           \young(12,2)\spy{0pt}{,} \young(12,3)\, \right\} \\
\mathcal{T} &=\left\{ \,\young(11,2)\spy{0pt}{,} \young(11,3)\spy{0pt}{,} 
           \young(12,2)\spy{0pt}{,} \young(13,2)\, \right\} .
%
%
\end{align*}
Each of $\mathcal{S}$ and $\mathcal{T}$
is the set of entries of a 
unique plethystic semistandard signed tableau 
of outer shape $(1^4)$ and inner shape $(2,1)$, and so
\smash{$\bigl| \PSSYT\bigl( (1^4), (2,1)
\bigr)_{(\varnothing, (6,4,2))}\bigr| = 2$}.
Moreover, there is no tableau family of shape $(2,1)$ and size $4$ with
signed weight strictly dominating $\bigl( \varnothing, (6,4,2) \bigr)$. (Note that
such a family has only positive entries.)
But, by the uniqueness part of
Lemma~\ref{lemma:stronglyMaximalSemistandardSignedTableauFamilyIsUnique},
the signed weight $\bigl(\varnothing, (6,4,2)\bigr)$ is not strongly maximal.
Theorem~\ref{thm:nuStable} is therefore inapplicable. If, ignoring that one of the hypotheses
fails to hold, we nonetheless
take $\nu = \varnothing$, $\mu = (2,1)$ and $\lambda = \varnothing$,
we wrongly conclude that $\langle s_{(M,M,M,M)} \circ s_{(2,1)}, s_{(6M,4M,2M)} \rangle$
is ultimately constant. 

To see this is false, first note, analogously to Example~\ref{ex:needMaximalBad}, that given $0 \le L \le M$, there
is a unique plethystic semistandard tableau $T_L$ of outer shape $(M,M,M,M)$ whose
first $L$ columns have entries $\mathcal{S}$ and whose final $M-L$ columns have entries
$\mathcal{T}$; the families occur in this order because, as seen 
after Definition~\ref{defn:signedColexicographicOrder}, we have
\[ \young(12,3) \,<\, \young(13,2) \]
in the signed colexicographic order.
The maximal tableau families of shape $(2,1)$, size $4$ and sign $+1$ have weights, in the usual sense for unsigned tableaux, as defined
after Definition~\ref{defn:signedWeight},
$(8,1,1,1,1), (7,3,1,1)$ and $(6,4,2)$. Since $(6,4,2)$ has the
least number of parts, it need not be compared with $(8,1,1,1,1)$ or $(7,3,1,1)$
in Definition~\ref{defn:stronglyMaximalSignedWeight}(c),
and so the \emph{only} reason why $\bigl(\varnothing,
(6,4,2)\bigr)$ fails
to be a strongly maximal signed weight is
that $\mathcal{S}$
and $\mathcal{T}$ have the same weight.
Since the signed weight $\bigl( \varnothing, (6,4,2) \bigr)$ \emph{is} maximal,
\emph{any} tableau family of shape $(2,1)$, size $4$, sign $+1$ and entries from $\{1,2,3\}$
has weight dominated by $\bigl(\varnothing, (6,4,2)\bigr)$. Hence
\begin{equation}\label{eq:whyStrictlyTL} 
\PSSYT\bigl( (1^4), (2,1) \bigr)_{(6M,4M,2M)} = \{ T_L : 0 \le L \le M \} \end{equation}
and $\PSSYT\bigl( (1^4), (2,1) \bigr)_\pi =\varnothing$ if
$\pi \rhd (6M,4M,2M)$.
By the basic result
on Kostka numbers mentioned before Lemma~\ref{lemma:twistedKostkaNumbers}
we have  $s_{(6M,4M,2M)} = h_{(6M,4M,2M)} + f$ where $f$ is a linear combination
of complete homogeneous symmetric functions $h_\pi$ with $\pi \rhd (6M,4M,2M)$.
Hence 
\[ \begin{split} \quad \langle s_{(M,M,M,M)} \circ &s_{(2,1)}, s_{(6M,4M,2M)} \rangle = 
\langle s_{(M,M,M,M)} \circ s_{(2,1)}, h_{(6M,4M,2M)} \rangle \\  
& =\bigl|\PSSYT\bigl( (M,M,M,M), (2,1) \bigr)_{(\varnothing,(6M,4M,2M))}\bigr| = M+1. \end{split} \]
In particular, the multiplicity is unbounded.

The previous paragraph also indicates where the proof of Theorem~\ref{thm:nuStable} breaks down.
Applying Definition~\ref{defn:exceptionalColumnAndRow} with respect
to the signed weight $\bigl(\varnothing, (6,4,2)\bigr)$ --- which according to this definition is illegitimate
as the signed weight is not strongly maximal --- each non-exceptional 
column 
of a plethystic semistandard tableau
$T \in \PSSYT\bigl( (M,M,M,M), (2,1)\bigr)$ may be either $\mathcal{S}$ or $\mathcal{T}$. Thus
there is no canonical tableau family that can be inserted
as the entries in a new column of height $4$ 
 to define a bijection
between the sets of plethystic semistandard signed tableaux for $M$ and $M+1$
and the key result Lemma~\ref{lemma:columnIsExceptionalOrBounded}(i)
fails in this attempt to adapt the proof of Theorem~\ref{thm:nuStable} in~\S\ref{subsec:nuStableProof}.
\end{example}

We remark that an alternative way to see that the conclusion of Theorem~\ref{thm:nuStable}
is false in the previous example uses 
the highest-weight vector methods in~\cite{deBoeckPagetWildon}.
Let $\nabla^\gamma$ denote the Schur functor for the partition $\gamma$, let~$E$
be a $3$-dimensional vector space, let $T_L$ be as defined in the previous example,
and let $F(T_L)$ be as defined in Definition~2.3
of~\cite{deBoeckPagetWildon}.
Generalizing Example~7.5 in~\cite{deBoeckPagetWildon}, for each $L$, the vector
\[ F(T_L) \in \nabla^{(M,M,M,M)} \bigl( \nabla^{(2,1)}(E)\bigr) \]
is highest weight of weight $(6M, 4M, 2M)$. The vectors $F(T_L)$ for $0 \le L \le M$ are linearly
independent because the multisets of semistandard $(2,1)$-tableau entries of each $T_L$ are distinct.
It again follows that
$\langle s_{(M,M,M,M)} \circ s_{(2,1)}, s_{(6M,4M,2M)} \rangle \ge M+1$
for each $M \in \N_0$.

\begin{example}\label{ex:17family}
In the previous example the problem was that there were two semistandard signed tableau families
of the same maximal weight. The other potential problem solved by
Definition~\ref{defn:stronglyMaximalSignedWeight} is seen only in relatively large
examples, such as the following.
Take $\mu = (3)$. There
are three maximal semistandard signed tableau families of shape~$(3)$ 
and size $17$ having
only positive entries,
each obtained from

\medskip
\centerline{ \scalebox{0.9}{$\young(111)\, ,\; \young(112)\, ,\;
\young(122)\, ,\; \young(222)\, ,\; \young(113)\, ,\;  \young(123)\, ,\;
\young(223)\, ,\;  \young(133)\, ,\; \young(233)$ }}

\medskip
\noindent by taking the union with the  eight tableaux shown in the table below.

\medskip
\begin{center}
\begin{tabular}{rl} \toprule 
Signed weight & Extend by  \\ \midrule 
$\bigl( \varnothing, (17,16,11,4,3) \bigr)$  & 
$\begin{matrix}
\scalebox{0.9}{$\young(333)\, , \; \young(114)\, , \; \young(124)\, , \; 
\young(224)\, , \; $} \\
\scalebox{0.9}{$\young(134)\, , \; \young(115)\, , \; \young(125)\, , \; 
\young(225) \phantom{\,,\;} $} \end{matrix} $
\\[10pt]
$\bigl( \varnothing, (18,15,10,5,3) \bigr)$  & 
$\begin{matrix}
\scalebox{0.9}{$\young(114)\, , \; \young(124)\, , \; \young(224)\, , \; 
\young(134)\, , \; $} \\
\scalebox{0.9}{$\young(234)\, , \; \young(115)\, , \; \young(125)\, , \; 
\young(135) \phantom{\,,\;} $} \end{matrix} $
\\[10pt]
$\bigl( \varnothing, (19,14,9,6,3) \bigr)$  & 
$\begin{matrix}
\scalebox{0.9}{$\young(114)\, , \; \young(124)\, , \; \young(224)\, , \; 
\young(134)\, , \; $} \\
\scalebox{0.9}{$\young(144)\, , \; \young(115)\, , \; \young(125)\, , \; 
\young(135) \phantom{\,,\;} $} \end{matrix} $
\\[9pt]
\bottomrule\end{tabular}
\end{center}

%
%
%

\medskip

\noindent That these families \emph{are} maximal can be checked by hand, or more quickly,
using the Haskell software \cite{PZYT} mentioned in the introduction
using \verb!display $! \verb!maximalTableauFamilies ColType Closed 17 (ssyts 5 [3])!.
Let $T$, $U$ and $V$ be the plethystic semistandard signed tableaux
of outer shape $(1^{17})$ and inner shape $3$ having as their 
entries the three families above. From the table above which
lists the $(3)$-tableau entries in the signed colexicographic
order from Definition~\ref{defn:signedColexicographicOrder},
one can see that any plethystic tableau of the form
\begin{equation} \raisebox{-5pt}{\begin{tikzpicture}[line width=0.6pt, x=0.6cm,y=-0.6cm] 
\draw(0,0)--(12,0)--(12,1)--(0,1)--(0,0);
\draw(1,0)--(1,1); \draw(3,0)--(3,1); \draw(4,0)--(4,1); \draw (5,0)--(5,1);
\draw(7,0)--(7,1); \draw(8,0)--(8,1); \draw(9,0)--(9,1);
\draw(11,0)--(11,1);
\node at (0.5,0.5) {$T$}; 
\node at (2,0.5) {$\cdots$};
\node at (3.5,0.5) {$T$}; 
\node at (4.5,0.5) {$U$}; 
\node at (6,0.5) {$\cdots$};
\node at (7.5,0.5) {$U$}; 
\node at (8.5,0.5) {$V$};
\node at (10,0.5) {$\cdots$}; 
\node at (11.5,0.5) {$V$}; 
\end{tikzpicture}} \label{eq:TUV} \end{equation}

\smallskip
\noindent is semistandard. (Here there is a mild abuse of notation: the columns, each of
length $17$ are $T$, $U$ and $V$; these plethystic semistandard tableaux are not themselves
entries in boxes.)
Observe that 
\[ 2(18,15,10,5,3) = (19,14,9,6,3) + (17,16,11,4,3). \]
Thus any two of the $2N$ columns $\young(UU)$ of length $17$ 
in a plethystic semistandard signed
tableau of
outer shape $(2N,\stackrel{17}{\ldots}, 2N)$ and inner shape $(3)$
may be replaced with two columns $\young(TV)$
without changing the weight. (The columns must then be reordered
as in~\eqref{eq:TUV} to respect the semistandard condition).
Hence there are at least $(N+1)$ plethystic semistandard signed tableaux
of outer shape  $(2N,\stackrel{17}{\ldots}, 2N)$ whose  signed weight is $2(18N,15N,10N,$ $5N,3N)$.
Similar arguments to the previous Example~\ref{ex:whyStronglyMaximalSignedWeights} now show
the plethysm coefficients $\langle s_{(M^{17})} \circ s_{(3)}, s_{(18M,15M,10M,5M,3M)} \rangle$
for even $M$
do not stabilise, even though the relevant signed weight is maximal.
\end{example}

We remark that each tableau family in
the previous example is a downset 
for the majorization partial order $\preceq$ 
defined by comparing tableaux entry by entry; this is a necessary, but not in general sufficient, condition for
maximality. For instance the family of weight $(17,16,11,4,3)$ 
is $\young(225)^{\preceqtab} \,\cup\, \young(134)^{\preceqtab} \,\cup\,\young(333)^{\preceqtab}$ with three incomparable maximals in the
dominance order. This leads to an efficient algorithm 
implemented in \cite{PZYT} for 
finding maximal, and so strongly maximal, tableau families.


\section{Symmetric functions and plethystic semistandard signed tableaux}\label{sec:symmetricFunctions}

\subsection{Basic results} \label{subsec:symmetricFunctionsBasics}
We refer the reader to Stanley's textbook
\cite[Ch.~7]{StanleyII} for an introduction to the Hopf algebra
$\Lambda$ of symmetric functions
and to~\cite{LoehrRemmel} for a careful development of plethysm and the
formalism of plethystic substitutions. We define the elementary and homogeneous
symmetric functions $e_\pi$ and~$h_\pi$ for arbitrary weights $\pi \in \W$
(as defined at the start of \S\ref{sec:preliminaryDefinitions})
while Schur functions $s_\lambda$ are labelled by partitions as usual. 
Beyond very basic results, the minimum we require 
is:
\begin{bulletlist}
\item Young's rule (horizontal strip addition) and Pieri's
rule (vertical strip addition) as stated in (7.65) and after Example~7.15.8 in \cite{StanleyII};
\item the coproduct $\Delta$ on $\Lambda$
satisfies $\Delta s_{\lambda/\lambda_\star} = \sum_\tau s_{\tau/\lambda_\star} \otimes s_{\lambda / \tau}$ and 
is compatible with the inner product on $\Lambda$
(see the proof of Lemma~\ref{lemma:twistedKostkaNumbers});
\item the formal definition of substitution by an alphabet with mixed signs, namely
\begin{equation}
\label{eq:negativePositiveSubstitution} 
\qquad f[-x_1,-x_2,\ldots, y_1,y_2,\ldots] = \sum_i f^\-_i[-x_1,-x_2, \ldots] f_i^\+[y_1,y_2,\ldots]\\[-3pt]
\end{equation}
where $\Delta f = \sum_i f_i^\- \otimes f_i^\+$
(this follows from the equation for $s_{\lambda/\nu}[A-B]$ on
page 177 of \cite{LoehrRemmel}, using the result on the coproduct
just mentioned, and that the Schur functions $s_\lambda$ are a basis
for $\Lambda$);
\item  the \emph{negation rule} 
\begin{equation}
\label{eq:negationRule} 
\qquad s_\lambda [-x_1,-x_2, \ldots ] = (-1)^{|\lambda|} s_{\lambda'}[x_1,x_2, \ldots]
\end{equation}
which is a special case of \cite[Theorem~6]{LoehrRemmel};
\item the rule for a general plethystic substitution into a Schur function
given in \cite[Theorem 10]{LoehrRemmel}.
\end{bulletlist}
We shall not state the final rule here, since it is lengthy and
we only need it once, in the proof of Lemma~\ref{lemma:negativePositivePlethysticSpecialisation}; the reader
may then either refer to \cite{LoehrRemmel}, or take it on trust that it has the effect we claim.

\begin{remark}\label{remark:nuSkew}
Let $\nu/\nus$ be a skew partition. By the adjointness relation in Corollary 7.15.4 of \cite{StanleyII}
we have $s_{\nu/\nus} = \sum_\sigma
\langle s_\nu, s_\nus s_\sigma\rangle s_\sigma$ where the sum is over all partitions
$\sigma$ of $|\nu/\nus|$.
Since the plethysm product is linear in its first component, i.e.~$(f + g)\circ h = f \circ h + g \circ h$
for all $f,g,h \in \Lambda$, it follows that for any skew partition $\muS$,
\[ s_{\nu/\nus} \circ s_\muS = \sum_{\sigma} \langle s_\nu, s_\nus s_\sigma \rangle s_\sigma \circ s_\muS \]
with the same condition on the sum. This reduces an arbitrary plethysm
of skew Schur functions to the case dealt with in this paper. On the other hand,
since the plethysm product is \emph{not} linear in its second component
(this is already clear from the negation rule)
there is no further reduction to plethysm products where both factors are Schur functions.
\end{remark}

The following definition is standard.

\begin{definition}\label{defn:supp}
Given a symmetric function $f$ expressed in the Schur
basis as $\sum_{\lambda} c_\lambda s_\lambda$ 
we define the \emph{support} of $f$ by 
$\supp(f) = \{ \lambda \in \Par : c_\lambda \not= 0 \}$.
\end{definition}

For example by~\eqref{eq:e2h51},
we have $\supp(e_{(2)}h_{(5,1)}) = \bigl\{
(6,2), (7,1), (6,1,1),$ $(5,2,1), (5,1,1,1) \bigr\}$.

\subsection{Enumerating semistandard signed tableaux}
Our first result is the twisted generalization of the basic result, see for instance
\cite[(7.30), (7.36)]{StanleyII} that $\langle s_{\beta/\betas}, h_\lambda \rangle$ is the number 
of semistandard tableaux of shape $\beta/\betas$ and weight~$\lambda$.
When $\betas = \varnothing$,
this quantity may also be familiar as the Kostka number $K_{\beta \lambda}$.
In our signed weight notation it is $|\SSYT(\beta/\betas)_{(\varnothing, \lambda)}|$.
(Semistandard signed Young tableaux are defined in Definition~\ref{defn:semistandardSignedTableau}
and their signed weights in Definition~\ref{defn:signedWeightTableau}.)

\begin{lemma}[Twisted Kostka numbers for skew shapes]\label{lemma:twistedKostkaNumbers}
Let $\beta/\betas$ be a skew partition and let $(\gamma^\-,\gamma^\+)$ be a signed weight of size 
$|\beta/\betas|$.
Then 
\[ \langle s_{\beta/\betas}, e_{\gamma^\-}h_{\gamma^\+} \rangle = |\mSSYT(\beta/\betas)_{(\gamma^\-,\gamma^\+)}|.\]
In particular $\beta \in \supp(e_{\gamma^\-}h_{\gamma^\+})$
if and only if\, $\mSSYT(\beta)_{(\gamma^\-,\gamma^\+)}$ is non-empty. 
\end{lemma}

\begin{proof}
The inner product on $\Lambda \otimes \Lambda$ is defined in the natural way by 
linear extension of $\langle f \otimes f', g \otimes g' \rangle = \langle f, g \rangle \langle f', g' \rangle$. In the following two steps we use subscripts to indicate
the relevant inner product.
By the identity $\langle f, gh \rangle_\Lambda = \langle \Delta f, g \otimes h \rangle_{\Lambda \otimes \Lambda}$ we have
\[ \langle s_{\beta/\betas}, e_{\gamma^\-}h_{\gamma^\+} \rangle_\Lambda =
\langle \Delta s_{\beta/\betas}, e_{\gamma^\-} \otimes h_{\gamma^\+} \rangle_{\Lambda \otimes \Lambda}
= \sum_{\tau} \langle s_{\tau/\betas} \otimes s_{\beta / \tau}, e_{\gamma^\-} \otimes h_{\gamma^\+} \rangle_{\Lambda \otimes \Lambda}
\]
where the sum is over all partitions $\tau$ of $|\betas| + |\gamma^\-|$.
The right-hand side is 
$\sum_\tau \langle s_{\tau/\betas}, e_{\gamma^\-} \rangle_\Lambda
\langle s_{\beta / \tau}, h_{\gamma^\+} \rangle_\Lambda$.
By the remark before the proof, the second factor is 
$|\mSSYT(\beta/\tau)_{(\varnothing, \gamma^\+)}|$.
Applying the omega-involution (see \cite[Theorem~7.14.5]{StanleyII})
to the first factor gives 
\[ \langle s_{\tau/\betas}, e_{\gamma^\-} \rangle = 
\langle s_{\tau'/\beta_\star'}, h_{\gamma^\-}
\rangle.\] 
The right-hand side is the number of semistandard tableaux of shape $\tau'/\beta_\star'$ with \emph{positive}
entries of weight $\gamma^\-$, and so equal to $|\mSSYT(\tau/\beta_\star)_{(\gamma^\-,\varnothing)}|$
by the obvious bijection conjugating tableaux and switching signs of
the integer entries.
Since negative entries always precede positive entries in the order in 
Definition~\ref{defn:semistandardSignedTableau}, the pairs of tableaux enumerated
by the two factors are in bijection
with $\SSYT(\beta/\betas)_{(\gamma^\-,\gamma^\+)}$. The final claim is now immediately obvious
on taking $\betas=\varnothing$.
\end{proof}

In the following two lemmas we use the standard notation $y^\gamma$ for $y_1^{\gamma_1}\ldots y_{\ell(\gamma)}^{\gamma_{\ell(\gamma)}}$
and $(-x)^\gamma$ for \smash{$(-x_1)^{\gamma_1} \ldots (-x_{\ell(\gamma)})^{\ell(\gamma)}$},
where $\gamma$ is a composition.

\begin{lemma}\label{lemma:signedMonomials}
Let $f$ be a symmetric function and let $(\alpha^\-,\alpha^\+)$ be a signed weight
of size $\deg f$ where $\alpha^\-$, $\alpha^\+$ are partitions. 
Then $\langle f, e_{\alpha^\-}h_{\alpha^\+} \rangle$ is the 
coefficient of $(-x)^{\alpha^\-}y^{\alpha^\+}$ in $f[-x_1,-x_2,\ldots, y_1,y_2,\ldots ]$.
\end{lemma}

\begin{proof}
Let $\Delta(f) = \sum_i f_i^\- \otimes f_i^\+$. 
Since $f$ is a symmetric function, so are each $f_i^\-$ and $f_i^\+$.
By \cite[(7.30)]{StanleyII}, 
the complete homogeneous and monomial symmetric functions are dual bases of $\Lambda$.
Hence the coefficient of $y^{\alpha^\+}$ in $f_i^\+[y_1,y_2,\ldots ]$ is
$\langle f^\+_i, h_{\alpha^\+} \rangle$.
Similarly, now also using  the negation rule~\eqref{eq:negationRule}, 
the coefficient of $(-x)^{\alpha^\-}$ in $f_i^\-[-x_1,-x_2,\ldots ]$ is $\langle f_i^\-, e_{\alpha^\-}
\rangle$. The lemma now follows by applying~\eqref{eq:negativePositiveSubstitution}.
\end{proof}

\subsection{Enumerating plethystic semistandard signed tableaux}
We can now extend Lemma~\ref{lemma:twistedKostkaNumbers} (Twisted Kostka Numbers)
 to plethystic signed tableaux.

\begin{lemma}\label{lemma:negativePositivePlethysticSpecialisation}
Let $\nu$ be a partition and let $\muS$ be a skew partition. Then
\[\begin{split} (s_\nu \circ s_\muS)&[-x_1,-x_2,\ldots, y_1,y_2, \ldots ] \\
&= \sum_{(\alpha^\-,\alpha^\+)}
|\mPSSYT(\nu,\muS)_{(\alpha^\-,\alpha^\+)}|\, (-x)^{\alpha^\-}\!y^{\alpha^\+} \end{split} \]
where the sum is over all signed weights $(\alpha^\-,\alpha^\+)$ of size $|\nu|  |\muS|$. 
Moreover, $\bigl|\PSSYTw{\nu}{\muS}{\alpha^\-}{\alpha^\+}\bigr| = \bigl|\srPSSYTw{\nu}{\muS}{\alpha^\-}{\alpha^\+}\bigr|$.

\end{lemma}

\begin{proof}
By Lemma~\ref{lemma:twistedKostkaNumbers} (Twisted Kostka Numbers) 
and Lemma~\ref{lemma:signedMonomials} we have
\[ s_\muS[-x_1,-x_2,\ldots, y_1,y_2,\ldots ] = 
\sum_{(\alpha^\-,\alpha^\+)} |\mSSYT(\muS)_{(\alpha^\-,\alpha^\+)}|\, (-x)^{\alpha^\-} \!y^{\alpha^\+} \]
where the sum is over all signed weights $(\alpha^\-,\alpha^\+)$ of size $|\muS|$.
Therefore 
$(s_\nu \circ s_\muS)[-x_1,-x_2,\ldots, y_1,y_2,\ldots ] = 
s_\nu[\A]$
where the plethystic alphabet~$\A$ is all semistandard tableaux of shape $\muS$ having entries
from $\Z \backslash \{0 \}$ ordered by the signed colexicographic order in Definition~\ref{defn:signedColexicographicOrder}. 
Note this alphabet has formal symbols (i.e.~tableaux) of both positive and negative
sign and that the sign of a semistandard  signed $\muS$-tableau of weight $(\alpha^\-,\alpha^\+)$
from the displayed equation above, namely $(-1)^{|\alpha^\-|}$, agrees with the sign defined by
Definition~\ref{defn:signTableau}. Moreover, negative tableaux are less than positive tableaux.
Therefore, by the definition of general plethystic
substitution~\cite[Theorem~8]{LoehrRemmel},
taking $D$ to be the negative tableaux in $\A$ and $E$ to be the 
positive tableaux in $\A$, 
$S_\nu[\A]$ is the generating function enumerating,
by their signed weight, the 
$\nu$-tableau $T$ having entries from $\A$, such that 
for some subpartition~$\beta$ of $\nu$,
\begin{itemize}
\item[(i)] the negative entries in $T$ form a subtableau of shape $\beta$ and are strictly increasing
along rows and weakly increasing down columns;
\item[(ii)] the positive entries in $T$ form a subtableau of skew shape $\nu / \beta$ and are
weakly increasing along rows and strictly increasing down columns.
\end{itemize}
Since (i) implies that all negative entries of $T$ are in boxes above or left of the positive entries of $T$,
it follows that $T$ is a plethystic semistandard signed tableau of outer shape $\nu$ and inner shape $\muS$,
as defined in Definition~\ref{defn:plethysticSemistandardSignedTableau}.
Moreover, since the weight of a plethystic tableau
is, by Definition~\ref{defn:signedWeightPlethystic}, the sum of the weights of its $\muS$-tableau
entries, the sign attached to each plethystic tableau in $\PSSYT(\nu,\muS)_{(\alpha^\-,\alpha^\+)}$
is $(-1)^{|\alpha^\-|}$. This completes the proof of the displayed equation in the statement
of the lemma. For the second claim,
observe that we could instead order $\A$ by the sign-reversed colexicographic order,
and take~$D$ to be the positive tableaux in $\A$ and $E$ to be the negative tableaux in $\A$.
We then obtain the displayed equation, modified by replacing
$\bigl|\PSSYTw{\nu}{\muS}{\alpha^\-}{\alpha^\+}\bigr|$ with $\bigl|\srPSSYTw{\nu}{\muS}{\alpha^\-}{\alpha^\+}\bigr|$.
\end{proof}

\begin{proposition}[Plethystic Signed Kostka Numbers] 
\label{prop:plethysticSignedKostkaNumbers}
Let $\nu$ be a partition and let $\muS$ be a skew partition.
Let $(\alpha^\-,\alpha^\+)$ be a signed weight of size $|\nu|  |\muS|$.
\[ \langle s_\nu \circ s_\muS, e_{\alpha^\-}h_{\alpha^\+} \rangle = |\mPSSYT(\nu, \muS)_{(\alpha^\-,\alpha^\+)}|.\]
\end{proposition}

\begin{proof}
This is immediate from Lemma~\ref{lemma:negativePositivePlethysticSpecialisation} and
Lemma~\ref{lemma:signedMonomials}.
\end{proof}

The special case $\nu = (1)$ of the proposition just proved
recovers Lemma~\ref{lemma:twistedKostkaNumbers}. Note also that, by the final part
of Lemma~\ref{lemma:negativePositivePlethysticSpecialisation}, Proposition~\ref{prop:plethysticSignedKostkaNumbers}
implies that
\smash{$\langle s_\nu \!\circ\! s_\muS, e_{\alpha^\-}h_{\alpha^\+} \rangle \!=\! \bigl|\srPSSYTw{\nu}{\muS}{\alpha^\-}{\alpha^\+}\bigr|$}.
For example, the three
elements of $\PSSYTw{(2,2)}{(3)}{(3)}{(7,2)}$
were seen after Definition~\ref{defn:plethysticSemistandardSignedTableau}; these may be compared with
the three elements of $\srPSSYTw{(2,2)}{(3)}{(3)}{(7,2)}$ shown below in the sign-reversed colexicographic order
\[            \pyoung{1.6cm}{0.75cm}{{{\young(112), \young(\oM 11)}, {\young(\oM 11), \young(\oM 12)}}}
\spy{20pt}{,} \pyoung{1.6cm}{0.75cm}{{{\young(111), \young(\oM 12)}, {\young(\oM 11), \young(\oM 12)}}}
\spy{20pt}{,} \pyoung{1.6cm}{0.75cm}{{{\young(111), \young(\oM 11)}, {\young(\oM 11), \young(\oM 22)}}}
\spy{20pt}{.} \]



\begin{remark}\label{remark:otherOrder}
The only property of the signed colexicographic order we used in the proof of Lemma~\ref{lemma:negativePositivePlethysticSpecialisation}
was that negative tableaux are always less than positive tableaux. 
We are therefore free to use any other order $\preceq$ that has this property,
obtaining plethystic semistandard signed tableaux defined as in Definition~\ref{defn:plethysticSemistandardSignedTableau}, but 
whose inner tableaux are instead semistandard with respect to $\preceq$. An analogous
remark holds for the sign-reversed colexicographic order. We
use this freedom in the proof of Theorem~\ref{thm:nuStable}
(see Definition~\ref{defn:adaptedSignedColexicographicOrder}). 
\end{remark}

 We end this section with an immediate application.

\subsection{A generalized Cayley--Sylvester formula}\label{subsec:twoRow}
By Proposition~\ref{prop:plethysticSignedKostkaNumbers},
using that $s_{(k-\ell,\ell)}  = h_{(k-\ell,\ell)} - h_{(k-\ell+1,\ell-1)}$,
for any $\ell$ with $1 \le \ell \le k/2$, we have
\begin{equation}
\label{eq:twoRow} 
\begin{split}
\langle &s_\nu \circ s_\muS,  s_{(mn-\ell,\ell)} \rangle \\ &= |\mPSSYT(\nu, \muS)_{(\varnothing, (mn-\ell,\ell))}|
- |\mPSSYT(\nu, \muS)_{(\varnothing, (mn-\ell+1,\ell-1))} | 
\end{split}
\end{equation}
for $1 \le \ell \le mn/2$.
Special cases of~\eqref{eq:twoRow} have appeared throughout the literature on plethysms,
especially in the context of representations of $\SL_2(\C)$. The most important case
occurs when
$\nu = (n)$ and $\mu = (m)$. Observe that an element of
\smash{$\PSSYTw{(n)}{(m)}{\varnothing}{(mn-\ell,\ell)}$}
is determined by the 
number of $2$s in each of its $n$ $(m)$-tableau entries. The corresponding non-negative sequence
of length $n$
may be interpreted as a partition of $\ell$ having at most~$n$ parts, each part having
size at most $m$.
For example when $n=4$ and $m=4$, 
\[ \pyoung{2cm}{0.775cm}{ {{\young(1111), \young(1122), \young(1122), \young(1222)}} }
\spy{9pt}{$\;\longleftrightarrow (3,2,2)$.}\]
This bijection shows that $\langle s_{(n)} \circ s_{(m)}, s_{(mn-\ell,\ell)} \rangle$
is the number of partitions of~$\ell$ contained in an $m \times n$-box \emph{minus}
the number of partitions of \hbox{$\ell-1$} contained in an $m \times n$ box.  This is
one form of the Cayley--Sylvester identity. The bijective proof just given is similar to that
in \cite{GiannelliArchMath}, where it is derived using symmetric group methods.
When $m$ and $n$ are large, and so $\ell \le \min(m,n)$,
this simplifies to the number of partitions of $\ell$ having no parts of size $1$:
see \cite[Proposition 8.4]{BowmanPaget} for an earlier proof of this fact.
For many further applications of~\eqref{eq:twoRow}, and related results such as Stanley's Hook Content Formula, see~\cite{PagetWildonSL2}. 

\section{Twisted dominance order and twisted symmetric functions}\label{sec:twisted}

In this section we define the $\ell^\-$-twisted dominance order and 
twisted symmetric functions. Twisted symmetric functions
interpolate between the homogeneous and
elementary symmetric functions, in an analogous way
(see Remark~\ref{remark:twistedDominanceOrderGeneralizesDominanceOrder})
to the way the $\ell^\-$-twisted dominance orders interpolate between the dominance order and its opposite order.
This is made precise by Lemma~\ref{lemma:twistedKostkaMatrix}.
The $1$-twisted dominance order was seen informally in the overview in \S\ref{sec:overview}.

\subsection{$\ell^\-$-decomposition}
The following definition and notation is shown diagrammatically in Figure~\ref{fig:ellDecomposition}.

\begin{definition}\label{defn:ellDecomposition}
Fix $\ell^\- \in \N_0$. Given a partition $\sigma$, we set $\sigma^\- = (\sigma'_1,\ldots, \sigma'_{\ell^\-})$
and $\sigma^\+ = \sigma - \sigma^{-\prime}$. We say that the ordered pair
 $\decs{\sigma^\-}{\sigma^\+}$
 is the \emph{$\ell^\-$-decomposition
of $\sigma$} and write $\sigma \decMap \decs{\sigma^\-}{\sigma^\+}$.
\end{definition}

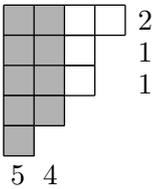
\begin{figure}[ht!]
\begin{center}
\begin{tikzpicture}[x=0.47cm,y=-0.47cm]
\newcommand{\xExtra}{2}
\newcommand{\yExtra}{1.5}
\newcommand{\lMx}{-0.5}
\newcommand{\lMy}{-0.5}
\draw (0,0)--(8+\xExtra,0)--(8+\xExtra,1)--(7+\xExtra,1)--(7+\xExtra,2)--(6+\xExtra,2);
\draw (0,0)--(0,7+\yExtra)--(2,7+\yExtra)--(2,6+\yExtra)--(3,6+\yExtra)--(3,5+\yExtra);
\draw (3,5)--(3,4);
\draw[dashed] (3,0)--(3,5);
\draw (3,4)--(5,4)--(5,3)--(6,3)--(6,2)--(6.5,2);
\node at (1.5,3) {$\sigma^{-\prime}$};
\node at (4,1.5) {$\sigma^\+$};
\node at (7.25,2) {$\ldots$};
\node at (3,5.5) {$\vdots$};

\draw (3,4)--(4,4)--(4,3)--(3,3)--(3,4);
\fill[color=gray] (3,4)--(4,4)--(4,3)--(3,3)--(3,4);

\draw[<-] (0,\lMy)--(1.0,\lMy);
\draw[->] (2,\lMy)--(3,\lMy);
\node at (1.55,\lMy-0.05) {$\scriptstyle \ell^\minus$};

\end{tikzpicture}
\end{center}
\caption{The partitions in the $\ell^\minus$-decomposition 
$\decs{\sigma^{\minus}}{\sigma^\+}$ of $\sigma \in \Par$. Note that $\sigma^\minus$ has at most $\ell^\minus$
parts and that $\sigma^\minus_{\ell^\minus} \ge \ell(\sigma^\+)$, so $\sigma$ is $\bigl( \ell(\sigma^\minus)+1,
\ell(\sigma^\+) \bigr)$-large in the sense of Definition~\ref{defn:large},
having $\bigl( \ell(\sigma^\+), \ell(\sigma^\minus) + 1 \bigr)$ as the
shaded box. In particular $\sigma$ is $\bigl( \ell(\sigma^\minus), \ell(\sigma^\+) \bigr)$-large.
\label{fig:ellDecomposition}
}
\end{figure}

The relevant $\ell^\-$ will always be clear from context. 
The $\ell^\-$-decomposition of a partition may be used as a signed weight (see 
Definition~\ref{defn:signedWeight}), but since
this is not always the case, we use angled brackets to make a visual distinction.
For example, the $0$-, $1$-, $2$-, $3$- and $4$-decompositions of $(4,3,3,2,1)$ are $\dec{\varnothing}{(4,3,3,2,1)}$, $\dec{(5)}{(3,2,2,1)}$,
$\dec{(5,4)}{(2,1,1)}$,
 $\dec{(5,4,3)}{(1)}$ and $\dec{(5,4,3,1)}{\varnothing}$; the 
 $2$-decomposition is shown in the margin, with the relevant part lengths marked.
\marginpar{
\begin{tikzpicture}[x=0.5cm, y=-0.5cm]
\node at (-1.1,2.5) {$\smaller{5}$}; \node at (-0.25,2.5) {$\smaller{4}$};
\node at (2.25,-1.6) {$\smaller{2}$};\node at (2.25,-0.75) {$\smaller{1}$};
\node at (2.25,0.1) {$\smaller{1}$};

\node at (0,0) {$\pyoung{0.4cm}{0.4cm}{ {{\youngBox{0.4}{grey}, \youngBox{0.4}{grey}, \ , \ },
{\youngBox{0.4}{grey}, \youngBox{0.4}{grey}, \ }, {\youngBox{0.4}{grey}, \youngBox{0.4}{grey}, \ },
{\youngBox{0.4}{grey}, \youngBox{0.4}{grey}}, {\youngBox{0.4}{grey}}} }$}; \end{tikzpicture}} 

\begin{remark}\label{remark:ellDecompositionLarge}
Not all ordered pairs of partitions are $\ell^\-$-decompositions: in fact
$(\sigma^\-, \sigma^\+)$ 
 is the $\ell^\-$-decomposition of a partition $\sigma$
if and only if \hbox{$\ell(\sigma^\-) \le \ell^\-$} and $\sigma^\-_{\ell^\-} \ge \ell(\sigma^\+)$.
Note that, by~\eqref{eq:oplus}, if these conditions hold then
$\varnothing \oplus (\sigma^\-,\sigma^\+) = \sigma$; this explains our preference
for joining first in the $\oplus$ operation.
Note also that if these conditions hold then the partition $\sigma$ such
that $\sigma \decMap \decs{\sigma^\-}{\sigma^\+}$ is $\bigl(\ell(\sigma^\-)+1, \ell(\sigma^\+)\bigr)$-large
in the sense of Definition~\ref{defn:large}, and hence,
as is sometimes all we need, $\bigl( \ell(\sigma^\-), \ell(\sigma^\+) \bigr)$-large.
This is clear from Figure~\ref{fig:ellDecomposition}.
\end{remark}

\begin{example}\label{ex:ellDecomposition}
Let $\mu$ be a partition.
Following Definition~\ref{defn:greatestSignedTableau},
$t^-$ is the tableau of shape $\mu^\minusprime$
having in row $i$ the first
$\min(\ell^\-,\mu_i)$ entries from $-1,\ldots, -\ell^\-$.
We then obtain 
$t_{\ell^\-}(\mu)$ from $t^-$ by putting $\mu'_j$ entries
from $1,2,\ldots $ into column $j$ for each $j > \ell^\-$.
Therefore row $i$ of $t_{\ell^\-}(\mu)$ 
has the first $\mu_i$ entries \marginpar{$\young(\oM\tM1111,\oM\tM22,\oM\tM33,\oM)$}
from the sequence $-1, \ldots, -\ell^\-$, $i,$ $i, \ldots$ 
ending with infinitely many $i$s. This can be seen in the first tableau in Example~\ref{ex:2greatest}
(repeated in the margin)
in the case when $\ell^\- = 2$
and $\mu = (6,4,4,1)$ and so $\mu^\- = (4,3)$.
The greatest semistandard signed tableau $t_{\ell^\-}(\mu)$
therefore has signed weight
\begin{equation}\label{eq:greatestSignedWeightPartitionCase}
\bigl(\omega_{\ell^\-}(\mu)^\-, \omega_{\ell^\-}(\mu)^\+ \bigr) = 
(\mu^\-, \mu^\+) .
\end{equation}
In particular, if $\ell^\- = 0$ then we have $\omega_{\ell^\-}(\mu) = \mu$;
this is the well-known fact that the greatest weight of a semistandard $\mu$-tableau is $\mu$.
\end{example}

Note that in the previous example, 
$\decs{\mu^\-}{\mu^\+}$ is the $\ell^\-$-decomposition
of the partition $\mu$.
More generally, we have the following lemma which is critical
in \S\ref{sec:twistedWeightBoundInner}; its generalization in
Proposition~\ref{prop:stronglyMaximalSignedWeightsAreEllDecompositions}
is important in~\S\ref{sec:twistedWeightBoundForStronglyMaximalWeight}.

\begin{lemma}\label{lemma:ellDecompositionGreatestSignedWeight}
Let $\tauS$ be a skew partition. 
Then the signed weight 
\[ \bigl(\omega_{\ell^\-}(\tauS)^\-, \omega_{\ell^\-}(\tauS)^\+ \bigr) \]
is the $\ell^\-$-decomposition of a partition.
\end{lemma}

\begin{proof}
The greatest semistandard signed tableau $t_{\ell^\-}(\tauS)$
of signed weight $\bigl(\omega_{\ell^\-}(\tauS), \omega_{\ell^\+}(\tauS) \bigr)$
is defined in Definition~\ref{defn:greatestSignedTableau}.
Let $d$ be the length of the partition $\omega_{\ell^\-}(\tauS)^\+$.
If $d=0$ then the result is immediate, so we may suppose that $d \in \N$.
Thus $d$ is the greatest positive entry of $t_{\ell^\-}(\tauS)$.
Choose the leftmost box of $t_{\ell^\-}(\tauS)$ containing~$d$, in position $(i,j)$ say.
By the construction of $t_{\ell^\-}(\tauS)$ in Definition~\ref{defn:greatestSignedTableau},
$t_{\ell^\-}(\tauS)$ has entries $1, \dots, d$ in positions $(i-d+1,j), \ldots, (i,j)$. In particular, each such row has a positive entry.
By the construction  of~$t_{\ell^\-}(\tauS)$
in which negative
entries from $-1,\ldots, \ell^\-$ are placed before 
positive entries, each of the rows $i-d+1, \ldots, i$ of
$t_{\ell^\-}(\tauS)$ begins, after skipping any  boxes 
not considered in $[\tau]$ because 
they are in $[\taus]$, with
\smash{\raisebox{-4.5pt}{\begin{tikzpicture}[x=0.55cm,y=-0.45cm]
\tableauBox{0}{0}{\oM} \tableauBox{1}{0}{\tM} \draw (1,0)--(3,0)--(3,1)--(1,1)--(1,0);
\node at (2,0.5) {$\ldots$}; \tableauBox{4}{0}{$\pmb{\ell}^{\pmb{-}}$}\end{tikzpicture}}}. 
Hence $\omega_{\ell^\-}(\tauS)^\-_{\ell^\-}$, which counts the number of entries of $-\ell^\-$
in $t(\tauS)$, is at least~$d$. Equivalently,
$\omega_{\ell^\-}(\tauS)_{\ell^\-}^\- \ge \ell\bigl(\omega_{\ell^\-}(\tauS)^\+ \bigr)$. 
The lemma follows.
\end{proof}

We have already seen in Lemma~\ref{lemma:stronglyMaximalSignedWeightsArePairsOfPartitions}
that if $\swtp{\kappa}$ is a strongly maximal signed weight then $\kappa^\-$ and $\kappa^\+$ are partitions.
We now build on this to show that one potentially nasty technicality does not arise:
in fact $\decs{\kappa^\-}{\kappa^\+}$ is an $\ell(\kappa^\-)$-decomposition.
This result  generalizes Lemma~\ref{lemma:ellDecompositionGreatestSignedWeight}
since, by Lemma~\ref{lemma:maximalAndStronglyMaximalSingletonSemistandardSignedTableauFamilies},
$\bigl(\omega_{\ell^\-}(\muS)^\-,\omega_{\ell^\-}(\muS)^\+\bigr)$ is the strongly maximal
signed weight of the singleton strongly maximal signed tableau family $\{ t_{\ell^\-}(\muS) \}$.
We use this result in the proof of Lemma~\ref{lemma:omegaKappaEllDecomposition}, part of the
proof of Theorem~\ref{thm:nuStable}.

\begin{proposition}\label{prop:stronglyMaximalSignedWeightsAreEllDecompositions}
Let $\swtp{\kappa}$ be a strongly
maximal signed weight.
Then $\decs{\kappa^\-}{\kappa^\+}$ is a well-defined $\ell(\kappa^\-)$-decomposition
of an $\bigl(\ell(\kappa^\-)+1, \ell(\kappa^\+) \bigr)$-large partition.
\end{proposition}

\begin{proof}
Set $\ell^\- = \ell(\kappa^\-)$.
If $\ell^\- = 0$ then $\kappa^\- =\varnothing$ and the result holds trivially.
Similarly if there are no positive entries then $\kappa^\+ = \varnothing$ and the result is obvious.
Therefore we may assume that $\kappa^\+ \not= \varnothing$. As we noted before this proposition, by Lemma~\ref{lemma:stronglyMaximalSignedWeightsArePairsOfPartitions},
$\kappa^\-$ and~$\kappa^\+$ are partitions.

Let $\muS$ be the shape and let $R$ be the size of $\swtp{\kappa}$.
By~(a) in Definition~\ref{defn:stronglyMaximalSignedWeight}
we have $\kappa^\- = R\omega_{\ell^\-}(\muS)^\-$.
Let $d = \ell\bigl(\omega_{\ell^\-}(\muS)^\+\bigr)$ and let $e = \ell(\kappa^\+)$;
note that $d$ is the greatest positive entry in $t_{\ell^\-}(\muS)$
and $e$ is the greatest positive entry in the unique tableau family
 of shape $\muS$, size~$R$ and signed weight $\swtp{\kappa}$.
 Denote this tableau family $\T$.
 
By maximality (Lemma~\ref{lemma:stronglyMaximalImpliesMaximal}),
$t_{\ell^-}(\muS)$ is an element of $\T$. Again by maximality,
there is a tableau in $\T$ containing $d+1$ obtained 
by incrementing the entry in a single box of $t_{\ell^-}(\muS)$. Repeating
this argument, we see that there exist 
distinct tableaux
$t_{(d)}, t_{(d+1)}, \ldots, t_{(e)} \in \T$ such that, for each $k \in \{d+1,\ldots, e\}$,
the tableau $t_{(k)}$ has $k$ as an entry.
(For instance this can be seen in Example~\ref{ex:LawOkitaniSignedWeightsAreStronglyMaximal}(ii),
by considering the entries in the box $(1,m)$.)
Therefore $R \ge e-d+1$. By Lemma~\ref{lemma:ellDecompositionGreatestSignedWeight},
every tableau in~$\T$ agrees with $t_{\ell^\-}(\muS)$ in its negative
entries. Hence the number of entries of tableaux in~$\T$ 
equal to $\ell^\-$  is $R \omega_{\ell^\-}(\muS)_{\ell^\-}$.
Using this for the first equality below, and recalling that by definition $d = \ell\bigl(\omega_{\ell^\-}(\muS)^\+\bigr)$, we obtain
\[ \kappa^\-_{\ell^\-} \hskip-1pt=\hskip-1pt R \omega_{\ell^\-}(\muS)_{\ell^\-} 
\hskip-1pt\ge\hskip-1pt R\ell(\omega_{\ell^\-}(\muS)^\+) \hskip-1pt=\hskip-1pt Rd\ge
d + (R-1) \ge d + (e-d) \hskip-1pt=\hskip-1pt e. \]
Hence $\kappa^\-_{\ell^-} \ge \ell(\kappa^\+)$ and by Remark~\ref{remark:ellDecompositionLarge},
$\decs{\kappa^\-}{\kappa^\+}$ is a well-defined $\ell(\kappa^\-)$-decomposition
of an $\bigl(\ell(\kappa^\-)+1, \ell(\kappa^\+) \bigr)$-large partition.
\end{proof}

\subsection{The $\ell^\-$-twisted dominance order}
\label{subsec:signedDominance}
The sets used in applications of the critical Signed Weight Lemma (see Lemma~\ref{lemma:SWL} below)
are subsets of intervals for 
a partial order on partitions
defined using the $\ell^\-$-decomposition in Definition~\ref{defn:ellDecomposition}
and the $\ell^\-$-signed dominance order in 
Definition~\ref{defn:ellSignedDominanceOrder}.

\begin{definition}[$\ell^\-$-twisted dominance order]\label{defn:ellTwistedDominanceOrder}
Fix $\ell^\- \in \N_0$. 
The $\ell^\-$-\emph{twisted dominance order} is the partial order defined 
on partitions of the same size by $\pi \unlhddot \sigma$ if and only if
$\decs{\pi^\-}{\pi^\+} \unlhd \decs{\sigma^\-}{\sigma^\+}$,
where $\unlhd$ is the $\ell^\-$-signed dominance order on the set 
$\mathcal{W}_{\ell^\-} \times \mathcal{W}$.
\end{definition}

The value of $\ell^\-$ will always be clear from context.
An example is given following Remark~\ref{remark:twistedDominanceOrderGeneralizesDominanceOrder} below.
In practice we shall often use the following lemma to work with the $\ell^\-$-twisted dominance order.
Recall that $\unLHD$ denotes the dominance order on partitions of arbitrary size.

\begin{lemma}[Characterization of the $\ell^\-$-twisted dominance order]\label{lemma:twistedDominanceOrderOldDefinition}
Let $\pi$ and $\sigma$ be partitions of the same size. 
We have $\pi \unlhddot \sigma$ in the $\ell^\-$-twisted 
dominance order if and only if both
\begin{itemize}
\item[(a)] $\pi^\- \unLHD\, \sigma^\-$ and
\item[(b)] $|\pi^\+| \ge |\sigma^\+|$ and $\pi^\+ \unlhd
\sigma^\+ + (|\pi^\+| - |\sigma^\+|)$.
\end{itemize}
\end{lemma}

\begin{proof}
From the equation 
\[ \sum_{i=1}^{\ell^\-} \sigma_i^\- - \sum_{i=1}^{\ell^\-} \pi_i^\-  = |\sigma^\-| - |\pi^\-| =
(|\sigma| - |\sigma^\+|) - (|\pi| - |\pi^\+|) = |\pi^\+| - |\sigma^\+| \]
we have 
$\sum_{i=1}^{\ell^\-} \pi_i^\- + \sum_{i=1}^k \pi^\+_i \le 
\sum_{i=1}^{\ell^\-} \sigma_i^\- + \sum_{i=1}^k \sigma^\+_i$
if and only if $\sum_{i=1}^k \pi^\+_i \le (|\pi^\+| - |\sigma^\+|) + \sum_{i=1}^k \sigma^\+_i$.
The lemma now follows from the definition of the dominance order and the $\ell^\-$-signed dominance order.
\end{proof}

\begin{remark}\label{remark:twistedDominanceOrderGeneralizesDominanceOrder}
It is obvious from Definition~\ref{defn:ellTwistedDominanceOrder} that
the $0$-twisted dominance order is the ordinary dominance order.
If $\ell^\- \ge p$ then the $\ell^\-$-twisted dominance order on partitions of size $p$
is the reverse of the usual dominance order. Whenever
$\ell^\- \ge 1$, the greatest partition in the $\ell^\-$-signed dominance
order is $(1^p)$.
\end{remark} 

\begin{example}\label{ex:62upset}
In the $1$-twisted dominance order on partitions of $8$, 
the negative component $\sigma^\-$ of each partition $\sigma$ has exactly one part.
Let \hbox{$\sigma^\- = (b)$} where $1 \le b \le 8$. By Lemma~\ref{lemma:twistedDominanceOrderOldDefinition},
$\sigma \unrhddot (6,2)$ if and only if $(b) \,\unRHD (2)$ and
$\sigma^\+ + (6- |\sigma^\+|) \unrhd (5,1)$. 
If $b \le 6$ then 
$\sigma^\+ \in \{ (8-b), (7-b,1) \}$; if $b = 7$ then $\sigma^\+ = (1)$
and if $b =8$ then $\sigma^\+ = \varnothing$. The up-set of $(6,2)$ is
therefore as claimed in~\eqref{eq:62upset} in the overview in \S\ref{sec:overview}.
It is shown in the Hasse diagram in Figure~\ref{fig:62upset}. 
See~\S\ref{subsec:cutUpsetForLawOkitani} for a continuation to the `cut' up-set used 
in \S\ref{subsec:cutUpsetsAreStable}.
\end{example}

\begin{figure}[ht!]
\hspace*{-0.4in}\scalebox{0.9}{\begin{tikzpicture}[x=0.5cm,y=0.4cm]
\newcommand{\leftup}[2]{\draw[thick] (#1,#2)--(#1-1,#2+1.4);}
\newcommand{\rightup}[2]{\draw[thick] (#1,#2)--(#1+1,#2+1);}

\node at (0,0) {$\ss(6,2) \atop \decs{(2)}{(5,1)}$}; 
\leftup{-0.5}{1}
\rightup{0.5}{1}

\begin{scope}[xshift=-0.125cm]
\node at (-2,3.4) {$\ss(5,2,1) \atop \decs{(3)}{(4,1)}$}; 
\node at (2,3) {$\ss(7,1) \atop \decs{(2)}{(6)}$}; 
\leftup{-2}{4.4}
\rightup{-1}{4.4}
\leftup{2}{4}
\end{scope}

\begin{scope}[xshift=-0.5cm]
\node at (-3,6.8) {$\ss(4,2,1,1) \atop \decs{(4)}{(3,1)}$}; 
\node at (1,6.4) {$\ss(6,1,1) \atop \decs{(3)}{(5)}$}; 
\leftup{-3}{7.8}
\rightup{-2}{7.8}
\leftup{1}{7.4}
\end{scope}

\begin{scope}[xshift=-0.75cm]
\node at (-4,10.2) {$\ss(3,2,1,1,1) \atop \decs{(5)}{(2,1)}$}; 
\node at (0,9.8) {$\ss(5,1,1,1) \atop \decs{(4)}{(4)}$}; 
\leftup{-4}{11.2}
\rightup{-3}{11.2}
\leftup{0}{10.8}
\end{scope}

\begin{scope}[xshift=-1cm]
\node at (-5,13.6) {$\ss(2,2,1,1,1,1) \atop \decs{(6)}{(1,1)}$}; 
\node at (-1,13.2) {$\ss(4,1,1,1,1) \atop \decs{(5)}{(3)}$}; 
\rightup{-4}{14.6}
\leftup{-1}{14.2}
\end{scope}

\begin{scope}[xshift=-1.25cm]
\node at (-2,16.6) {$\ss(3,1,1,1,1,1) \atop \decs{(6)}{(2)}$}; 
\draw[thick] (-2,17.6)--(-2,18.8);

\node at (-2,19.8) {$\ss(2,1,1,1,1,1,1) \atop \decs{(7)}{(1)}$}; 
\draw[thick] (-2,20.8)--(-2,22);

\node at (-2,23) {$\ss(1,1,1,1,1,1,1,1) \atop \decs{(8)}{\varnothing}$}; 
\end{scope}

\node at (12,14){
$\setcounter{MaxMatrixCols}{12}
\begin{pNiceMatrix}[first-row,first-col,nullify-dots] &
\rb{(1^8)} & \rb{(2,1^6)} & \rb{(3,1^5)} & \rb{(2,2,1^4)} & \rb{(4,1^4)} & \rb{(3,2,1,1,1)}  &
\rb{(5,1,1,1)} & \rb{(4,2,1,1)} & \rb{(6,1,1)} & \rb{(5,2,1)} & \rb{(7,1)} & \rb{(6,2)}  \\
\decss{(8)}{\varnothing} & 1 & \cdot & \cdot & \cdot & \cdot & \cdot & \cdot & \cdot & \cdot &\cdot & \cdot & \cdot \\
\decss{(7)}{(1)}   & 1 & 1 & \cdot & \cdot & \cdot & \cdot & \cdot & \cdot & \cdot & \cdot &\cdot & \cdot  \\
\decss{(6)}{(2)}   &  0 & 1 & 1 & \cdot & \cdot & \cdot & \cdot & \cdot & \cdot & \cdot & \cdot &\cdot   \\  
\decss{(6)}{(1,1)} &  1 & 2 & 1 & 1 & \cdot & \cdot & \cdot & \cdot & \cdot & \cdot & \cdot & \cdot  \\  
\decss{(5)}{(3)}   &  0 & 0 & 1 & 0 & 1 & \cdot & \cdot & \cdot & \cdot & \cdot & \cdot & \cdot   \\  
\decss{(5)}{(2,1)} &  0 & 1 & 2 & 1 & 1 & 1 & \cdot & \cdot & \cdot & \cdot & \cdot & \cdot  \\  
\decss{(4)}{(4)}   &  0 & 0 & 0 & 0 & 1 & 0 & 1 & \cdot & \cdot & \cdot & \cdot & \cdot   \\  
\decss{(4)}{(3,1)} &  0 & 0 & 1 & 2 & 0 & 1 & 1 & 1 & \cdot & \cdot & \cdot & \cdot  \\  
\decss{(3)}{(5)}   &  0 & 0 & 0 & 0 & 0 & 0 & 1 & 0 & 1 & \cdot & \cdot & \cdot    \\
\decss{(3)}{(4,1)} &  0 & 0 & 0 & 0 & 1 & 0 & 2 & 1 & 1 & 1 & \cdot & \cdot      \\
\decss{(2)}{(6)}   &  0 & 0 & 0 & 0 & 0 & 0 & 0 & 0 & 1 & 0 & 1 & \cdot  \\
\decss{(2)}{(5,1)} &  0 & 0 & 0 & 0 & 0 & 0 & 1 & 0 & 2 & 1 & 1 & 1      
\end{pNiceMatrix}$};
\end{tikzpicture}}

\caption{Hasse diagram of the up-set $(6,2)^\unlhddot$ in the $1$-twisted dominance order on $\Par(8)$,
as seen in \eqref{eq:62upset} in the overview of the proof in \S\ref{sec:overview}.
By Remark~\ref{remark:twistedDominanceOrderGeneralizesDominanceOrder},
this up-set is also the interval $[(6,2),(1^8)]_\unlhddotS$.
The total order $\ledot$ refining $\unlhddot$ is indicated by vertical height.
The matrix with entries $|\SSYT(\sigma)_{(\pi^\minus,\pi^\+)}|$ for $\pi$, $\sigma \in (6,2)^\unlhddotS$ 
relevant to condition (b) in the definition of a stable partition
system (Definition~\ref{defn:stablePartitionSystem})
is shown to the right, with row and column labels ordered by the total order in 
Definition~\ref{defn:twistedTotalOrder}. 
It is lower unitriangular by Lemma~\ref{lemma:twistedKostkaMatrix}. We use $\cdot$ to show a zero 
implied by this lemma.\label{fig:62upset}
}
\end{figure}

For a further example of the twisted dominance order see
Example~\ref{ex:444support}.
In practice, we find the following
informal interpretation, using the standing notation shown in
Figure~\ref{fig:ellDecomposition} is helpful: 
\emph{the partitions larger than $\pi$ in the $\ell^\-$-signed
dominance order are exactly those  obtained
from $\pi$ by a combination of box moves that are either:
down and left within~$\pi^\-$,
up and right within~$\pi^\+$,
or from $\pi^\+$ to~$\pi^\-$}. The final possibility is responsible
for the equalization of sizes in condition (b) in Lemma~\ref{lemma:twistedDominanceOrderOldDefinition}.

In particular we have the analogue of the well-known property of
the normal dominance order that $\alpha \unlhd \beta$ implies $\ell(\alpha) \ge \ell(\beta)$.

\begin{lemma}\label{lemma:signedDominancePositiveLength}
Let $\alpha$ and $\beta$ be partitions.
If $\alpha \unlhddot \beta$
in the $\ell^\-$-twisted dominance order then $\ell(\alpha^\+) \ge \ell(\beta^\+)$.
\end{lemma}

\begin{proof}
By Lemma~\ref{lemma:twistedDominanceOrderOldDefinition}(b) we have
$\alpha^\+ \unlhd \beta^\+ + (|\alpha^\+| - |\beta^\+|)$. Hence by the property of 
the dominance order just mentioned,
$\ell(\alpha^\+) 
\ge \ell(\beta^\+)$.
\end{proof}

\subsection{Twisted symmetric functions and twisted
Kostka numbers}\label{subsec:symmetricFunctionsForSignedWeightLemma}

\begin{definition}[$\ell^\-$-twisted symmetric function]\label{defn:ellTwistedSymmetricFunction}
Fix $\ell^\- \in \N_0$. 
We define the $\ell^\-$-\emph{twisted symmetric function} $g_\pi$ for a partition $\pi$ by
$g_\pi = e_{\pi^\-}h_{\pi^\+}$.
\end{definition}

For example if $\ell^\- = 0$ then $g_\pi = h_{\pi^\+}$, or equivalently, $g_\pi = h_\pi$, and if $\ell^\- \ge a(\pi)$ then
$g_\pi = e_{\pi^\-}$, or equivalently, $g_\pi = e_{\pi'}$. Thus as claimed at the start
of this section, the $\ell^\-$-twisted symmetric functions interpolate between the homogeneous
and elementary symmetric functions.

The following lemma is vital when
verifying condition~(i) in the Signed Weight Lemma (Lemma~\ref{lemma:SWL}).
Example~\ref{ex:444support} following illustrates the iterative part of the proof.
We require Young's rule and Pieri's rule: see references
in \S\ref{subsec:symmetricFunctionsBasics}.
The support of a symmetric function is defined in Definition~\ref{defn:supp}.

\begin{lemma}[Twisted Kostka matrix]\label{lemma:twistedKostkaMatrix}
Let $\pi \in \Par(n)$ have $\ell^\-$-decomposi\opthyphen{}tion $\decs{\pi^\-}{\pi^\+}$
where $\ell(\pi^\-) = \ell^\-$.
If $\sigma \in \supp(e_{\pi^\-}h_{\pi^\+})$ 
then
$\sigma \unrhddot \pi$. Moreover we have \hbox{$\langle e_{\pi^\-}h_{\pi^\+}, s_\pi \rangle = 1$}.
\end{lemma}

\begin{proof}
We describe the summands of $e_{\pi^\-}h_{\pi^\+}$ combinatorially.
By Pieri's rule, if $s_\beta \in \supp(e_{\pi^\-})$ then $\beta \unlhd \pi^\minusprime$,
or equivalently, $\beta' \unrhd \pi^\-$. Since $\beta'$ 
has at most $\ell^\-$ parts, we have 
$\beta^\- = \beta'$ and hence $\beta^\- \unrhd \pi^\-$.
Set $k = \ell(\pi^\+)$. The product $s_\beta h_{\pi^\+}$
may be computed by repeated applications of Young's rule: starting with $\gamma(0) = \beta$,
let $s_{\gamma(i+1)}$ be a chosen Schur function summand of \smash{$s_{\gamma(i)} h_{\pi^\+_i}$} for
each $i$ such that $0 \le i < k$.
To find the possible $\gamma(i+1)$,  we fix $b_i^\-$ and $b_i^\+ \in \N_0$ with
$b_i^\- + b_i^\+ = \pi_i^\+$, then 
\begin{itemize}
\item add a horizontal strip of length~$b_i^\-$ to~\smash{$[\gamma(i)^{\minusprime}]$} to obtain $[\gamma(i+1)^{\minusprime}]$;
\item add a horizontal strip of length~$b_i^\+$ to~\smash{$[\gamma(i)] \backslash [\gamma(i+1)^{\minusprime}]$}.
\end{itemize}
Let $\sigma = \gamma(k)$ be 
the Schur function obtained after iteratively applying this procedure
to all parts of $\pi^\+$.
Since $\beta^\- \,\unrhd\, \pi^\-$, and subsequent steps add boxes to each 
$[\gamma(i)^\minusprime]$,
we have $\sigma^\- \unRHD \pi^\-$ 
and condition~(a) in Lemma~\ref{lemma:twistedDominanceOrderOldDefinition} holds.
The horizontal additions to each~successive~$[\gamma(i)^\minusprime]$ in this sequence
used in 
total $|\pi^\+| - |\sigma^\+|$ boxes. 
Moreover, at step $i$, the $b_i^\+$ boxes added to $[\gamma(i)] \backslash [\gamma(i)^{\minusprime}]$ 
 lie in rows $1$ up to $i$
of $[\gamma(i)] \backslash [\gamma(i)^{\minusprime}]$. It follows that
$\sigma^\+$ satisfies 
$\sigma^\+ + (|\pi^\+| - |\sigma^\+|) \,\unrhd\, \pi^\+$. 
This gives~(b) in Lemma~\ref{lemma:twistedDominanceOrderOldDefinition}. Hence, by this
lemma,  $\sigma \unrhddot \pi$.
Finally, if $\sigma = \pi$ then $\beta = \pi^\-$ 
and $\gamma(i) = (\pi^\+_1, \ldots, \pi^\+_i)$ for each $i$.
Since the sequence $\gamma(0), \ldots, \gamma(k)$ is uniquely determined,
we have $\langle e_{\pi^\-}h_{\pi^\+}, s_\pi \rangle = 1$, as required.
\end{proof}

\begin{example}\label{ex:444support}
Take $\ell^\- = 2$
and let $\pi = (4,4,4)$ with $2$-decomposition 
\[ \decs{\pi^\-}{\pi^\+}
= \dec{(3,3)}{(2,2,2)}. \]
The Schur function summands of $e_{\pi^\-}$ are all $s_\beta$ such that $\beta \unlhd (3,3)'$. 
For this example we take $\beta = (2,2,1,1)$.
The partition~$\pi^\+$ specifies  three  Young's rule additions 
of two boxes \raisebox{-1.5pt}{\scalebox{0.8}{\yng(2)}}\hskip2pt. 
The sequence of partitions $\gamma(0)$, $\gamma(1)$, $\gamma(2)$, $\gamma(3)$ 
in the proof is, for one particular
choice of Young's rule additions, $(2,2,1,1)$, $(4,2,1,1)$, 
$(4,3,2,1)$, $(5,3,3,1)$. The final partition
$\sigma = \gamma(3)$ has $2$-decomposition $\dec{(4,3)}{(3,1,1)}$.
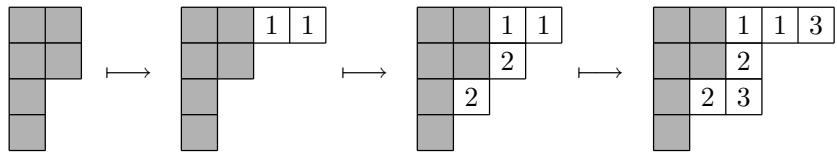
\newcommand{\gb}{\youngBox{0.5}{grey}}
\[ \pyoungScaled{0.5cm}{0.5cm}{0.95}{ {{\gb, \gb}, {\gb, \gb}, {\gb}, {\gb}} } 
		\spy{30pt}{$\longmapsto$}
   \pyoungScaled{0.5cm}{0.5cm}{0.95}{ {{\gb, \gb, 1 , 1 }, {\gb, \gb}, {\gb}, {\gb}} } 
  		\spy{30pt}{$\longmapsto$}
   \pyoungScaled{0.5cm}{0.5cm}{0.95}{ {{\gb, \gb, 1 , 1 }, {\gb, \gb, 2 }, {\gb, 2 }, {\gb}} } 
  		\spy{30pt}{$\longmapsto$}		
   \pyoungScaled{0.5cm}{0.5cm}{0.95}{ {{\gb, \gb, 1 , 1, 3}, {\gb, \gb, 2 }, {\gb, 2, 3}, {\gb}} } 
\]
At step $2$ we added one box to $[\gamma(1)^{-'}]$ and one box to $[\gamma(1)^\+]$,
taking $b_2^\- = b_2^\+ = 1$; in the other two steps $b_1^\- = b_3^\- = 0$.
As expected, conditions~(a) and (b) 
in Lemma~\ref{lemma:twistedDominanceOrderOldDefinition} hold,
with $(3,3) \unLHD (4,3)$ and $(2,2,2) \unlhd (3,1,1) + (1)$.
Moreover, 
\[ \dec{(4,3)}{(3,1,1)} \decMap  (5,3,3,1)  \,\unrhddot\, \pi \decMap \dec{(3,3)}{(2,2,2)} \]
as expected from the conclusion of Lemma~\ref{lemma:twistedKostkaMatrix}.
If instead we had chosen $\beta = (1^6)$ then
a possible sequence ending with $\sigma = (4,3,2,1,1,1)$ is
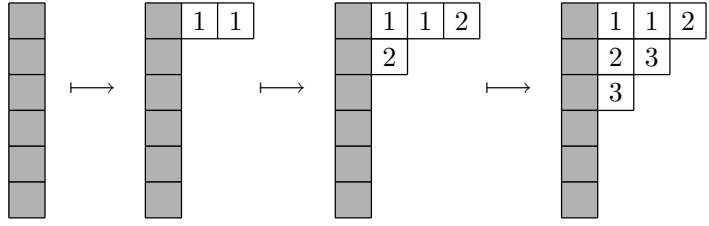
\[ \pyoungScaled{0.5cm}{0.5cm}{0.95}{ {{\gb}, {\gb}, {\gb}, {\gb}, {\gb}, {\gb}} } 
		\spy{50pt}{$\longmapsto$}
   \pyoungScaled{0.5cm}{0.5cm}{0.95}{ {{\gb, 1 , 1 }, {\gb}, {\gb}, {\gb}, {\gb}, {\gb}} } 
  		\spy{50pt}{$\!\!\longmapsto$}
   \pyoungScaled{0.5cm}{0.5cm}{0.95}{ {{\gb, 1 , 1, 2 }, {\gb, 2 }, {\gb}, {\gb}, {\gb} , {\gb}} } 
  		\spy{50pt}{$\!\!\longmapsto$}		
   \pyoungScaled{0.5cm}{0.5cm}{0.95}{ {{\gb, 1 , 1, 2}, {\gb, 2, 3 }, {\gb, 3}, {\gb}, {\gb}, {\gb}} } 
\]
in which $b_1^\- = b_2^\- = b_3^\- = 1$ (since column $2$ grows by one box in each addition)
and correspondingly 
$|\pi^\+| - |\sigma^\+| = 3$.
Again conditions (a) and (b) 
in Lemma~\ref{lemma:twistedDominanceOrderOldDefinition} hold, 
now with $(3,3) \unLHD (6,3)$ and $(2,2,2) \unlhd (2,1) + (3)$
and again the conclusion of Lemma~\ref{lemma:twistedKostkaMatrix} holds
since $\dec{(6,3)}{(2,1)} \decMap (4,3,2,1,1,1) \unrhddot \pi$.
\end{example}

Figure~\ref{fig:62upset} has an example of the matrix $\langle e_{\pi^\-}h_{\pi^\+}, s_\sigma \rangle$
in Lemma~\ref{lemma:twistedKostkaMatrix}.
It is an instructive exercise to show that the many zeros in this matrix correspond to pairs
of partitions incomparable in the $1$-twisted dominance order.
For a further example of the conclusion of 
Lemma~\ref{lemma:twistedKostkaMatrix},
calculation shows
that, cut to partitions of length at most $3$, 
$e_{(3,3)}h_{(3,3)}$ and $e_{(3,3)}h_{(4,1,1)}$  have supports 
\begin{align*}
& \{(5,5,2), (6,4,2), (7,3,2), (8,2,2)\}, \\
& \{ (6,3,3), (6,4,2), (7,3,2), (8,2,2) \}\end{align*}
respectively,
corresponding to the part of the up-sets seen in Figure~\ref{fig:Hasse444and6442}
lying above $(5,5,2)$ and $(6,3,3)$, respectively.

\subsection{Twisted total order}\label{subsec:twistedTotalOrder}
While not logically essential, it is useful to have a total order that makes
the twisted Kostka matrix 
seen in Figure~\ref{fig:62upset} lower-triangular.
These matrices are used in (ii) in the critical Signed Weight Lemma (Lemma~\ref{lemma:SWL}).

\begin{definition}
\label{defn:twistedTotalOrder}
Fix $\ell^\- \in \N_0$. We define the $\ell^\-$-\emph{twisted total order} by
$\pi \ledot \sigma$ if and only if $(\pi^\-, \pi^\+) \le (\sigma^\-, \sigma^\+)$
where $\le$ is the lexicographic order on compositions.
\end{definition} 

Equivalently, $\pi \ledot \sigma$ if and only if
$\pi^\- < \sigma^\-$ or $\pi^\- = \sigma^\-$ and
$\pi^\+ \le \sigma^\+$, where $<$ and $\le$ are the 
lexicographic order on partitions (now possibly of different sizes).
It is easily seen that
$\ledot$ is a total order refining the $\ell^\-$-twisted dominance order.
For example, in the total order on compositions we have
\[ \dec{(3,3)}{(3,2,1)} < \dec{(3,3)}{(3,3)} < \dec{(3,3)}{(4,1,1)} < 
\dec{(3,3)}{(4,2)} \]
and hence in the $2$-twisted total order we have
$(5,4,3) \ltdot (5,5,2) \ltdot (6,3,3)$ $\ltdot (6,4,2)$. (See Figure~\ref{fig:Hasse444and6442}
for the relevant Hasse diagram.)
Moreover 
\[ \dec{(3,3)}{(4,2)} \le \dec{(4,2)}{\pi^\+} \le \dec{(4,3)}{\sigma^\+} \]
and hence
$(6,4,2) \ltdot (2,2,1,1) + \pi^\+ \ltdot (2,2,2,1) + \sigma^\+$
for any partitions $\pi^\+$ of $6$ and $\sigma^\+$ of $5$
with $\ell(\pi^\+) \le 2$ and $\ell(\sigma^\+) \le 3$.


\subsection{Up-sets and twisted intervals}\label{subsec:upsetsAndTwistedIntervals}
For a fixed $\ell^\- \in \N_0$, and partitions $\gamma$,~$\delta$ of the
same size we define the \emph{twisted interval} $[\gamma, \delta]_{\unlhddotS}$ by
\[ [\gamma,\delta]_\unlhddotS = \{ \sigma \in \Par(p) : \gamma \unlhddot \sigma \unlhddot \delta \}.
\]
where $\unlhddot$ is the $\ell^\-$-twisted dominance order. We define
the \emph{up-set} of a partition $\lambda$ of size $p$ by
\[ \lambda^\unlhddotS = \{ \sigma \in \Par(p) : \sigma \unrhddot \lambda \}. \]
As ever, the value of $\ell^\-$ will be clear from context.
Equivalently, by Remark~\ref{remark:twistedDominanceOrderGeneralizesDominanceOrder}, 
$\lambda^\unlhddotS = [\lambda, (p)]_\unlhd$
when $\ell^\- = 0$ and
$\lambda^\unlhddotS =
[\lambda, (1^p)]_\unlhddotS$ when $\ell^\- \ge 1$.


\begin{example}[Length bound recast as an interval]\label{ex:lengthBound}
In the overview in \S\ref{sec:overview} we used (without giving full details) the stable
partition system $(\PSeq{M})_{M \in \N_0}$ defined using the $1$-twisted dominance order by 
\[ \PSeq{M} = 
\bigl\{ \sigma \in \Par(8+2M) : \sigma \unrhddot (6+M,2,1^M),\, \ell(\sigma) \le 4+M \bigr\}. \]
Let $\sigma \in \Par(8+2M)$. Observe that
\[ \ell(\sigma) \le 4+M \iff \sigma^\- \unLHD (4+M) \iff \sigma \unlhddot (5+M,1^{3+M}), \]
where the final implication holds since $(5+M,1^{3+M}) \decMap \bigl\langle (4+M), (4+M) \bigr\rangle$
is the greatest partition in the $1$-twisted dominance order with negative part $(4+M)$.
Therefore an equivalent definition of $\PSeq{M}$ is 
\begin{align*} \PSeq{M} &= [(6+M,2,1^M), (5+M,1^{3+M})]_\unlhddotS \\
&= [(6,2) \oplus M\bigl((1),(1)\bigr), (5,1,1,1) \oplus M\bigl((1),(1)\bigr)]_\unlhddotS
\end{align*}
where $\unlhddot$ is the $1$-twisted dominance order, as claimed in \S\ref{subsec:overviewSWL}.
\end{example}

It is a special feature of the $1$-twisted dominance order that the \emph{only} restriction imposed
by the comparison on negative parts is a bound on the length of the partition. 
See \S\ref{subsec:stablePartitionSystemsAsIntervals} for an extended example more typical of the general case.

\addtocontents{toc}{\smallskip}
\addtocontents{toc}{\textbf{Signed Weight Lemma and stable partition systems}}

\section{Signed Weight Lemma}\label{sec:SWL}

In this section we prove the critical Signed Weight Lemma (Lemma~\ref{lemma:SWL}) and give
the related results and definitions needed to apply it to prove our main theorems.
We end with an extended example.

\subsection{Stable partition systems}\label{subsec:stablePartitionSystem}
We isolate the two more technical hypotheses of the Signed Weight Lemma in the following definition.

\begin{definition}\label{defn:stablePartitionSystem}
A \emph{partition system} 
is a sequence $(\PSeq{M})_{M \in \N_0}$ of sets of partitions
such that all partitions in each $\PSeq{M}$
have the same size, together with
a function $\pmap : \Par \rightarrow \Par$ such that
$\pmap(\PSeq{M}) \subseteq \PSeq{M+1}$ for all $M \in \N_0$. For each $\pi \in \Par$, let $g_\pi$ be a symmetric function of degree $|\pi|$.
We say the partition system is \emph{stable} with respect to the family $g_\pi$ if

\smallskip
\begin{defnlist}
\item $\pmap : \PSeq{M} \rightarrow \PSeq{M+1}$ is a bijection for all $M$ sufficiently large;
\vspace*{1pt}
\item if $M$ is sufficiently large then
$\langle g_\pi, s_\sigma \rangle \!=\! \langle g_{\pmap(\pi)}, s_{\pmap(\sigma)} \rangle$
for all \hbox{$\pi$}, \hbox{$\sigma \in \PSeq{M}$} and moreover,
the matrix\, $\KM(M)$ with rows and columns labelled by $\PSeq{M}$ and entries
$\KM(M)_{\pi\sigma} = \langle g_\pi, s_\sigma \rangle$ is invertible.
\end{defnlist}
If (a) and (b) hold for $M \ge L$ then we say the system is \emph{stable for $M \ge L$}.
Given $k\in \N$, the $k$-\emph{subsystem} of $(\PSeq{M})_{M \in \N_0}$ is 
$(\PSeq{kM})_{M \in \N_0}$ with function $\pmap^k$, i.e.~the $k$-fold composition of $\pmap$.
\end{definition}

Note in particular that the conditions imply that the matrix $K(M)$ is constant for $M$ sufficiently large.
It is routine to check that a $k$-subsystem of a stable partition system for the family $g_\pi$ 
is  a stable
partition system, again for the family $g_\pi$.
The general results we need on stable partition systems are in \S\ref{sec:stablePartitionSystems}.

\begin{example}\label{ex:FoulkesStablePartitionSystem}
Let $g_\pi = h_\pi$ for all $\pi \in \Par$
and let $\pmap : \Par \rightarrow \Par$ be defined by $\pmap(\sigma) = \sigma + (1)$.
Fix a partition $\lambda$ and let 
\[ \PSeq{M} = \bigl\{ \sigma \in \Par(|\lambda| + M) :
\sigma \,\unrhd\, \lambda + (M) \bigr\}.\] 
We claim that the 
sets $\PSeq{M}$ form a stable partition system.
First note that, provided $M$ is sufficiently large, every partition $\mu$ such
that $\mu \unrhd \lambda + (M+1)$ satisfies $\mu_1 > \mu_2$ and so is in the image 
of the map $\pmap$. (Explicitly, it suffices to take $M \ge |\lambda| - 2a(\lambda)$ where
as usual $a(\lambda)$ is the first part of $\lambda$;
this is the bound $L$ from Corollary~\ref{cor:signedIntervalStable};
%
we explain why it applies after this example.)
Hence (a) holds. 
By a special case of Lemma~\ref{lemma:twistedKostkaNumbers} (Twisted Kostka Numbers), 
the matrix 
$\KM(M)$ in condition~(b) 
is the matrix of Kostka numbers: 
\[ \KM(M)_{\pi \sigma} = \langle h_\pi, s_\sigma  \rangle
= |\SSYT(\sigma)_{(\varnothing, \pi)}|\]
 for $\sigma$, $\pi \in \PSeq{M}$.
 Provided $M$ is sufficiently large,
we have $K_{\pi+(1) \hskip1pt\sigma+(1)} = K_{\pi \sigma}$;
 the relevant semistandard tableaux have the form shown below
with $1$s in the shaded region, and so are in bijection by removing
the hatched box and shifting the boxes right of it one position left.
Hence (b) holds. 

\smallskip
\begin{center}
\begin{tikzpicture}[x=0.55cm, y=-0.55cm]
\fill[color=lightgray] (0,0)--(9,0)--(9,1)--(0,1)--(0,0);
\fill[color=lightgray] (11,0)--(13,0)--(13,1)--(11,1)--(11,0);
\draw (0,0)--(9,0); \draw (11,0)--(17,0)--(17,1)--(11,1);
\draw (1,0)--(1,1); \draw (12,0)--(12,1);
\node at (0.5, 0.5) {$1$};
\node at (6.5,0.5) {$1$};
\node at (12.5, 0.5) {$1$};
\node at (10, 0.5) {$\cdots$};
\draw (13,0)--(13,1);
\draw (0,1)--(9,1); 
\draw (7,0)--(7,1); \draw (6,0)--(6,1);
\draw (7,1)--(6,1)--(6,3)--(5,3)--(5,4)--(2,4)--(2,5)--(0,5)--(0,0);
\fill[pattern = north west lines] (6,0)--(7,0)--(7,1)--(6,1)--(6,0);
\node at (15,0.5) {$> 1$};
\node at (3,2.5) {$> 1$};
\end{tikzpicture}
\end{center}

\noindent The argument for (b) is seen in more generality
and detail in the extended example in 
\S\ref{subsec:444to822partitionSystem}.
The proof of Proposition~\ref{prop:signedTableauStable}
shows that for (b) the same bound $L \ge |\lambda| - 2a(\lambda)$
as (a) suffices.
\end{example}

We leave it as an instructive 
exercise to use the 
Signed Weight Lemma with the stable partition system in Example~\ref{ex:FoulkesStablePartitionSystem}
to prove the stability of the plethysm
coefficients $\langle s_{(n+M)} \circ s_{(m)}, s_{\lambda + mM} \rangle$ and
$\langle s_{(n)} \circ s_{(m+M)}, s_{\lambda + nM} \rangle$ in Foulkes' Conjecture.
Of course this also follows from our main theorems; in the context of their proofs, one should think
of $\bigl\{ \sigma \in \Par(mn + M) :
\sigma \,\unrhd\, \lambda + (M) \bigr\}$ as the interval $[\lambda + (M), (|\lambda| + M)]_\unlhd$ for the 
$0$-dominance order. With this interpretation the stability of the partition
system follows from
Corollary~\ref{cor:signedIntervalStable} applied with $\kappa^\+ = (1)$, $\kappa^\- = \varnothing$
and $\omega = (|\lambda|)$,
giving the bound $M \ge L\bigl([\lambda, (|\lambda|)], (1)\bigr) = |\lambda| - 2 a(\lambda)$.
(This is the first bound in the corollary; the remaining three impose no restriction,
as is generally the case when $\ell(\kappa^\-) = 0$. Alternatively
since the interval is `unsigned' one can use Proposition~\ref{prop:intervalStable}.)
See \S\ref{subsec:stablePartitionSystemsAsIntervals}
for a related example where we reinterpret a stable partition system as a sequence of intervals.

\subsection{Signed weight lemma}\label{subsec:signedWeightLemma}
The following key lemma specifies the overall strategy of the proofs of
Theorem~\ref{thm:muStable} and~\ref{thm:nuStable}. 
In this lemma $\decs{\pi^\-}{\pi^\+}$ is the $\ell^\-$-decomposition of
the partition $\pi$ as defined in Definition~\ref{defn:ellDecomposition}. 

\begin{lemma}[Signed Weight Lemma]\label{lemma:SWL}
Let $\nuSeq{M}$ be a sequence
of partitions and let \smash{$\muSSeq{M}$} be a sequence of skew partitions,
indexed by $M \in \N_0$.
Fix $\ell^\- \in \N$. 
Set $g_\pi = e_{\pi^\-}h_{\pi^\+}$ for each $\pi \in \Par$.
Let $\PSeq{M}$ be a stable partition system for $M \ge L$ with respect to the symmetric functions
$g_\pi$ and the function $\pmap : \Par \rightarrow \Par$, such that
the common size of all partitions in $\PSeq{M}$ is $|\nuSeq{M}||\muSSeq{M}|$.
Suppose that
\begin{thmlist}
\item if $M$ is sufficiently large and $\pi \in \PSeq{M}$ then 
\[ \quad \supp(g_\pi) 
\cap \supp(s_\nuSeq{M} \circ s_\muSSeqs{M}) \subseteq \PSeq{M},\]
\item if $M$ is sufficiently large then, for all $\pi \in \PSeq{M}$,
\[\  \bigl|\PSSYTw{\nuSeq{M}}{\muSSeq{M}}{\pi^\-}{\pi^\+}\bigr|
\!=\! \bigl|\PSSYTw{\nuSeq{M+1}}{\muSSeq{M+1}}{\pmap(\pi)^\-}{\pmap(\pi)^\+}\bigr|.
\]
\end{thmlist}
Then, provided $M \ge L$ and $M$ meets the bounds required by \emph{(i)} and~\emph{(ii)},
\[ \langle s_\nuSeq{M} \circ s_\muSSeqs{M}, s_\sigma \rangle = 
\langle s_{\nuSeq{M+1}} \circ s_\muSSeqs{M+1}, s_{\pmap(\sigma)} \rangle \]
for all $\sigma \in \PSeq{M}$.
\end{lemma}

 We hope to convince the reader, 
both by the proofs of our main theorems, 
and the extended examples in \S\ref{subsec:444to822partitionSystem},
\S\ref{subsec:cutUpsetForLawOkitani} and \S\ref{subsec:stablePartitionSystemsAsIntervals}
below
that Lemma~\ref{lemma:SWL} 
is both powerful and practical, and not as technical as it appears at first sight.
In particular we note that
by Lemma~\ref{lemma:twistedKostkaMatrix}, if
$\sigma \in \supp(e_{\pi^\-}h_{\pi^\+})$ then $\sigma \unrhddot \pi$
in the $\ell^\-$-twisted dominance order, and so condition (i)
can be tested in practice when, as usual, $g_\pi$ is the twisted symmetric
function $e_{\pi^\-}h_{\pi^\+}$ in Definition~\ref{defn:ellTwistedSymmetricFunction}.

\begin{proof}[Proof of Lemma~\ref{lemma:SWL}]
To simplify notation 
we set \smash{$\pf{M} = s_\nuSeq{M} \circ s_\muSSeqs{M}$} for $M \in \N_0$.
Let $M \ge L$ be given
and let $\pi \in \PSeq{M}$.
Recall the matrix $\KM(M)$ from  Definition~\ref{defn:stablePartitionSystem}(b).
By hypothesis (i), we have 
\[ g_\pi = \sum_{\tau \in \PSeqs{M}} \KM(M)_{\pi\tau}s_\tau + G_\pi \]
where the symmetric function $G_\pi$ satisfies
$\langle \pf{M}, G_\pi \rangle = 0$.
By~Definition~\ref{defn:stablePartitionSystem}(b)  $\KM(M)$ is invertible, hence for $\sigma \in \PSeq{M}$ we have
\[ \sum_{\pi \in \PSeqs{M}} \KM(M)^{-1}_{\sigma\pi} g_\pi 
= s_\sigma + \sum_{\pi \in \PSeqs{M}} \KM(M)^{-1}_{\sigma \pi} G_\pi. \]
Substituting $g_\pi = e_{\pi^\-}h_{\pi^\+}$ 
we obtain $s_\sigma = \sum_{\pi \in \PSeqs{M}} \KM(M)^{-1}_{\sigma\pi}\,
 e_{\pi^\-}h_{\pi^\+} + E_\sigma$ where, since $E_\sigma$
 is a linear combination of the $G_\pi$, we have
 $\langle \pf{M}, E_\sigma \rangle = 0$.
By this equation for $s_\sigma$
and Proposition~\ref{prop:plethysticSignedKostkaNumbers}
we get
\begin{align}
\langle \pf{M}, s_\sigma \rangle &=
\sum_{\pi \in \PSeqs{M}} \KM(M)^{-1}_{\sigma\pi} 
	\bigl|\PSSYTw{\nuSeq{M}}{\muSSeq{M}}{\pi^\-}{\pi^\+} \bigr|. \label{eq:WSLN}
\intertext{The same argument applies with $M$ replaced with $M+1$
and $\sigma \in \PSeq{M}$ replaced with $\pmap(\sigma) \in \PSeq{M+1}$. Hence we also have}
\langle \pf{M'}, s_{\pmap(\sigma)} \rangle &= 
\sum_{\rho \in \PSeqs{\Mn}} \KM(\Mn)^{-1}_{\pmap(\sigma)\rho} 
	\bigl|\PSSYTw{\nuSeq{\Mn}}{\muSSeq{\Mn}}{\rho^\-}{\rho^\+}\bigr|. \label{eq:WSLNn}
\end{align}
(Here we reduce clutter by writing $M'$ for $M+1$.)
By Definition~\ref{defn:stablePartitionSystem}(a),
the set $\PSeq{M}$ labelling the rows
and columns of $\KM(M)$ and the set $\PSeq{\Mn}$ labelling 
the rows and columns of $\KM(\Mn)$ are in bijection by~$\pmap$. Therefore
we may take~\eqref{eq:WSLNn} and replace each $\rho$ with $\pmap(\pi)$
and the sum over $\rho \in \PSeq{\Mn}$ with
a sum over $\pi \in \PSeq{M}$.
By Definition~\ref{defn:stablePartitionSystem}(b) we have $\KM(M)_{\sigma\pi }
= \KM(\Mn)_{\pmap(\sigma)\pmap(\pi)}$, and so
\smash{$\KM(M)^{-1}_{\sigma\pi} = \KM(\Mn)^{-1}_{\pmap(\sigma)\pmap(\pi)}$}
for all $\pi, \sigma \in \PSeq{M}$.
This matches up the first factors after the sums in the
right-hand sides of~\eqref{eq:WSLN} and~\eqref{eq:WSLNn}, and hypothesis~(ii) immediately
implies the second factors are equal.
Therefore the right-hand sides agree. Comparing the left-hand sides gives the 
Signed Weight Lemma.
\end{proof}

\subsection{A stable partition system defined by a length bound}
\label{subsec:444to822partitionSystem}
We continue in the setting of Example~\ref{ex:444support}, so $\ell^\- = 2$. 
In this subsection we illustrate 
Definition~\ref{defn:stablePartitionSystem} 
by show that the partition system
\begin{equation}
\label{eq:444to822partitionSystem} \PSeq{M} = \bigl\{ \sigma \in \Par(12+4M) : \sigma 
\unrhddot (4+2M,4,4,2^M), \ell(\sigma) \le 3+M \bigr\} \end{equation}
is stable with respect to the injective map
$\pmap : \Par \rightarrow \Par$ defined by
\[ \lambda \stackrel{\pmap}{\longmapsto} \lambda \oplus
\bigl( (1^2), (2) \bigr) = \lambda \sqcup (2) + (2) . \]

Stability is not immediate. Indeed, 
from the Hasse diagrams in Figure~\ref{fig:Hasse444and6442}
we see that $\pmap$ is not surjective when $M=0$: for example,
the partition $(6,6,2,2)$ is  not in its image.
Suppose that $N \ge 2$
and take $\sigma \in \Par(12 + 4N)$. By hypothesis
\[ \sigma \unrhddot (4+2N,4,4,2^N) \decMap
\dec{(3+N,3+N)}{(2+2N,2,2)}. \]
Since $\sigma^\- \unRHD (3+N, 3+N)$ and, by definition of $\PSeq{M}$, we have
$\ell(\sigma) \le 3+N$, we have 
$\ell(\sigma) = 3+N$.
By definition of the $2$-twisted dominance order, we have
$\sigma^\- = (3+N,3+N)$ and hence
$\sigma^\+ \unrhd (2+2N,2,2)$. Since $N \ge 2$,
it follows that $\sigma^\+_1 - \sigma^\+_2 \ge 2$.
Therefore every partition in $\PSeq{N}$ is of the
form $\lambda \oplus \bigl( (1^2), (2) \bigr)$
and so $\pmap$ is bijective for $M \ge 1$.
This verifies condition (a) in the definition of a stable partition system 
(Definition~\ref{defn:stablePartitionSystem}).

\begin{figure}[ht!]
\begin{minipage}{6in}
\hspace*{-0.6in}\begin{tikzpicture}[x=0.5cm,y=1cm]
\node at (0,7) {$\ss(8,2,2) \atop \decs{(3,3)}{(6)}$}; 
\draw[thick] (0,6.1)--(0,6.6);
\node at (0,5.7) {$\ss(7,3,2) \atop \decs{(3,3)}{(5,1)}$};
\draw[thick] (0,4.8)--(0,5.3);
\node at (0,4.35) {$\ss(6,4,2) \atop \decs{(3,3)}{(4,2)}$};
\draw[thick] (-1.7,3)--(-0.5,3.95);
\draw[thick] (1.7,3.4)--(0.5,3.95);
\node at (-2.5,2.5) {$\ss(5,5,2) \atop \decs{(3,3)}{(3,3)}$};
\node at (2.5,3.0) {$\ss(6,3,3) \atop \decs{(3,3)}{(4,1,1)}$};
\draw[thick] (0.5,1.6)--(1.5,2.6);
\draw[thick] (-0.5,1.6)--(-1.5,2.1);
\node at (0,1.2) {$\ss(5,4,3)\atop\decs{(3,3)}{(3,2,1)}$};
\draw[thick] (0,-0.6)--(0,0.75);
\node at (0,-1) {$\ss(4,4,4)\atop\decs{(3,3)}{(2,2,2)}$};
\end{tikzpicture}
\hspace*{-0.2in}
\begin{tikzpicture}[x=0.5cm,y=1cm]
\node at (0,7) {$\ss(10,2,2,2) \atop \decs{(4,4)}{(8)}$};
\draw[thick] (0,6.1)--(0,6.6);
\node at (0,5.7) {$\ss(9,3,2,2) \atop \decs{(4,4)}{(7,1)}$};
\draw[thick] (0,4.8)--(0,5.3);
\node at (0,4.35) {$\ss(8,4,2,2) \atop \decs{(4,4)}{(6,2)}$};

\draw[thick] (-3.9,1.1)--(-2.5,2.1);
\node at (-3.8,0.7) {$\ss\mbf{(6,6,2,2)} \atop \mbf{\decs{(4,4)}{(4,4)}}$};
\draw[thick] (-2,0.3)--(-1.6,0.1);

\draw[thick] (-1.7,3)--(-0.5,3.95);
\draw[thick] (1.7,3.4)--(0.5,3.95);
\node at (-2.5,2.5) {$\ss(7,5,2,2) \atop \decs{(4,4)}{(5,3)}$};
\node at (2.5,3.0) {$\ss(8,3,3,2) \atop \decs{(4,4)}{(6,1,1)}$};
\draw[thick] (0.5,1.6)--(1.5,2.6);
\draw[thick] (-0.5,1.6)--(-1.5,2.1);
\node at (0,1.2) {$\ss(7,4,3,2)\atop\decs{(4,4)}{(5,2,1)}$};
\draw[thick] (0,0.4)--(0,0.8);
\node at (0,0.0) {$\ss\mbf{(6,5,3,2)}\atop\mbf{\decs{(4,4)}{(4,3,1)}}$};
\draw[thick] (0,-0.65)--(0,-0.35);
\node at (0,-1) {$\ss(6,4,4,2)\atop\decs{(4,4)}{(4,2,2)}$};
\end{tikzpicture}
\hspace*{-0.2in}
\begin{tikzpicture}[x=0.5cm,y=1cm]
\node at (0,7) {$\ss(8+2M,2,2,2^M) \atop \decs{(M',M')}{(6+2M)}$};
\draw[thick] (0,6.1)--(0,6.6);
\node at (0,5.7) {$\ss(7+2M,3,2,2^M) \atop \decs{(M',M')}{(5+2M,1)}$};
\draw[thick] (0,4.8)--(0,5.3);
\node at (0,4.35) {$\ss(6+2M,4,2,2^M) \atop \decs{(M',M')}{(4+2M,2)}$};
\draw[thick] (-3.9,1.1)--(-2.5,2.1);
\node at (-5,0.7) {$\ss{(4+2M,6,2,2^M)} \atop {\decs{(M',M')}{(2+2M,4)}}$};
\draw[thick] (-3,0.3)--(-2.6,0.1);

\draw[thick] (-1.7,3)--(-0.5,3.95);
\draw[thick] (1.7,3.4)--(0.5,3.95);
\node at (-3.25,2.5) {$\ss(5+2M,5,2,2^M) \atop \decs{(M',M')}{(3+2M,3)}$};
\node at (2.75,3.0) {$\ss(6+2M,3,3,2^M) \atop \decs{(M',M')}{(4+2M,1,1)}$};
\draw[thick] (0.5,1.6)--(1.5,2.6);
\draw[thick] (-0.5,1.6)--(-1.5,2.1);
\node at (1,1.2) {$\ss(5+2M,4,3,2^M)\atop\decs{(M',M')}{(3+2M,2,1)}$};
\draw[thick] (0,0.4)--(0,0.8);
\node at (0,0.0) {$\ss{(4+2M,5,3,2^M)}\atop{\decs{(M',M')}{(2+2M,3,1)}}$};
\draw[thick] (0,-0.65)--(0,-0.35);
\node at (0,-1) {$\ss(4+2M,4,4,2^M)\atop\decs{(M',M')}{(2+2M,2,2)}$};
\end{tikzpicture}
\end{minipage}
\caption{Hasse diagrams of up-sets in the 
$2$-twisted dominance order. 
The total order $\ledot$ refining $\unlhddot$ defined in Definition~\ref{defn:twistedTotalOrder} 
is indicated by vertical height.
On the left is the up-set of $(4,4,4) \decMap
\dec{(3,3)}{(2,2,2)}$ restricted
to partitions of length at most~$3$. (This is part of the up-set relevant
to Example~\ref{ex:444support} and the following remark.)
This poset maps under $\lambda \mapsto \lambda \opluss \bigl( (1,1), (2) \bigr)$
into the 
up-set of $(6,4,4,2) \decMap
\dec{(4,4)}{(4,2,2)}$ restricted
to partitions of length at most~$4$, shown
in the middle; the two partitions
not in the image of the map are highlighted.
In turn, for each $M \ge 1$, 
the middle poset is in bijection, by iterating this map, with the
up-set of $(4,4,4) \oplus M\bigl((1^2),(2)) = (4+2M,4,4,2^M) \decMap
\dec{(3+M,3+M)}{(2+2M,2,2)}$  cut  to partitions of length at most $M+3$,
as shown on the right. (To save space we write $M'$ for $M+3$.)
\label{fig:Hasse444and6442}\medskip
}
\end{figure}


Continuing we now check condition (b) in this definition.
We have $g_\pi = e_{\pi^\-}h_{\pi^\+}$, where~$\decs{\pi^\-}{\pi^\+}$
is the $\ell^\-$-decomposition of $\pi$ from Definition~\ref{defn:ellDecomposition}.
The key result we need is Lemma~\ref{lemma:twistedKostkaNumbers} (Twisted Kostka Numbers).
 By this
lemma, for $\pi$, $\sigma \in \Par$, we have
$\langle g_\pi, s_\sigma \rangle = |\SSYT(\sigma)_{(\pi^\-,\pi^\+)}|$.
Since $\KM(M)_{\pi \sigma} = \langle g_\pi, s_\sigma \rangle$ by definition,
the matrix $\KM(M)$ is invertible for all $M$
by Lemma~\ref{lemma:twistedKostkaMatrix}, and
it only remains to show, if $M \ge 1$, then there is a bijection 
\begin{equation}
\label{eq:SSYTbijection} \SSYT(\sigma)_{(\pi^\-,\pi^\+)} \rightarrow \SSYT\bigl( \pmap(\sigma) \bigr)_{(\pmap(\pi)^\-,\pmap(\pi)^\+)} \end{equation}
for each pair $\sigma$, $\pi \in \PSeq{M}$.

\medskip
\begin{example}\label{ex:444tableauBijection}
To illustrate why there is a natural bijection for~\eqref{eq:SSYTbijection},
we continue with the stable partition system in \eqref{eq:444to822partitionSystem}
and take \hbox{$M=1$} and 
$\sigma = (7,5,2,2)$ and $\pi = (6,4,4,2)$. Figure~\ref{fig:7522} shows 
the two elements of each of \smash{$\SSYT\bigl( (7,5,2,2) \bigr)_{((4,4),(4,2,2))}$}
and
\smash{$\SSYT\bigl( (9,5,2,2,2) \bigr)_{((5,5),(6,2,2))}$}.
For~\eqref{eq:SSYTbijection}, we want a bijection
\[ \SSYT\bigl( (7,5,2,2) \bigr)_{((4,4),(4,2,2))} \rightarrow 
\SSYT\bigl( (9,5,2,2,2) \bigr)_{((5,5),(6,2,2))}. \]
Observe that, when $M=2$, the two tableaux of shape $(9,5,2,2,2)$ each have a removable 
\raisebox{0.5pt}{$\young(\oM\tM)$} 
in positions $(3,1)$ and $(3,2)$ and two adjacent boxes \raisebox{0.5pt}{$\young(11)$} 
in positions $(1,6)$ and $(1,7)$ in its top row.
Removing these boxes and shifting the remaining boxes in the first column strictly below row $3$ up by one row 
and the remaining boxes in the top row strictly right of column $5$ left by two columns
pairs up the sets of tableaux.
(We admit it might be more natural here to define the bijection by removing
\raisebox{0.5pt}{$\young(\oM\tM)$}  from  positions $(3+M,1)$ and $(3+M,2)$; we
choose the 
complicated specification to agree with the proof of Proposition~\ref{prop:signedTableauStable} (Tableau Stability)
using Lemma~\ref{lemma:signedTableauPositions}(ii).)
Note also that the two tableaux for $M=1$ have
\raisebox{0.5pt}{$\young(13)$} and \raisebox{0.5pt}{$\young(12)$} 
in boxes $(1,6)$ and $(1,7)$, so removing the four boxes from the positions $(3,1), (3,2), (1,6), (1,7)$
hatched in the lower part of Figure~\ref{fig:7522}
gives tableaux of signed weight $\bigl( (3,3), (3,2,1) \bigr)$
and $\bigl( (3,3), (3,1,2) \bigr)$,
not $\bigl( (3,3), (2,2,2) \bigr)$
as required. Correspondingly, the unique element of \smash{$\SSYT\bigl( 5,5,2) \bigr)_{((3,3), (2,2,2))}$}
is as shown in the margin, \marginpar{\qquad \raisebox{36pt}{$\young(\oM\tM112,\oM\tM233,\oM\tM)$}}
so when $\sigma = (5,5,2)$ and $\pi = (4,4,4)$ 
the insertion map is a canonical injection
\[ \SSYT\bigl( (5,5,2) \bigr)_{((3,3),(2,2,2))} \rightarrow 
\SSYT\bigl( (7,5,2,2) \bigr)_{((4,4),(4,2,2))} \]
that is not a bijection.
\end{example}

\begin{figure}[t!]
\begin{center}
\begin{tikzpicture}[x=0.5cm,y=-0.5cm]
\tB{1}{1}{\oM}\tB{1}{2}{\tM}\tB{1}{3}{1}\tB{1}{4}{1}\tB{1}{5}{1}\tB{1}{6}{1}\tB{1}{7}{3}
\tB{2}{1}{\oM}\tB{2}{2}{\tM}\tB{2}{3}{2}\tB{2}{4}{2}\tB{2}{5}{3}
\tB{3}{1}{\oM}\tB{3}{2}{\tM}
\tB{4}{1}{\oM}\tB{4}{2}{\tM}
\end{tikzpicture}
\qquad
\begin{tikzpicture}[x=0.5cm,y=-0.5cm]
\tB{1}{1}{\oM}\tB{1}{2}{\tM}\tB{1}{3}{1}\tB{1}{4}{1}\tB{1}{5}{1}\tB{1}{6}{1}\tB{1}{7}{2}
\tB{2}{1}{\oM}\tB{2}{2}{\tM}\tB{2}{3}{2}\tB{2}{4}{3}\tB{2}{5}{3}
\tB{3}{1}{\oM}\tB{3}{2}{\tM}
\tB{4}{1}{\oM}\tB{4}{2}{\tM}
\end{tikzpicture}
\end{center}
\begin{center}
\begin{tikzpicture}[x=0.5cm,y=-0.5cm]
\tB{1}{1}{\oM}\tB{1}{2}{\tM}\tB{1}{3}{1}\tB{1}{4}{1}\tB{1}{5}{1}\tB{1}{6}{1}\tB{1}{7}{1}\tB{1}{8}{1}
	\tB{1}{9}{3}
\tB{2}{1}{\oM}\tB{2}{2}{\tM}\tB{2}{3}{2}\tB{2}{4}{2}\tB{2}{5}{3}
\tB{3}{1}{\oM}\tB{3}{2}{\tM}
\tB{4}{1}{\oM}\tB{4}{2}{\tM}
\tB{5}{1}{\oM}\tB{5}{2}{\tM}
\fill[pattern = north west lines] (1,4)--(3,4)--(3,5)--(1,5)--(1,4);
\fill[pattern = north west lines] (6,2)--(8,2)--(8,3)--(6,3)--(6,2);
\end{tikzpicture}
\qquad
\begin{tikzpicture}[x=0.5cm,y=-0.5cm]
\tB{1}{1}{\oM}\tB{1}{2}{\tM}\tB{1}{3}{1}\tB{1}{4}{1}\tB{1}{5}{1}\tB{1}{6}{1}\tB{1}{7}{1}\tB{1}{8}{1}
	\tB{1}{9}{2}
\tB{2}{1}{\oM}\tB{2}{2}{\tM}\tB{2}{3}{2}\tB{2}{4}{3}\tB{2}{5}{3}
\tB{3}{1}{\oM}\tB{3}{2}{\tM}
\tB{4}{1}{\oM}\tB{4}{2}{\tM}
\tB{5}{1}{\oM}\tB{5}{2}{\tM}
\fill[pattern = north west lines] (1,4)--(3,4)--(3,5)--(1,5)--(1,4);
\fill[pattern = north west lines] (6,2)--(8,2)--(8,3)--(6,3)--(6,2);
\end{tikzpicture}
\qquad

\end{center}

%
\caption{The two semistandard signed tableaux in the sets
\smash{$\SSYT\bigl( (7,5,2,2) \bigr)_{((4,4),(4,2,2))}$}
and \smash{$\SSYT\bigl( (9,5,2,2,2) \bigr)_{((5,5),(6,2,2))}$}. The hatched
boxes are inserted by the $\Fmap$ insertion map. \label{fig:7522}}
\end{figure}

\bigskip
To generalize this example to arbitrary $M$ it is most convenient to consider
the inverse map.
Fix $N \ge 2$, let $\sigma$, $\pi \in \PSeq{N}$ and let
$t \in \SSYT(\sigma)_{(\pi^\-,\pi^\+)}$. 
We know that $\sigma^\- = \pi^\- = (3+N,3+N)$.
Hence $t$ has $3+N$ entries of $-1$ and $3+N$ entries of $-2$ which, since
negative entries cannot be repeated in a row, must form the first two columns of $t$.
Therefore position $(1,2)$ $t$ contains $\tM$. 
Moreover, by Definition~\ref{defn:stablePartitionSystem}(a), 
$\sigma$ satisfies $\sigma_1 - \sigma_2 \ge 1$
and since $\pi^\- = (4+2N,4,4,2^N)^\-$, it is immediate from Definition~\ref{defn:ellTwistedDominanceOrder}
that 
\[ \pi^\+ \unrhd (4+2N, 4,4,2^N)^\+ = (2+2N,2,2). \]
and so $\pi^\+_1 \ge 2+2N$. Therefore $t$ has at least $2+2N$ entries
of~$1$, necessarily in its first row, and 
we see that boxes \hbox{$(1,6)$} and $(1,7)$
of $t$ both contain~$1$. Removing this \raisebox{1pt}{$\young(11)$} 
and deleting \raisebox{1pt}{$\young(\oM\tM)$} from positions $(3,1)$ and $(3,2)$ and then shifting
boxes left or up (as seen  when $N = 2$)
defines a bijection
$\SSYT(\sigma)_{(\pi^\-,\pi^\+)} \rightarrow \SSYT(f^{-1}(\sigma))_{(f^{-1}(\pi)^\-,f^{-1}(\pi)^\+)}$.

\begin{figure}
\[  
\begin{pNiceMatrix}[first-row,first-col,nullify-dots] & \rb{(10,2,2,2)} & \rb{(9,3,2,2)} & \rb{(8,4,2,2)}
& \rb{(8,3,3,2)} & \rb{(7,5,2,2)} & \rb{(7,4,3,2)} & \rb{(6,6,2,2)} & \rb{(6,5,3,2)} & \rb{(6,4,4,2)} \\
\decss{(4,4)}{(8)}& 1 & \cdot & \cdot & \cdot & \cdot & \cdot & \cdot & \cdot & \cdot \cr
\decss{(4,4)}{(7,1)}       & 1 & 1     & \cdot & \cdot & \cdot & \cdot & \cdot & \cdot & \cdot \cr
\decss{(4,4)}{(6,2)}        & 1 & 1     & 1     & \cdot & \cdot & \cdot & \cdot & \cdot & \cdot \cr
\decss{(4,4)}{(6,1,1)}       & 1 & 2     & 1     & 1     & \cdot & \cdot & \cdot & \cdot & \cdot \cr
\decss{(4,4)}{(5,3)}         & 1 & 1     & 1     & 0     & 1     & \cdot & \cdot & \cdot & \cdot \cr
\decss{(4,4)}{(5,2,1)}       & 1 & 2     & 2     & 1     & 1     & 1     & \cdot & \cdot & \cdot \cr
\decss{(4,4)}{(4,4)}         & 1 & 1     & 1     & 0     & 1     & 0     & 1     & \cdot & \cdot \cr
\decss{(4,4)}{(4,3,1)}       & 1 & 2     & 2     & 1     & 2     & 1     & 1     & 1     & \cdot \cr
\decss{(4,4)}{(4,2,2)}       & 1 & 2     & 3     & 1     & \mbf{2} & 2     & 1     & 1     & 1
\end{pNiceMatrix} \]
\caption{The stable transition matrix $\KM(1)$ in Example~\ref{ex:444matrix} with entries $\KM(1)_{\pi \sigma}
= |\hskip-1pt\SSYT(\sigma)_{(\pi^\minus,\pi^\+)}|$. 
Columns are labelled by the partition~$\sigma$, rows by the $2$-decomposition
(see Definition~\ref{defn:ellDecomposition})
of $\pi$ and are ordered by the total order $\ledot$ (see Definition~\ref{defn:twistedTotalOrder}) 
refining the $2$-twisted dominance order $\unlhddot$ (see 
Definition~\ref{defn:ellTwistedDominanceOrder}).
We use $\cdot$ to denote a zero entry implied by Lemma~\ref{lemma:twistedKostkaMatrix}. The
entry highlighted in bold counting $\SSYT\bigl( (7,5,2,2) \bigr)_{((4,4),(4,2,2))}$ 
is used in Example~\ref{ex:444tableauBijection}.
\label{fig:444matrix}
}
\end{figure}

\begin{example}\label{ex:444matrix}
The stable
transition matrix $\KM(1)$ is shown in Figure~\ref{fig:444matrix} below. It was computed 
using the \textsc{Magma} code available as part of the arXiv submission of this paper using
\hbox{\texttt{TwistedIntervalMatrix(2, [6,4,4,2] :}} \texttt{q := [10,2,2,2]);}.
The entry relevant to Example~\ref{ex:444tableauBijection} is highlighted in bold in the bottom row.
As remarked at the end of \S\ref{subsec:symmetricFunctionsForSignedWeightLemma} 
it is instructive to check that the entries of $0$ correspond
to pairs of partitions incomparable in the $2$-twisted dominance order.
\end{example}

To finish this example,
note that the greatest partition in the $2$-twisted dominance order
(see  Definition~\ref{defn:ellTwistedDominanceOrder})
of $12+4M$ having first two columns of length $M+3$ is $(8+2M,2,2,2^M)$. 
By the definition in~\eqref{eq:444to822partitionSystem},
the least element of $\PSeq{M}$ is $(4+2M,4,4,2^M)$.
Therefore for each $M \in \N_0$ we have
\[ \PSeq{M} = \bigl[(4+2M,4,4,2^M), (8+2M,2,2,2^M)\bigr]_\unlhddotS
\]
and each $\PSeq{M}$ is an interval for the $2$-twisted dominance order.
The stability of $\PSeq{M}$ 
is therefore a special case of Corollary~\ref{cor:signedIntervalStable}.
See \S\ref{subsec:stablePartitionSystemsAsIntervals} for a related example.




%
%

\subsection{Cut up-sets and the plethysm  $\langle s_{(3) + (M)} \circ s_{(4)}, s_{(4,4,4) \oplus M ((1^2), (2)} \rangle$}\label{subsec:cutUpsetForLawOkitani}
The special case of Theorem~\ref{thm:nuStable} for the strongly maximal signed
weights \smash{$\bigl( (1^d), (m-d) \bigr)$} seen in Example~\ref{ex:LawOkitaniSignedWeightsAreStronglyMaximal}(i)
asserts that, if $d$ is even, then the plethysm coefficients
$\langle s_{\nu + (M)} \circ s_{(m)}, s_{\lambda \oplus M((1^d), (m-d))} \rangle$
are ultimately constant.
To prove this using the Signed Weight Lemma, we need a stable
partition system $( \PSeq{M})_{M \in \N_0}$ such that $\lambda \oplus M\bigl( (1^d), (m-d)\bigr) \in \PSeq{M}$
for each $M \in \N_0$.
As we saw in the overview in \S\ref{sec:overview}, we cannot expect to
define $\PSeq{M}$ to be the up-set
\smash{$\bigl(\lambda \oplus M(\kappa^\-,\kappa^\+) \bigr)^{\raisebox{-1pt}{$\unlhddotS$}}$}, 
where $\unlhddot$ is the $d$-twisted
dominance order, because typically the sizes of the up-sets 
grow, ruling out any bijection between them. 
In this subsection we shall see this problem
in the particular case of the plethysm coefficients 
\begin{equation}\label{eq:LawOkitaniExample}
\langle s_{(3) + (M)} \circ s_{(4)}, s_{(4,4,4) \opluss M ((1^2), (2)} \rangle 
\end{equation}
and resolve it using the stable partition system constructed in \S\ref{subsec:444to822partitionSystem}.

\subsubsection*{Up-sets are not stable}
We take
$\ell^\- = 2$ in the Signed Weight Lemma. The $2$-decomposition of a partition $\pi$ (see
Definition~\ref{defn:ellDecomposition}) is defined by
$\pi^\- = (\pi'_1, \pi'_2)$ and
$\pi^\+ = (\pi_1-2,\pi_2-2,\ldots, \pi_r - 2)$, where $r$ is maximal such that $\pi_r > 2$.
For example, if $\pi = (4+2M, 4, 4, 2^M)$ then $\pi^\- = (M+3,M+3)$ and $\pi^\+ = (2+2M,2,2)$.
By Lemma~\ref{lemma:twistedKostkaMatrix}, if $s_\sigma$ is a summand of $e_{\pi^\-}h_{\pi^\+}$
then $\sigma \unrhddot \pi$. Therefore, taking $g_\pi = e_{\pi^\-}h_{\pi^\+}$, 
the up-set of $(4+2M,4,4,2^M)$ in the $2$-twisted dominance
order satisfies condition (i) 
in the Signed Weight Lemma. We cannot 
take the up-sets $(4+2M,4,4,2^M)^\unlhddotS$ as our partition system because,
they are not stable.
Indeed, since $\pi \decMap \bigl\langle (3+M,3+M), (2+2M,2,2) \bigr\rangle$
and $(2^{3+b+M}, 1^{6-2b+2M}) \decMap \bigl\langle (9-b+3M,3+b+M), \varnothing \bigr\rangle$
we have
\[  \pi \unlhddot
(2^{3+b+M}, 1^{9-2b+2M}) \]
for all $b \le 6 + 2M$ 
and hence $\bigl|(4+2M,4,4,2^M)^{\!\unlhddotS}\hskip-1pt\bigr| \ge 6 + 2M$ and the 
sizes of the up-sets 
tend to infinity with $M$. This behaviour, that some `cut' is necessary before a sequence of up-sets
becomes stable, is typical. 

\subsubsection*{Cut up-sets are stable}
To get around the problem 
we use that condition~(i) in the Signed Weight Lemma (Lemma~\ref{lemma:SWL})
does not require that  
$\supp(g_\pi) \subseteq \PSeq{M}$ for all $\pi \in \PSeq{M}$
but instead, since  $\nu = (3)$ and $\muS = (4)$, only
the weaker condition that $\supp(g_\pi) \cap \supp(s_{(3+M)} \circ s_{(4)}) \subseteq \PSeq{M}$
for all $\pi \in \PSeq{M}$.
Since the support of the plethysm $s_{(3+M)} \circ s_{(4)}$ is contained in the support
of \smash{$s_{(4)} \times \stackrel{3+M}{\ldots} \times s_{(4)}$}, 
each partition in $\supp(s_{(3+M)} \circ s_{(4)})$ has at most $3+M$
parts. We therefore only need to consider partitions such as $(4+2M, 4, 4, 2^M)$ for which $\ell(\sigma)
\le 3+M$. This motivates the definition 
\[ \PSeq{M} = (4+2M,4,4,2^M)^\unlhddotS \cap \{ \sigma \in \Par(12+4M) : \ell(\sigma) \le 3+M \} \]
already given in~\eqref{eq:444to822partitionSystem} in an obviously equivalent form.
We saw in \S\ref{subsec:444to822partitionSystem} that $(\PSeq{M})_{M \in \N_0}$ 
is a stable partition system for $M \ge 1$ with respect to 
$\pmap : \PSeq{M} \rightarrow \PSeq{M+1}$ defined by 
$\pmap(\lambda) = \lambda \oplus \bigl( (1,1), (2) \bigr) = \lambda \sqcup (2) +(2) $ and 
the symmetric functions $g_\pi$.

\begin{proof}[Proof that $\langle s_{(3) + (M)} \circ s_{(4)}, s_{(4,4,4) \oplus M ((1^2), (2)} \rangle$
is ultimately constant.]
We shall check conditions (i) and (ii) in the
Signed Weight Lemma (Lemma~\ref{lemma:SWL}).
Let $\pi \in \PSeq{M}$. 
As we saw in \S\ref{subsec:444to822partitionSystem} we have 
$\pi^\- = (3+M,3+M)$. Hence
\begin{align*}
\supp& (g_\pi) \cap \supp (s_{(3+M)} \circ s_{(4)}) \\
&\subseteq \pi^\unlhddotS \cap \{ \sigma \in \Par(12+4M) : \ell(\sigma) \le 3+M \} \\
&\subseteq \{ \sigma \in (4+2M,4,4,2^M)^{\unlhddotS} : \ell(\sigma) \le 3+M \} \\
&= \PSeq{M}
\end{align*}
where the second line uses Lemma~\ref{lemma:twistedKostkaMatrix} on $\supp(g_\pi)$
and the length bound in
the previous paragraph on partitions in $\supp (s_{(3+M)} \circ s_{(4)})$, and the third line follows from
$\pi \unrhddot (4+2M,4,4,2^M)$.
Hence (i) holds for all $M \in \N_0$.

Now fix $M \in \N_0$ with $M \ge 1$ (so meeting the stability bound) 
and $\pi \in \PSeq{M}$. 
For (ii), it suffices to define
a bijection
\[ \begin{split}
\Hmap : &\PSSYTw{(3+M)}{(4)}{(3+M,3+M)}{\pi^\+} \\ &\hspace*{1in}
\rightarrow \PSSYTw{(3+M+1)}{(4)}{(3+M+1,3+M+1)}{\pi^\+ + (2)}. \end{split} \]
Let $T$ be in the codomain of $\Hmap$.
Observe that $T$ has $3+M+1$ integer entries of $-1$.
necessarily lying in distinct $(4)$-tableau entries. A similar argument considering $-2$ now shows
that each inner $(4)$-tableau in $T$ is of the form \raisebox{1pt}{$\young(\oM\tM\rx\ry)$} where $1 \le x \le y$.
Since $\pi \in \PSeq{M}$ we have $\pi \unrhddot (4+2M,4,4,2^M)$ and hence
$\pi^\+ + (2) \unrhd (2+2M + 2,2,2)$. Therefore $a(\pi^\+ + (2)) \ge 4+2M$ and, of the $8+2M$ positions
in the $(4)$-tableau entries of $T$
containing a positive entry, all but four positions contain $1$. In particular, since $M \ge 1$, the
leftmost $(4)$-tableau in $T$ is \raisebox{1pt}{$\young(\oM\tM 11)$}\,. Hence we may
define $\Hmap$ by inserting this inner $(4)$-tableau as a new leftmost inner tableau in a given
plethystic semistandard signed tableau in 
$\PSSYTw{(3+M)}{(4)}{(3+M,3+M)}{\pi^\+}$. 

We have now checked conditions
(i) and (ii) in the Signed Weight Lemma (Lemma~\ref{lemma:SWL}).
We
conclude that $\langle s_{(3) + (M)} \circ s_{(4)}, s_{(4,4,4) \oplus M ((1,1), (2)} \rangle$
is constant for \hbox{$M \ge 1$}.
\end{proof}

Computation shows that the stable multiplicity is in fact $1$.
The stability of this plethysm is a special case of Theorem~\ref{thm:nuStable}
and the map $\Hmap$ is as in the proof of condition (ii)
of the Signed Weight Lemma in the proof of Theorem~\ref{thm:nuStableSharp}.

\begin{example}\label{ex:omegaTwistGivesMore}
Applying the $\omega$-involution to the result just proved, we obtain that
$\langle s_{(3+M)} \circ s_{(1^4)}, s_{(3,3,3,3) \oplus M((2),(1,1))}
\rangle$ is constant
for $M \ge 1$. This is not an instance of Theorem~\ref{thm:nuStable} since,
according to Definition~\ref{defn:stronglyMaximalSignedWeight}, 
the singleton strongly maximal signed tableau families of shape $(1^4)$ have
as their unique elements the tableaux shown in the margin
of signed weights $\bigl( (4), \varnothing \bigr)$ and $\bigl( \varnothing, (1^4) \bigr)$
respectively. Therefore \marginpar{ \ $\young(\oM,\oM,\oM,\oM)\quad \young(1,2,3,4)$} 
$\bigl( (2), (1,1) \bigr)$ is not the signed weight of a strongly maximal signed tableau
family of shape $(1^4)$ and size $1$. This illustrates Remark~\ref{remark:asymmetry}.
\end{example}

\subsection{Stable partition systems as intervals}\label{subsec:stablePartitionSystemsAsIntervals}
A special feature of the stable partition system $\PSeq{M}$ in our running example is that
all the partitions $\pi \in \PSeq{M}$ have the same negative part in their $2$-decomposition,
namely $(3+M,3+M)$.
This was a deliberate choice in order to give a system that was not immediately stable,
but still of manageable size and useful for proving stability results.
To give a more typical example we
suppose that instead of the plethysm coefficients 
$\langle s_{(3) + (M)} \circ s_{(4)}, s_{(4,4,4) \oplus M ((1^2), (2)} \rangle$
in~\eqref{eq:LawOkitaniExample}, we want to prove that
\[ \langle s_{(4)+(M)} \circ s_{(4)}, s_{(6,6,4) \oplus M((1^2),(2))} \rangle \]
is ultimately constant.
Since the Schur functions constituents of $s_{(4+M)} \circ s_{(4)}$ have at
most $4+M$ parts we must relax the length bound in $\PSeq{M}$, and so we now define
\[ \RSeq{M} = \bigl\{ \sigma \in \Par(16+4M) : \sigma \unrhddot (6+2M,6,4,2^M), \ell(\sigma) \le 4+M \bigr\} \]
still working with the $2$-twisted dominance order.
Observe that $\RSeq{0}$ contains $(6,6,4)$, $(7,7,1,1)$, $(6,5,4,1)$,
$(5,5,4,2)$ 
with increasing
 negative parts $(3,3)$, $(4,2)$, $(4,3)$ and $(4,4)$ respectively.  
To show that $(\RSeq{M})_{M \in \N_0}$ is stable
we reinterpret each set $\RSeq{M}$ as an interval for the $2$-twisted dominance
order using the idea seen in Example~\ref{ex:lengthBound}. Observe that 
$\ell(\sigma) \le 4+M$ 
if and only if $\sigma^\- \unlhd (4+M,4+M)$. Since 
$(10+2M,2,2,2,2^M) \decMap \dec{(4+M,4+M)}{(8+2M)}$ is the greatest partition
of $16+4M$
in the $2$-twisted dominance order satisfying this condition, we have
\[ \RSeq{M} = \bigl[(6+2M,6,4,2^M), (10+2M,2,2,2,2^M)\bigr]_\unlhddotS. \] 
The stability of $(\RSeq{M})_{M \in \N_0}$ is then a special case of Corollary~\ref{cor:signedIntervalStable}.
The four bounds in this corollary are $M \ge -2$, $M \ge 1$, $M \ge 2$ and $M \ge 1$, 
respectively, so the stability bound is $M \ge 2$. By this corollary,
this bound is a sufficient condition for $\pmap : \RSeq{M} \rightarrow \RSeq{M+1}$
to be a bijection;
computation using the Magma code mentioned after
Definition~\ref{defn:F} shows that this bound is also necessary:
$|\RSeq{M}| = 40, 57, 60, 60$ for $0 \le M \le 3$.


%

\section{Stable partition systems defined by twisted 
intervals}\label{sec:stablePartitionSystems}

In this section we prove the technical result, Corollary~\ref{cor:signedIntervalStable},
that suitable sequences of intervals in
the $\ell^\-$-twisted dominance order define
stable partition systems. These are the stable partition systems
 we use in the Signed Weight Lemma (Lemma~\ref{lemma:SWL}) to prove Theorems~\ref{thm:muStable} and~\ref{thm:nuStable}.

\subsection{Unsigned intervals}\label{subsec:unsignedIntervals}
Recall from \S\ref{sec:preliminaryDefinitions} 
that  we write $\unLHD$ for the dominance order extended to partitions possibly of different sizes.
Given partitions $\gamma$ and $\delta$ each with at most $\ellp$ parts, 
we define the \emph{unsigned interval} 
$[\gamma,\delta]^\ellpb_\unLHDS$ by
\[ [\gamma,\delta]^\ellpb_\unLHDS = \{ \sigma \in \Par : \gamma \unLHD \sigma \unLHD \delta
\text{ and $\ell(\sigma) \le \ellp$} \}.
\]
Note that unless $|\gamma| \le |\delta|$ the interval is empty.

\begin{remark}\label{remark:intervalNotation}
If $|\gamma| = |\delta|$ and $\ell(\gamma) \le \ellp$  then
\smash{$ [\gamma,\delta]^\ellpb_\unLHDS = \{ \sigma \in \Par : \gamma \unlhd\, \sigma \unlhd \delta \}$}
since any partition $\sigma$ such that $\gamma \unlhd\, \sigma$ satisfies
$\ell(\sigma) \le \ell(\gamma)$; thus in this case we have
\smash{$[\gamma,\delta]^\ellpb_\unLHDS = [\gamma, \delta]_\unlhd$ }
and there is no ambiguity in using this simpler notation.
\end{remark}

In our applications, whenever $|\gamma| < |\delta|$, we shall take $\ellp = \ell^\-$
where $\ell^\-$ is the length of the negative part of the relevant signed weight, and so
the partitions in the unsigned interval $[\gamma,\delta]^\ellpb_\unLHDS$ 
all have at most $\ell^\-$ parts, as in the $\ell^\-$-decomposition
(see Definition~\ref{defn:ellDecomposition}).

\begin{definition}\label{defn:LBound}
Given partitions $\lambda$ and $\omega$ 
with $\lambda \unLHD \omega$ and a non-empty partition~$\kappa$,
each having at most $\ellp$ parts, let $\ell = \ell(\kappa)$ and set
\[ L_k = \frac{2\sum_{i=1}^{k-1} \omega_i + \omega_k + \omega_{k+1} - 2\sum_{i=1}^k \lambda_i}{\kappa_k - \kappa_{k+1}}
\]
for $k$ such that $1 \le k \le \ell$ and $\kappa_k > \kappa_{k+1}$.
Set $L_k = 0$ if $\kappa_k = \kappa_{k+1}$.
Define $\LBound\bigl( [\lambda, \omega]^{\ellpb}_\unLHDS , \kappa \bigr)$ 
to be the maximum of $L_1, \ldots, L_\ell$ if $\ellp > \ell$
and the maximum of $L_1,\ldots, L_{\ell-1}$ and $(|\omega| - |\lambda| - \omega_\ell)/\kappa_\ell$
if $\ellp = \ell$.
Set \smash{$\LBound\bigl( [\lambda, \omega]^\ellpb_\unLHDS, \varnothing \bigr) = 0$}.
\end{definition}

We remark that if $\ell(\lambda) \le \ell$ and $\ell(\omega) \le \ell$ then
$L_\ell = (2|\omega|- 2 |\lambda|  - \omega_\ell)/\kappa_\ell$. Thus
the bound in Definition~\ref{defn:LBound} may in this case
be strictly less
than the maximum of~$L_1,\ldots, L_\ell$.

Let $\kappa$ be a partition with $\ell(\kappa) \le \ellp$.
Since $\alpha \unLHD \beta$ implies $\alpha + \kappa \,\unLHD \beta + \kappa$ for any partitions
$\alpha$, $\beta$,
adding $\kappa$ defines an injective map from \smash{$[\gamma, \delta]^\ellpb_\unLHDS$}
to \smash{$[\gamma+\kappa,\delta+\kappa]^\ellpb_\unLHDS$}.


\begin{proposition}\label{prop:intervalStable}
Let $\lambda$ and $\omega$  be partitions and let $\kappa$ be a non-empty partition,
each having at most $\ellp$ parts.
Let $\pmap : \Par \rightarrow \Par$
be defined by $\pmap(\sigma) = \sigma + \kappa$. Let $M \in \N_0$. 
The injective map 
\[ \pmap : [\lambda + M\kappa, \omega + M\kappa]^\ellpb_\unLHDS
\longhookrightarrow [\lambda + (M+1)\kappa, \omega + (M+1)\kappa]^\ellpb_\unLHDS \]
is bijective provided $M \ge \LBound\bigl( [\lambda, \omega]^\ellpb_\unLHDS, \kappa \bigr)$.
\end{proposition}

\begin{proof}
Let $\ell = \ell(\kappa)$ and let $N = M+1$. Let $\tau \in 
[\lambda + N\kappa, \omega + N\kappa]^\ellpb_\unLHDS$.
Observe that $\tau$ is of the form $\sigma + \kappa$ for a partition $\sigma$ if and only if
all $\ell$ inequalities in the chain
\begin{equation}
\label{eq:chain} 
\tau_1 - \kappa_1 \ge \tau_2 - \kappa_2 \ge \ldots \ge \tau_\ell - \kappa_\ell \ge \tau_{\ell+1} 
\end{equation}
hold. (For instance if $\tau_k < \kappa_k$ for some $k$ then since $\tau_k - \kappa_k < 0 \le 
\tau_{\ell+1}$, at least one inequality fails to hold.) 
Fix $k \le \ell$. Using the hypotheses
$\tau \unRHD \lambda + N\kappa$ and $\tau \unLHD \omega + N\kappa$ we have
\begin{align} \label{eq:kMle}
\sum_{i=1}^{k-1} \tau_i &\le \sum_{i=1}^{k-1}\omega_i + N\sum_{i=1}^{k-1}\kappa_i, \\
\label{eq:kge}
\sum_{i=1}^{k} \tau_i  & \ge \sum_{i=1}^k \lambda_i + N\sum_{i=1}^k \kappa_i, \\
\label{eq:kPle}
\sum_{i=1}^{k+1} \tau_i  &\le \sum_{i=1}^{k+1} \omega_i + N\sum_{i=1}^{k+1}\kappa_i.
\end{align}
Subtracting~\eqref{eq:kMle} from~\eqref{eq:kge} 
we get $\tau_k \ge -\sum_{i=1}^{k-1} \omega_i + \sum_{i=1}^k \lambda_i + N\kappa_k$ 
and subtracting~\eqref{eq:kge} from~\eqref{eq:kPle} we get
$\tau_{k+1} \le \sum_{i=1}^{k+1} \omega_i - \sum_{i=1}^k \lambda_i + N\kappa_{k+1}$.
Subtracting these two equations in turn, to form the linear combination
$-\eqref{eq:kPle} + 2\eqref{eq:kge} - \eqref{eq:kMle}$, we get
\begin{equation}\label{eq:tauBig} 
\tau_k - \tau_{k+1} \ge -2\sum_{i=1}^{k-1} \omega_i - \omega_k -\omega_{k+1} 
+2 \sum_{i=1}^k \lambda_i + N(\kappa_k - \kappa_{k+1}).
\end{equation}
Recalling that $M = N-1$, we deduce that
\begin{equation}\label{eq:tauBigB} (\tau_k - \kappa_k) -(\tau_{k+1} - \kappa_{k+1}) \ge
B_k + M(\kappa_k -\kappa_{k+1}) \end{equation}
where $B_k = -2\sum_{i=1}^{k-1} \omega_i - \omega_k -\omega_{k+1} + 2\sum_{i=1}^k \lambda_i$.
Note that if $\kappa_k = \kappa_{k+1}$, the inequality $\tau_k - \kappa_k \ge \tau_{k+1} - \kappa_{k+1}$
holds simply because $\tau$ is a partition. Therefore by taking 
$M \ge -B_k / (\kappa_k - \kappa_{k+1})$ for each $k$ such that $\kappa_k > \kappa_{k+1}$,
we deduce from~\eqref{eq:tauBigB} that every inequality in the chain~\eqref{eq:chain}
holds. Hence, provided $M \ge L_1, \ldots, L_\ell$,
we may define $\sigma = \tau - \kappa$, knowing that
$\sigma$ is a well-defined partition. 

If $\ellp = \ell$ then rather than $M \ge L_\ell$, we have only the weaker hypothesis 
that $M \ge (|\omega| - |\lambda| - \omega_\ell)/ \kappa_\ell$. However, in this case
$\ell(\lambda) \le \ell$, $\ell(\tau) \le \ell$ and $\ell(\omega) \le \ell$ and
\[ \begin{split} \tau_\ell = \sum_{i=1}^\ell \tau_i - \sum_{i=1}^{\ell-1} \tau_i  & 
\ge \bigl( \sum_{i=1}^\ell \lambda_i + \sum_{i=1}^\ell N\kappa_i \bigr) -
     \bigl( \sum_{i=1}^{\ell-1} \omega_i + \sum_{i=1}^{\ell-1} N\kappa_i \bigr) 
\\ &\quad = \sum_{i=1}^\ell \lambda_i - \sum_{i=1}^{\ell-1} \omega_i + N\kappa_\ell
= |\lambda| - |\omega| + \omega_\ell + N\kappa_\ell.
\end{split} \]
Hence $\tau_\ell \ge \kappa_\ell$, as we require, 
provided $(N-1)\kappa_\ell \ge |\omega| - |\lambda| - \omega_\ell$.
Therefore in the case $\ellp = \ell$ we may replace $L_\ell$ with
the weaker bound $(|\omega| - |\lambda| - \omega_\ell)/ \kappa_\ell$, 
and again $\sigma$ is a well-defined partition.

It remains to show that \smash{$\sigma \in [\lambda + M \kappa, \omega + M\kappa]^\ellpb_\unLHDS$}.
Since $\tau \,\unRHD  \lambda + (M+1)\kappa$ we have
$\sum_{i=1}^k \tau_i \ge \sum_{i=1}^k \lambda_i + (M+1)\sum_{i=1}^k \kappa_i$ for each $k \in \N$.
Therefore
$\sum_{i=1}^k \sigma_i \ge \sum_{i=1}^k \lambda_i + M\sum_{i=1}^k \kappa_i$
for each $k \in \N$, and hence $\sigma \unRHD\! \lambda + M\kappa$.
Very similarly one shows that $\sigma \unLHD \omega + M\kappa$. 
Finally since \smash{$\tau \in 
[\lambda + N\kappa, \omega + N\kappa]^\ellpb_\unLHDS$} 
we have $\ell(\tau) \le \ellp$, and since $\ell(\kappa) = \ell \le \ellp$,
it follows that $\ell(\sigma) \le \ellp$.
Therefore
\smash{$\sigma \in [\lambda + M\kappa, \omega + M\kappa]^\ellpb_\unLHDS$} is a preimage of $\tau$
under  $\pmap$ and since $\pmap$ is injective, it follows that
$\pmap$ is bijective for 
\smash{$M \ge \LBound\bigl( [\lambda, \omega]^\ellpb_\unLHDS, \kappa \bigr)$}.
\end{proof}

We give one of the smallest examples in which the bound in Definition~\ref{defn:LBound}
and Proposition~\ref{prop:intervalStable} is $2$:
see Examples~\ref{ex:F} and~\ref{ex:unsignedTableauStable}
for cases where two parts of~$\kappa$ agree.

\begin{example}\label{ex:321special}
We take $\lambda = (1,1,1)$, $\omega = (3)$ and $\kappa = (3,2,1)$. Routine calculations show that
the unsigned intervals $[(1,1,1),(3)]_\unlhd$, $[(4,3,2),(6,2,1)]_\unlhd$, $[(7,5,3), (9,4,2)]_\unlhd$ and 
$[(10,7,4), (12,6,3)]_\unlhd$
are as shown below
\[ \left\{ \begin{matrix} (3) \\ (2,1) \\ (1,1,1) \end{matrix} \right\} \hookrightarrow
   \left\{ \begin{matrix} (6,2,1) \\ (5,3,1) \\ \mbf{(5,2,2)} \\ \mbf{(4,4,1)} 
   				\\ (4,3,2) \end{matrix} \right\} \hookrightarrow
   \left\{ \begin{matrix} (9,4,2) \\ \mbf{(9,3,3)} \\ (8,5,2) \\ (8,4,3) \\ (7,6,2)				 
   				\\ (7,5,3) \end{matrix} \right \} \hookrightarrow
   \left\{ \begin{matrix} (12,6,3) \\ (12,5,4) \\ (11,7,3) \\ (11,6,4) \\ (10,8,3) 
   				\\ (10,7,4) \end{matrix} \right \}.
\]
The elements not in the image of the map $\sigma \stackrel{\pmap}{\longmapsto} \sigma + (3,2,1)$ are highlighted. 
Setting $\PSeq{M} = [(1,1,1) + M(3,2,1), (3) + M(3,2,1)]_\unlhd$ we see that $\pmap : \PSeq{2} \rightarrow \PSeq{3}$ is a
bijection. Correspondingly, by Proposition~\ref{prop:intervalStable},
$\pmap : \PSeq{M} \rightarrow \PSeq{M+1}$ is
bijective provided $M \ge \LBound\bigl( [(1,1,1), (3)]_{\unlhd}, (3,2,1)\bigr)$ and the right-hand
side is the maximum of 
$\max \{ \frac{3-2}{3-2}, \frac{6-4}{2-1} \} =
\max \{ 1, 2 \} = 2$ and \hbox{$\frac{3-3-0}{1-0} = 0$}.
\end{example}

\subsection{Twisted intervals}\label{subsec:signedIntervals}
We now extend Proposition~\ref{prop:intervalStable} to the twisted case.
Recall from~\S\ref{subsec:upsetsAndTwistedIntervals}
that for fixed $\ell^\- \in \N_0$, and partitions $\gamma$, $\delta$ of the
same size we defined the twisted interval 
$[\gamma,\delta]_\unlhddotS = \{ \sigma \in \Par(p) : \gamma \unlhddot \sigma \unlhddot \delta \}$,
where $\unlhddot$ is the $\ell^\-$-twisted dominance order.
It is obvious that addition of partitions preserves the dominance order.
By conjugating partitions, 
the same result holds for joining. Despite this, addition
\emph{does not} preserve the $\ell^\-$-twisted dominance order. For instance, taking $\ell^\- = 1$
we have 
\begin{align*} \dec{(1)}{(1)} \decMap (2) &\hbox{$\,\unlhddot\,$} (1,1) \decMap 
\dec{(2)}{\varnothing},\\
\intertext{whereas after adding $(1,1)$, }
\dec{(2)}{(2)} \decMap (3,1) &\hbox{$\,\unrhddot\,$} (2,2) \decMap \dec{(2)}{(1,1)}. \end{align*}
The problem does not arise for addition of $\delta$ when the partitions involved are $\bigl( \ell^\-, \ell(\delta) \bigr)$-large, in the sense of
Definition~\ref{defn:large}. Moreover, joining is better behaved. We establish this
in a series of easy lemmas.

\begin{lemma}\label{lemma:addAndJoinToLarge}
Fix $\ell^\- \in \N_0$. Let $\alpha$, $\gamma$ and $\delta$ be partitions.
\begin{thmlist}
\item If $\alpha$ is $\bigl( \ell^\-, \ell(\delta) \bigr)$-large then
$(\alpha + \delta)^\- = \alpha^\-$ and $(\alpha+\delta)^\+ = \alpha^\+ + \delta$.
\item If $\ell(\gamma) \le \ell^\-$ then $(\alpha \sqcup \gamma')^\- = \alpha^\- + \gamma^\-$
and $(\alpha \sqcup \gamma')^\+ = \alpha^\+$.
\end{thmlist}
\end{lemma}

\begin{proof}The most transparent proof uses Young diagrams.
By hypothesis $[\alpha]$ contains the boxes $(i,j)$ for $1 \le i \le \ell(\delta)$
and $1 \le j \le \ell^\-$. Hence
addition of $\delta$ creates no new boxes in the first $\ell^\-$ columns of $\alpha$.
Similarly joining $\gamma'$ creates no new boxes outside the first $\ell^\-$ columns of $\alpha$.
\end{proof}

\begin{lemma}\label{lemma:adjoinToLarge}
Let $\kappa^\-$ and $\kappa^\+$ be partitions. 
If $\alpha$ is a \smash{$\bigl( \ell(\kappa^\-), \ell(\kappa^\+) \bigr)$}-large partition then
in the $\ell(\kappa^\-)$-decomposition of $\alpha \,\oplus\, \swtp{\kappa}$ we have
\smash{$\bigl( \alpha \oplus \swtp{\kappa} \bigr)^\-$} $= \alpha^\- + \kappa^\-$ and
\smash{$\bigl( \alpha \oplus \swtp{\kappa} \bigr)^\+ = \alpha^\+ + \kappa^\+$}.
Moreover, adding and joining to $\alpha$ are commuting operations.
\end{lemma}

\begin{proof}
This is immediate from Lemma~\ref{lemma:addAndJoinToLarge}.
\end{proof}

In particular, if $\KK \in \N$ and $\alpha$ is a \smash{$\bigl( \ell(\kappa^\-), \ell(\kappa^\+) \bigr)$}-large partition then $\alpha \oplus (\KK-1)\swtp{\alpha}$ is a partition having
$\ell^\-$-decomposition $\KK\decs{\alpha^\-}{\alpha^\+}$.
We use this remark in the proof of Lemma~\ref{lemma:omegaKappaEllDecomposition}.

\begin{lemma}[Twisted dominance order on large partitions is preserved by adjoining]\label{lemma:signedDominancePreservedByOplusWhenLarge}
Let $\kappa^\-$, $\kappa^\+$ be partitions. Set $\ell^\- = \ell(\kappa^\-)$. 
Suppose that $\alpha$ and $\beta$ are $\bigl( \ell(\kappa^\-), \ell(\kappa^\+) \bigr)$ large.
Then, working in the $\ell^\-$-twisted dominance order, $\alpha \unlhddot \beta$ if and only if
$\alpha \oplus \swtp{\kappa} \unlhddot \beta \oplus \swtp{\kappa}$.
\end{lemma}

\begin{proof}
By Lemma~\ref{lemma:adjoinToLarge} we have
\begin{align}
\label{eq:oplusM}
\bigl( \lambda \oplus (\kappa^\-, \kappa^\+) \bigr)^\- &= \lambda^\- + \kappa^\-, \\
\label{eq:oplusP}
\bigl( \lambda \oplus (\kappa^\-, \kappa^\+) \bigr)^\+ &= \lambda^\+ + \kappa^\+.
\end{align}
Therefore it is
equivalent to show that $\decs{\alpha^\-}{\alpha^\+} \unlhd \decs{\beta^\-}{\beta^\+}$ 
if and only if
$\decs{\alpha^\- + \kappa^\-}{\alpha^\+ + \kappa^\+} \unlhd \decs{\beta^\- + \kappa^\-}{\beta^\+ + \kappa^\+}$, which is obvious.
\end{proof}

\begin{lemma}\label{lemma:downsetIsLarge}
Fix $\ell^\- \in \N_0$ and let $\ell^\+ \in \N_0$.
Let $\omega$ be a $(\ell^\-+1, \ell^\+)$-large partition. If $\pi \,\unlhddotS\, \omega$
then $\pi$ is $(\ell^\-+1, \ell^\+)$-large.
\end{lemma}

\begin{proof}
By Remark~\ref{remark:ellDecompositionLarge}, a 
partition $\alpha$ is $(\ell^\-+1, \ell^\+)$-large if and only if 
\hbox{$\ell(\alpha^\+) \ge \ell^\+$}.
Therefore $\ell(\omega^\+) \ge \ell^\+$ 
and  since
$\pi \unlhddots \omega$, Lemma~\ref{lemma:signedDominancePositiveLength} implies that
 $\ell(\pi^\+) \ge \ell^\+$. 
\end{proof}

\begin{lemma}\label{lemma:twistedDownsetLarge}
Fix $\ell^\- \in \N_0$ and let $\ell^\+ \in \N_0$.
Let $\lambda$ and $\omega$ be partitions such that $\omega$ is $(\ell^\-+1,\ell^\+)$-large.
If $\pi \in [\lambda, \omega]_\unlhddotS$ then $\pi$ is $(\ell^\-,\ell^\+)$-large. 
In particular $\lambda$ is $(\ell^\-,\ell^\+)$-large.
\end{lemma}

\begin{proof}
This is immediate from Lemma~\ref{lemma:downsetIsLarge}
since, as used in Remark~\ref{remark:ellDecompositionLarge}, 
if a partition is $(\ell^\-+1,\ell^\+)$-large then it is $(\ell^\-,\ell^\+)$-large.
\end{proof}

When $\ell^\- \not=0$, the hypothesis in the previous lemma cannot be weakened
to the apparently more natural condition that $\omega$ is $(\ell^\-, \ell^\+)$-large. For example,
taking $\ell^\- = 2$, both $(3,2)$ and $(2,2,1)$ are $(2,2)$-large, but since
\[ (3,2) \decMap \dec{(2,2)}{(1)},\ (3,1,1) \decMap \dec{(3,1)}{(1)},\ (2,2,1) \decMap \dec{(3,2)}{\varnothing}, \]
the twisted interval $[(3,2), (2,2,1)]_\unlhddotS$ for the $2$-twisted dominance order 
contains $(3,1,1)$ which is not $(2,2)$-large. See Remark~\ref{remark:ellDecompositionLarge} for 
one sign that the hypothesis in Lemma~\ref{lemma:twistedDownsetLarge} is in fact the correct one
when $\ell^\- \not= 0$. By Remark~\ref{remark:becomesLarge}, when $\ell^\- \not= 0$, any partition
can be made $\bigl( \ell(\kappa^\-)+1, \ell(\kappa^\+) \bigr)$-large by 
sufficiently many applications of the adjoining map $\lambda \mapsto \lambda
\oplus (\kappa^\-, \kappa^\+)$ so, as usual, any `largeness' assumption are made without loss of generality.

The $\mathrm{L}$ bounds  in the following proposition are defined in Definition~\ref{defn:LBound}.
Remark~\ref{remark:intervalNotation} explains the different notations
for intervals in the dominance order in the first two bounds in the lemma.

\begin{proposition}[Partition Stability]\label{prop:signedIntervalStable}
Let $\kappa^\-$ and $\kappa^\+$ be partitions. Fix $\ell^\- = \ell(\kappa^\-)$ and $\ell^\+ = \ell(\kappa^\+)$.
Let $\omega$ be a partition
 and let $\lambda \unlhddots \omega$ in the $\ell^\-$-twisted
dominance order.
If $\ell^\- \not=0$ then suppose that $\omega$ is $\bigl(\ell^\-+1, \ell^\+\bigr)$-large.
 For each $M \in \N_0$,
there is an injective map of intervals for the $\ell^\-$-twisted
dominance order
\[ \begin{split} 
\pmap : \bigl[\lambda  \,\opluss\, &M(\kappa^\-, \kappa^\+), \omega \opluss M(\kappa^\-,\kappa^\+)\bigr]_\unlhddotS \\
&\hookrightarrow \bigl[\lambda \opluss (M+1)(\kappa^\-,\kappa^\+), \omega 
\opluss (M+1)(\kappa^\-,\kappa^\+)\bigr]_\unlhddotS
\end{split} \]
defined, using the $\ell^\-$-decomposition, 
by $\pmap(\sigma) = \sigma \oplus (\kappa^\-, \kappa^\+)$. This map
is bijective provided $M \ge L$ where $L$ is the maximum of
\begin{bulletlist}
\item $\LBound\bigl([\lambda^\-,\omega^\-]^\ellmb_\unLHDS, \kappa^\-\bigr)$,
\item $\LBound\bigl([\lambda^\+,\omega^\++(|\lambda^\+|-|\omega^\+|)]_\unlhd, \kappa^\+\bigr)$,
\item $\bigl( \omega_1^\+ + \omega_2^\+ - 2\lambda_1^\+ + 2|\lambda^\+| - 2|\omega^\+| 
\bigr)/(\kappa^\+_1-\kappa^\+_2)$,
\item $\bigl( \max( \ell(\lambda^\+), \ell^\+ ) + |\omega^\-| - |\lambda^\-| - \omega^\-_{\ell^\-} \bigr)/
 \kappa^\-_{\ell^\-}$
\end{bulletlist}
where the third is omitted if $\kappa^\+_1 = \kappa^\+_2$ and the fourth is 
omitted if $\kappa^\- = \varnothing$.
\end{proposition}

\begin{proof}
If $\ell^\- \not=0$ then,
by hypothesis the partitions $\omega \opluss M(\kappa^\-,\kappa^\+)$
and $\omega \opluss (M+1)(\kappa^\-,\kappa^\+)$ 
are $\bigl(\ell^\-+1, \ell^\+\bigr)$-large.
Hence, by Lemma~\ref{lemma:twistedDownsetLarge}, every partition in each
twisted interval 
$\bigl(\ell^\-, \ell^\+\bigr)$-large. If $\ell^\- = 0$
then \emph{any} partition is $\bigl(0, \ell^\+\bigr)$ large.
By the `only if' direction of 
Lemma~\ref{lemma:signedDominancePreservedByOplusWhenLarge} it now follows that the map $\pmap$ 
on these 
twisted intervals
preserves the $\ell^\-$-twisted dominance order. Hence
the image of the left-hand twisted interval under~$\pmap$ is contained in the
right-hand twisted interval. 
Set $N = M +1$ and suppose that $M$ satisfies the inequalities in the proposition.
By Lemma~\ref{lemma:adjoinToLarge} and Lemma~\ref{lemma:twistedDominanceOrderOldDefinition} we have
$\tau \in \bigl[\lambda  \opluss N(\kappa^\-, \kappa^\+), \omega \opluss N(\kappa^\-,\kappa^\+)\bigr]_\unlhddotS$
if and only if
\begin{defnlistE}
\item[(a)] $\lambda^\- + N\kappa^\- \unLHD \tau^\- \unLHD\, \omega^\- + N\kappa^\-$;
\item[(b)(i)] $\lambda^\+ + N\kappa^\+ \unlhd \tau^\+ + \bigl( |\lambda^\+| + N|\kappa^\+| - |\tau^\+| \bigr)$ and $|\tau^\+| \le |\lambda^\+| + N|\kappa^\+|$;
\item[(b)(ii)] $\tau^\+ \unlhd \omega^\+ + N\kappa^\+ + \bigl( |\tau^\+| - |\omega^\+| - N|\kappa^\+| \bigr)$ and $|\omega^\+| + N|\kappa^\+| \le |\tau^\+|$.
\end{defnlistE}
It is easily seen that (b)(i) and (b)(ii) are equivalent to the two conditions
$\lambda^\+ + N\kappa^\+ \unlhd \tau^\+ + \bigl( |\lambda^\+| + N|\kappa^\+| - |\tau^\+| \bigr)
\unlhd \omega^\+ + N\kappa^\+ + \bigl( |\lambda^\+| - |\omega^\+| \bigr)$
and 
\begin{equation}\label{eq:signedIntervalStableSizeBound} 
|\omega^\+| + N|\kappa^\+| \le |\tau^\+| \le |\lambda^\+| + N|\kappa^\+|. \end{equation}
Note that by definition of the $\ell^\-$-decomposition (see Definition~\ref{defn:ellDecomposition}),
$\lambda^\-$ and~$\omega^\-$ have at most $\ell^\-$ parts, where $\ell^\- = \ell(\kappa^\-)$.
Thus (a), (b)(i) and (b)(ii)  hold if and only if~\eqref{eq:signedIntervalStableSizeBound} holds
and
\[ \tau^\- \in [\lambda^\- + N\kappa^\-, \omega^\- + N\kappa^\-]^{\ellmb}_\unLHDS \] 
and
\[ \tau^\+ + \bigl( |\lambda^\+| + N|\kappa^\+| - |\tau^\+| \bigr) \in 
\bigl[\lambda^\+ + N\kappa^\+, \omega^\+ + N\kappa^\+ + (|\lambda^\+|-|\omega^\+|) \bigr]_{\unlhd}.\]

By Proposition~\ref{prop:intervalStable}, the map
\[ [\lambda^\- + (N-1)\kappa^\-, \omega^\- + (N-1)\kappa^\-]^\ellmb_\unLHDS
\,\rightarrow [\lambda^\- + N\kappa^\-, \omega^\- + N\kappa^\-]^\ellmb_\unLHDS \]
defined by adding $\kappa^\-$
is bijective if
$N-1 \ge \LBound\bigl([\lambda^\-,\omega^\-]^\ellmb_\unLHDS, \kappa^\-\bigr)$, as we have assumed.
Similarly, the map
\[ \begin{split}
[\lambda^\+ + (N-1)\kappa^\+, &\omega^\+ + (N-1)\kappa^\+ + (|\lambda^\+| - |\omega^\+|)]_\unlhd \\
&\rightarrow [\lambda^\+ + N\kappa^\+, \omega^\+ + N\kappa^\+ + (|\lambda^\+| - |\omega^\+|)]_\unlhd
\end{split} \]
defined by adding~$\kappa^\+$
is bijective if
$N-1 \ge \LBound\bigl([\lambda^\+,\omega^\++(|\lambda^\+|-|\omega^\+|)]_\unlhd, \kappa^\+\bigr)$,
again as we have assumed.
 Hence there exist unique partitions 
\begin{equation}
\label{eq:signedIntervalStableTauM}
\sigma^\- \in 
[\lambda^\- + (N-1)\kappa^\-, \omega^\- + (N-1)\kappa^\-]^\ellmb_\unLHDS \end{equation}
such that $\tau^\- = \sigma^\- + \kappa^\-$
and 
\begin{equation}
\label{eq:signedIntervalStableTheta} 
\theta \in [\lambda^\+ + (N-1)\kappa^\+, \omega^\+ + (N-1)\kappa^\+ + (|\lambda^\+| - |\omega^\+|)\bigr]_\unlhd \end{equation}
such that 
\begin{equation}
\label{eq:signedIntervalExtra}
\tau^\+ + (|\lambda^\+| + N|\kappa^\+| - |\tau^\+|) = \theta + \kappa^\+.\end{equation}
The unique integer sequence $\sigma^\+$ such that $\sigma^\+ + \kappa^\+ = \tau^\+$ is
\begin{equation}\label{eq:signedIntervalStableTauP}
\sigma^\+ = \theta - (|\lambda^\+| + N|\kappa^\+| - |\tau^\+|) .
\end{equation}
We shall show that $\sigma^\+$ is a partition, provided $N$ is sufficiently large. Suppose first of all that
$\kappa_1^\+ = \kappa_2^\+$. Then by~\eqref{eq:signedIntervalExtra}, $\theta_1 - \theta_2 = 
\tau_1^\+ + (|\lambda^\+| + N|\kappa^\+| - |\tau^\+|) - \tau_2^\+
\ge |\lambda^\+| + N|\kappa^\+| - |\tau^\+|$ and hence by~\eqref{eq:signedIntervalStableTauP},
$\sigma^\+_1 - \sigma^\+_2 \ge 0$,
with no condition on $N$.
Now suppose that $\kappa_1^\+ > \kappa_2^\+$. 
By~\eqref{eq:signedIntervalStableTheta}, we have
\begin{align*}
\theta_1 - \theta_2 &= 2\theta_1 -  (\theta_1  +\theta_2)  \\
& \ge 2\bigl( \lambda^\+_1 + (N-1)\kappa^\+_1 \bigr)  \\
& \qquad - \bigl( \omega^\+_1 + \omega^\+_2 + (N-1)(\kappa^\+_1 + \kappa^\+_2) + (|\lambda^\+| - |\omega^\+|) \bigr) \\
& = 2 \lambda^\+_1 - \omega^\+_1 - \omega^\+_2 + (N-1)(\kappa^\+_1 - \kappa^\+_2) - |\lambda^\+| + |\omega^\+|
\intertext{
 By~\eqref{eq:signedIntervalStableTauP},}
\sigma^\+_1   - \sigma^\+_2 &= \theta_1 - \theta_2  - |\lambda^\+| - N|\kappa^\+| +|\tau^\+|
\\ &\ge \theta_1 - \theta_2 - |\lambda^\+| + |\omega^\+| 
\\ &\ge 2 \lambda^\+_1 - \omega^\+_1 - \omega^\+_2 + (N-1)(\kappa^\+_1 - \kappa^\+_2) - 2|\lambda^\+| + 2|\omega^\+|
\end{align*}
where the middle line follows from the first inequality in~\eqref{eq:signedIntervalStableSizeBound}
that $|\omega^\+|+N|\kappa^\+| \le |\tau^\+|$ and the third line by substituting the expression for
$\theta_1-\theta_2$ just found. 
Hence it suffices if
\[ N-1 \ge \frac{\omega_1^\+ + \omega_2^\+ - 2\lambda_1^\+ + 2|\lambda^\+| - 2|\omega^\+|}{\kappa^\+_1-\kappa^\+_2} \]
which, setting $M = N-1$, is the third condition.


We have now defined partitions $\sigma^\-$ and $\sigma^\+$ such that, 
provided $\decs{\sigma^\-}{\sigma^\+}$ is a well-defined $\ell^\-$-decomposition, 
the partition $\sigma$ defined by $\sigma \decMap \decs{\sigma^\-}{\sigma^\+}$
satisfies $\sigma \oplus (\kappa^\-,\kappa^\+) = \tau$. 
If $\ell^\- = 0$ (or equivalently, $\kappa^\- = \varnothing$) this is immediate, and
similarly it is immediate if $\kappa^\+ = \varnothing$.
We may therefore assume that $\ell^\- \ge 1$ and $\kappa^\+ \not= \varnothing$.
We then require
$\sigma^\-_{\ell^\-} \ge \ell(\sigma^\+)$. 
By~\eqref{eq:signedIntervalStableTauM} we have $\sigma^\- \unLHD \omega^\- + (N-1)\kappa^\-$.
Define an integer sequence $\psi$ by 
$\psi_j = \sigma^\-_j$ for $1 \le j < \ell^\-$ and $\psi_{\ell^\-} = 
\sigma^\-_{\ell^\-} + |\omega^\-| + (N-1)|\kappa^\-| - |\sigma^\-|$. 
Since $\sigma^\- \unLHD \omega^\- + (N-1)\kappa^\-$ by~\eqref{eq:signedIntervalStableTauM},
$\sigma^\-$ is a weight, having non-negative entries.
Moreover, after this equalization of sizes,
we have $\psi \,\unlhd\, \omega^\- + (N-1)\kappa^\-$. 
Since each side has at most $\ell^\-$ parts, it follows from the dominance order that
$\psi_{\ell^\-} \ge \omega^\-_{\ell^\-} + (N-1)\kappa^\-_{\ell^\-}$.
Now using that $|\omega^\-| + (N-1)|\kappa^\-| - |\sigma^\-| = |\omega^\-| + N|\kappa^\-| - |\tau^\-|
\le \bigl(|\omega^\-| + N|\kappa^\-|\bigr) - \bigl(|\lambda^\-| + N|\kappa^\-|\bigr) = |\omega^\-| - |\lambda^\-|$
we obtain
\[ \sigma^\-_{\ell^\-} 
\ge \omega^\-_{\ell^\-} + (N-1)\kappa^\-_{\ell^\-} - ( |\omega^\-| - |\lambda^\-|).\]
By~\eqref{eq:signedIntervalStableTheta} we have
$\theta \unrhd \lambda^\+ + (N-1)\kappa^\+$, and since $\ell(\kappa^\+) = \ell^\+$ we have
\[ \ell(\sigma^\+) = \ell(\theta) \le \max( \ell(\lambda^\+), \ell^\+ ).\]
Therefore, comparing the two previous displayed equations,
 a sufficient condition for $\sigma$ to be well-defined is
\[ \omega_{\ell^\-}^\- + (N-1)\kappa^\-_{\ell^\-} - (|\omega^\-| - |\lambda^\-|) 
\ge \max ( \ell(\lambda^\+), \ell^\+ ).\]
Rearranging and, as before, setting $M=N-1$, this becomes the fourth condition.
\end{proof}

This shows that twisted intervals for the $\ell^\-$-twisted dominance order, defined for suitable
large partitions, satisfy
condition~(i) in the definition of a stable partition system
(Definition~\ref{defn:stablePartitionSystem})
for the 
map $\pmap$ in Proposition~\ref{prop:signedIntervalStable} (Partition Stability).
We remark that the example in \S\ref{subsec:stablePartitionSystemsAsIntervals} shows
one case where the third bound, required in the middle part of the proof, is 
the only bound that is tight
and Example~\ref{ex:F} below shows that the most technical fourth bound may also be 
the only bound that is tight.

\subsection{Positions for tableau stability}
We must now verify condition (b) in the definition of a stable partition system 
(Definition~\ref{defn:stablePartitionSystem}).
The critical positions in tableaux are defined below.
In this section, only the case where $\mus = \varnothing$ is needed: the 
definition is used in full generality in \S\ref{subsec:stablePlethysticTableau} below.
Recall that $a(\lambda)$ denotes the first part of a partition $\lambda$.

\begin{definition}\label{defn:skewPositions}{\ }
Let $\muS$ be a skew partition. 
Let $\kappa^\-$ and $\kappa^\+$ be partitions. Fix $\ell^\- = \ell(\kappa^\-)$
and let $\ell^\+ = \ell(\kappa^\+)$.
For $1 \le r^\- \le \ell^\-$ and $1 \le r^\+ \le \ell^\+$,
\begin{defnlist}
\item the $\rM$-\emph{top position} of $\muS$ is 
$\bigl( \max( \ell(\mus),  \ell(\mu^\+), \ell^\+, \mu^\-_{r^\-+1}),r^\-)$,
\item the $\rM$-\emph{bottom position} of $\muS$ is 
$(k+\kappa^\-_{r^\-}-\kappa^\-_{r^\-+1}, r^\-)$ where
$(k, r^\-)$ is the $\rM$-top position of $\muS$.
\item the $\rP$-\emph{left position} of $\muS$ is  $\bigl( r^\+, \ell^\- + \max(a(\mus), \mu^\+_{r^\+ +1})\bigr)$,
\item the $\rP$-\emph{right position} of $\muS$ is 
$(\rP, k + \kappa^\+_{r^\+} - \kappa^\+_{r^\++1})$ where $(\rP, k)$
is the $\rP$-left position of $\muS$. 
\end{defnlist}
\end{definition}

Note that if $\mu_\rP < \ell^\-$ then $\mu^\+_{\rP+1} = 0$
and so the $\rP$-left position of $\muS$ is $(\rP, \ell^\- + a(\mus))$ and is 
not contained in $[\mu]$. Similarly, if $\mus = \varnothing$ and $\mu^\+ = \varnothing$ and $\kappa^\+ = \varnothing$
then
since $\mu^\-$ has at most $\ell^\-$ parts, the $\ell^\-$-top position
is $(0,\ell^\-)$. We therefore refer to `positions' rather than `boxes'.

\begin{example}
\label{ex:positions}
Take $\kappa^\- = (1,1)$ and $\kappa^\+ = (2)$. 
The map \hbox{$\pmap : \Par \rightarrow \Par$} in
Proposition~\ref{prop:signedIntervalStable} is defined by  $\pmap(\sigma) = 
\sigma \oplus \bigl( (1,1), (2) \bigr) = \sigma \sqcup (2) + (2)$.
The numbers in the 
diagrams below show the $1$-top, $2$-top and $1$-left positions 
in the partitions obtained from $(2)$ and $(1,1)$ by adjoining according 
to~$\pmap$.
Following our usual convention, top positions, relevant to the insertion
of negative entries, are marked by bold numbers.
For instance the $2$-top position is $(1,2)$ in every partition.
The $2$-bottom and $1$-right positions are indicated by shading;
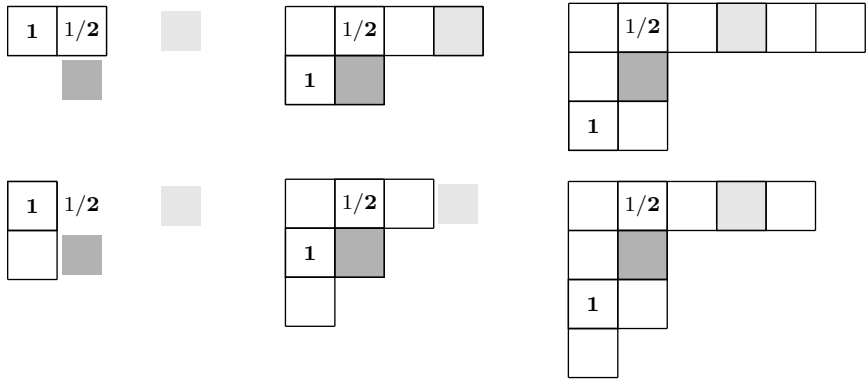
\begin{align*}
&\raisebox{0cm}{\begin{tikzpicture}[x=\smallBoxWidth,y=-\smallBoxWidth]
\pyoungInner{ {{\ ,\ }} }
\tBF{0}{1}{}{optdarkgrey}
\tBF{0}{2}{}{optdarkgrey}
\tBFN{1}{2}{}{grey}
\tBFN{0}{4}{}{verylightgrey}
\node at (1.5,1.5) {$\phantom{\scriptstyle 1/\mbf{1}}$};
\node at (1.5,1.5) {$\scriptstyle \mbf{1}$};
\node at (2.5,1.5) {$\scriptstyle 1/\mbf{2}$};
\end{tikzpicture}}\spy{7pt}{\qquad}
\raisebox{-0.075cm}{
\begin{tikzpicture}[x=\smallBoxWidth,y=-\smallBoxWidth]
\tBF{0}{2}{}{optdarkgrey}
\tBF{1}{1}{}{optdarkgrey}
\tBF{1}{2}{}{grey}
\tBF{0}{4}{}{verylightgrey}
\pyoungInner{ {{\ ,\ ,\ ,\ }, {\ ,\ }} }
\node at (1.5,2.5) {$\scriptstyle \mbf{1}$};
\node at (2.5,1.5) {$\scriptstyle 1/\mbf{2}$};
\end{tikzpicture}}\spy{7pt}{\qquad}
\raisebox{-0.68cm}{
\begin{tikzpicture}[x=\smallBoxWidth,y=-\smallBoxWidth]
\tBF{2}{1}{}{optdarkgrey}
\tBF{0}{2}{}{optdarkgrey}
\tBF{1}{2}{}{grey}
\tBF{0}{4}{}{verylightgrey}
\pyoungInner{ {{\ ,\ ,\ ,\ ,\ ,\ }, {\ ,\ }, {\ ,\ }} }
\node at (2.5,1.5) {$\scriptstyle 1/\mbf{2}$};
\node at (1.5,3.5) {$\scriptstyle \mbf{1}$};
\end{tikzpicture}}\\[6pt]
&\begin{tikzpicture}[x=\smallBoxWidth,y=-\smallBoxWidth]
\tBF{0}{1}{}{optdarkgrey}
\tBFN{0}{2}{}{optdarkgrey}
\tBFN{1}{2}{}{grey}
\tBFN{0}{4}{}{verylightgrey}
\pyoungInner{ {{\ }, {\ }} }
\node at (1.5,1.5) {$\phantom{\scriptstyle 1/\mbf{1}}$};
\node at (1.5,1.5) {$\scriptstyle \mbf{1}$};
\node at (2.5,1.5) {$\scriptstyle 1/\mbf{2}$};
\end{tikzpicture}\spy{7pt}{\qquad}
\raisebox{-0.62cm}{
\begin{tikzpicture}[x=\smallBoxWidth,y=-\smallBoxWidth]
\tBF{0}{2}{}{optdarkgrey}
\tBF{1}{1}{}{optdarkgrey}
\tBFN{1}{2}{}{grey}
\tBF{0}{4}{}{verylightgrey}
\pyoungInner{ {{\ ,\ ,\ ,\ }, {\ }, {\ }} }
\node at (1.5,2.5) {$\scriptstyle \mbf{1}$};
\node at (2.5,1.5) {$\scriptstyle 1/\mbf{2}$};
\end{tikzpicture}}\spy{7pt}{\qquad}
\raisebox{-1.3cm}{\hspace*{2pt}
\begin{tikzpicture}[x=\smallBoxWidth,y=-\smallBoxWidth]
\tBF{2}{1}{}{optdarkgrey}
\tBF{0}{2}{}{optdarkgrey}
\tBF{1}{2}{}{grey}
\tBF{0}{4}{}{verylightgrey}
\pyoungInner{ {{\ ,\ ,\ ,\ ,\ ,\ }, {\ ,\ }, {\ }, {\ }} }
\node at (1.5,3.5) {$\scriptstyle \mbf{1}$};
\node at (2.5,1.5) {$\scriptstyle 1/\mbf{2}$};
\end{tikzpicture}}
\end{align*}
%
Since $\kappa^\-_1 = \kappa^\-_2$ the $1$-top and $1$-bottom positions coincide
in every case. (We shall see in  Definition~\ref{defn:F} that this makes them irrelevant to our application.)
Since $\kappa^\-_2 - \kappa^\-_3 = 1 - 0 = 1$, the $2$-bottom position is
always one position below the $2$-top position and 
since $\kappa^\+_1 - \kappa^\+_2 = 2 - 0 = 2$, the $1$-right position
is always two positions right of the $1$-left position.

For a further example in the general skew case, also showing the behaviour when $\ell(\mu^\+) > \ell(\kappa^\+)$, see Example~\ref{ex:skewPositions}.
\end{example}

\subsection{The $\Fmap$ insertion map on tableaux}
We now show how these positions can be used to define a bijection between semistandard
signed tableaux. We admit the following results
are technical, and so we give two substantial examples.
See also
Example~\ref{ex:FoulkesStablePartitionSystem}
and the end of
\S\ref{subsec:444to822partitionSystem}; the 
bijections in these examples can now be seen
to be instances of $\Fmap$ and its inverse.

\begin{example}\label{ex:F}
Consider the twisted intervals
\[ \PSeq{M} = \bigl[(4,2) \sqcup (2^M) + (2M) , (3,2,1)  \sqcup (2^M) + (2M)\bigr]_\unlhddotS. \]
for the $2$-twisted dominance order.
By Proposition~\ref{prop:signedIntervalStable} (Partition Stability), the map $\pmap : \PSeq{M} \rightarrow \PSeq{M+1}$ defined by $\lambda \mapsto \lambda \oplus \bigl( (1,1), (2) \bigr)$
is bijective for $M \ge 0$. (Note that $(3,2,1)$ is $(3,1)$-large; the four bounds
on $M$ are respectively $M\ge -1$, $M \ge -1$, $M \ge -\mfrac{1}{2}$ and $M \ge 0$.)
This gives a bijection between the row and column labels of the
matrices $\KM(M)$ in condition (b) of a stable partition system 
(Definition~\ref{defn:stablePartitionSystem}), as indicated below. We include
the set $\PSeq{-1}$, defined to be $[(2), (1,1)]_\unlhddotS$ below: even though $(2)$ is not $(3,1)$-large,
the proof of Proposition~\ref{prop:signedIntervalStable} still applies; the bounds on $M$
are now $M \ge 0$, $M \ge 0$, $M \ge 0$ and $M \ge 1$, so the necessary restriction on $M$
comes from the technical final paragraph of the proof.
\[ \hspace*{-6pt}
\begin{pNiceMatrix}[first-row,first-col,nullify-dots] & \rb{(1,1)} & \rb{(2)} \\
\decss{(2)}{\varnothing} & 1 & \cdot \\
\decss{(1,1)}{\varnothing} & 1 & 1 \end{pNiceMatrix}\  \
\begin{pNiceMatrix}[first-row,first-col,nullify-dots] & \rb{(3,2,1)} & \rb{(4,1,1)} & \rb{(4,2)} \\
\decss{(3,2)}{(1)} & 1 & \cdot & \cdot \\
\decss{(3,1)}{(2)}& 1 & 1 & \cdot \\
\decss{(2,2)}{(2)} & 2 & 1 & 1 \end{pNiceMatrix}\  \
\begin{pNiceMatrix}[first-row,first-col,nullify-dots] & \rb{(5,2,2,1)} & \rb{(6,2,1,1)} & \rb{(6,2,2)} \\
\decss{(4,3)}{(3)} & 1 & \cdot & \cdot \\
\decss{(4,2)}{(4)}& 1 & 1 & \cdot \\
\decss{(3,3)}{(4)} & 2 & 1 & 1 \end{pNiceMatrix}
\]
The tableaux enumerated by 
the bottom left matrix entries of $1$, $2$ and $2$ are shown below.

\smallskip
\centerline{
\raisebox{3.025cm}{\begin{tikzpicture}[x=0.55cm,y=-0.55cm] 
\tB{1}{1}{$\mbf{1}$}\tBFN{1}{2}{}{darkgrey}
\tB{2}{1}{$\mbf{2}$}
\end{tikzpicture}} \qquad\qquad
\raisebox{0cm}{
\begin{tikzpicture}[x=0.55cm,y=-0.55cm] 
\begin{scope}[yshift=2.5cm]
\tB{1}{1}{$\mbf{1}$}\tBF{1}{2}{$\mbf{2}$}{darkgrey}\tB{1}{3}{$1$}
\tB{2}{1}{$\mbf{1}$}\tB{2}{2}{$\mbf{2}$}
\tB{3}{1}{$1$}
\end{scope}
\tB{1}{1}{$\mbf{1}$}\tBF{1}{2}{$\mbf{2}$}{darkgrey}\tB{1}{3}{$1$}
\tB{2}{1}{$\mbf{1}$}\tB{2}{2}{$1$}
\tB{3}{1}{$\mbf{2}$}
\end{tikzpicture}}
\quad\raisebox{1cm}{\begin{tikzpicture}[x=0.55cm,y=-0.55cm]
\node at (0,-3) {$\stackrel{\mathcal{F}}{\longmapsto}$};
\node at (0,1.25) {$\stackrel{\mathcal{F}}{\longmapsto}$};
\end{tikzpicture}}\quad
\raisebox{-0.55cm}{
\begin{tikzpicture}[x=0.55cm,y=-0.55cm] 
\begin{scope}[yshift=2.5cm]
\tB{1}{1}{$\mbf{1}$}\tBF{1}{2}{$\mbf{2}$}{darkgrey}\tB{1}{5}{$1$}
\tB{3}{1}{$\mbf{1}$}\tB{3}{2}{$\mbf{2}$}
\tB{4}{1}{$1$}
{\renewcommand{\tableauLineWidth}{1pt}
\tB{1}{3}{$1$}\tBF{1}{4}{$1$}{lightgrey}
\tB{2}{1}{$\mbf{1}$}\tBF{2}{2}{$\mbf{2}$}{grey}}
\fill[pattern = north west lines] (1,3)--(3,3)--(3,4)--(1,4);
\fill[pattern = north east lines] (3,2)--(5,2)--(5,3)--(3,3);
\end{scope}
\tB{1}{1}{$\mbf{1}$}\tBF{1}{2}{$\mbf{2}$}{darkgrey}\tB{1}{5}{$1$}
\tB{3}{1}{$\mbf{1}$}\tB{3}{2}{$1$}
\tB{4}{1}{$\mbf{2}$}
{\renewcommand{\tableauLineWidth}{1pt}
\tB{1}{3}{$1$}\tBF{1}{4}{$1$}{lightgrey}
\tB{2}{1}{$\mbf{1}$}\tBF{2}{2}{$\mbf{2}$}{grey}}
\fill[pattern = north west lines] (1,3)--(3,3)--(3,4)--(1,4);
\fill[pattern = north east lines] (3,2)--(5,2)--(5,3)--(3,3);
\end{tikzpicture}}
}

\smallskip
We saw in Example~\ref{ex:positions} that the $2$-top and $1$-left
positions of $(3,2,1)$ are both $(1,2)$; these positions are shaded 
dark grey in all tableaux.
Insertion of $\young(\oM\tM)$ 
in the two positions $(2,1)$, $(2,2)$ below the $2$-top position,
moving each box in columns $1$ and $2$ one row down,
gives a semistandard tableau. 
Similarly 
insertion of $\young(11)$ into the positions $(1,3)$, $(1,4)$, right of the $1$-left
position, moving each box in row $1$ two columns right, 
again gives a semistandard tableau.
These operations commute. The inverse map is defined 
by deleting $\young(\oM\tM)$ and $\young(11)$ from
the $2$-bottom and $1$-right positions in the $(5,2,2,1)$-tableau; these 
positions are again shaded and the newly inserted boxes which should be deleted
are hatched.
We therefore have a bijection
\[ \SSYT\bigl( (3,2,1) \bigr)_{((2,2),(2))}  \stackrel{\mathcal{F}}{\longrightarrow}
\SSYT \bigl( (5,2,2,1)\bigr)_{((3,3),(4))}. \]
This bijection
establishes, via Lemma~\ref{lemma:twistedKostkaNumbers} (Twisted Kostka Numbers), 
that the bottom-left entries $\langle e_{(2,2)}h_{(2)}, s_{(3,2,1)}  \rangle$ and
$\langle e_{(3,3)}h_{(4)}, s_{(5,2,2,1)} \rangle$ of the final two matrices above are equal.
This bijection is generalized in Definition~\ref{defn:F}: in general $\kappa^\-_r - \kappa^\-_{r+1}$
rows of length~$r$ and $\kappa^\+_r - \kappa^\+_{r+1}$ columns of length $r$ are inserted/deleted.
This feature may be seen in this example: for instance, 
since $\kappa^\-_1 = \kappa^\-_2$, there was no need to consider the $1$-top and $1$-bottom positions.
\end{example}



Generalizing this example,
we now define the insertion map $\mathcal{F}$  in the general skew case; this 
generality is needed later in the proof of
Proposition~\ref{prop:plethysticSignedTableauInnerStable}. Note that when $\sigmas = \varnothing$
then the only hypothesis needed is that~$\sigma$ is $\bigl(\ell(\kappa^\-),\ell(\kappa^\+)\bigr)$-large.
Recall from Definition~\ref{defn:signedTableau} that
 $\YT\bigl( \sigmaS)$ is the set of signed tableaux of shape $\sigmaS$;
 note the tableaux in $\YT\bigl( \sigmaS \bigr)$ are not necessarily semistandard.

\begin{definition}\label{defn:F}
Let $\kappa^\-$ and $\kappa^\+$ be partitions. 
Let $\sigmaS$ be a 
 $\bigl(\ell(\kappa^\-)+a(\sigmas), \ell(\kappa^\+)\bigr)$-large
and $\bigl(\ell(\kappa^\-), \ell(\mus)\bigr)$-large
skew partition.
Define 
\[ \Fmap : \SSYT(\sigmaS) \rightarrow \YT\bigl( \sigmaS \opluss (\kappa^\-,\kappa^\+)\bigr)\] 
by performing (1) then~(2) below:
\begin{itemize}
\item[(1)] starting with $\rM = 1$ and finishing with $\rM = \ell(\kappa^\-)$,
insert $\kappa^\-_\rM - \kappa^\-_{\rM+1}$ new rows each with entries $-1,\ldots, -\rM$,
each with their right-most box immediately below the $\rM$-top position of $\sigma$;\\[-12pt]
\item[(2)] starting with $\rP = 1$ and finishing with $\rP = \ell(\kappa^\+)$,
insert $\kappa^\+_\rP - \kappa^\+_{\rP+1}$ new columns each with entries $1, \ldots, \rP$,
each with their bottom box immediate right of the $\rP$-left position of $\sigma$.
\end{itemize}
If $\kappa^\-_\rM = \kappa^\-_{\rM+1}$ or $\kappa^\+_\rP = \kappa^\+_{\rP+1}$ then
there is nothing to do in that step.
\end{definition}

The partitions $\kappa^\-$ and $\kappa^\+$ will always be clear from context.
It has to be checked that $\Fmap$ is well-defined (meaning that the insertions give a tableau of skew partition shape), 
but as we shall see in Lemma~\ref{lemma:FisWellDefinedAndBumpsWeight},
this is not hard to prove.
Our aim, achieved in Proposition~\ref{prop:signedTableauStable}, is to show that
\[ \Fmap : \SSYTw{\sigma}{\pi^\-}{\pi^\+} \rightarrow 
\SSYTw{\sigma \oplus (\kappa^\-,\kappa^\+)}{\pi^\- + \kappa^\-}{\pi^\+ + \kappa^\+}. \]
is a well-defined bijection for $\sigma$ and $\pi$ suitable elements of a twisted interval
for the $\ell(\kappa^\-)$-twisted dominance order.
Example~\ref{ex:F} shows the special case
where $\sigma = (3,2,1)$ and $\pi = \dec{(2,2)}{(2)} \decMap (4,2)$.

To help guide the reader through the remaining technicalities we
give a further `unsigned' example below.
This
example, like many others in this paper, was created with the help of the Magma
code mentioned in the introduction using
\verb!TwistedIntervalInjectionM(! \verb![], [3,3,1], [2,1,1] : q :=! \verb![4],! 
\verb!NSteps := 2)!;
varying the parameters to \verb![1,1]!, \verb![2]!, \verb![4,2]!,
\verb!q :=! \verb![3,2,1]! gives the bijection in Example~\ref{ex:F}.

\begin{example}\label{ex:unsignedTableauStable}
Take $\kappa^\- = \varnothing$, $\kappa^\+ = (3,3,1)$ so $\ell^\- = 0$ and $\ell^\+ = 3$.
By Definition~\ref{defn:LBound},
$\LBound\bigl([(2,1,1), (4)]_\unlhd, (3,3,1)\bigr) = 1$; the only strictly positive quantity 
comes from the case $k=2$.
(Note that we disregard the case $k=1$ because $\kappa_1 = \kappa_2$.)
Therefore, by Proposition~\ref{prop:intervalStable}, the $\Fmap$ map 
adding $(3,3,1)$ is an injection
\[ [(2,1,1), (4)]_\unlhd   \stackrel{+(3,3,1)}{\verylongrightarrow} [(5,4,2), 
(7,3,1)]_\unlhd
\stackrel{+(3,3,1)}{\verylongrightarrow} [(8,7,3), (10,6,2)]_\unlhd
 \]
and the second map is a bijection.
(We remark that Proposition~\ref{prop:signedIntervalStable} could also be used;
the intervals are then interpreted  for the $0$-twisted
dominance order,
which by Remark~\ref{remark:twistedDominanceOrderGeneralizesDominanceOrder}
is the usual dominance order on partitions, and the additional bounds are, as expected,
irrelevant.)
Figure~\ref{fig:unsignedTableauStable}  shows the
Kostka matrices $\langle h_\pi, s_\sigma \rangle$ for $\pi$, $\sigma$ in each interval;
and several features of the~$\mathcal{F}$ bijection 
\[ \SSYT\bigl((7,3,1)\bigr)_{(\varnothing, (5,4,2))} \stackrel{\mathcal{F}}{\longrightarrow}
\SSYT\bigl( (10,6,2)\bigr)_{(\varnothing, (8,7,3))} \]
establishing the equality $\langle h_{(5,4,2)}, s_{(7,3,1)} \rangle = \langle h_{(8,7,3)},
s_{(10,6,2}\rangle$ of the bottom-left matrix entries of $2$.
\end{example}

\begin{figure}[th!]
\centerline{\hspace*{-0.5in}
\begin{tikzpicture}[x=0.55cm,y=-0.55cm]
\node at (0,0) {\scalebox{0.9}{$\begin{pNiceMatrix}[first-row,first-col,nullify-dots] & \rb{(4)} & \rb{(3,1)} & \rb{(2,2)} & \rb{(2,1,1)} \\
& 1 & \cdot & \cdot & \cdot \\
& 1 & 1     & \cdot & \cdot \\
& 1 & 1     & 1     & \cdot \\
& \mbf{1} & 2     & 1     & 1     
\end{pNiceMatrix}$}};

\node at (9,0) {\scalebox{0.9}{$\begin{pNiceMatrix}[first-row,first-col,nullify-dots] & \rb{(7,3,1)} & \rb{(7,2,2)} & \rb{(6,4,1)} & \rb{(6,3,2)} & \rb{(5,5,1)} & \rb{(5,4,2)}  \\
& 1 & \cdot & \cdot & \cdot & \cdot & \cdot \\
& 1 & 1     & \cdot & \cdot & \cdot & \cdot \\
& 1 & 0     & 1     & \cdot & \cdot & \cdot \\
& 2 & 1     & 1     & 1     & \cdot & \cdot \\
& 1 & 0     & 1     & 0     & 1     & \cdot \\
& \mbf{2}   & 1     & 2     & 1     & 1     & 1 \\
\end{pNiceMatrix}$}};

\node at (19,0) {\scalebox{0.9}{$\begin{pNiceMatrix}[first-row,first-col,nullify-dots] & \rb{(10,6,2)} & \rb{(10,5,3)} & \rb{(9,7,2)} & \rb{(9,6,3)} & \rb{(8,8,2)} & \rb{(8,7,3)}  \\
& 1 & \cdot & \cdot & \cdot & \cdot & \cdot \\
& 1 & 1     & \cdot & \cdot & \cdot & \cdot \\
& 1 & 0     & 1     & \cdot & \cdot & \cdot \\
& 2 & 1     & 1     & 1     & \cdot & \cdot \\
& 1 & 0     & 1     & 0     & 1     & \cdot \\
& \mbf{2}   & 1     & 2     & 1     & 1     & 1 \\
\end{pNiceMatrix}$}};

\begin{scope}[yshift=-1cm,xshift=-0.8cm]
\tBFN{1}{0}{$\scriptstyle 1$}{grey}\tB{1}{1}{}{}\tB{1}{2}{}{}\tB{1}{2}{}\tB{1}{3}{}\tB{1}{4}{}
\tBFN{2}{0}{$\scriptstyle 2$}{grey}\tBFN{2}{2}{}{lightgrey}
\tBFN{3}{0}{$\scriptstyle 3$}{grey}\tBFN{3}{1}{}{lightgrey}
\end{scope}

\begin{scope}[yshift=-1cm,xshift=3.25cm]
\tB{1}{1}{}\tB{1}{2}{}\tB{1}{2}{}\tBF{1}{3}{$\scriptstyle 1$}{grey}\tB{1}{4}{}\tB{1}{5}{}\tB{1}{6}{}\tB{1}{7}{}
\tBF{2}{1}{$\scriptstyle 2$}{grey}\tB{2}{2}{}\tBF{2}{3}{}{lightgrey}
\tBFN{3}{0}{$\scriptstyle 3$}{grey}\tBF{3}{1}{}{lightgrey}
\end{scope}

\begin{scope}[yshift=-1cm,xshift=8.5cm]
\tB{1}{1}{}\tB{1}{2}{}\tB{1}{2}{}\tB{1}{3}{}\tB{1}{4}{}\tB{1}{5}{}\tBF{1}{6}{$\scriptstyle 1$}{grey}\tB{1}{7}{}\tB{1}{8}{} \tB{1}{9}{} \tB{1}{10}{} 
\tB{2}{1}{}\tBF{2}{2}{$\scriptstyle 2$}{grey}\tB{2}{3}{}\tBF{2}{4}{}{lightgrey}
\tB{2}{5}{}\tB{2}{6}{}
\tBFN{3}{0}{$\scriptstyle 3$}{grey}\tBF{3}{1}{}{lightgrey}\tB{3}{2}{}
\end{scope}

\begin{scope}[yshift=-3.25cm,xshift=3.25cm]
\tB{1}{1}{1}\tB{1}{2}{1}\tB{1}{2}{1}\tBF{1}{3}{1}{grey}\tB{1}{4}{1}\tB{1}{5}{1}\tB{1}{6}{2}\tB{1}{7}{$\scriptstyle 3/2$}
\tBF{2}{1}{2}{grey}\tB{2}{2}{2}\tBF{2}{3}{$\scriptstyle 2/3$}{lightgrey}
\tBFN{3}{0}{}{grey}\tBF{3}{1}{3}{lightgrey}
\end{scope}

\begin{scope}[yshift=-3.25cm,xshift=8.5cm]
\tB{1}{1}{1}\tB{1}{2}{1}\tB{1}{2}{1}\tB{1}{3}{1}\tB{1}{4}{1}\tB{1}{5}{1}\tBF{1}{6}{1}{grey}\tB{1}{7}{1}\tB{1}{8}{1} \tB{1}{9}{2} \tB{1}{10}{$\scriptstyle 3/2$} 
\tB{2}{1}{2}\tBF{2}{2}{2}{grey}\tB{2}{3}{2}\tBF{2}{4}{2}{lightgrey}
\tB{2}{5}{2}\tB{2}{6}{$\scriptstyle 2/3$}
\tBFN{3}{0}{}{grey}\tBF{3}{1}{3}{lightgrey}\tB{3}{2}{3}
\fill[pattern = north east lines] (1,2)--(2,2)--(2,5)--(1,5);
\fill[pattern = north east lines] (3,2)--(5,2)--(5,4)--(3,4);

\end{scope}

\end{tikzpicture}}

\caption{Kostka matrices for the intervals
$[(2,1,1), (4)]_\unlhd$, $[(5,4,2), (7,3,1)]_\unlhd$,  
$[(8,7,3), (10,6,2)]_\unlhd$ showing the bijection~$\pmap$
(see Proposition~\ref{prop:signedIntervalStable})
between the two larger intervals defined by adding $(3,3,1)$.
Observe that while $(7,3,1) = (4) + (3,3,1)$ and $(5,4,2) = (2,1,1) + (3,3,1)$,
we have $|\SSYT\bigl((4)\bigr)_{(\varnothing, (2,2,1))}| = 1$
but $|\SSYT\bigl((7,3,1)\bigr)_{(\varnothing, (5,4,2))}\bigr| = 2$
so in the step from the first interval to the second we do not have tableau stability, in the sense of Proposition~\ref{prop:signedTableauStable}, 
\emph{even if we consider
only those partitions in the image of the addition map}.
Below the matrices we show the $1$-, $2$- and $3$-left and $1$-, $2$- and $3$-right position for the partitions $(4), (7,3,1), (10,6,2)$;
note the $1$-left and $1$-right positions coincide.
At the bottom we show
the bijection $\mathcal{F} :\!
\SSYT\bigl((7,3,1)\bigr)_{(\varnothing, (5,4,2))}\stackrel{\mathcal{F}}{\longrightarrow}\SSYT\bigl( (10,6,2)\bigr)_{(\varnothing, (8,7,3))}$
defined by inserting two columns of length $2$ immediately right of the $2$-left position $(2,1)$ and a single column of height $3$ immediately
right of the $3$-left position $(3,0)$, using $2/3$ and $3/2$ to indicate the two boxes that have a choice of entry.
The  shading and hatching conventions are as in Example~\ref{ex:F}. This gives
a bijective proof of the equality of the bottom left entries of $2$ in the
two larger matrices marked in bold.
\label{fig:unsignedTableauStable}
}
\end{figure}

\subsection{Technical lemmas on positions}

In the following lemma we use the bounds
\smash{$\LBound\bigl([\lambda,\omega]^\ellpb_\unLHDS, \kappa\bigr)$}
and \smash{$\LBound\bigl( \bigl[\lambda^\+, \omega^\+ 
+ (|\lambda^\+| - |\omega^\+|)\bigr], \kappa^\+\bigr)_\unlhd$},
defined in Definition~\ref{defn:LBound}. (See Remark~\ref{remark:intervalNotation}
for the difference in notation.) We remark that the bounds in the following lemma 
are the first, second and fourth from Proposition~\ref{prop:signedIntervalStable} (Partition Stability), so
whenever the conditions for this proposition hold, so do the conditions for this lemma.

\begin{lemma}
\label{lemma:signedTableauPositions}
Let $\kappa^\-$ and  $\kappa^\+$ be partitions. Fix $\ell^\- = \ell(\kappa^\-)$
and $\ell^\+ = \ell(\kappa^\+)$.
Let~$\lambda$ and $\omega$ be $(\ell^\-, \ell^\+)$-large partitions
and let $\lambda \unlhddot \omega$ in the $\ell^\-$-twisted dominance order.
Let $L$ be the maximum of the twisted interval bounds
\begin{bulletlist}
\item $\LBound\bigl([\lambda^\-, \omega^\-]^{\ellmb}_\unLHDS, \kappa^\-\bigr)$,
\item $\LBound\bigl( \bigl[\lambda^\+, \omega^\+ 
+ (|\lambda^\+| - |\omega^\+|)\bigr]_\unlhd, \kappa^\+\bigr)$ 
\item \smash{$\bigl( \max( \ell(\lambda^\+), \ell^\+) + |\omega^\-| - |\lambda^\-| - \omega^\-_{\ell^\-} \bigr)/ \kappa^\-_{\ell^\-}$}
\end{bulletlist}
omitting the third if $\kappa^\- = \varnothing$.
Let $\sigma$ and~$\pi$ be partitions in the interval
\[ [\lambda  \opluss M(\kappa^\-, \kappa^\+), \omega \opluss M(\kappa^\-,\kappa^\+)\bigr]_\unlhddotS \]
for the $\ell^\-$-twisted dominance order
such that $\sigma$ is $(\ell^\-,\ell^\+)$-large. Let
$t \in \SSYT(\sigma)_{(\pi^\-,\pi^\+)}$. 
If $M - 1  \ge L$ then
\begin{thmlist}
\item the $\rM$-bottom position of $t$ contains $-\rM$ 
if $\rM < \ell^\-$ and $\kappa^\-_\rM > \kappa^\-_{\rM+1}$;
\item if $\kappa^\- \not= \varnothing$ then the $\ell^\-$-bottom position of $t$ contains $-\ell^\-$; 
\item if $\kappa^\- \not= \varnothing$ and $\kappa^\+ \not= \varnothing$ then
the box $(\ell^\+,\ell^\-)$ of $t$ contains $-\ell^\-$; 
\item the $\rP$-right position of $t$ contains $\rP$ if $\rP < \ell^\+$ and 
$\kappa^\+_\rP > \kappa^\+_{\rP+1}$.
\item the $\ell^\+$-right position of $t$ contains $\ell^\+$.
\end{thmlist}
Moreover if $M \ge L$ then the same results hold replacing `bottom' with `top' and `right' with `left', except that
\begin{thmlist}
\item[\emph{(ii)}] if $\sigma^\+ = \varnothing$ and $\kappa^\+ = \varnothing$ 
then the $\ell^\-$-top position is $(0,\ell^\-)$;
\item[\emph{(iv)}] and \emph{(v)} if $\sigma_{\rP+1} \le \ell^\-$, 
and so the $\rP$-left position
is $(\rP, \ell^\-)$, then it contains a negative entry.
\end{thmlist}
\end{lemma}

\begin{proof}
First note that, by Lemma~\ref{lemma:adjoinToLarge}, we have 
$\bigl( \lambda \oplus M(\kappa^\-,\kappa^\+) \bigr)^\-
= \lambda^\- + M\kappa^\-$ and three further analogous equations replacing $-$ with $+$
or $\lambda$ with~$\omega$. We also record a key observation on where negative entries lie in $t$:
\begin{itemize}
\item[$(-)$] The negative entries of $t$ lie in the boxes in $[\alpha]$
where $\alpha$ is a subpartition of $\sigma$ such that $a(\alpha) \le \ell^\-$ 
and $|\alpha| = |\pi^\-|$.
\end{itemize}

For (i), there is nothing to prove if $\kappa^\- = \varnothing$.
Let $\rM < \ell^\-$. Since $\sigma$ is $(\ell^\-,\ell^\+)$-large, we have $\sigma^\-_{\rM+1} \ge \ell^\+$.
Hence the $\rM$-top position of $t$ is $(\sigma^\-_{\rM+1},\rM)$.
Suppose for a contradiction that this position
has either a positive entry, or some $-s$ with $-s \succ -\rM$ 
in the order in Definition~\ref{defn:semistandardSignedTableau}, meaning that $s > r$.
In either case, ($-$) implies
that the total number of entries of $t$ in the set $\{-1,\ldots, -\rM\}$
is at most $\sigma^\-_1  + \cdots + \sigma^\-_{\rM-1}+ 
\sigma^\-_{\rM+1}-1$. (At this point we suggest that the reader refers to
Figure~\ref{fig:ellDecompositionTableau} to see ($-$) graphically: it is
also  helpful to note that $\sigma^\-_j = \sigma'_j$
for $1\le j \le \ell^\-$. See  Figure~\ref{fig:ellDecomposition} for
a reminder of this notation.)
On the other hand, $t$ has exactly
$\pi^\-_1 + \cdots + \pi^\-_{\rM-1} + \pi^\-_{\rM}$
such entries. Hence
\begin{equation}
\label{eq:signedTableauPositionsMinus}
\sum_{j=1}^{\rM-1} \sigma^\-_j + \sigma^\-_{\rM+1} > \sum_{j=1}^\rM \pi^\-_j. 
\end{equation}
Using 
\hbox{$\pi \unrhddots \lambda \oplus M(\kappa^\-,\kappa^\+)$} and so, by
 Lemma~\ref{lemma:twistedDominanceOrderOldDefinition}(a), $\pi^\- \unRHD \lambda^\- + M\kappa^\-$
we have $\sum_{j=1}^\rM \pi^\-_j \ge \sum_{j=1}^\rM \lambda^\-_j + M \sum_{j=1}^\rM \kappa^\-_j$. Hence 
\begin{equation}\label{eq:piMinusCorollary}
\sum_{j=1}^{\rM-1} \sigma^\-_j + \sigma^\-_{\rM+1} > \sum_{j=1}^\rM \lambda^\-_j + M \sum_{j=1}^\rM \kappa^\-_j.
\end{equation}
Since $\sigma \in [\lambda  \opluss M(\kappa^\-, \kappa^\+), \omega \opluss M(\kappa^\-,\kappa^\+)\bigr]_\unlhddotS$
we have $\sigma \unlhddots \omega \opluss M(\kappa^\-,\kappa^\+)$ 
and so by Lemma~\ref{lemma:twistedDominanceOrderOldDefinition}(a), $\sigma^\- \unLHD \omega^\- + M\kappa^\-$, we also have
for each $k$,
\begin{equation}\label{eq:sigmaMinus} 
\sum_{j=1}^k \lambda^\-_j + M\sum_{j=1}^k \kappa^\-_j 
\le \sum_{j=1}^k \sigma^\-_j \le \sum_{j=1}^k \omega^\-_j + M\sum_{j=1}^k \kappa^\-_j.
\end{equation}
Take $k = \rM$ in the left-hand inequality in~\eqref{eq:sigmaMinus} 
and $k = \rM+1$ in the right-hand inequality in~\eqref{eq:sigmaMinus}
and subtract to get 
$\sigma^\-_{r+1} \le \sum_{j=1}^{\rM+1} \omega^\-_j  - \sum_{j=1}^\rM \lambda^\-_j  + M\kappa^\-_{\rM+1}$.
Hence by another use of the right-hand inequality in~\eqref{eq:sigmaMinus} taking $k = \rM-1$,
\[ \sum_{j=1}^{\rM-1} \sigma^\-_j + \sigma^\-_{\rM+1} \le
2 \sum_{j=1}^{\rM-1} \omega^\-_j + \omega^\-_\rM + \omega^\-_{\rM+1}
- \sum_{j=1}^\rM \lambda^\-_j + M\sum_{j=1}^{\rM-1} \kappa^\-_j + M\kappa^\-_{\rM+1}. \]
Now \eqref{eq:piMinusCorollary} and the previous inequality imply
\begin{equation}
\label{eq:signedTableauPositionsMinusBound} 
2 \sum_{j=1}^{\rM-1} \omega^\-_j + \omega^\-_\rM + \omega^\-_{\rM+1} - 2\sum_{j=1}^\rM \lambda^\-_j
 >  M(\kappa^\-_\rM - \kappa^\-_{\rM+1}). \end{equation}
Taking $k = \rM$ in the definition of $\LBound\bigl([\lambda^\-,\omega^\-]^\ellmb_\unLHDS,\kappa^\-\bigr)$
in Definition~\ref{defn:LBound} we get 
\smash{$2\sum_{j=1}^{\rM-1} \omega^\-_j + \omega^\-_k + \omega^\-_{k+1} - 2\sum_{j=1}^\rM \lambda^\-_j \le M(\kappa_\rM - \kappa_{\rM+1})$}.
This contradicts~\eqref{eq:signedTableauPositionsMinusBound}. Hence, provided we have the
first condition on $M$ that
\smash{$M \ge$} \smash{$\LBound\bigl([\lambda^\-,\omega^\-]^\ellmb_\unLHDS,\kappa^\-\bigr)$},~(i) holds
for top positions. 

The $\rM$-bottom position lies $\kappa^\-_\rM - \kappa^\-_{\rM+1}$ boxes
below the $\rM$-top position. Supposing similarly that it does not contain $\rM$
we deduce that the total number of entries of~$t$ in the set $\{-1,\ldots, -\rM\}$
is at most \smash{$\kappa^\-_\rM - \kappa^\-_{\rM+1}$} plus 
the left-hand side of~\eqref{eq:signedTableauPositionsMinus}.
Running the 
same argument, using the same inequalities~\eqref{eq:piMinusCorollary} and 
\eqref{eq:sigmaMinus}, 
we
obtain \eqref{eq:signedTableauPositionsMinusBound} with $\kappa^\-_{\rM} -\! \kappa^\-_{\rM+1}$ subtracted from  the right hand side,
which is therefore $(M-1)(\kappa_\rM^\- - \kappa_{\rM+1}^\-)$. 
We then get a contradiction as before from
\smash{$M-1 \ge \LBound\bigl([\lambda^\-,\omega^\-]^\ellmb_\unLHDS,\kappa^\-\bigr)$}. 

For (ii), we may assume that $\kappa^\- \not=\varnothing$; then by Definition~\ref{defn:skewPositions}, the
$\ell^\-$-\emph{top position} of $t$ is 
$\bigl(\max(\ell(\sigma^\+), \ell^\+\bigr), \ell^\-)$.
If $\sigma^\+ = \varnothing$  
then we are in the exceptional case 
at the end of the statement of the lemma; otherwise,
since~$\sigma$ is $(\ell^\-,\ell^\+)$-large, this is a box of $t$.
Suppose for a contradiction that
this box does not contain $-\ell^\-$. The analogue of~\eqref{eq:piMinusCorollary} 
is 
\[ \sum_{j=1}^{\ell^- -1} \sigma^\-_j + \max(\ell(\sigma^\+), \ell^\+) > 
\sum_{j=1}^{\ell^\-}\lambda^\-_j +M\sum_{j=1}^{\ell^-}\kappa_j^\-
= |\lambda^\-| + M|\kappa^\-|.\]
By
\begin{equation}\label{eq:topContradiction} 
\sum_{j=1}^{\ell^\- -1} \sigma^\-_j \le \sum_{j=1}^{\ell^\- -1} \omega^\-_j +  M\sum_{j=1}^{\ell^\- -1} \kappa^\-_j
= |\omega^\-| - \omega^\-_{\ell^\-} + M|\kappa^\-| - M\kappa^\-_{\ell^\-} 
\end{equation}
obtained from the upper bound in~\eqref{eq:sigmaMinus}
we deduce $|\lambda^\-|  -\max(\ell(\sigma^\+), \ell^\+) <
|\omega^\-| - \omega^\-_{\ell^\-} - M\kappa^\-_{\ell^\-}$.
By Lemma~\ref{lemma:signedDominancePositiveLength}, since $\sigma \unrhddots \lambda \oplus M(\kappa^\-, \kappa^\+)$,
we have $\ell(\sigma^\+) \le \ell(\lambda^\+)$. Therefore
\begin{equation} \label{eq:topContradiction2} M\kappa^\-_{\ell^\-} < |\omega^\-| - |\lambda^\-| - \omega^\-_{\ell^\-}   + \max(\ell(\lambda^\+),\ell^\+). \end{equation}
This contradicts the third bound in the statement of this lemma, namely
$M \ge \bigl( \max( \ell(\lambda^\+), \ell^\+ ) + |\omega^\-| - |\lambda^\-| - \omega^\-_{\ell^\-} \bigr) / \kappa^\-_{\ell^\-}$.
This proves (ii) for the top position. The modifications for the $\ell^\-$-bottom position are precisely analogous
to~(i), leading to~\eqref{eq:topContradiction} with $\kappa^\-_{\ell^\-}$ subtracted from the right-hand side,
and~\eqref{eq:topContradiction2} with $M$ replaced by $M-1$, as required. 

Part (iii) follows from (ii) because the $\ell^\-$-top position is $(k,\ell^\-)$ 
where $k \ge \ell^\+$, and since this position contains $-\ell^\-$, so does position $(\ell^\+, \ell^\-)$.
This argument is indicated in the caption to Figure~\ref{fig:ellDecompositionTableau}.

For (iv) and (v), we first note that if $\kappa^\+ = \varnothing$ then there is nothing
to prove. Suppose that $\kappa^\+ \not=\varnothing$.
By (iii) we have
\begin{itemize}
\item[(+)] The positive entries of $t$ in $\{1,2,\ldots, \ell^\+\}$ lie either in boxes in the first $\ell^\-$ columns of~$t$
in rows strictly below row $\ell^\+$, or in boxes $(i,j)$ with $i \le \ell^\+$ and $j > \ell^\-$.
\end{itemize}
This restrictions from $(-)$ and $(+)$ are 
shown diagrammatically in Figure~\ref{fig:ellDecompositionTableau}. 

\begin{figure}[ht!]
\begin{center}\begin{tikzpicture}[x=0.5cm,y=-0.5cm]
\draw(0,0)--(13,0)--(13,1)--(12,1)--(12,1.5);
\draw (5,4)--(8,4)--(8,3)--(8.5,3);
\draw(8,4)--(8,5)--(7.5,5);
\draw(5,6.5)--(6,6.5)--(6,6);
\draw(5,7)--(5,10)--(4,10)--(4,10.5);
\draw(0,0)--(0,13)--(2,13)--(2,12)--(2.5,12);

\draw[thick] (0,0)--(5,0)--(5,8.5)--(3,8.5)--(3,9.5)--(2.5,9.5); 
\draw[thick] (0,0)--(0,11)--(1,11)--(1,10.5);

\node at (2.5,10.5) {$\+$};
\node at (6.25,5.25) {$\+$};
\node at (8,1.5) {$\+$};

\node at (2.5,0.5) {$\ldots$};
\node at (2.5,3.5) {$\ldots$};
\node at (0.5,1.75) {$\vdots$};
\node at (4.5,1.75) {$\vdots$};
\node at (1.6,9.6) {$\iddots$};
\node at (3.25,11) {$\iddots$};
\node at (6.9,5.5) {$\iddots$};
\node at (10,2) {$\iddots$};

\node at (4.5,6.2) {$\bullet$};

\tableauBox{1}{0}{$\oM$}
\tableauBox{5}{0}{$\scriptstyle \pmb{\ell}^{\pmb{-}}$}
\tableauBox{6}{0}{$\scriptstyle 1$}

\tableauBox{1}{3}{$\oM$}
\tableauBox{5}{3}{$\scriptstyle \pmb{\ell}^{\pmb{-}}$}

\newcommand{\xExtra}{2}
\newcommand{\yExtra}{1.5}
\newcommand{\lMx}{-0.5}
\newcommand{\lMy}{-0.5}

\draw[<-] (-0.5,0)--(-0.5,1.5); 
\draw[->] (-0.5,2.5)--(-0.5,4);
\node at (-0.5,2) {$\scriptstyle \ell^\+$};
\draw[<-] (0,-0.5)--(2,-0.5);
\draw[->] (3,-0.5)--(5,-0.5);
\node at (2.5,-0.5-0.05) {$\scriptstyle \ell^\minus$};

\draw[dashed] (5,0)--(5,13);
\draw[dashed] (0,4)--(12,4);

\end{tikzpicture}
\end{center}
\caption{Entries in a tableau $t \in \SSYT(\sigma)_{(\pi^\minus,\pi^\+)}$ 
when~$\sigma$ is $(\ell^-,\ell^\+)$-large 
showing
the conditions $(-)$ and $(+)$ in the proof of Lemma~\ref{lemma:signedTableauPositions}.
The positive entries not in the first~$\ell^{\scriptstyle +}$
rows lie in the regions marked $+$. 
Note that by  (iii) in Lemma~\ref{lemma:signedTableauPositions} 
the box in position $(\ell^{\scriptstyle -},\ell^{\scriptstyle +})$ contains $-\ell^{\scriptstyle -}$. 
We have shown the case where $\ell(\sigma^{\scriptstyle +}) > \ell^{\scriptstyle +}$,
and so the $\ell^{\scriptstyle -}$-top position is $\bigl( \ell(\sigma^{\scriptstyle +}), 
\ell^{\scriptstyle -})$ marked $\bullet$.
By part (ii) of the lemma, when $M$ is sufficiently large, this position also contains 
$-\ell^{\scriptstyle -}$,
and so the first $\ell(\sigma^{\scriptstyle +})$ rows of $t$ are equal in their first $\ell^{\scriptstyle -}$ columns.
The $\ell^{\scriptstyle -}$-bottom position is $\kappa^{\scriptstyle -}_{\ell^{\scriptstyle -}}$ rows below the $\ell^{\scriptstyle -}$-top position.
Observe that since the $\ell^{\scriptstyle -}$-top position is in row $\ell(\sigma^{\scriptstyle +})$, deleting 
a row of length $\ell^{\scriptstyle -}$ strictly below the $\ell^{\scriptstyle -}$-top position and weakly above the $\ell^{\scriptstyle -}$-bottom position
preserves partition shape: this is relevant to the bijection $\mathcal{F}$ defined in Definition~\ref{defn:F}.
\label{fig:ellDecompositionTableau}}
\end{figure}

By $(+)$, there are exactly $|\sigma^\-| - |\pi^\-|$ positive entries in the first $\ell^\-$ columns of $t$.
Let $\rP \le \ell^\+$ and suppose, as we may, that $\kappa_\rP > \kappa_{\rP+1}$.
The  $\rP$-left position of~$t$ is $(\rP, \ell^\- + \sigma^\+_{\rP+1})$.
If $\sigma^\+_{\rP+1} = 0$ then (iii) implies that this 
position contains a negative entry, as required in the exceptional cases for left-positions.
We may therefore assume that $\sigma^\+_{\rP+1} > 0$, so the $\rP$-left position is not in the first $\ell^\-$ columns of~$t$.
Suppose, for a contradiction, that this position does not contain $\rP$.
The total number of entries of~$t$
lying in the set $\{1,\ldots, \rP\}$ is then at most
$|\sigma^\-| - |\pi^\-| + \sigma^\+_1 + \cdots + \sigma^\+_\rP
+ \sigma^\+_{\rP+1} - 1$. On the other hand~$t$ has exactly $\pi^\+_1 + \cdots + \pi^\+_{\rP-1}
+\pi^\+_\rP$ such entries. Hence
\begin{equation}
\label{eq:signedTableauPositionsPlus}
|\sigma^\-| - |\pi^\-| + 
\sum_{i=1}^{\rP-1} \sigma^\+_i + \sigma^\+_{\rP+1} > \sum_{i=1}^\rP \pi^\+_i .
\end{equation}
Now using 
$\pi \unrhddot \lambda \oplus M(\kappa^\-,\kappa^\+)$ and so, by Lemma~\ref{lemma:twistedDominanceOrderOldDefinition}(b), 
$\pi^\+ + \bigl( (|\lambda^\+| + M|\kappa^\+|) - |\pi^\+|\bigr) \unrhd \lambda^\+ + M\kappa^\+$,
we have
\begin{equation}
\label{eq:signedTableauPositionsPlusLowerBoundPi} 
\sum_{i=1}^\rP \pi^\+_i \ge \sum_{i=1}^\rP \lambda^\+_i + M \sum_{i=1}^\rP \kappa^\+_i
- |\lambda^\+| - M |\kappa^\+| + |\pi^\+|. \end{equation}
In exactly the same way, since $\sigma \unrhddot \lambda \oplus M(\kappa^\-, \kappa^\+)$, we have, for each $k \ge 1$,
\begin{equation}
\label{eq:piPlusCorollary}
\sum_{i=1}^k \sigma^\+_i \ge \sum_{i=1}^k \lambda^\+_i + M \sum_{i=1}^k \kappa^\+_i 
- |\lambda^\+| - M |\kappa^\+| + |\sigma^\+| \end{equation}
and using $\sigma \unlhddot \omega \oplus M(\kappa^\-, \kappa^\+)$ and so
$\sigma^\+ \unlhd \omega^\+ + \bigl( |\sigma^\+| - |\omega^\+| - M|\kappa^\+|) + M\kappa^\+$, we have
\begin{equation}
\label{eq:sigmaPlus} 
\sum_{i=1}^k \sigma_i^\+ \le \sum_{i=1}^k \omega_i^\+ + M \sum_{i=1}^k \kappa_i^\+ + |\sigma^\+| - |\omega^\+|
- M |\kappa^\+| \end{equation}
for each $k$. 
Taking $k = \rP$ in~\eqref{eq:piPlusCorollary} 
and $k = \rP+1$ in~\eqref{eq:sigmaPlus}
and subtracting (as before for the negative case) we get
\[ \sigma^\+_{\rP+1} \le 
\sum_{i=1}^{\rP+1} \omega_i^\+ - \sum_{i=1}^\rP \lambda_i^\+ + M\kappa_{\rP+1} + |\lambda^\+| - |\omega^\+|. \]
Hence 
\begin{align*} \sum_{i=1}^{\rP-1} \sigma^\+_i + \sigma^\+_{\rP+1} &\le 
\sum_{i=1}^{\rP-1} \omega_i^\+  + M\sum_{i=1}^{\rP-1} \kappa_i^\+ + |\sigma^\+| - |\omega^\+| - M|\kappa^\+|
\\ & \hspace*{1in} + \sum_{i=1}^{\rP+1} \omega_i^\+ - \sum_{i=1}^\rP \lambda_i^\+ + M\kappa^\+_{\rP+1} + |\lambda^\+| - |\omega^\+| \\
&= 2\hskip-0.5pt \sum_{i=1}^{\rP-1}\hskip-0.5pt  \omega_i^\+ \hskip-0.5pt +\hskip-0.5pt  \omega^\+_\rP 
\hskip-0.5pt +\hskip-0.5pt  \omega^\+_{\rP+1} \hskip-0.5pt -\hskip-0.5pt  
\sum_{i=1}^\rP \lambda^\+_i 
\hskip-0.5pt +\hskip-0.5pt  M\sum_{i=1}^{\rP-1}\hskip-0.5pt  \kappa_i^\+ \hskip-0.5pt +\hskip-0.5pt  M\kappa^\+_{\rP+1} \\
&\hspace*{1in} + |\sigma^\+| + |\lambda^\+| - 2|\omega^\+| - M |\kappa^\+|
\end{align*}
and so, writing $A$ for the right-hand side above,~\eqref{eq:signedTableauPositionsPlus}
and~\eqref{eq:signedTableauPositionsPlusLowerBoundPi} imply
\[ |\pi^\-| - |\sigma^\-| + \sum_{i=1}^\rP \lambda_i^\+ + M\sum_{i=1}^\rP \kappa_i^\+ - |\lambda^\+| - M|\kappa^\+|
+ |\pi^\+| 
< A. \]
We now rearrange and simplify
using $|\sigma^\-| + |\sigma^\+| = |\pi^\-| + |\pi^\+|$ to get
\begin{align} 
M(\kappa^\+_\rP - \kappa^\+_{\rP+1}) &< 2\sum_{i=1}^{\rP-1} \omega^\+_i + \omega^\+_\rP + \omega^\+_{\rP+1} 
- 2 \sum_{i=1}^\rP \lambda^\+_i \nonumber \\ &\hspace*{0.5in} + |\sigma^\+| + 2|\lambda^\+| - 2|\omega^\+| - |\pi^\-| + |\sigma^\-| - |\pi^\+| \nonumber \\
&= 2\sum_{i=1}^{\rP-1} \omega^\+_i + \omega^\+_\rP + \omega^\+_{\rP+1} - 2\sum_{i=1}^\rP \lambda^\+_i + 2|\lambda^\+|
-2 |\omega^\+| . \label{eq:rightFinal} 
\end{align}
But from the hypothesis
$M \ge \LBound\bigl( [\lambda^\+, \omega^\+ + (|\lambda^\+| - |\omega^\+|)]_\unlhd, \kappa^\+ \bigr)$ we have
\begin{equation} \label{eq:topFinal2} 
M \ge \frac{2\sum_{i=1}^{\rP-1} \omega^\+_i + \omega^\+_\rP + \omega^\+_{\rP+1} - 2\sum_{i=1}^\rP \lambda^\+_i
+ 2|\lambda^\+| - 2 |\omega^\+|}{\kappa^\+_\rP - \kappa^\+_{\rP+1}}. 
\end{equation}
Comparing with the previous inequality we get a contradiction.
This proves (iv) and (v) for left positions.
The modifications for (iv) and (v) for right positions are precisely analogous to the negative case; the relevant
quantity to subtract from the left-hand side of~\eqref{eq:rightFinal} is $\kappa^\+_\rP - \kappa^\+_{\rP+1}$, so
we obtain~\eqref{eq:topFinal2} with~$M$ replaced by $M-1$, as already seen twice before. This completes the proof.
%
\end{proof}

For later use in the proof
of Lemma~\ref{lemma:plethysticSignedTableauPositions} 
we give the following lemma in the general skew case.

\begin{lemma}\label{lemma:positions}
Let $\kappa^\-$ and $\kappa^\+$ be partitions. Set $\ell^\- = \ell(\kappa^\-)$
and $\ell^\+ = \ell(\kappa^\+)$.
Let $\sigmaS$ be an $\bigl(\ell^\- + a(\sigmas), \ell^\+ \bigr)$-large
and $\bigl(\ell^\-, \ell(\sigmas) \bigr)$-large
skew partition.
The $\rM$-top position of $\sigmaS$ is
\begin{align*}
&\begin{cases} \bigl( \sigma'_{\rM+1}, \rM \bigr) & \text{if $\rM < \ell^\-$} \\
\bigl( \max( \ell(\sigmas), \ell^\+, \ell(\sigma^\+) ),
\ell^\- \bigr)\hspace*{10pt} & \text{if $\rM = \ell^\-$} \end{cases} \\
\intertext{and the $\rP$-left position of $\sigma$ is}
&\begin{cases}  \bigl( \rP, \sigma_{\rP+1} \bigr) & \text{if $\rP < \ell^\+$} \\
\bigl( \ell^\+, \max(  \ell^\- + a(\sigmas), \sigma_{\ell^\++1}) \bigr) & \text{if $\rP = \ell^\+$} \end{cases}
\end{align*}
Either $\sigmas = \varnothing$ and $\sigma^\+ = \varnothing$ and $\ell^\+ = 0$ or
all top positions are in row $\max(\ell(\sigmas), \ell(\sigma^\+), \ell^\+ )$ of $\sigma$ or further below.
All left positions are in column $\ell^\- + a(\sigmas)$ of $\sigma$ or further right.
\end{lemma}

\begin{proof}
Since $\sigmaS$ is $\bigl( \ell^\-,\ell(\sigmas) \bigr)$-large, we have $\sigma^\-_{\ell^\-} \ge \ell(\sigmas)$.
Since $\sigmaS$ 
is $\bigl(\ell^\- + a(\sigmas) , \ell^\+\bigr)$-large
we have $\sigma^\-_{\ell^\-} \ge \ell^\+$
and  $\ell^\-  + \sigma^\+_{\ell^\+} \ge \ell^\- + a(\sigmas)$. Moreover,
by Remark~\ref{remark:ellDecompositionLarge}, we have $\sigma_{\ell^\-}^\- \ge \ell(\sigma^\+)$. Summarising
\begin{align}
\sigma^\-_{\ell^\-} &\ge \max( \ell(\sigmas), \ell(\sigma^\+), \ell^\+) \label{eq:ellMinusPart}, \\
\sigma^\+_{\ell^\+} &\ge a(\sigmas). \label{eq:ellPlusPart}
\end{align}
By Definition~\ref{defn:skewPositions}, the $\rM$-top position of $\sigmaS$ is 
\[ \bigl( \max(  \ell(\sigmas), \ell(\sigma^\+), \ell^\+, \sigma^\-_{\rM+1}), 
\rM \bigr) \]
for each $\rM \le \ell^\-$. The claim on the $\ell^\-$-top position is now immediate.
If $\rM < \ell^\-$, then by~\eqref{eq:ellMinusPart},
the maximum defining the row is at least $\sigma^\-_{\ell^\-}$, and so 
the position is $(\sigma^\-_{\rM+1}, \rM)$, as required.
This also proves the claim on the rows of these positions. 
Again by Definition~\ref{defn:skewPositions},
the $\rP$-left position of $\sigmaS$ is 
\[ \bigl( r^\+, \ell^\- + \max(a(\sigmas), \sigma^\+_{r^\+ +1})\bigr) \]
for $r^\+ \le \ell^\+$. If $\rP < \ell^\+$, then by~\eqref{eq:ellPlusPart}, 
the maximum is at least $\sigma^\+_{\ell^\+}$ and so the position
is $\bigl( r^\+, \ell^\- + \sigma^\+_{r^\+ + 1} \bigr)$, which is as required.
If $\sigma^\+_{\ell^\+ + 1} = 0$ then the column of the $\ell^\+$-left position is 
$\ell^\- + a(\sigmas)$, as claimed, while if \smash{$\sigma^\+_{\ell^\+ + 1} > 0$} then
\smash{$\sigma_{\ell^\+ + 1} = \ell^\- + \sigma^\+_{\ell^\+ +1} $} and so
\[ \ell^\- + \max( a(\sigmas) , \sigma^\+_{\ell^\++1} )
= \max( \ell^\- + a(\sigmas), \sigma_{\ell^\+ + 1} \bigr)\] 
as required. This
also proves the claim on the columns of these positions.
\end{proof}

Recall from Definitions~\ref{defn:signedTableau} and~\ref{defn:semistandardSignedTableau} that $\YT(\muS)$
is the set of signed tableaux of shape $\muS$
and $\sSSYT(\muS)$ is the subset of signed semistandard tableaux of shape $\muS$.

\begin{lemma}\label{lemma:FisWellDefinedAndBumpsWeight}
Let $\kappa^\-$ and $\kappa^\+$ be partitions. 
Let $\sigmaS$ be a 
 $\bigl(\ell(\kappa^\-)+a(\sigmas),$ $\ell(\kappa^\+)\bigr)$-large
and $\bigl(\ell(\kappa^\-), \ell(\sigmas)\bigr)$-large
skew partition. 
\begin{thmlist}
\item The map
$\Fmap : \sSSYT(\sigmaS) \rightarrow \YT\bigl( \sigmaS \opluss (\kappa^\-,\kappa^\+)\bigr)$ is well-defined.
\item If $t \in \sSSYT(\sigmaS)$ has signed weight $(\pi^\-,\pi^\+)$ then $\Fmap(t)$
has signed weight $(\pi^\-+ \kappa^\-, \pi^\+ + \kappa^\+)$.
\end{thmlist}
\end{lemma}

\begin{proof}
For (i) we must check that the insertions preserve partition shape.
By Lemma~\ref{lemma:positions},
each $\rM$-top position is immediately above a row of length at most $\rM$
not meeting $[\sigmas]$. (Note particularly that this holds when $\rM = \ell^\-$  because
the $\ell^\-$-top position lies in row 
$\max(\ell(\sigmas), \ell(\sigma^\+), \ell^\+)$,
as remarked on in the caption to Figure~\ref{fig:ellDecompositionTableau}.)
The definition of $\Fmap$ in Definition~\ref{defn:F} specifies
that insertions are performed working from bottom to top and 
top positions used for later insertions in (1) are not changed
by earlier insertions in (1).
Therefore the row insertions are well-defined.
By the claim on the rows of the top positions in the lemma, the left positions are
not changed by these insertions.
Therefore the collumn insertions in (2) commute with the row insertions in (1), 
and a similar argument shows
that the row insertions in~(2) are well-defined. Hence $\Fmap$ is well-defined as required in (i).
Part (ii) is obvious from the definition of~$\Fmap$.
\end{proof}

\subsection{Tableau stability}
We now show that $\Fmap$ is bijective in the case relevant to our stable partition system.
We later quote the main part of this proof of the following lemma
in the proof of the extension to the skew case 
in Proposition~\ref{prop:plethysticSignedTableauInnerStable}. (See Remark~\ref{remark:recap}
for the recapped version.) Since this extension has other details
that are somewhat fiddly, we do not attempt to continue the unified exposition.
We remark that, by Lemma~\ref{lemma:downsetIsLarge}, 
the hypothesis that $\omega$ is $\bigl(\ell^\-+1, \ell^\+\bigr)$-large
implies the same condition on $\lambda$.

\begin{proposition}[Tableau Stability]
\label{prop:signedTableauStable}
Let $\kappa^\-$ and $\kappa^\+$ be partitions. Fix $\ell^\- = \ell(\kappa^\-)$ and $\ell^\+ = \ell(\kappa^\+)$.
Let $\omega$ be a partition
 and let $\lambda \unlhddots \omega$ in the $\ell^\-$-twisted
dominance order.
If $\ell^\- \not=0$ then suppose that $\omega$ is $\bigl(\ell^\-+1, \ell^\+\bigr)$-large.
Let $\sigma$ and $\pi$ be partitions in the twisted interval
\[ \bigl[\lambda  \opluss M(\kappa^\-, \kappa^\+), \omega \opluss M(\kappa^\-,\kappa^\+)\bigr]_\unlhddotS \]
for the $\ell^\-$-twisted dominance order.
Provided $M$ is at least the maximum of
\begin{bulletlist}
\item $\LBound\bigl([\lambda^\-,\omega^\-]^\ellmb_\unLHDS, \kappa^\-\bigr)$ 
\item $\LBound\bigl( \bigl[\lambda^\+,\omega^\+ + (|\lambda^\+| - |\omega^\+|) \bigr]_\unlhd, \kappa^\+\bigr)$
\item \smash{$\bigl( \max( \ell(\lambda^\+), \ell^\+) + |\omega^\-| - |\lambda^\-| - \omega^\-_{\ell^\-} \bigr)/ \kappa^\-_{\ell^\-}$}
\end{bulletlist}
the map $\Fmap$ is a well-defined bijection
\[ \Fmap : \SSYTw{\sigma}{\pi^\-}{\pi^\+} \rightarrow 
\SSYTw{\sigma \oplus (\kappa^\-,\kappa^\+)}{\pi^\- + \kappa^\-}{\pi^\+ + \kappa^\+}. \]
\end{proposition}

\begin{proof}
Suppose that $\ell^\- \not= 0$. Then,
by Lemma~\ref{lemma:downsetIsLarge} the partitions $\lambda$ and $\sigma$ are both $(\ell^\-+1,\ell^\+)$-large
and so $(\ell^\-,\ell^\+)$-large. If $\ell^\- = 0$
then \emph{any} partition is $\bigl(0, \ell^\+\bigr)$ large.
Hence, by Lemma~\ref{lemma:FisWellDefinedAndBumpsWeight}(i), the map $\Fmap$ is well-defined 
\emph{when defined with codomain $\YT\bigl( \sigmaS \opluss (\kappa^\-,\kappa^\+)$}.

Fix a semistandard signed tableau $t \in \SSYTw{\sigma}{\pi^\-}{\pi^\+}$.
By hypothesis  the three bounds on~$M$ required to apply Lemma~\ref{lemma:signedTableauPositions} all hold, 
and we have just seen the required largeness conditions hold.
Therefore we have properties~(i), (ii), (iii), (iv) and~(v) in this lemma
for top and left position in any $t \in \SSYTw{\sigma}{\pi^\-}{\pi^\+}$.

By~(i) and (ii) for top positions, when we apply the insertion map $\Fmap$,
each of the $\kappa^\-_r - \kappa^\-_{r+1}$ new rows of length $\rM \le \ell^\-$ with entries $-1, \ldots, -\rM$
are inserted below a row of $t$ having $-\rM$ in column $\rM$ and so have 
the same entries in their first~$\rM$ positions.
These row insertions therefore preserve the semistandard condition for columns.

By (iv) and (v) for left positions, each new column of height~$\rP \le \ell^\+$ with entries $1,\ldots, \rP$
is inserted to the right of a column having~$\rP$ in row~$\rP$ and so having the same entries as the inserted
column in its first~$\rP$ positions,
or as a new column $\ell^\- + 1$, immediately to the right of a column having only negative entries.
These column insertions therefore preserve the semistandard condition for rows.

It is clear that the overall effect of these insertions is to 
change the signed weight of $t$ from $\swtp{\pi}$ to $(\pi^\- + \kappa^\-, \pi^\+ 
+\kappa^\+)$.
Hence~$\Fmap$ has image in the set 
\smash{$\SSYTw{\sigma \opluss (\kappa^\-,\kappa^\+)}{\pi^\- + \kappa^\-}{\pi^\+ +\kappa^\+)}$} as claimed,
and so is well-defined.

The map $\Fmap$ is defined by inserting certain rows and columns into fixed positions in a tableau, so 
is clearly injective.

To see that $\Fmap$ is surjective, 
let $u \in \SSYTw{\sigma \opluss (\kappa^\-,\kappa^\+)}{\pi^\- + \kappa^\-}{\pi^\+ +\kappa^\+}$.
Suppose that $\kappa^\-_{\rM} > \kappa^\-_{\rM+1}$. Then, by definition of the 
$\rM$-top position, the row containing the $\rM$-bottom position of $u$, and each
of the $\kappa^\-_{\rM} - \kappa^\-_{\rM+1}$ rows weakly above it (including the row itself)
has length $\rM$.
By~(i) and~(ii) for bottom positions, the $\rM$-bottom position in $u$ contains~$-\rM$;
since boxes in column $\rM$ contain entries at least $-\rM$ 
in the order in Definition~\ref{defn:semistandardSignedTableau}, all entries in this column
are $-\rM$. Therefore all $\kappa^\-_{\rM} - \kappa^\-_{\rM+1}$ rows have the 
form $-1,\ldots, -\rM$.
Deleting these rows and shifting the remaining boxes in lower rows up gives a signed semistandard
tableau because (as remarked at the start of the proof of Lemma~\ref{lemma:FisWellDefinedAndBumpsWeight}), by
Lemma~\ref{lemma:positions},
each $\rM$-top position is immediately above a row of length at most $\rM$
not meeting $[\sigmas]$. 
Therefore this 
deletion undoes the insertion map in (1). Our assumption that $\kappa^\-_\rM > \kappa^\-_{\rM+1}$
is now seen to be without loss of generality, since if equality holds then no rows were inserted.
The argument for right positions and column deletion is very similar:
if $\kappa^\+_{\rP} > \kappa^\+_{\rP+1}$ then the column containing the $\rP$-right
position of $u$ and each of the $\kappa^\+_{\rP} - \kappa^\+_{\rP+1}$ columns weakly left of it  (including the column itself)
has length~$\rP$ and entries  $1,\ldots, \rP$. Deleting these columns undoes
the insertion map in (2). (Here it is obvious that deletion preserves the signed semistandard condition.)
Hence $\Fmap$ is surjective and so
bijective.
\end{proof}

\subsection{Stable partition systems from twisted intervals}

We summarise this section in the following corollary.
Recall that the first two bounds are defined in Definition~\ref{defn:LBound}.
We remark that (as seen
at the start of proof of Proposition~\ref{prop:signedIntervalStable}),
the hypotheses below imply, via
Lemma~\ref{lemma:signedDominancePositiveLength},
that $\lambda$ is $(\ell^\- + 1, \ell^\+)$-large; 
thus adjoining to $\lambda$ behaves as expected from Lemma~\ref{lemma:adjoinToLarge}, and we do no need an explicit
hypothesis that $\lambda$ is suitably large.

\begin{corollary}\label{cor:signedIntervalStable}
Let $\kappa^\-, \kappa^\+$ be partitions. Fix $\ell^\- = \ell(\kappa^\-)$ and $\ell^\+ = \ell(\kappa^\+)$.
Let $g_\pi = e_{\pi^\-}h_{\pi^\+}$
for each $\pi \in \Par$. 
Let $\omega$ be a partition
 and let $\lambda \unlhddots \omega$ in the $\ell^\-$-twisted
dominance order.
If $\ell^\- \not=0$ then suppose that $\omega$ is $\bigl(\ell^\-+1, \ell^\+\bigr)$-large.
Let $L$ be the maximum of the quantities
\begin{bulletlist}
\item $\LBound\bigl([\lambda^\-,\omega^\-]^\ellmb_\unLHDS, \kappa^\-\bigr)$,
\item $\LBound\bigl([\lambda^\+,\omega^\++(|\lambda^\+|-|\omega^\+|)]_\unlhd, \kappa^\+\bigr)$,
\item $\bigl( \omega_1^\+ + \omega_2^\+ - 2\lambda_1^\+ + 2|\lambda^\+| - 2|\omega^\+| 
\bigr)/(\kappa^\+_1-\kappa^\+_2)$,
\item $\bigl( \max( \ell(\lambda^\+), \ell^\+ ) + |\omega^\-| - |\lambda^\-| - \omega^\-_{\ell^\-} \bigr)/
\kappa^\-_{\ell^\-}$
\end{bulletlist}
omitting the third if $\kappa^\+_1 = \kappa^\+_2$ and the fourth if $\kappa^\- = \varnothing$.
Let
\[ \PSeq{M} = \bigl[ \lambda \oplus M(\kappa^\-,\kappa^\+), \omega \oplus M(\kappa^\-, \kappa^\+) \bigr]_\unlhddotS \]
for each $M \in \N_0$. Then $(\PSeq{M} )_{M\in \N_0}$ is a stable partition system with respect to
the map $\pmap : \Par \rightarrow \Par$ defined by
$\pmap(\sigma) = \sigma \oplus (\kappa^\-, \kappa^\+)$
and the twisted symmetric functions $g_\pi$. The system is stable for $M \ge L$.
\end{corollary}

\begin{proof}
We check the two conditions in the definition of a stable partition system 
in Definition~\ref{defn:stablePartitionSystem}.
The four bounds above 
give the hypotheses
required in Proposition~\ref{prop:signedIntervalStable} (Partition Stability). Therefore condition (a)
holds for $M \ge K$.
The hypotheses on $M$ in Proposition~\ref{prop:signedTableauStable} (Tableau Stability)  are the first two bounds above and 
$M \ge \bigl( |\omega^\-| - \omega^\-_{\ell^\-} + \ell^\+ - |\lambda^\-|\bigr) / \kappa^\-_{\ell^\-}$,
which is implied by the fourth bound above. The condition that  $\omega$ is $(\ell^\-+1,\ell^\+)$-large holds by assumption.
Hence $|\mSSYT(\sigma)_{(\pi^\-,\pi^\+)}| = |\mSSYT(\pmap(\sigma))_{(\pmap(\pi)^\-,\pmap(\pi)^\+)}|$ 
for all $\pi$, $\sigma \in \mathcal{P}(M)$ provided $M \ge L$. By Lemma~\ref{lemma:twistedKostkaNumbers} (Twisted Kostka Numbers)
it follows that
$\langle e_{\pi^\-}h_{\pi^\+}, s_\sigma \rangle = \langle e_{\pmap(\pi)^\-}h_{\pmap(\pi)^\+}, s_{\pmap(\sigma)} \rangle$
for all $\pi$, $\sigma \in \mathcal{P}(M)$ provided $M \ge L$. Therefore condition~(b)
holds for~\hbox{$M \ge L$}.  
\end{proof}

\addtocontents{toc}{\smallskip}
\addtocontents{toc}{\textbf{Theorem~\ref{thm:muStable}: inner stability}}

\section{Twisted weight bound for Theorem~\ref{thm:muStable}}\label{sec:twistedWeightBoundInner}
To apply Corollary~\ref{cor:signedIntervalStable} in the proofs of our two main theorems
we need an upper bound in the $\ell^\-$-twisted dominance
order on the constituents of an arbitrary plethysm. 
For instance, in the overview in \S\ref{sec:overview},
we implicitly used (see Example~\ref{ex:lengthBound}) 
the $1$-twisted dominance order with the twisted intervals 
$\bigl[(6,2) \oplus M\bigl((1),(1)\bigr), (5,1,1,1) \oplus M\bigl( (1), 
(1) \bigr)\bigr]_\unrhddot$,
arguing that if $s_\lambda$ is a constituent in $s_{(3,1,1^M)} \circ s_{(2)}$ then
$\lambda \,\unlhddot\, (5,1,1,1) \oplus M\bigl( (1), (1) \bigr)$.
The aim of this section is to prove 
Corollary~\ref{cor:twistedWeightBoundInnerGrowing} which gives the upper bound
we use for Theorem~\ref{thm:muStable}. En route we obtain Proposition~\ref{prop:twistedWeightBoundInner}
which is of independent interest.
We show in Example~\ref{ex:lengthBoundOmega} that, in the case of the plethysm $s_{(3,1,1^M)} \circ s_{(2)}$,
Proposition~\ref{prop:twistedWeightBoundInner} specializes
to give the upper bound obtained earlier by ad-hoc arguments; Example~\ref{ex:cutUpsetForLawOkitani}
shows the connection between our upper bound and the extended example in \S\ref{subsec:cutUpsetForLawOkitani}.

\subsection{Weight large skew partitions}\label{subsec:weightLarge}
There is an analogous technicality to that pointed out before
Definition~\ref{defn:large} about adjoining to partitions.
Recall from Definition~\ref{defn:greatestSignedWeight} and Lemma~\ref{lemma:greatestSignedWeight}
 that $t_{\ell^\-}(\tauS)$ is the semistandard tableau of shape
$\tauS$ of greatest signed weight in the $\ell^\-$-signed dominance order.
By Lemma~\ref{lemma:ellDecompositionGreatestSignedWeight}
the signed weight $\bigl( \omega_{\ell^\-}(\tauS)^\-,  \omega_{\ell^\-}(\tauS)^\+ \bigr)$
of $t_{\ell^\-}(\tauS)$
is the $\ell^\-$-decomposition of a partition.

\begin{definition}\label{defn:weightLarge}
Fix $\ell^\- \in \N_0$.
Let $\tauS$ be a skew partition and 
let $\sigma$ be the partition with 
$\ell^\-$-decomposition $\dec{\omega_{\ell^\-}(\tauS)^\-}{\omega_{\ell^\-}(\tauS)^\+}$.
Let $a \in \N_0$. We say $\tauS$ is 
\begin{defnlist}
\item $(a,\ell^\+)$-\emph{weight large for $\ell^\-$} if 
$\sigma$ is $(a,\ell^\+)$-large,
\item $(\ell^\-,\ell^\+)$-\emph{weight large} if
$\sigma$ is $(\ell^\-,\ell^\+)$-large.
\end{defnlist}
\end{definition}

We state the definition in this form to emphasise the connection with Definition~\ref{defn:large}.
When $\ell^\- \ge 1$, the paragraph after this earlier definition
implies that the skew partition $\tauS$ is $(\ell^\-,\ell^\+)$-weight large
if and only if part~$\ell^\-$ of
$\omega_{\ell^\-}(\tauS)^\-$
is at least $\ell^\+$, or equivalently,
if and only if $t_{\ell^\-}(\tauS)$ has at least $\ell^\+$ entries
of $-\ell^\-$. This is the interpretation we need most often.

\begin{example}
\label{ex:2greatestWeight}
From Example~\ref{ex:2greatest}, where $\ell^\- = 2$, we see that
$(6,4,4,1)$, $(6,4,4,1)/(1,1)$ and $(6,4,4,1)/(2,1)$ 
are $(2,\ell^\+)$-weight large if and only if $\ell^\+ \le 3$ and $(6,4,4,1)/(3,3)$
is $(2,\ell^\+)$-weight large if and only if $\ell^\+ \le 2$.
This is most easily seen using the final  characterisation
just mentioned: for example the tableau $t_2\bigl( (6,4,4,1)/(2,1) \bigr)$ 
shown in the margin evidently has three entries of $-2$.\marginpar{$\quad\young(::\oM\tM11,:\oM\tM1,\oM\tM12,\oM)$}
It is easily checked from the other tableaux in Example~\ref{ex:2greatest}
that $(6,4,4,1)$
and $(6,4,4,1)/(1,1)$  are $(3,3)$-weight large for~$2$ because the relevant partitions
$\sigma$ in Definition~\ref{defn:weightLarge} have $(3,3)$ as a box (or equivalently
the greatest tableaux both have $3$ as an entry)
but $(6,4,4,1)/(2,1)$ is not, because, as the marginal tableau shows,
$3$ is not an entry.  Working directly from Definition~\ref{defn:weightLarge}, we would instead compute
\[ \dec{\omega_{\ell^\-}(\scalebox{0.9}{$(6,4,4,1)/(2,1)$})^\-}
{\omega_{\ell^\-}(\scalebox{0.9}{$(6,4,4,1)/(2,1)$})^\+} =
\dec{(4,3)}{(4,1)} \decMap (6,3,2,1) \]
and note that $(6,3,2,1)$ does not have a box in position $(3,3)$. Thus $(6,4,4,1)/(2,1)$ is $(2,3)$-weight large and $(3,2)$-weight large for $2$
but not $(3,3)$-weight large.
Note also that while $(6,4,4,1)/(3,3)$ is not $(2,3)$-weight large,
it is $(2,3)$-large, since $(6,4,4,1)_{3} = 4 \ge 2$.
\end{example}

We have just seen
that `$(\ell^\-,\ell^\+)$-large' does not imply `$(\ell^\-,\ell^\+)$-weight large'.
The following lemma gives the complete picture.

\begin{lemma}\label{lemma:skewLargeImpliesWeightLargeImpliesLarge}
Let $\tauS$ be a skew partition. 
\begin{thmlist}
\item If $\tauS$ is $(\ell^\- \hskip-0.5pt +  a(\taus), \ell^\+)$-large 
 then $\tauS$ is $(\ell^\-, \ell^\+)$-weight large.
\item If $\tauS$ is $(\ell^\-, \ell^\+)$-weight large then $\tauS$ is $(\ell^\-, \ell^\+)$-large
\end{thmlist}
Moreover, if $\taus = \varnothing$ then the converses also hold.
\end{lemma}

\begin{proof}
For (i), by hypothesis $[\tau]\backslash [\taus]$ contains all the boxes $(i, a(\tauS) + j)$ for $1 \le i \le \ell^\+$
and $1 \le j \le \ell^\-$. As in the proof of Lemma~\ref{lemma:ellDecompositionGreatestSignedWeight},
rows $1, 2, \ldots, \ell^\+$ of $t_{\ell^\-}(\tauS)$ begin
\smash{\raisebox{-4.5pt}{\begin{tikzpicture}[x=0.55cm,y=-0.45cm]
\tableauBox{0}{0}{\oM} \tableauBox{1}{0}{\tM} \draw (1,0)--(3,0)--(3,1)--(1,1)--(1,0);
\node at (2,0.5) {$\ldots$}; \tableauBox{4}{0}{$\pmb{\ell}^{\pmb{-}}$}\end{tikzpicture}}}.
(These boxes form part of the heavy marked region $[\alpha]$ in Figure~\ref{fig:ellDecompositionSkewTableau}.)
Hence $t_{\ell^-}(\tauS)$ has at least $\ell^\+$ entries of $-\ell^\-$ and so
$\tauS$ is $(\ell^\-,\ell^\+)$-weight large.
For (ii), we have just seen that 
$t_{\ell^\-}(\tauS)$ has at least $\ell^\+$ entries equal to $-\ell^\-$.
Hence there are at least $\ell^\+$ rows of the Young diagram of $\tauS$ having at least $\ell^\-$ boxes.
It easily follows that $(\ell^\+,\ell^\-) \in [\tau]$ and so $\tauS$ is $(\ell^\-,\ell^\+)$-large.
Finally, if $\taus = \varnothing$ then $a(\taus) = 0$ and so (i) and~(ii) are opposite directions of the required
implication.
\end{proof}

We also have the following remark, analogous to Remark~\ref{remark:becomesLarge}.

\begin{remark}\label{remark:becomesWeightLarge}
Fix partitions $\kappa^\-$ and $\kappa^\+$ and let $\ell^\- = \ell(\kappa^\-)$ and $\ell^\+= \ell(\kappa^\+)$.
When $\kappa^\- \not=\varnothing$, each application
of the map $\tauS \mapsto \tauS \oplus (\kappa^\-, \kappa^\+)$ 
inserts $\kappa^\-_{\ell^\-}$ new parts of length $\ell^\-$, and so
increases the number of parts of length at least $\ell^\-$ by at least $\kappa^\-_{\ell^\-}$.
Therefore after $\lceil \ell(\taus) / \kappa^\-_{\ell^\-} \rceil$ steps, the skew partition obtained 
has $(\ell(\taus), \ell^\-)$ as a box. On each subsequent step we insert~$\kappa^\-_{\ell^\-}$ new boxes in column $\ell^\-$ which in the greatest tableau all
contain~$-\ell^\-$. 
Therefore the original skew partition $\tauS$ becomes $(\ell^\-, \ell^\+)$-weight large after
at most $\lceil \ell(\taus) / \kappa^\-_{\ell^\-} \rceil + \lceil \ell^\+ / \kappa^\-_{\ell^\-} \rceil$
steps. Each further step creates at least $\kappa^\+_{\ell^\+}$ new boxes containing $\ell^\+$.
Therefore, when $\kappa^\+ \not=\varnothing$,
for any $a \in \N$, the original skew partition $\tauS$ becomes $(\ell^\- + a, \ell^\+)$-weight large,
meaning that $(\ell^\+, \ell^\- + a)$ is a box of the partition corresponding to the signed
weight of $t_{\ell^\-}\bigl( \tauS \oplus M(\kappa^\-, \kappa^\+) \bigr)$, after
at most $M = \lceil \ell(\taus) / \kappa^\-_{\ell^\-} \rceil 
+ \lceil \ell^\+ / \kappa^\-_{\ell^\-} \rceil
+ \lceil a/ \kappa^\+_{\ell^\+} \rceil$ steps. 
\end{remark}

By this remark, 
there is no loss of generality in
assuming in all the results below that the partitions involved are suitably weight large.

For example, take $\kappa^\- = (2,1,1)$, $\kappa^\+ = (1,1)$ and $\tauS = (1,1,1)/(1,1,1)$,
so that $\ell^\- = 3$ and $\ell^\+ = 2$.
The tableaux in Figure~\ref{fig:becomesWeightLarge} show that three applications of the 
adjoining map
$\tauS \mapsto \tauS \oplus \bigl( (2,1,1), (1,1) \bigr)$
are necessary and sufficient
to obtain a $(3,2)$-weight large skew partition; this skew partition is also $(3,3)$-weight large. One further
application gives a $(3,4)$-weight large skew partition.
As is typically the case, this beats the bound in Remark~\ref{remark:becomesWeightLarge},
which specifies $\lceil 3/1 \rceil + \lceil 2/1 \rceil = 5$ adjoinings.

\begin{figure}[h!t] 
\vspace*{-9pt}
\centerline{\begin{tikzpicture}[x=0.5cm,y=-0.5cm]
\node at (-0.5,-0.65) {$\scriptstyle\langle \varnothing, \varnothing \rangle$};
\node at (-0.5,-1.4) {$\scriptstyle \varnothing$};
\tableauBoxFilled{0}{0}{}{lightgrey} \tableauBoxFilled{0}{1}{}{lightgrey}  \tableauBoxFilled{0}{2}{}{lightgrey}
\end{tikzpicture}\spy{12pt}{$\,\;\mapsto\;\,$}
\raisebox{-27.5pt}{\begin{tikzpicture}[x=0.5cm,y=-0.5cm]
\node at (0.5,-1.4) {$\scriptstyle(3,1,1,1)$};
\node at (0.5,-0.65) {$\scriptstyle\langle (4,1,1), \varnothing \rangle$};
\tableauBoxFilled{0}{0}{}{lightgrey} \tableauBoxFilled{0}{1}{}{lightgrey}  \tableauBoxFilled{0}{2}{}{lightgrey}
\tableauBox{1}{0}{\oM} \tableauBox{1}{1}{\oM} \tableauBox{2}{0}{\tM} 
 \tableauBox{3}{0}{\dM} 
\tableauBox{0}{3}{\oM}
\tableauBox{0}{4}{\oM}
\end{tikzpicture}}\spy{12pt}{$\,\;\mapsto\;\,$}
\raisebox{-55pt}{\begin{tikzpicture}[x=0.5cm,y=-0.5cm]
\node at (1,-1.4) {$\scriptstyle(4,3,1^5)$};
\node at (1,-0.65) {$\scriptstyle\langle (7,2,2), (1) \rangle$};
\tableauBoxFilled{0}{0}{}{lightgrey} \tableauBoxFilled{0}{1}{}{lightgrey}  \tableauBoxFilled{0}{2}{}{lightgrey}
\tableauBox{1}{0}{\oM} \tableauBox{1}{1}{\oM} \tableauBox{2}{0}{\tM} \tableauBox{0}{3}{\oM}
\tableauBox{1}{2}{\oM} \tableauBox{3}{1}{\dM}
\tableauBox{0}{4}{\oM} 
\tableauBox{2}{0}{\tM} \tableauBox{2}{1}{\tM}  \tableauBox{3}{0}{\dM} 
\tableauBox{4}{0}{1}  \tableauBox{0}{5}{\oM}
\tableauBox{0}{6}{\oM}
\end{tikzpicture}}\spy{12pt}{$\,\;\mapsto\;\,$}
\raisebox{-75pt}{\begin{tikzpicture}[x=0.5cm,y=-0.5cm]
\node at (1.5,-1.4) {$\scriptstyle(5,4,2,2,1^5)$};
\node at (1.5,-0.65) {$\scriptstyle\langle (9,4,2),(2,1)\rangle$};
\tableauBoxFilled{0}{0}{}{lightgrey} \tableauBoxFilled{0}{1}{}{lightgrey}  \tableauBoxFilled{0}{2}{}{lightgrey}
\tableauBox{1}{0}{\oM} \tableauBox{1}{1}{\oM} \tableauBox{1}{2}{\oM} \tableauBox{2}{0}{\tM} \tableauBox{0}{3}{\oM}
\tableauBox{0}{4}{\oM}
\tableauBox{2}{0}{\tM} \tableauBox{2}{1}{\tM} \tableauBox{2}{2}{\tM} \tableauBox{3}{0}{\dM} \tableauBox{0}{5}{\oM}
\node[left] at (-0.5,6.75) {$\scriptstyle 2$};
\node at (-0.5,6.5) {$\vdots$};
\tableauBox{0}{7.5}{\oM}\tableauBox{1}{3}{\tM}
\tableauBox{3}{1}{\dM}  \tableauBox{4}{0}{1} \tableauBox{5}{0}{1} \tableauBox{4}{1}{2} 
\end{tikzpicture}}\spy{12pt}{$\,\;\mapsto\;\,$}
\raisebox{-89.5pt}{\begin{tikzpicture}[x=0.5cm,y=-0.5cm]
\node at (1.95,-1.4) {$\scriptstyle(6,5,3,2,2,1^6)$};
\node at (1.95,-0.65) {$\scriptstyle\langle (11,5,3),(3,2)\rangle$};
\tableauBoxFilled{0}{0}{}{lightgrey} \tableauBoxFilled{0}{1}{}{lightgrey}  \tableauBoxFilled{0}{2}{}{lightgrey}
\tableauBox{1}{0}{\oM} \tableauBox{1}{1}{\oM} \tableauBox{1}{2}{\oM} \tableauBox{2}{0}{\tM} \tableauBox{0}{3}{\oM}
\tableauBox{0}{4}{\oM} \tableauBox{5}{0}{1} \tableauBox{4}{1}{2} \tableauBox{2}{3}{\dM}
\tableauBox{2}{0}{\tM} \tableauBox{2}{1}{\tM} \tableauBox{2}{2}{\tM} \tableauBox{3}{0}{\dM} \tableauBox{0}{5}{\oM}
\tableauBox{0}{6}{\oM}
\node[left] at (-0.5,7.75) {$\scriptstyle 3$};
\node at (-0.5,7.5) {$\vdots$};
\tableauBox{0}{8.5}{\oM}\tableauBox{1}{3}{\tM}\tableauBox{1}{4}{\tM}
\tableauBox{3}{1}{\dM}  \tableauBox{4}{0}{1} \tableauBox{5}{0}{1}
 \tableauBox{6}{0}{1}\tableauBox{5}{1}{2}
\end{tikzpicture}}
}\vspace*{-6pt}\caption{The map 
 $\tauS \mapsto \tauS \oplus \bigl( (2,1,1), (1,1) \bigr)$ applied repeatedly to the
skew partition $(1,1,1)/(1,1,1)$, showing 
the tableaux $t_3\bigl( (1,1,1) \oplus M\bigl( (2,1,1), (1,1) \bigr) \bigr)$ 
for $M \in \{0,1,2,3,4\}$. The corresponding `weight' partitions
$\omega_3\bigl( (1,1,1) \,\oplus\, M\bigl( (2,1,1), (1,1) \bigr) \bigr)$
and their $3$-decompositions
are shown above the tableaux. In each subsequent step the weight partition grows
by $\oplus \bigl( (2,1,1), (1,1) \bigr)$; note that this is not the case until the 
weight partition becomes $(3,2)$-large, which is first the case
when it is $(5,4,2,2,1^5)$,
thus the technical nature of Remark~\ref{remark:becomesWeightLarge}.
\label{fig:becomesWeightLarge}
}
\end{figure}

Finally we have the expected analogue of Lemma~\ref{lemma:adjoinToLarge}.

\begin{lemma}\label{lemma:adjoinToWeightLarge}
Let $\kappa^\-$ and $\kappa^\+$ be partitions. Fix $\ell^\- = \ell(\kappa^\-)$.
If $\tauS$ is an $\bigl(\ell(\kappa^\-), \ell(\kappa^\+)\bigr)$-weight large skew partition then,
in the $\ell^\-$-decomposition of 
$\omega_{\ell^\-}\bigl( \tauS \oplus (\kappa^\-, \kappa^\+)\bigr)$
we have
\begin{align*} 
\omega_{\ell^\-}\bigl( \tauS \oplus (\kappa^\-, \kappa^\+) \bigr)^\- &= \omega_{\ell^\-}(\tauS)^\- + \kappa^\- \\[-2pt]
\omega_{\ell^\-}\bigl( \tauS \oplus (\kappa^\-, \kappa^\+) \bigr)^\+ &= \omega_{\ell^\-}(\tauS)^\+ + \kappa^\+ .
\end{align*}
\end{lemma}

\begin{proof}
By Lemma~\ref{lemma:skewLargeImpliesWeightLargeImpliesLarge}(ii), 
$\tauS$ is $\bigl(\ell(\kappa^\-), \ell(\kappa^\+)\bigr)$-large. 
Hence by Lemma~\ref{lemma:adjoinToLarge} we have
$\bigl(\tau \oplus (\kappa^\-,\kappa^\+) \bigr)^\- = \tau^\- + \kappa^\-$ and
\smash{$\bigl(\tau \oplus (\kappa^\-,\kappa^\+) \bigr)^\+ = \tau^\+ + \kappa^\+$}.
This implies the two equations.
\end{proof}

\subsection{Bounding plethysms by greatest weights}\label{proof:muStablePlethysmBound}

If $\decs{\alpha^\-}{\alpha^\+}$ is an $\ell^\-$-decomposition then so is 
$\decs{n\alpha^\-}{n\alpha^\+}$ for any $n \in \N$. 
Therefore, by Lemma~\ref{lemma:ellDecompositionGreatestSignedWeight},
the following definition is well posed.

\begin{definition}\label{defn:plethysticGreatestSignedWeight}
Let $\ell^\- \in \N_0$ and let $n \in \N_0$. Given a skew partition $\tauS$ 
we define $\omega^{(n)}_{\ell^\-}(\tauS)$ to be the unique partition whose
$\ell^\-$-decomposition is 
$n\dec{\omega_{\ell^\-}(\tauS)^\-}{\omega_{\ell^\-}(\tauS)^\+}$.
\end{definition}

We give examples after
the next proposition, which is the main result in this section,
giving an upper bound in the $\ell^\-$-twisted dominance order (see Definition~\ref{defn:ellTwistedDominanceOrder}) on the 
constituents of an arbitrary plethysm.

\begin{proposition}\label{prop:twistedWeightBoundInner}
Let $\ell^\- \in \N_0$.
Let $\rho$ be a partition of $n$ and let $\tauS$ be a skew partition.
If $s_\pi$ is  a constituent of $s_\rho \,\circ\, s_\tauS$
then \smash{$\pi \unlhddot \omega_{\ell^\-}^{(n)}(\tauS)$}
in the $\ell^\-$-twisted dominance order.
\end{proposition} 

\begin{proof}
By Lemma~\ref{lemma:twistedKostkaMatrix}, $s_\pi$ is a summand of $e_{\pi^\-}h_{\pi^\+}$ with multiplicity $1$.
Hence, using Proposition~\ref{prop:plethysticSignedKostkaNumbers} (Plethystic Signed Kostka Numbers) for
the first equality below, we have
\[ |\PSSYT(\rho, \tauS)_{(\pi^\-,\pi^\+)}\bigr| = 
\langle e_{\pi^\-}h_{\pi^\+}, s_\rho \circ s_\tauS \rangle \ge \langle s_\pi, s_\rho \,\circ\, s_\tauS \rangle \ge  1. \]
Let $T \in \PSSYTw{\rho}{\tauS}{\pi^\-}{\pi^\+}$
and let $t$ be an inner $\tauS$-tableau in $T$.
By Lemma~\ref{lemma:greatestSignedWeight} 
\smash{$\bigl(\omega_{\ell^\-}(\tauS)^\-, \omega_{\ell^\-}(\tauS)^\+\bigr)$} is the greatest weight
(in the $\ell^\-$-signed dominance order on $\W_{\ell^\-} \times \W$ defined in Definition~\ref{defn:ellSignedDominanceOrder}) 
of all signed weights of $\tauS$-tableaux.
Thus, writing $\swt(t)$ for the signed weight of $t$, we have
\[ \bigl( \swt(t)^\-, \swt(t)^\+ \bigr) \unlhd \dec{\omega_{\ell^\-}(\tauS)^\-}{\omega_{\ell^\-}(\tauS)^\+} \]
where we regard either side as a composition,
as in the definition of the $\ell^\-$-twisted dominance order in Definition~\ref{defn:ellTwistedDominanceOrder}.
Hence, summing over all inner $\tauS$-tableaux in $T$, we have
\[ \dec{\pi^\-}{\pi^\+} \unlhd \dec{n\omega_{\ell^\-}(\tauS)^\-}{n\omega_{\ell^\-}(\tauS)^\+}. \]
(Note that $\decs{\pi^\-}{\pi^\+}$ is an $\ell^\-$-decomposition simply because $\pi$ is a partition.)
By definition of the $\ell^\-$-twisted dominance order this inequality holds if and only if
\smash{$\pi \unlhddot \omega^{(n)}_{\ell^\-}(\tauS)$}, as required.
\end{proof}

By Remark~\ref{remark:singletonSimpler}, Proposition~\ref{prop:twistedWeightBoundInner}
also follows from
a special case of Corollary~\ref{cor:signedWeightBoundForStronglyMaximalSignedWeight};
this alternative proof brings in many technicalities
irrelevant to Theorem~\ref{thm:muStable}, and so we much prefer the proof above
which is self-contained to this section.
We pause to give two examples.

\begin{example}\label{ex:lengthBoundOmega}
Fix $\ell^\- = 1$. Take $\rho = (3,1,1^M)$ and $\tauS = (2)$. Then $t_1\bigl( (2) \bigr) 
= \raisebox{0pt}{$\young(\oM1)$}\spy{1pt}{,}$ and so $\omega_1\bigl( (2) \bigr) = \bigl((1),(1)\bigr)$.
Hence
\[ \omega_1^{(4+M)}\bigl( (2) \bigr) \decMap (4+M)\dec{1}{1}  = \dec{(4+M)}{(4+M)} \decMap  (5+M,1^{3+M}).\]
and $\omega_1^{(4+M)}\bigl( (2) \bigr) = (5+M,1^{3+M})$.
Note, as claimed at the start of this section, that the right-hand side is the partition used as the upper bound in \S\ref{subsec:overviewSWL} (see Example~\ref{ex:lengthBound}).
\end{example}

\begin{example}\label{ex:cutUpsetForLawOkitani}
Fix $\ell^\- = 2$.
Taking $\rho = (3+M)$ and $\tauS = (4)$ in Proposition~\ref{prop:twistedWeightBoundInner} we obtain
\[ \supp s_{(3+M)} \circ s_{(4)} \subseteq \bigl\{ 
\lambda \in \Par(12+4M) : \lambda \unlhddot \omega_{2}^{(3+M)}(4) \bigr\} \]
where $\unlhddot$ is the $2$-twisted dominance order.
Since $t_2\bigl( (4) \bigr) = \young(\oM\tM11)$ has signed weight $\bigl((1^2), (2)\bigr)$ we have
\[ \begin{split} \omega_{2}^{(3+M)}\bigl( (3+M) \bigr) &
\decMap (3+M) \dec{(1^2)}{(2)} \\ &\ = \dec{(3+M,3+M)}{(6+2M)} \decMap (8+2M,2^{2+M})\end{split} \]
and so $\omega_2^{(3+M)}\bigl((3+M)\bigr) = 
(8+2M,2^{2+M})$ is an upper bound in the $2$-twisted dominance order
for the constituents of the plethysm $s_{(3+M)} \circ s_{(4)}$. 
In particular, if $s_\sigma$ appears in $s_{(3+M)} \circ s_{(4)}$
then $\ell(\sigma) \le 3+M$, as used earlier in \S\ref{subsec:cutUpsetForLawOkitani}.
Correspondingly, $(8+2M,2^{2+M})$ is the upper bound in the twisted interval defining
the stable partition system $\PSeq{M}$ used in \S\ref{subsec:cutUpsetForLawOkitani}: see
\S\ref{subsec:stablePartitionSystemsAsIntervals} for the interpretation
of this stable partition system using intervals for the $2$-twisted dominance order.
\end{example}

The following corollary is used
in Lemma~\ref{lemma:stablePartitionSystemForMuVarying}
to verify condition~(i) in the Signed Weight Lemma (Lemma~\ref{lemma:SWL}).

\begin{corollary}[Inner Twisted Weight Bound]\label{cor:twistedWeightBoundInnerGrowing}
Let $\rho$ be a partition of $n$.
Let $\kappa^\-$ and $\kappa^\+$ be partitions. Fix $\ell^\- = \ell(\kappa^\-)$ 
and let $\ell^\+ = \ell(\kappa^\+)$.
Let $\muS$ be an $(\ell^\-,\ell^\+)$-weight large
skew partition.
If $s_\sigma$ is a constituent of 
the plethysm $s_\rho  \circ s_{\muS \oplus M(\kappa^\-,\kappa^\+)}$ then
\[ \sigma \unlhddot \omega^{(n)}_{\ell^\-}(\muS) \oplus nM(\kappa^\-, \kappa^\+).\]
\end{corollary}

\begin{proof}
By Proposition~\ref{prop:twistedWeightBoundInner} taking
$\tauS = \muS \oplus M(\kappa^\-,\kappa^\+)$ we have
\begin{equation}
\label{eq:sigmaBound}
\sigma \unlhddot \omega^{(n)}_{\ell^\-}\bigl(\hskip0.5pt\muS \oplus M(\kappa^\-,\kappa^\+)\bigr)
\end{equation}
Since $\muS$ is $(\ell^\-,\ell^\+)$-weight large, by Lemma~\ref{lemma:adjoinToWeightLarge}
we have 
\[ \omega_{\ell^\-}\bigl(\muS \oplus M(\kappa^\-,\kappa^\+) \bigr)
= \bigl(\omega_{\ell^\-}(\muS)^\- + M\kappa^\-,  \omega_{\ell^\-}(\muS)^\+ + M\kappa^\+ \bigr) \]
and so 
\[ \omega_{\ell^\-}^{(n)}\bigl(\muS \oplus  M(\kappa^\-,\kappa^\+) \bigr) =
\omega_{\ell^\-}^{(n)}\bigl( \muS) \oplus nM (\kappa^\-,\kappa^\+). \]
Therefore~\eqref{eq:sigmaBound} is equivalent to
$\sigma \unlhddot \omega^{(n)}_{\ell^\-}(\muS) \oplus nM (\kappa^\-,\kappa^\+)$.
\end{proof}

This result should
be compared to Corollary~\ref{cor:signedWeightBoundForStronglyMaximalSignedWeight},
which gives a more sophisticated bound used in the proof of Theorem~\ref{thm:nuStable}.

\section{Proof of Theorem~\ref{thm:muStable}}\label{sec:muStable}
We begin in \S\ref{subsec:muStableZero}
by proving the second part of this theorem where the stable multiplicity is zero.
We then use the Signed Weight Lemma
(Lemma~\ref{lemma:SWL}) to prove the remaining part of the theorem.
In \S\ref{subsec:stablePartitionSystemMu} we construct a suitable
stable partition system. In \S\ref{subsec:stablePlethysticTableau}
we construct a bijection satisfying~(ii) in the Signed Weight Lemma. Finally
in \S\ref{subsec:muStableProof} we put together all the pieces
proving Theorem~\ref{thm:muStableSharp} which restates Theorem~\ref{thm:muStable}
with an explicit bound.

\subsection{The vanishing case of Theorem~\ref{thm:muStable}}\label{subsec:muStableZero}
We require the following statistic. Recall that $\dec{\lambda^\-}{\lambda^\+}$ denotes the $\ell^\-$-decomposition
of a partition~$\lambda$, as defined in Definition~\ref{defn:ellDecomposition}.

\begin{definition}\label{defn:LZBound}
Let $\lambda$ and $\omega$ be partitions of the same size.
Let $\swtp{\kappa}$ and $\swtp{\eta}$ be signed weights.
Fix $\ell^\- = \max( \ell(\eta^\-), \ell(\kappa^\-) )$.
We define $\LZBound\bigl([\lambda, \omega]_\unlhddotS, \swtp{\kappa}, \swtp{\eta}\bigr)$ to be the minimum of the quantities
\smallskip
\begin{bulletlist}
\item $\displaystyle \frac{\sum_{i=1}^k \omega^\-_i - \sum_{i=1}^k \lambda^\-_i}{ \sum_{i=1}^k
\eta^\-_i - \sum_{i=1}^k \kappa^\-_i }$.\\[1pt]
\item $\displaystyle 
\frac{|\omega^\-| + \sum_{i=1}^k \omega^\+_i - |\lambda^\-| - \sum_{i=1}^k
\lambda^\+_i}{ |\eta^\-| + \sum_{i=1}^k \eta^\+_i - |\kappa^\-| - \sum_{i=1}^k \kappa^\+_i}$
\end{bulletlist}

\smallskip\noindent
taken in each case over those $k$ for which the denominator is strictly positive; if there
are no such $k$, we leave $\LZBound\bigl([\lambda, \omega]_\unlhddotS, \swtp{\kappa}, \swtp{\eta}\bigr)$ undefined.
\end{definition}

If, as in the first case of Theorem~\ref{thm:muStable},
we have $\swtp{\eta} \notunlhd \swtp{\kappa}$
in the $\ell^\-$-signed dominance order (see Definition~\ref{defn:ellSignedDominanceOrder}), then 
it immediately follows from the definition of this order
that $\LZBound\bigl([\lambda, \omega]_\unlhddotS, \swtp{\kappa}, \swtp{\eta}\bigr)$ is defined for
any partitions $\lambda$ and $\omega$.

In the following proposition we prove a 
generalization of
the final part of Theorem~\ref{thm:muStable}, with an explicit bound from Definition~\ref{defn:LZBound}.
Recall from Definition~\ref{defn:plethysticGreatestSignedWeight} that
\smash{$\omega^{(n)}_{\ell^\-}(\muS)$} is the
unique partition whose $\ell^\-$-decomposition is 
$n\dec{\omega_{\ell^\-}(\tauS)^\-}{\omega_{\ell^\-}(\tauS)^\+}$,
where $\dec{\omega_{\ell^\-}(\tauS)^\-}{\omega_{\ell^\-}(\tauS)^\+}$
is the signed weight of the $\ell^\-$-greatest tableau: see Definition~\ref{defn:greatestSignedWeight}
and Lemma~\ref{lemma:greatestSignedWeight}.
See Definitions~\ref{defn:large} and~\ref{defn:weightLarge} for the definitions
of `large' and `weight large'.

\begin{proposition}\label{prop:muStableZero}
For each $M \in \N_0$, let $\nu^{(M)}$ be a partition of $n \in \N$.
Let~$\kappa^\-$ and $\kappa^\+$ be partitions. Let $\ell^\- = \ell(\kappa^\-)$.
Let $\eta^\-$ and $\eta^\+$ be partitions
with $\ell(\eta^\-) \le \ell^\-$ and $|\eta^\-| + |\eta^\+| = |\kappa^\-| + |\kappa^\+|$. 
Let $\ell^\+ = \max(\ell(\kappa^\+), \ell(\eta^\+))$ 
and let $\muS$ be an $(\ell^\-,\ell^\+)$-weight large skew partition.
Let $\lambda$ be a $(\ell^\-,\ell^\+)$-large partition.
If $\swtp{\eta} \notunlhd \swtp{\kappa}$
then 
\[ \bigl\langle s_\nuSeq{M} \circ s_{\muS \opluss M(\kappa^\-, \kappa^\+)}, 
s_{\lambda \opluss nM(\eta^\-, \eta^\+)} \bigr\rangle = 0 \]
for all 
\[ M > \LZBound\bigl([\lambda, \omega_{\ell^\-}^{(n)}(\muS)]_\unlhddotS, \swtp{\kappa},\swtp{\eta}\bigr)\bigl/ n. \]
\end{proposition}

\begin{proof}
By Corollary~\ref{cor:twistedWeightBoundInnerGrowing},
applied with $\rho = \nuSeq{M}$,
if $s_\sigma$ is a constituent of the plethysm
\smash{$s_\nuSeq{M} \circ s_{\muS \opluss M(\kappa^\-, \kappa^\+)}$}
then 
\smash{$\sigma \unlhddot \omega^{(n)}_{\ell^\-}(\muS) \oplus nM(\kappa^\-, \kappa^\+)$}.
Therefore, by Definition~\ref{defn:plethysticGreatestSignedWeight} and
the definition of the $\ell^\-$-twisted dominance order (see Definition~\ref{defn:ellTwistedDominanceOrder})
we have
\begin{equation}\label{eq:sigmaSignedWeightBound} \swtp{\sigma} \unlhd n \bigl( \omega_{\ell^\-}(\muS)^\-, \omega_{\ell^\-}(\muS)^\+ \bigr) 
+ nM\swtp{\kappa}\end{equation}
in the $\ell^\-$-signed dominance order (see Definition~\ref{defn:ellSignedDominanceOrder}).
Since $\lambda$ is $(\ell^\-,\ell^\+)$-large and $\ell(\eta^\-) \le \ell^\-$
and $\ell(\eta^\+) \le \ell^\+$ by hypothesis, we may apply
Lemma~\ref{lemma:adjoinToLarge} to get that the $\ell^\-$-decomposition of 
$\lambda \oplus nM\swtp{\eta}$ is $\decs{\lambda^\- + nM \eta^\-}{\lambda^\+ + nM\eta^\+}$.
We now substitute $\swtp{\lambda} + nM\swtp{\eta}$ for $\swtp{\sigma}$ in~\eqref{eq:sigmaSignedWeightBound}
to obtain the inequality
\begin{equation}\label{eq:signedEtaBound} \swtp{\lambda} + nM\swtp{\eta} \unlhd 
n \bigl( \omega_{\ell^\-}(\muS)^\-, \omega_{\ell^\-}(\muS)^\+ \bigr)  + nM\swtp{\kappa} \end{equation}
Thus, by Lemma~\ref{lemma:twistedDominanceOrderOldDefinition}(a), we have, for each $k \le \ell^\-$,
\[ \sum_{i=1}^k \lambda^\-_i + nM \sum_{i=1}^k \eta^\-_i \le  \sum_{i=1}^k
n\omega_{\ell^\-}(\muS)^\-_i + nM \sum_{i=1}^k \kappa^\-_i  \]
and so
\[ nM \Bigl(\, \sum_{i=1}^k \eta^\-_i - \sum_{i=1}^k \kappa^\-_i \Bigr)
\le \sum_{i=1}^k n\omega_{\ell^\-}(\muS)^\-_i - \sum_{i=1}^k \lambda^\-_i. \]
Hence, if $\sum_{i=1}^k \eta^\-_i > \sum_{i=1}^k \kappa^\-_i$ then 
$nM$ is at most the relevant bound from the first case of Definition~\ref{defn:LZBound}.
The proof for the second family of bounds is very closely analogous: by~\eqref{eq:signedEtaBound}
and Lemma~\ref{lemma:twistedDominanceOrderOldDefinition}(b) we obtain
\[ \begin{split}
&|\lambda^\-| + \sum_{i=1}^k \lambda^\+_i + nM \sum_{i=1}^k \eta^\+_i \\
&\qquad \le n  |\omega_{\ell^\-}(\muS)^\-| + \sum_{i=1}^k
n\omega_{\ell^\-}(\muS)^\+_i + nM \sum_{i=1}^k \kappa^\+_i + nM\bigl( |\eta^\+| - |\kappa^\+|\bigr)
\end{split} \]
and so using $|\eta^\+| - |\kappa^\+| = -|\eta^\-| + |\kappa^\-|$ we have
\[ \begin{split}
 nM \Bigl(\, \sum_{i=1}^k \eta^\+_i &- \sum_{i=1}^k \kappa^\+_i + |\eta^\-| - |\kappa^\-| \Bigr)\\
 &\quad \le n |\omega_{\ell^\-}(\muS)^\-| + \sum_{i=1}^k n\omega_{\ell^\-}(\muS)^\+_i - |\lambda^\-| - \sum_{i=1}^k \lambda^\+_i \end{split} \]
showing that $nM$ is at most the relevant bound from the second case of Definition~\ref{defn:LZBound}.
\end{proof}

\subsection{Stable partition system for Theorem~\ref{thm:muStable}}
\label{subsec:stablePartitionSystemMu}
The stable partition system
we require to prove the main part of Theorem~\ref{thm:muStable}
again comes from Corollary~\ref{cor:twistedWeightBoundInnerGrowing}.

\begin{lemma}\label{lemma:stablePartitionSystemForMuVarying}
Let $\rho$ be a partition of $n \in \N$.
Let $\kappa^\-$ and $\kappa^\+$ be partitions. Fix $\ell^\- = \ell(\kappa^\-)$
and $\ell^\+ = \ell(\kappa^\+)$.
Let $\muS$ be a skew partition;
if $\ell^\- \not=0$ then suppose
that $\muS$ is $(\ell^\-+1,\ell^\+)$-weight large for $\ell^\-$.
Let $\lambda$ be a partition of $n|\muS|$
such that \smash{$\lambda \,\unlhddot\, \omega^{(n)}_{\ell^\-}(\muS)$}, where $\unlhddot$ is the $\ell^\-$-twisted dominance order.
Let 
\[ \PSeq{M} = [\lambda \oplus nM(\kappa^\-,\kappa^\+), \omega^{(n)}_{\ell^\-}(\muS) 
\oplus nM(\kappa^\-, \kappa^\+)]_{\unlhddotS}. \]
 Then $(\PSeq{M})_{M \in \N_0}$ is a stable partition system for the symmetric functions $g_\pi = e_{\pi^\-}h_{\pi^\+}$.
Moreover, if $\pi \in \PSeq{M}$ and~$s_\sigma$ is a summand of $e_{\pi^\-}h_{\pi^\+}$ appearing
in the plethysm $s_\rho  \circ s_{\muS \oplus M(\kappa^\-,\kappa^\+)}$ then
$\sigma \in \PSeq{M}$.
\end{lemma}

\begin{proof}
Suppose that $\ell^\- \not=0$.
Then, by the hypothesis that  $\muS$ is $(\ell^\-+1,\ell^\+)$-weight large for $\ell^\-$, the
partition \smash{$\omega_{\ell^\-}(\muS)$ is $(\ell^\-+1,\ell^\+)$}-large
(see Definition~\ref{defn:weightLarge})
and so
\smash{$\omega^{(n)}_{\ell^\-}(\muS)$} is also $(\ell^\-+1,\ell^\+)$-large. 
We therefore may
apply Corollary~\ref{cor:signedIntervalStable} to deduce that the partition system is stable.
For the final claim, by Lemma~\ref{lemma:twistedKostkaMatrix}, we have
$\sigma \unrhddot \pi$, and so,
since $\pi \unrhddot \lambda \oplus nM\swtp{\kappa}$, we have
$\sigma \unrhddot \lambda \oplus nM\swtp{\kappa}$.
By Corollary~\ref{cor:twistedWeightBoundInnerGrowing},
we have \smash{$\sigma \unlhddot \omega^{(n)}_{\ell^\-}(\muS) 
\oplus nM(\kappa^\-, \kappa^\+)$}. Hence $\sigma \in \PSeq{M}$.
\end{proof}

\subsection{Positions for plethystic tableaux: motivating example}
\label{subsec:plethysticTableauGMotivation}
The aim in the next three subsections is to prove
Proposition~\ref{prop:plethysticSignedTableauInnerStable}, the
plethystic analogue of Proposition~\ref{prop:signedTableauStable} (Tableau Stability), 
using the map $\Gmap$ on
plethystic semistandard signed tableaux
defined in Definition~\ref{defn:G}.
We begin with a motivating example
that gives a good idea how to define $\Gmap$ using
the earlier map $\Fmap$ on semistandard tableaux
from Definition~\ref{defn:F}.
We use this example
throughout this section.

\begin{example}\label{ex:22motivation}
The special case of Theorem~\ref{thm:muStable} taking
$\nu = (2)$, $\mu = (2,2)$, $\mus = \varnothing$,
$(\kappa^\-,\kappa^\+) = \bigl((1,1), (1) \bigr)$ and $\lambda = (8-b,b)$ with
$b \in \{2,3,4\}$ is that 
\begin{equation}
\label{eq:plethysticMultiplicityOverview} 
\langle s_{(2)} \circ s_{(2+M,2,2^M)}, s_{(8-b+2M,b,2^{2M})}  \rangle \end{equation}
is ultimately constant. 
In our proof using the Signed Weight Lemma (Lemma \ref{lemma:SWL})  
we must verify condition (ii), that 
\begin{equation}
\label{eq:plethysticWeightOverview}
\bigl| \PSSYT\bigl( (2), (2+M,2,2^M)_{((2+2M,2+2M),(6-b+2M,b-2))} \bigr| \end{equation}
is constant for all $M$ sufficiently large. (Here $\bigl((2+2M,2+2M), (6-b+2M,b-2)\bigr)$
is the $2$-decomposition of $(8-b,b) \oplus 2M\bigl( (1,1), (1) \bigr)$: 
see Definition~\ref{defn:ellDecomposition}.) 
In any
plethystic semistandard tableau $\young(st)$ of signed weight 
$\bigl((2+2M,2+2M),(6-b+2M,b-2)\bigr)$, the inner tableaux $s$ and $t$ have
$2+2M$ entries of  $-1$ and~$-2$ and $6-b+2M$ entries of $1$ between them.
 By the signed semistandard condition in Definition~\ref{defn:plethysticSemistandardSignedTableau},
the entries of~$-1$ and~$-2$ lie in the first two columns.
Thus when $M$ is large, 
both $s$ and $t$ have the form
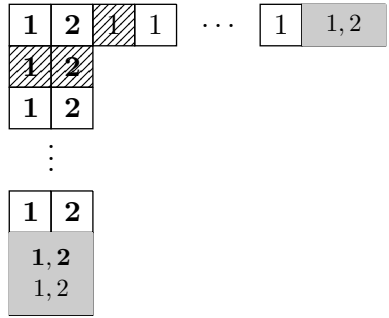
\[ \begin{tikzpicture}[x=0.55cm,y=-0.55cm] 
\tB{1}{1}{\oM} \tB{1}{2}{\tM} \tB{1}{3}{1} \tB{1}{4}{1} \node at (6,2.5) {$\ldots$};
\tB{1}{7}{1} \tBDF{1}{9}{\scalebox{0.85}{$1,2$}}{lightgrey};
\fill[pattern = north east lines] (3,2)--(4,2)--(4,3)--(3,3)--(3,2);

\tB{2}{1}{\oM} \tB{2}{2}{\tM} 
\fill[pattern = north east lines] (1,3)--(3,3)--(3,4)--(1,4)--(1,3);

\tB{3}{1}{\oM} \tB{3}{2}{\tM} 
\node at (2,5.5) {$\vdots$};
\tB{5.5}{1}{\oM} \tB{5.5}{2}{\tM} 
\draw (1,7.5)--(3,7.5)--(3,9.5)--(1,9.5)--(1,7.5);
\fill[color=lightgrey] (1,7.5)--(3,7.5)--(3,9.5)--(1,9.5)--(1,7.5);
\node at (2,8.5) 
{\scalebox{0.85}{$ \begin{matrix} \mbf{1},\mbf{2} \\ 1,2\end{matrix}$}};

\end{tikzpicture} \]
where the two  shaded regions in the bottom $2$ rows and rightmost $2$ columns
are marked with the possible entries. (As ever negative entries may be repeated in a column, and positive
entries repeated in row.) Clearly, almost all the entries of $s$ and $t$ are determined
by their weight. In particular both~$s$ and~$t$
are obtained from a tableau for the case $M-1$ by
inserting $\young(\oM\tM)$ as a new
complete second row and $\young(1)$ as a new complete third column.
By Definition~\ref{defn:skewPositions}, the $2$-top and $1$-left positions
of $(2+M,2,2^M)$ are both $(1,2)$. 
Thus these insertions are exactly as specified by the map
$\Fmap$ from Definition~\ref{defn:F}. We have shown that provided $M$ is
sufficiently large, the map
from $\PSSYT\bigl( (2), (2+M-1,2,2^{M-1})_{((2+2(M-1),2+2(M-1)),(6-b+2(M-1),b-2))}$
to $\PSSYT\bigl( (2), (2+M,2,2^M)_{((2+2M,2+2M),(6-b+2M,b-2))}$ defined by
\[ \pyoung{0.6cm}{0.6cm}{ {{$s$, $t$}} }\ \, \raisebox{6pt}{$\longmapsto$} 
\raisebox{-0.5pt}{\pyoung{1cm}{0.6cm}{
{{$\Fmap(s)$, $\Fmap(t)$}} }} \]
is surjective (with inverse defined by deleting the hatched boxes and performing
suitable shifts) and so 
the cardinality in~\eqref{eq:plethysticWeightOverview} is ultimately constant.
In fact $\Fmap(s)$ and $\Fmap(t)$ are semistandard provided
that the $2$-top positions of both $s$ and $t$ both contain $-2$. 
As we show in Example~\ref{ex:22positions},
using Lemma~\ref{lemma:plethysticSignedTableauPositions} in
the following subsection, this holds provided $M \ge \max(3, b+2)$,

\end{example}

\subsection{Lemma on positions for plethystic tableaux}\label{subsec:plethysticTableauPositions}
The critical positions are defined in  Definition~\ref{defn:skewPositions}
and were seen in the non-skew case in the previous subsection.
To remind the reader of the general definition we begin with an example 
in the skew case.

\begin{example}\label{ex:skewPositions}
Let $\muS = (7,6,3,1) / (4,3,1)$.
We take $\kappa^\- = (2,1)$ and $\kappa^\+ = (1,1)$ so $\ell^\- = \ell^\+ = 2$.
The diagrams below show the $1$-left, $2$-left, $1$-top and $2$-top positions
in the partitions $\muS \oplus M(\kappa^\-, \kappa^\+)$ for $M \in \{0,1,2\}$.
Following our usual convention, top positions, relevant to the insertion
of negative entries, are marked by bold numbers. The skew partition
$\muS$ is $(a(\mus)+\ell^\-, \ell^\+)$-large as it contains $(2, 6)$,  
and $(\ell^\-, \ell(\mus))$-large as it contains $(3,2)$; all
these positions are contained within $[\mu]$.
Moreover, also as promised by Lemma~\ref{lemma:positions},
the top positions are no higher than row $\max( \ell(\mus), \ell(\mu^\+), \ell^\+ )$ 
and left positions are no further left than column~$\ell^\- + a(\mus)$.
The relevant boxes $(2,6)$ and $(3,2)$ are hatched,
as in Figure~\ref{fig:ellDecompositionSkewTableau}.

\smallskip
\hspace*{-0.5in}
\renewcommand{\d}{0.5}
\newcommand{\yb}{\youngBox{\d}{grey}}
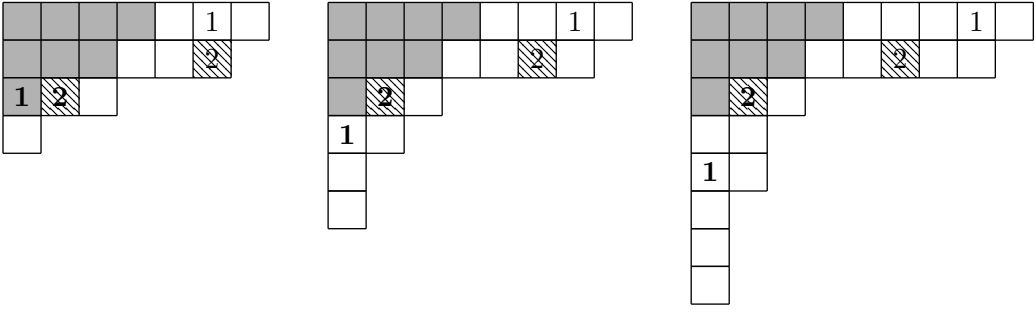
\begin{tikzpicture}[x=\d cm,y=-\d cm]
\pyoungInner{ {{\yb,\yb,\yb,\yb,\ ,\ , \ }, {\yb,\yb,\yb,\ ,\ , \ }, {\yb, \ , \ }, {\ }} }
\node at (1.5,3.5) {$\mbf{1}$};
\node at (2.5,3.5) {$\mbf{2}$};
\node at (6.5,1.5) {$1$};
\node at (6.5,2.5) {$2$};
\fill[pattern = north west lines] (6,2)--(7,2)--(7,3)--(6,3);
\fill[pattern = north west lines] (2,3)--(3,3)--(3,4)--(2,4);
\end{tikzpicture}\spy{7pt}{\quad}
\raisebox{-1cm}{\begin{tikzpicture}[x=\d cm,y=-\d cm]
\pyoungInner{ {{\yb,\yb,\yb,\yb,\ ,\ , \ ,\ }, {\yb,\yb,\yb,\ ,\ , \ ,\ }, {\yb, \ , \ }, {\ , \ }, {\ },
{\ }} }
\node at (1.5,4.5) {$\mbf{1}$};
\node at (2.5,3.5) {$\mbf{2}$};
\node at (7.5,1.5) {$1$};
\node at (6.5,2.5) {$2$};
\fill[pattern = north west lines] (6,2)--(7,2)--(7,3)--(6,3);
\fill[pattern = north west lines] (2,3)--(3,3)--(3,4)--(2,4);
\end{tikzpicture}}\spy{7pt}{\quad}
\raisebox{-2cm}{\begin{tikzpicture}[x=\d cm,y=-\d cm]
\pyoungInner{ {{\yb,\yb,\yb,\yb,\ ,\ , \ ,\ ,\ }, {\yb,\yb,\yb,\ ,\ , \ ,\ ,\ }, {\yb, \ , \ }, 
{\ , \ }, {\ ,\ }, {\ }, {\ }, {\ }} }
\node at (1.5,5.5) {$\mbf{1}$};
\node at (2.5,3.5) {$\mbf{2}$};
\node at (8.5,1.5) {$1$};
\node at (6.5,2.5) {$2$};
\fill[pattern = north west lines] (6,2)--(7,2)--(7,3)--(6,3);
\fill[pattern = north west lines] (2,3)--(3,3)--(3,4)--(2,4);
\end{tikzpicture}}

\smallskip
\noindent We remark that if we changed $\mus = (4,3,1)$ to $(a)$ with $a \in \{0,1,2,3,4\}$ then
the $1$-top and $2$-top positions remain unchanged, because they always lie
in rows weakly below $\max(\ell(\mus), \ell^\+, \ell(\mu^\+))$, and this statistic
remains $3$ since 
$\mu^\+ = (5,4,1)$. The $1$-left position is also unchanged, but the $2$-left position
is now $\bigl(2,\ell^\- + \max(a, \mu^\+_3)\bigr) = \bigl( 2, 2 + \max(a, 1) \bigr)$
which is $(2,3)$ if and only if $a \le 1$.
Note that if $a \le 1$ then
it is possible to insert a column of height $2$ immediately right of column $3$,
and the entries of this column \emph{must} be positive.
\end{example}

\begin{definition}\label{defn:LPBound}
Let $\kappa$, $\mu$ and $\lambda$ be partitions. Let $A \in \N_0$.
We define $\LPBound\bigl(n, \mu \,: \lambda, \kappa \,: A)$
to be $0$ if $\kappa = \varnothing$ and otherwise 
to be the maximum of
\begin{align*}
&\frac{n\sum_{i=1}^\rk \mu_i - \sum_{i=1}^\rk \lambda_i - \mu_\rk + \max(A, \mu_{\rk+1})}{\kappa_\rk - \kappa_{\rk+1}}.
\end{align*}
for $1 \le \rk \le \ell(\kappa)$, omitting any expressions with zero denominator.
\end{definition}

The following lemma is the analogue of Lemma~\ref{lemma:signedTableauPositions}.
The statement, apart from the change of bounds, is very similar, but the proof is somewhat easier, 
apart from some technicalities arising from the skew case,
because the shape is known to be $\muS \oplus M(\kappa^\-,\kappa^\+)$, rather than
an arbitrary partition in an interval for the $\ell^\-$-twisted dominance order.
The relevant positions are defined in Definition~\ref{defn:skewPositions}.

\begin{lemma}
\label{lemma:plethysticSignedTableauPositions}
Let $\kappa^\-, \kappa^\+$ be partitions. Fix $\ell^\- = \ell(\kappa^\-)$
and $\ell^\+ = \ell(\kappa^\+)$.
Let $\muS$ be an $\bigl(\ell^\-+a(\mus), \ell^\+\bigr)$-large and $\bigl(\ell^\-, \ell(\mus)\bigr)$-large
skew partition. 
Let~$\lambda$ and $\omega$ be $(\ell^\-, \ell^\+)$-large partitions of $n|\muS|$.
Let 
\[ \pi \in \bigl[\lambda \oplus nM(\kappa^\-, \kappa^\+), \omega \oplus nM(\kappa^\-,\kappa^\+) \bigr]_\unlhddotS. \] 
Let $T \in \PSSYTw{\rho}{\muS \oplus M(\kappa^\-,\kappa^\+)}{\pi^\-}{\pi^\+}$
and let~$t$ be a $\muS \oplus M(\kappa^\-,\kappa^\+)$-inner tableau of~$T$.
Let $L$ be the maximum of 
\begin{bulletlist}
\item $\LPBound\bigl(n,\mu' : \lambda^\-, \kappa^\- : \max(\ell(\mus), \ell(\mu^\+), \ell^\+\bigr))$ 
\item 
$\LPBound\bigl(n, \mu : \lambda^\+, \kappa^\+ : a(\mus) + \ell^\-\bigr) + |\lambda^\+| - |\omega^\+|$.
\end{bulletlist}
If $M-1 \ge L$ then 
\begin{thmlist}
\item the $\rM$-bottom position of $t$ contains $-\rM$ 
if $\rM < \ell^\-$ and $\kappa^\-_{\rM} > \kappa^\-_{\rM+1}$;
\item the $\ell^\-$-bottom position of $t$ contains $-\ell^\-$; 
\item if $\kappa^\- \not=\varnothing$ and \emph{either} $\kappa^\+\not=\varnothing$ \emph{or} $\mus\not=\varnothing$ \emph{or} $\mu^\+ \not=\varnothing$
 then
the box $\bigl(\max(\ell(\mus),\ell(\mu^\+),\ell^\+),\ell^\-\bigr)$ of $t$ contains a negative entry; 
\item the $\rP$-right position of $t$ contains $\rP$ if $\rP < \ell^\+$ and $\kappa^\+_\rP > \kappa^\+_{\rP+1}$;
\item the $\ell^\+$-right position of $t$ contains $\ell^\+$.
\end{thmlist}
Moreover if $M \ge L$ then the same results hold replacing `bottom' with `top' and `right' with `left', except that 
\begin{thmlist}
\item[\emph{(ii)}] if $\mus = \varnothing$ and $\kappa^\+ = \varnothing$ and $\mu^\+ = \varnothing$ 
then the $\ell^\-$-top position is $(0,\ell^\-)$;
\item[\emph{(iv)}] and \emph{(v)}  if $\mu_{r^\++1} \le \ell^- + a(\mus)$
and so the 
$r^\+$-left-position
is $\bigl(r^\+, \ell^\- + a(\mus)\bigr)$, then it may contain~$-\ell^\-$.
\end{thmlist}
\end{lemma}


\begin{proof}
Since $\mu$ is $\bigl(\ell^\- + a(\mus), \ell^\+\bigr)$-large, it is  $(\ell^\-,\ell^\+)$-large. Therefore,
by Lemma~\ref{lemma:adjoinToLarge}, we have 
$\bigl( \mu \opluss M(\kappa^\-,\kappa^\+) \bigr)^\-
= \mu^\- + M\kappa^\-$ and 
$\bigl( \mu \opluss M(\kappa^\-,\kappa^\+) \bigr)^\+
= \mu^\+ + M\kappa^\+$.
For use throughout 
the proof we define a subpartition $\alpha$ of $\mu \opluss M(\kappa^\-,\kappa^\+)$ 
containing $\mus$ by
\begin{equation}
\label{eq:alpha}
\alpha_i = \min(\mus_i + \ell^\-, \mu_i) \text{ for $1 \le i \le \ell(\mu) + Ma(\kappa^\-)$}.
\end{equation}
Thus $[\alpha/\mus]$ consists of the first $\ell^\-$ boxes in each row of $\mu/\mus \opluss M(\kappa^\-,\kappa^\+)$,
or the whole row if it has fewer than $\ell^\-$ boxes.
We show $[\alpha/\mus]$ in Figure~\ref{fig:ellDecompositionSkewTableau}. As a final preliminary,
for ease of reference,
we record
the following immediate corollary of Lemma~\ref{lemma:positions}: 
the $\rM$-top position of $\muS \oplus M(\kappa^\-,\kappa^\+)$ is
\begin{align}
\label{eq:top} 
&\begin{cases} \bigl( \mu'_{\rM+1} + M\kappa^\-_{\rM+1}, \rM \bigr)\hspace*{29pt} & \text{if $\rM < \ell^\-$} \\
\bigl( \max( \ell(\mus), \ell(\mu^\+), \ell^\+\bigr), \ell^\- ) & \text{if $\rM = \ell^\-$.} \end{cases} 
\intertext{and the $\rP$-left position of $\muS \oplus M(\kappa^\-,\kappa^\+)$ is }
\label{eq:left}
&\begin{cases}  \bigl( \rP, \mu_{\rP+1} + M\kappa^\+_{\rP+1} \bigr) & \text{if $\rP < \ell^\+$} \\
\bigl( \ell^\+, \max(\ell^\- + a(\mus), \mu_{\ell^\++1}) \bigr) & \text{if $\rP = \ell^\+$}. \end{cases}
\end{align}

\begin{figure}[t!]
\begin{center}\begin{tikzpicture}[x=0.5cm,y=-0.5cm]

\fill[color=grey] (0,0)--(0,6)--(2,6)--(2,3)--(2,2)--(3,2)--(3,1)--(5,1)--(5,0)--(0,0);
\draw[line width=1pt] (0,16)--(1,16)--(1,12)--(2,12)--(2,6)--(4,6)--(4,2)--(5,2)--(5,1)--(7,1)--(7,0)
	--(5,0)--(5,1)--(3,1)--(3,2)--(2,2)--(2,6)--(0,6)--(0,16);
\fill[color=blue!25] (0,16)--(1,16)--(1,12)--(2,12)--(2,6)--(4,6)--(4,2)--(5,2)--(5,1)--(7,1)--(7,0)
	--(5,0)--(5,1)--(3,1)--(3,2)--(2,2)--(2,6)--(0,6)--(0,16);

\draw[dashed] (0,3)--(13,3);
\draw[dashed] (6,0)--(6,15);

\begin{scope}[xshift=-0.8cm]
\draw[<-] (-0.5,0)--(-0.5,1); 
\draw[->] (-0.5,2)--(-0.5,3);
\node at (-0.5,1.5) {$\scriptstyle \ell^\+$};
\end{scope}

\begin{scope}[xshift=-0.2cm]
\draw[<-] (-0.5,0)--(-0.5,2.5); 
\draw[->] (-0.5,3.5)--(-0.5,6);
\node at (-0.5,3) {$\scriptstyle \ell(\mus)$};
\end{scope}

\draw[<-] (0,-0.5)--(1.5,-0.5);
\draw[->] (3.5,-0.5)--(4.9,-0.5);
\node at (2.5,-0.5-0.05) {$\scriptstyle a(\mus)$};

\newcommand{\lMinus}[2]{%
\begin{scope}[xshift=#1cm, yshift=#2cm]
\draw[<-] (0.1,0)--(2.5,0);
\draw[->] (3.5,0)--(5.9,0);
\node at (3,-0.05) {$\scriptstyle \ell^-$};
\end{scope}}

\lMinus{2.5}{0.25}
\lMinus{0}{0.6}
\lMinus{1}{-1.3}

\draw[thick] (3.5,12)--(3.5,13)--(2,13)--(2,15)--(1,15)--(1,16)--(0,16)--(0,6)--(2,6)--(2,3)--(2,2)--(3,2)--(3,1)
	--(5,1)--(5,0)--(11,0);
\begin{scope}[xshift=1cm]
\node at (7,3.5) {$\+$};

\draw (-2,0)--(6,0);
\draw (-2,0)--(-2,6);
\draw (9,0)--(13,0);
\draw (12,2.5)--(12,3)--(8,3)--(8,4)--(7.5,4);
\draw (13,0)--(13,1)--(12,1)--(12,1.5);

\draw[thick] (9,0)--(9,1)--(7,1)--(7,2)--(6,2)--(6,4)--(5,4)--(5,6)--(4,6)--(4,11)--(3,11)--(3,11.5);
\draw[dashed] (9,0)--(9,3);

\draw (9,3)--(8,3)--(8,2)--(9,2);
\fill[pattern = north west lines] (9,3)--(8,3)--(8,2)--(9,2);

\draw (4,6)--(3,6)--(3,5)--(4,5)--(4,6);
\fill[pattern = north west lines] (4,6)--(3,6)--(3,5)--(4,5);

\node at (10,1.5) {$\+$};
\end{scope}

\node at (1.5,12.5) {$\bullet$};


\node at (4,8.5) {\scalebox{0.8}{$[\alpha/\mus]$}};
\node at (1,11) {\scalebox{0.8}{$[\tau/\mus]$}};


\draw[->] (0.5,16.5)--(0,16.5);
\draw[->] (1.35,16.5)--(2,16.5);
\node at (1,16.4) {$\scriptstyle \rM$};

\end{tikzpicture}
\end{center}
\caption{Entries in a tableau \smash{$t \in \SSYT\bigl(\muS \oplus M(\kappa^{\-},\kappa^{\+})\bigr)_{(\pi^\minus,\pi^\+)}$} when
$\muS$ is \smash{$(\ell^-+a(\muS), \ell^\+)$}-large and 
\smash{$(\ell^-, \ell(\mus))$}-large, and so contains the hatched boxes. 
The case $\ell(\mus) > \ell(\mu^\+) > \ell^\+$, in which the skew part
of the partition is most important in determining
the quantity $\max(\ell(\mus),\ell(\mu^{\scriptstyle +}),\ell^{\scriptstyle +})$ defining
the row of the $\ell^{\scriptstyle -}$-top position
is shown. The heavy lines
show the region $[\alpha/\mus]$ defined in~\eqref{eq:alpha}
that contains all negative entries of $t$ in the proof of 
Lemma~\ref{lemma:plethysticSignedTableauPositions}(i).
It contains the region $[\tau/\mus]$ shaded in blue
 which contains all entries of $\{-1,\ldots, -\rM\}$,
under the assumption in this part of the proof
that the $\rM$-top position marked $\bullet$ does not contain $-\rM$.
In the figure we have taken $\rM = 2$. 
\label{fig:ellDecompositionSkewTableau}
}
\end{figure}

For (i), we have $\rM < \ell^\-$ so may suppose $\kappa^\- \not= \varnothing$.  
If $s$ is an arbitrary $\muS \oplus M(\kappa^\-,\kappa^\+)$-inner 
tableau in $T$ then
the total number
of entries of $s$ in the set $\{-1,\ldots,-\rM\}$ is at most $\sum_{j=1}^\rM \mu'_j + M\sum_{j=1}^\rM\kappa^\-_j$.
To simplify some arithmetic steps later, we  write this quantity as
$\DM + \mu'_\rM + M\kappa^\-_\rM$ where
\begin{equation}\label{eq:DM} \DM = \sum_{j=1}^{\rM-1} \mu'_j + M \sum_{j=1}^{\rM-1} \kappa^\-_j \end{equation}
is (as we have  just seen) an upper bound on the number of entries of $s$
in $\{-1,\ldots, -(\rM-1)\}$.
By~\eqref{eq:top}
the $\rM$-top position of $s$ is $\bigl( \mu'_{\rM+1} + M\kappa^\-_{\rM+1}, \rM \bigr)$.
We assume, for a contradiction, that, in \emph{some} tableau $t$ in $T$, this position 
has either a positive entry, or some $-q$ with $-q \succ -\rM$ 
in the order in Definition~\ref{defn:semistandardSignedTableau}, meaning that $q > \rM$.
Define
a subpartition $\tau$ of $\alpha$ by
\begin{equation}
\label{eq:tau}
\tau_i = \begin{cases} \min(\mus_i + \rM, \mu_i) & \text{if $1 \le i < \mu'_{\rM+1} + M\kappa^\-_{\rM+1}$} \\
\min(\mus_i + \rM-1, \mu_i) & \text{if $ i \ge \mu'_{\rM+1} + M\kappa^\-_{\rM+1}$}.\end{cases}
\end{equation}
Thus $[\tau/\mus]$ consists of the first $\rM$ boxes in each row of $\mu/\mus \opluss M(\kappa^\-,\kappa^\+)$,
or the whole row if it has fewer than $\rM$ boxes, \emph{except} that the boxes of $\mu/\mus$
at or below the $\rM$-top position known, by our assumption, not to contain~$-\rM$, are excluded. 
By construction $\tau$ contains $\mus$ and,
by our assumption, all the entries of $t$ in $\{-1,\ldots, -\rM\}$ are contained in $[\tau]$.
(This is the blue shaded region in
Figure~\ref{fig:ellDecompositionSkewTableau}.)
%
Hence using~\eqref{eq:DM}, the total number of entries of $t$ in the set $\{-1,\ldots,-\rM\}$ is
strictly less than
$\DM + \mu'_{\rM+1} + M\kappa^\-_{\rM+1}$.
The $n-1$ tableaux other than $t$ forming $T$ have at most
$(n-1)\DM + (n-1)\bigl( \mu'_\rM + M\kappa^\-_\rM \bigr)$ entries in $\{-1,\ldots,-\rM\}$.
Therefore~$T$ has strictly less than
\begin{align} &(n-1)\DM
 + (n-1) \bigl( \mu'_\rM + M\kappa^\-_\rM \bigr) + \DM + \mu'_{\rM+1} + M\kappa^\-_{\rM+1} \nonumber \\
 &\quad\  \ \ =
 n\sum_{j=1}^\rM \mu'_j - \mu'_\rM + nM \sum_{j=1}^\rM \kappa^\-_j - M\kappa^\-_\rM + \mu'_{\rM+1} + M\kappa^\-_{\rM+1} 
\label{eq:piLessThanMinus} \end{align}
entries in the set $\{-1,\ldots,-\rM\}$.
On the other hand, as $\pi \unrhddot \lambda \,\oplus\, nM(\kappa^\-, \kappa^\+)$,
and $\lambda$ is $(\ell^\-,\ell^\+)$ large, it follows from 
Lemma~\ref{lemma:twistedDominanceOrderOldDefinition}(a) and Lemma~\ref{lemma:adjoinToLarge}
that there are at least
\begin{equation}
\label{eq:piAtLeastMinus} 
\sum_{j=1}^\rM \lambda^\-_j + nM\sum_{j=1}^\rM \kappa^\-_j \end{equation}
entries of $T$ in $\{-1,\ldots,-\rM\}$. From~\eqref{eq:piLessThanMinus} and~\eqref{eq:piAtLeastMinus} we obtain
\begin{equation}
\label{eq:piAtLeastMinusCorollary} n\sum_{j=1}^\rM \mu'_j - \mu'_\rM  + \mu'_{\rM+1} - \sum_{j=1}^\rM \lambda^\-_j > 
M(\kappa^\-_\rM - \kappa^\-_{\rM+1}). \end{equation}
This contradicts the first bound.
Therefore (i) holds for top positions.
For bottom positions we mimic the proof of Lemma~\ref{lemma:signedTableauPositions}, and  run the same argument, 
replacing each $\mu'_{\rM+1}$ with $\mu'_{\rM+1} \hskip0.5pt+\hskip1pt \kappa^\-_\rM - \kappa^\-_{\rM+1}$, and obtain~\eqref{eq:piAtLeastMinusCorollary}
with $\kappa^\-_\rM - \kappa^\-_{\rM+1}$ subtracted from the right-hand side, which therefore becomes $(M-1)(\kappa^\-_\rM - \kappa^\-_{\rM+1})$. 
We then get a contradiction as before from the first bound. 

For (ii), we may again suppose $\kappa^\- \not=\varnothing$; then by~\eqref{eq:top}
the $\ell^\-$-top position of $t$ is
 $\bigl(\max(\ell(\mus), \ell(\mu^\+), \ell^\+\bigr), \ell^\-)$.
We may suppose that either $\mus \not=\varnothing$ or 
$\kappa^\- \not= \varnothing$ or $\mu^\+ \not=\varnothing$, so that this is a box of $\muS$.  
The proof is then almost identical to (i)
using~\eqref{eq:top} to replace every appearance of $\mu'_{\rM+1}$ in the argument 
for (i) with $\max(\ell(\mus), \ell(\mu^\+), \ell^\+ )$.
This also proves (iii), since the hypotheses for (iii) imply that  $\bigl(\max(\ell(\mus), \ell(\mu^\+), \ell^\+\bigr), \ell^\-)$
is a box of~$\mu/\mu^\star$.

Comparing~\eqref{eq:top} and~\eqref{eq:left} for $\rM < \ell^\-$ and $\rP < \ell^\+$, we see
that they are symmetric with respect to conjugation. While a further non-symmetric change is necessary later on,
this indicates the most conceptual way to prove~(iv). Read the proof of (i),
replacing $\mu'$ with $\mu$ and 
$\kappa^\-$ with $\kappa^\+$. Thus the subpartition $\tau$ is now defined by 
\[
\label{eq:tauPlus}
\tau'_j = \begin{cases} \min(\mus'_j + \rP, \mu_j) & \text{if $1 \le j < \mu'_{\rP+1} + M\kappa^\+_{\rP+1}$} \\
\min(\mus'_j + \rP-1, \mu'_j) & \text{if $j \ge \mu'_{\rP+1} + M\kappa^\+_{\rP+1}$}\end{cases}
\]
and the analogue of~\eqref{eq:piLessThanMinus}
is that there are strictly less than
\begin{equation}
\label{eq:piLessThanPlus}
n\sum_{i=1}^\rP \mu_j - \mu_\rP + nM \sum_{i=1}^\rP \kappa^\+_j - M\kappa^\+_\rP 
+ \mu_{\rP+1} + M\kappa^\+_{\rP+1} 
\end{equation}
entries of $T$ in $\{1,\ldots, \rP\}$.
The only non-symmetric change  is that
from $\pi \unrhddot \lambda \oplus nM(\kappa^\-,\kappa^\+)$ and 
Lemma~\ref{lemma:twistedDominanceOrderOldDefinition}(b) we now
get 
\[ \pi^\+ + \bigl(|\lambda^\+| + nM|\kappa^\+| - |\pi^\+| \bigr) \unrhd \lambda^\+ + nM\kappa^\+ \]
where $|\pi^\+| \le |\lambda^\+| + nM|\kappa^\+|$.
We must therefore bring in the upper bound $\pi \unlhddot \omega \oplus nM(\kappa^\-,\kappa^\+)$ to get, 
again by Lemma~\ref{lemma:twistedDominanceOrderOldDefinition}(b),
$|\pi^\+| \ge |\omega^\+| + nM|\kappa^\+|$. Hence
\[|\lambda^\+| + nM|\kappa^\+| - |\pi^\+| \le \bigl(|\lambda^\+| + nM|\kappa^\+|\bigr) - \bigl(|\omega^\+|
+ nM|\kappa^\+|\bigr) = |\lambda^\+| - |\omega^\+|.\] 
The replacement for~\eqref{eq:piAtLeastMinus} is therefore 
that there are at least
\begin{equation}
\label{eq:piAtLeastPlus} \sum_{i=1}^\rP \lambda_i^\+ + nM\sum_{i=1}^\rP \kappa_i^\+ - |\lambda^\+| + |\omega^\+| 
\end{equation}
entries of $T$ in $\{1,\ldots, \rP\}$. From~\eqref{eq:piLessThanPlus} and~\eqref{eq:piAtLeastPlus} we get
\begin{equation}
\label{eq:piFinalPlus} n\sum_{i=1}^\rP \mu_i - \mu_\rP  + \mu_{\rP+1} - \sum_{i=1}^\rP \lambda^\+_i 
+ |\lambda^\+| - |\omega^\+| > 
M(\kappa^\+_\rP - \kappa^\+_{\rP+1}). \end{equation}
This contradicts the second bound. 
The modifications for right positions are precisely analogous to the negative case;
the relevant quantity to subtract from the right-hand side of~\eqref{eq:piFinalPlus} is $\kappa^\+_\rP - \kappa^\+_{\rP+1}$,
and as before we get a contradiction from the second bound.

For (v) we first note that, by~\eqref{eq:left}, the $\ell^\+$-left position of $t$ is
$\bigl( \ell^\+, \max(\ell^\- + a(\mus), \mu_{\ell^\++1})  \bigr)$.
Since $\mu$ is $\bigl(\ell^\- + a(\mus), \ell^\+\bigr)$-large,
we have $(\mu - \mus)_{\ell^\+} \ge \ell^\- + a(\mus) - \mu_{\star \ell^\+} \ge \ell^\-$.
Therefore if this position contains a negative entry, equality holds and the entry is $-\ell^\-$.
In the remaining case, and for the $\ell^\+$-right position, the proof of (iv) adapts routinely.
This completes the proof.
\end{proof}

\begin{example}\label{ex:22positions}
Take $\mu = (2,2)$, $\mus = \varnothing$, $\swtp{\kappa} = \bigl((1,1),(1)\bigr)$ and $\lambda = (8-b,b)
\decMap \dec{(2,2)}{(6-b,b-2)}$ with $b \in \{2,3,4\}$ and $n=2$
as in Example~\ref{ex:22motivation}. 
Since 
$\omega_{2}\bigl((2,2)\bigr) = \bigl((2,2),\varnothing\bigr)$
is the signed weight of the $2$-greatest semistandard signed tableau $t_{2}\bigl((2,2)\bigr)$
shown in the margin, 
\marginpar{ \raisebox{0pt}{\qquad\qquad \young(\oM\tM,\oM\tM)}}
we have 
\[ \omega_2^{(2)}\bigl((2,2)\bigr)
\decMap 2\dec{(2,2)}{\varnothing} = \dec{(4,4)}{\varnothing} \] 
and so we take $\omega = (2,2,2,2)$.
Since $\kappa^\-_1 = \kappa^\-_2$, the important cases of Lemma~\ref{lemma:plethysticSignedTableauPositions}
are (ii) and (v). 
We saw in Example~\ref{ex:22motivation}
that the $2$-top position and $1$-left position are both $(1,2)$,
and so the $2$-bottom position is $(2,2)$ and the $1$-right position
is $(1,3)$.
From (ii) we get that the $2$-bottom position $(1,2)$ contains $-2$ 
in every $(2+M,2,2^M)$-tableau entry in a plethystic semistandard signed tableau of outer shape $(2)$ 
provided 
\begin{align*} M-1 &\ge \mathrm{LP}(2, (2,2) : (2,2), (1,1) : 1) = \frac{2(2\!+\!2) \!-\! (2\!+\!2) \!-\! 2 \!+\! \max(1,0)}{1-0} = 3. \\
\intertext{(Note that (ii) was proved using only the first bound in the statement of
Lemma~\ref{lemma:plethysticSignedTableauPositions}.)
From (v), using that $|\lambda^\+| = 4$ and $|\omega^\+| = 0$,
 we get that the $1$-right position $(1,3)$ contains 
$1$ in every
$(2+2M,2,2^M)$-tableau entry in a plethystic semistandard signed tableau of outer shape $(2)$ 
provided}
M-1 & \ge \mathrm{LP}(2, (2,2) : (6-b, b-2), (1) : 2) + 4 - 0 
\\ & \hspace*{0.85in} \hspace*{-1pt} = \hspace*{-1pt}\frac{2.2 - (6-b) - 2 + \max(2,2)}{1} + 4
= \frac{b-2}{1} + 4 = b+2. \end{align*}
(Again note that (v) was proved using only the second bound in the statement of
Lemma~\ref{lemma:plethysticSignedTableauPositions}.)
We now take $\nu = (2)$. From these bounds it follows that, provided $M-1 \ge \max(3, b+2)$,
the hatched boxes in the diagram in Example~\ref{ex:22motivation}
contain $\young(\oM\tM)$ and $\young(1)$ and the
insertion map 
\[ \begin{split}& \PSSYTw{(2)}{(2+M,2,2^M)}{(2+2M,2+2M)}{(6-b+2M,b-2)} 
\\ & \rightarrow
\PSSYTw{(2)}{(2+M+1,2,2^{M+1})}{(2+2(M+1),2+2(M+1))}{(6-b+2(M+1),b-2)}
\end{split}  \]
is surjective for $M \ge  \max(3, b+2)$,
as stated in Example~\ref{ex:22motivation}.
The table below shows the number of plethystic
semistandard signed tableaux in the domain above
and the stable
values of~$\langle s_{(2)} \circ s_{(2+M,2,2^M)}, s_{(8-b+2M,b,2^{2M})} \rangle $
as in~\eqref{eq:plethysticMultiplicityOverview},
for each $0 \le M \le 3$ and $b \in \{2,3,4\}$.

\medskip
\centerline{
\begin{tabular}{cccccccccc} \toprule  $b$ & $M= 0$ & $M=1$ & $M=2$ & $M=3$ & 
\eqref{eq:plethysticMultiplicityOverview} & $b+2$ & $(\mathrm{LP}, \mathrm{LP}')$ \\ \midrule
2 & 1 & 1 & 1 & 1  & 0 & 4 & $(3,4)$ \\
3 & 4 & 5 & 5 & 5&   0 & 5 & $(3,5)$ \\
4 & 10 & 19 & 20 & 20 & 1 & 6 & $(3,6)$ \\ \bottomrule
\end{tabular}}

%
%
%
%

\medskip
\noindent 
By Proposition~\ref{prop:plethysticSignedTableauInnerStable} below,
the values are constant for the step from $M$ to $M+1$ provided
$M$ is at least each of the bounds from Lemma~\ref{lemma:plethysticSignedTableauPositions}
denoted $\mathrm{LP}$ and $\mathrm{LP}'$ in the table above; we have seen
the first is $3$ for $M \ge 2$ and the second is $b+2$.
Therefore each row is constant for $M \ge \max(3, b\,+\,2)$.
To show a case where the constant value is not yet reached, note
from the column for $M=2$ above that there are 20 plethystic semistandard signed tableaux in
 \smash{$\PSSYTw{(2)}{(4,2,2,2)}{(6,6)}{(6,2)}$}.
\marginpar{\raisebox{1pt}{\scalebox{0.8}{$ \pyoung{2.1cm}{2.1cm}{ {{\youngL{(\oM\tM11,\oM\tM,\oM\tM,11)}, \youngL{(\oM\tM22,\oM\tM,\oM\tM,11)}}} }$}}}
The plethystic semistandard signed tableau in the margin is the unique one 
having~$2$ in the $1$-right box $(1,3)$ of its rightmost inner $(4,2,2,2)$-tableau entry, and so
%
%
%
%
%
%
this plethystic semistandard signed tableau is the unique one not in the
image of the insertion map from the $19$ plethystic tableaux for $M=1$ to the $20$ 
plethystic tableaux for $M=2$.
The insertion map is then surjective at each subsequent step.
The final column in the table above is relevant to Example~\ref{ex:22sharper} below.
\end{example}

\subsection{The $\Gmap$ insertion map on plethystic tableaux}\label{subsec:stablePlethysticTableau}
We now define the plethystic extension of the insertion map $\Fmap$ in Definition~\ref{defn:F}.
Recall from after Definition~\ref{defn:plethysticSignedTableau} 
that $\PYT(\nurho, \muS)$
denotes the set of plethystic 
signed tableaux of shape $\nurho$ having entries from the
set $\YT(\muS)$ of all signed tableaux of shape $\muS$.
Given a plethystic signed tableau $T \in \PYT(\nurho, \muS)$, we define its \emph{conjugate} $T' \in \PYT(\nurho', \muS)$
by $T'(i,j) = T(j,i)$ for $(i,j) \in [\nurho]$. Note that the conjugation is defined with respect to the outer
Young diagram $[\nurho$]; it does not change the shape of the inner tableaux of $T$. 

\begin{definition}\label{defn:G}
Let $\nurho$ be a partition.
Let $\kappa^\-$ and $\kappa^\+$ be partitions. 
Let $\muS$ be an $\bigl(\ell(\kappa^\-) + a(\mus), \ell(\kappa^\+) \bigr)$-large
and $\bigl( \ell(\kappa^\-), \ell(\mus) \bigr)$-large skew partition. 
Let $\nurhod = \nurho$ if $|\kappa^\-|$ is even and let $\nurhod = \nurho'$ if $|\kappa^\-|$ is odd.
We define
\[ \Gmap : \PSSYT(\nurho, \muS) \rightarrow \PYT\bigl( \nurho^\dagger, \muS \oplus (\kappa^\-,\kappa^\+) \bigr) \]
by applying $\Fmap$ to each  inner $\muS$-tableau entry of
 $T \in \PSSYT(\nurho, \muS)$ to
obtain $U \in \PYT\bigl(\nurho, \muS \oplus (\kappa^\-,\kappa^\+) \bigr)$. If $|\kappa^\-|$ is even then we define $\Gmap(T) = U$;
if $|\kappa^\-|$ is odd then we define $\Gmap(T) = U'$.
\end{definition}

By Lemma~\ref{lemma:FisWellDefinedAndBumpsWeight}(i), using the largeness
 hypotheses on $\muS$, the map $\Gmap$ is well-defined.

The following example is relevant to the special
case of Theorem~\ref{thm:muStable}, taking $\nu = (n)$, $\mu = (m)$, 
$\swtp{\kappa} = \bigl((1), \varnothing\bigr)$,  that
\smash{$\langle s_{(n)^{(M)}} \circ s_{(m,1^M)}, s_{\lambda \sqcups (1^{nM})} \rangle$}
is ultimately constant for any partition $\lambda$ of $mn$,
where \smash{$(n)^{(M)} = (n)$} if~$M$ is even and 
\smash{$(n)^{(M)} =(1^n)$} if $M$ is odd. 
In fact, as we mentioned in the discussion
of Theorem~1.1 in \cite{deBoeckPagetWildon} in \S\ref{subsec:earlierWork}, 
provided $\ell(\lambda) \ge n$,
the plethysm coefficient is constant for all $M \in \N_0$. Correspondingly, 
one can check that $\Gmap : \PSSYTw{(n)}{(m,1^M)}{\lambda^\-}{\lambda^\+}
\rightarrow \PYT\bigl( (1^{n}), (m,1^{M+1}) \bigr)$ 
is a bijection onto \smash{$\PSSYTw{(1^n)}{(m,1^{M+1})}{\lambda^\- + (n)}{\lambda^\+}$}.

\begin{example}\label{ex:GFlip}
We take $\nurho = (1^3)$, $\mu = (4)$ and $\lambda = (8,3,1)$.
Take $\ell^\- = 1$ and note that $\lambda$ has $1$-decomposition $\dec{(3)}{(7,2)}$.
The map $\Gmap : \PSSYT\bigl((1^3), (4) \bigr) \rightarrow \PSSYT\bigl(
(3), (4,1) \bigr)$  in Definition~\ref{defn:G} for $\swtp{\kappa} = \bigl((1), \varnothing\bigr)$ 
is shown below on both elements
of $\PSSYTw{(1^3)}{(4)}{(3)}{(7,2)}$:
\begin{align*}
\raisebox{-28pt}{\pyoung{2.05cm}{0.75cm}{ {{ \youngL{(\oM111)} }, \youngL{(\oM112)}, \youngL{(\oM112)} }  } }
&\stackrel{G}{\longmapsto}
\raisebox{-14pt}{\pyoung{2.05cm}{1.09cm}{
 {{ \youngL{(\oM111,\oM)}, {\youngL{(\oM122,\oM)}},  {\youngL{(\oM112,\oM)}} }} } } \\
\raisebox{-28pt}{\pyoung{2.05cm}{0.65cm}{ {{ \youngL{(\oM111)} }, \youngL{(\oM111)}, \youngL{(\oM122)} }  } }
&\stackrel{G}{\longmapsto}
\raisebox{-14pt}{\pyoung{2.05cm}{1.09cm}{
 {{ \youngL{(\oM111,\oM)}, {\youngL{(\oM111,\oM)}},  {\youngL{(\oM122,\oM)}} }} }}.
 \end{align*}
As expected from the fact that $\Gmap$ is a bijection onto
\smash{$\PSSYT\bigl((3),(4,1)\bigr)$}, 
the image is \smash{$\PSSYTw{(3)}{(4,1)}{(6)}{(7,2)}$}.
Note that the conjugation in the outer shape is essential: in each plethystic tableau there is a repeated
inner tableau, and this is permitted because each has sign $-1$ in the plethystic
tableau of outer shape $(1^3)$ and each has sign $+1$ in the plethystic tableaux of outer shape $(3)$. 
Moreover, 
the example
\begin{align*}
\raisebox{-30pt}{\pyoung{2.05cm}{0.75cm}{ {{ \youngL{(\oM111)} }, \youngL{(\oM112)}, \youngL{(\oM\tM11)} }  } }
&\stackrel{G}{\longmapsto}
\raisebox{-15pt}{\pyoung{2.05cm}{1.09cm}{
 {{ \youngL{(\oM111,\oM)}, {\youngL{(\oM122,\oM)}},  {\youngL{(\oM\tM11,\oM)}} }} } }
\end{align*}
shows that the convention that, in the right-hand plethystic semistandard
signed tableau, negative inner tableaux 
are greater than positive inner tabl\opthyphen{}eaux, is essential to make the plethystic tableau semistandard.
%
%
%
%
%
%
%
%
%
%
\end{example}

The following two lemmas and proposition generalize Example~\ref{ex:GFlip}
and show that $\Gmap$ respects 
the semistandard condition on inner tableaux in different positions
in a plethystic semistandard signed tableau,
up to the technical order
reversal required in Lemma~\ref{lemma:FrespectsSemistandard}(iii).
The signed colexicographic order $<$ is defined in
Definition~\ref{defn:signedColexicographicOrder}.

\begin{lemma}
\label{lemma:insertionPreservesSignedColexicographicOrder}
Let $\tauS$ be a skew partition and let $s$ and $t$ be semistandard signed tableaux
of shape $\tauS$ and the same sign such that $s < t$
in the signed colexicographic order.
\begin{thmlist}
\item Let $\tilde{s}$ and $\tilde{t}$ be the signed tableaux of shape $\tau \,\sqcup\, (\rM) / \taus$
obtained from~$s$ and~$t$ by inserting a single new row with entries $-1,\ldots, -\rM$ into each. If~$\tilde{s}$ and $\tilde{t}$ are semistandard
then~$\tilde{s} < \tilde{t}$. 
\item Let $\tilde{s}$ and $\tilde{t}$ be the signed tableaux of shape $\tau + (1^\rP) / \taus$
obtained from~$s$ and~$t$ by inserting a single new column with entries $1, 2, \ldots, \rP$.
If~$\tilde{s}$ and~$\tilde{t}$ are semistandard 
then $\tilde{s} < \tilde{t}$.
\end{thmlist}
\end{lemma}

\begin{proof}
Let $m$ be the rightmost column  in which the multisets of entries in~$s$ and $t$ differ.
For (i), since a single new entry of~$-k$ is added to column~$k$ for each $k \le \rM$ in both $s$ and $t$, 
it is clear that $m$ is again the rightmost column in which the multisets of entries
in $\tilde{s}$ and $\tilde{t}$ differ. 
Since $\tilde{s}$ and $\tilde{t}$ have the same sign,
the relative order of $\tilde{s}$ and $\tilde{t}$ is determined 
by the multisets $C(s)$ and $C(t)$ of entries of $s$ and $t$ 
in column $m$.
If $m > \rM$ then $C(\tilde{s}) = C(s)$ and $C(\tilde{t}) = C(t)$ and hence the greatest entry
(taken with multiplicity) still lies in $\tilde{t}$, and hence $\tilde{s} < \tilde{t}$.
If $m \le \rM$ then $C(\tilde{s}) = C(s) \cup \{-m\}$ and $C(\tilde{t}) = C(t) \cup \{-m\}$, 
and again $\tilde{s} < \tilde{t}$.
For (ii), suppose that the new column with entries $1, 2, \ldots, \rP$ was inserted 
as column $c$,
moving the existing columns $c, c+1, \ldots $ one box to the right.
Thus if $m \ge c$ then we compare $\tilde{s}$ and $\tilde{t}$ on column $m+1$ and get $\tilde{s} < \tilde{t}$.
Otherwise $m < c$ and since
the inserted column $c$ is equal in $\tilde{s}$ and $\tilde{t}$, we still compare on column $m$ and
again get $\tilde{s} < \tilde{t}$.
\end{proof}

As mentioned earlier, the following lemma re-uses a large part of the 
proof of Proposition~\ref{prop:signedTableauStable} (Tableau Stability)
which we recap at a high level in the following remark.

\begin{remark}\label{remark:recap}
In the proofs of Proposition~\ref{prop:signedTableauStable} above
and Lemma~\ref{lemma:FrespectsSemistandard} and Proposition~\ref{prop:plethysticSignedTableauInnerStable} below
we have sets of tableaux
parametrised by $M \in \N_0$
and a map from objects for $M$ to objects for $M+1$.
There is a bound~$L$ such that the map is
well-defined provided $M \ge L$; this requires control of top- and left-positions.
The map is then clearly injective. The map is surjective onto objects for $N$ provided $N-1 \ge L$;
this requires control of right- and bottom-positions. Thus the map is bijective
from objects for $M$ to objects for $M+1$ is bijective provided $M \ge L$.
\end{remark}

Again we use
the signed colexicographic order $<$ from
Definition~\ref{defn:signedColexicographicOrder}.

\begin{lemma}\label{lemma:FrespectsSemistandard}
Let $\kappa^\-$ and $\kappa^\+$ be partitions. 
Let $\muS$ be an $\bigl(\ell(\kappa^\-) + a(\mus),$ $\ell(\kappa^\+) \bigr)$-large
and $\bigl( \ell(\kappa^\-), \ell(\mus) \bigr)$-large skew partition. 
Let~$\lambda$ and $\omega$ be 
$\bigl(\ell(\kappa^\-), \ell(\kappa^\+)\bigr)$-large partitions of $n|\muS|$.
Let 
\[ \pi \in \bigl[\lambda \oplus nM(\kappa^\-, \kappa^\+), \omega \oplus nM(\kappa^\-,\kappa^\+) \bigr]_\unlhddotS. \] 
Let $T \in \PSSYTw{\nurho}{\muS \oplus M(\kappa^\-,\kappa^\+)}{\pi^\-}{\pi^\+}$
and let~$s$ and $t$ be  $\muS \oplus M(\kappa^\-,\kappa^\+)$-inner tableaux  of~$T$ such that $s < t$.
Suppose that 
$M$ is at least the maximum of
\begin{bulletlist}
\item $\LPBound\bigl(n,\mu' : \lambda^\-, \kappa^\- : \max(\ell(\mus), \ell(\mu^\+), \ell(\kappa^\+)\bigr))$ 
\item 
$\LPBound\bigl(n, \mu : \lambda^\+, \kappa^\+ : a(\mus) + \ell(\kappa^\-)\bigr) + |\lambda^\+| - |\omega^\+|$.
\end{bulletlist}
Then 
\begin{thmlist}
\item We have $\Fmap(s)$, $\Fmap(t) \in \SSYT\bigl( \muS \oplus (M+1)\swtp{\kappa} \bigr)$.
\item If $|\kappa^\-|$ is even then $\Fmap(s)$ and $\Fmap(t)$
have the same sign and~\hbox{$\Fmap(s) < \Fmap(t)$}.
\item If $|\kappa^\-|$ is odd and $s$ and $t$ have the same sign then 
$\Fmap(s)$ and $\Fmap(t)$ have the same sign and~$\Fmap(s) < \Fmap(t)$.
Otherwise $s$ is negative, $t$ is positive, $\Fmap(s)$ is positive, $\Fmap(t)$ is negative, and 
$\Fmap(s) > \Fmap(t)$.
\end{thmlist}
\end{lemma}

\begin{proof}
We have all the hypotheses for Lemma~\ref{lemma:plethysticSignedTableauPositions}. Using the
results on left and top positions from this lemma,
the proof of~Proposition~\ref{prop:signedTableauStable} generalizes to show that 
$\Fmap(s)$ and $\Fmap(t)$ are semistandard $\muS \opluss (M+1)\swtp{\kappa}$-tableaux.
The only extra points to note are that, by Lemma~\ref{lemma:positions},
each top position is in row $\ell(\mus)$ or lower, and each left position is in column
$\ell(\kappa^\-) + a(\mus)$ or further right, and hence the inserted columns do not meet
$\mus$, \emph{and}, by (iii) in the lemma, setting $\ell^\- = \ell(\kappa^\-)$,
inserting or deleting a new row with entries $-1, \ldots, -\ell^\-$
immediately below the $\ell^\-$-top position gives a well-defined semistandard signed tableau.
(This is where we need that this position is in row $\ell(\mu^\+)$ or lower,
as remarked on in the caption to Figure~\ref{fig:ellDecompositionTableau} and seen in the proof
of the earlier lemma.) This proves (i).

Suppose that $s$ and $t$ have the same sign.
Then, by Lemma~\ref{lemma:insertionPreservesSignedColexicographicOrder}, applied to each
row and column insertion in turn, we have $\Fmap(s) < \Fmap(t)$. To prove (ii) and (iii) in the
remaining case where $s$ and $t$ have opposite sign, observe that since $s < t$, it follows
immediately from Definition~\ref{defn:signedColexicographicOrder} that $s$ is negative
and~$t$ is positive. If $|\kappa^\-|$ is even then, by Lemma~\ref{lemma:FisWellDefinedAndBumpsWeight}(ii),
$\Fmap$ is sign preserving, and so $\Fmap(s)$ is negative and $\Fmap(t)$ is positive, 
and $\Fmap(s) < \Fmap(t)$, proving (ii).
Finally for~(iii), when $|\kappa^\-|$ is odd then, again by Lemma~\ref{lemma:FisWellDefinedAndBumpsWeight}(ii),
$\Fmap$ is sign reversing, and so~$\Fmap(s)$ is positive and $\Fmap(t)$ is negative, and $\Fmap(s) > \Fmap(t)$.
\end{proof}

We are now ready to prove the analogue of Proposition~\ref{prop:signedTableauStable}
(Tableau Stability). 
Again we re-use part of its proof.
This is the point where we need the sign-reversed colexicographic order on semistandard
signed tableaux, defined in Definition~\ref{defn:signedColexicographicOrder}
and the corresponding set $\PSSYT^\mp$ of sign-reversed plethystic semistandard signed tableaux
defined in Definition~\ref{defn:plethysticSemistandardSignedTableau}
and motivated in~Example~\ref{ex:GFlip}.
Recall, as seen in this example, in $\PSSYT^\mp$, negative inner tableaux
are \emph{greater} than positive inner tableaux. 
The $\mathrm{LP}$ bounds are defined in Definition~\ref{defn:LPBound}.
The map~$\Gmap$ is defined
in Definition~\ref{defn:G}.

\begin{proposition}[Inner stability for plethystic tableaux]
\label{prop:plethysticSignedTableauInnerStable}
Let $\nurho$ be a partition
of $n \in \N$.
Let $\kappa^\-$ and  $\kappa^\+$ be partitions.  Let
$\muS$ be an $\bigl(\ell(\kappa^\-) + a(\mus), \ell(\kappa^\+) \bigr)$-large
and $\bigl( \ell(\kappa^\-), \ell(\mus) \bigr)$-large skew partition.
Let $\nurhod = \nurho$ if $|\kappa^\-|$ is even and let $\nurhod = \nurho'$ if $|\kappa^\-|$ is odd.
Let~$\omega$ be a $(\ell(\kappa^\-), \ell(\kappa^\+))$-large partition and let $\lambda \unlhddots \omega$ 
in the $\ell(\kappa^\-)$-twisted dominance order. Let $\pi$ be a partition in the 
interval
\[  \bigl[\lambda \oplus nM(\kappa^\-, \kappa^\+), \omega \oplus nM(\kappa^\-,\kappa^\+) \bigr]_{\unlhddotS} \]
for the $\ell(\kappa^\-)$-twisted dominance order. If $M$ is at least
\begin{bulletlist}
\item $\LPBound\bigl(n,\mu' : \lambda^\-, \kappa^\- : \max(\ell(\mus), \ell(\mu^\+),
\ell(\kappa^\+)\bigr))$ 
\item  
$\LPBound\bigl(n, \mu : \lambda^\+, \kappa^\+ : a(\mus) + \ell(\kappa^\-)\bigr) + |\lambda^\+| - |\omega^\+|$.
\end{bulletlist}
then the map $\Gmap$ is a well-defined bijection
\[ \begin{split} &\Gmap {}:{}  \PSSYTw{\nurho}{\muS \oplus M\swtp{\kappa}}{\pi^\-}{\pi^\+}
\\[-3pt] & \longrightarrow \begin{cases}
\PSSYT\bigl(\nurhod, \muS \!\oplus\! (M\!+\!1)\swtp{\kappa}\bigr)_{(\pi^\- + n\kappa^\-,\pi^\+ +n\kappa^\+)}
& \text{\!\!\!\!if $|\kappa^\-|$ is even} \\
\PSSYT^\mp\bigl(\nurhod, \muS \!\oplus\! (M\!+\!1)\swtp{\kappa}\bigr)_{(\pi^\- + n\kappa^\-,\pi^\+ +n\kappa^\+)}
& \text{\!\!\!\!if $|\kappa^\-|$ is odd.} \end{cases}\\
 \end{split} \]
\end{proposition}

\begin{proof}
Note that we have all the hypotheses for Lemma~\ref{lemma:FrespectsSemistandard}.
Let $T \in \PSSYTw{\nurho}{\muS \opluss M\swtp{\kappa}}{\pi^\-}{\pi^\+}$.
By this lemma, after applying $\Fmap$ to each inner $\muS \oplus M\swtp{\kappa}$-tableau in $T$
we have a plethystic signed tableau $U$ of shape $\nurho$ having well-defined  entries
from $\SSYT^\pm\bigl(\muS \oplus (M+1)\swtp{\kappa}\bigr)$. 
By Lemma~\ref{lemma:FrespectsSemistandard}(i),
$U$ has signed weight $(\pi^\- + n\kappa^\-, \pi^\+ + n\kappa^\+)$, as required.

Suppose that $|\kappa^\-|$ is even. Then $\Gmap(T) = U$.
Since, by Lemma~\ref{lemma:FrespectsSemistandard}(ii),
$\Fmap$ preserves strict equality in the signed colexicographic order, we have $U \in \PSSYTw{\nurho}{\muS
\oplus M(\kappa^\-,\kappa^\+)}{\pi^\-+n\kappa^\-}{\pi^\+ +n\kappa^\+}$, as required.

Suppose that $|\kappa^\-|$ is odd. Then~$\Gmap(T) = U'$.
Since, by Lemma~\ref{lemma:FrespectsSemistandard}(iii), $\Fmap$ preserves
strict inequality in the signed colexicographic order
on inner tabl\opthyphen{}eaux of the same sign, the conjugate plethystic
 signed tableau $U'$ is semistandard
with respect to inner $\muS \oplus (M+1)\swtp{\kappa}$-tableaux \emph{of the same sign}. Moreover,
since in $T$, equal positive inner tableaux are repeated only in the same
row, and equal negative inner tableaux are repeated only in the same column,
the same holds in $U$, swapping `row' and `column'.
Let $\beta$ be the subpartition of $\nurho$ such that the negative
$\muS \oplus M(\kappa^\-,\kappa^\+)$-tableau entries in $T$ lie in $[\beta]$.
In $U$, since $\Fmap$ is sign reversing, the positive inner $\muS \oplus (M+1)(\kappa^\-,\kappa^\+)$-tableaux are in the boxes in $[\beta]$ and the negative inner $\muS \oplus (M+1)(\kappa^\-,\kappa^\+)$-tableaux
are in the boxes in $[\nurho/\beta]$.
Therefore the conjugate plethystic signed tableau $U'$
is semistandard with respect to the sign-reversed colexicographic order;
that is $U' \in \PSSYT^\mp\bigl(\nurhod, \muS \!\oplus\! (M\!+\!1)\swtp{\kappa}\bigr)_{(\pi^\- + n\kappa^\-,\pi^\+ +n\kappa^\+)}$
as required.

We have now shown that the image $\Gmap(T)$ is in the set specified in the proposition.
Let $t$
be an inner $\muS \oplus (M+1)(\kappa^\-,\kappa^\+)$-tableau entry of~$\Gmap(T)$.
Using the results on right and bottom positions from Lemma~\ref{lemma:plethysticSignedTableauPositions},
and noting that the removed rows are strictly below row $\ell(\mus)$ and the removed columns are strictly to the
right of column $\ell(\kappa^\-) + a(\mus)$, it follows as in the proof of Proposition~\ref{prop:signedTableauStable}
that $t$ is in the image of $\Fmap$. Hence the map $\Gmap$ is surjective.
\end{proof}

We remark that the bounds in Proposition~\ref{prop:plethysticSignedTableauInnerStable}
are typically not optimal.

\begin{example}\label{ex:22sharper}
We continue Example~\ref{ex:22positions}
to show how the general bounds coming ultimately 
from Lemma~\ref{lemma:plethysticSignedTableauPositions}
can be sharpened by considering 
negative and positive entries together. 
In this example we got the bound $M \ge \max(3, b+2)$.
From the diagram in Example~\ref{ex:22motivation}
we see that the entries of $1$ in the two 
 semistandard tableaux $s$ and $t$ forming 
\[ T \in \PSSYT\bigl((2),(2+M,2,2^M)\bigr)_{((2+2M,2+2M),(6-b+2M,b-2))} \]
lie either in the bottom two rows of their first two columns, or in
the $M$ boxes ending their first rows.
At most two $1$s can be in the first two columns of each.
Therefore each of~$s$ and $t$ has at most $2+M$ entries of $1$,
and if the
$1$-right position $(1,3)$ in either~$s$ or~$t$ does not contain~$1$ then
the total number of entries of $1$ in $T$ is at most $(2+M) + 2 = 4+M$. Therefore
$4+N \ge 6-b+2N$ and we deduce that $N \le b-2$. Hence if $M \ge b-1$ 
then the $1$-right position $(1,3)$ in both tableaux $s$ and $t$ contains $1$.
Therefore, provided $M \ge b-1$, and $M \ge 3$ (as we needed
from Lemma~\ref{lemma:plethysticSignedTableauPositions} to ensure 
that the $2$-bottom position $(1,2)$ contains $-2$),
insertion of $\young(1)$ and $\young(\oM\tM)$ defines a surjective map
\[ \begin{split} &\scalebox{0.9}{$\PSSYT\bigl((2),(2+M-1,2,2^{M-1})\bigr)_{((2+2(M-1),2+2(M-1)),
(6-b+2(M-1),b-2))}$}
\\ &\quad\qquad\quad\qquad\longrightarrow \scalebox{0.9}{$\PSSYT\bigl((2),(2+M,2,2^M)\bigr)_{((2+2M,2+2M),
(6-b+2M,b-2))}$}.\end{split} \]
This gives the improved bound $\max(3, b-1)$.
Note that the bound from Proposition~\ref{prop:signedTableauStable} is $M \ge 2$, so this
bound still holds; this bound is relevant in the proof of Theorem~\ref{thm:muStable} following.
\end{example}

\subsection{Proof of Theorem~\ref{thm:muStable}}\label{subsec:muStableProof}
We prove Theorem~\ref{thm:muStable} with an explicit stability bound.
By Remarks~\ref{remark:becomesLarge} and~\ref{remark:becomesWeightLarge} there is no loss
of generality in the `largeness' hypotheses in the theorem. The $\mathrm{L}$ and $\mathrm{LP}$ bounds
are defined in Definitions~\ref{defn:LBound} and~\ref{defn:LPBound}, respectively.
(Remark~\ref{remark:intervalNotation} explains the small difference in notation for the intervals
in the first two bounds.)
As long promised, we use the Signed Weight Lemma (Lemma~\ref{lemma:SWL}) for the main part of the proof.
An example of the six bounds, proving~\eqref{eq:dBPW} in \S\ref{subsec:earlierWork},
is given after the proof.

\begin{theorem}[Signed inner stability with bound]\label{thm:muStableSharp}
Let $\nu$ be a partition of $n \in \N$. Let $\kappa^\-$ and $\kappa^\+$ be partitions.
Fix $\ell^\- = \ell(\kappa^\-)$ and $\ell^\+ = \ell(\kappa^\+)$.
Let $\muS$ be an $\bigl(\ell^\- +  a(\mus), \ell^\+ \bigr)$-large
and $\bigl( \ell^\-, \ell(\mus) \bigr)$-large  skew partition.
If $\ell^\- \not=0$ then suppose also that
$\muS$ is $\bigl(\ell^\- + 1, \ell^\+ \bigr)$-weight 
large for $\ell(\kappa^\-)$.
 Let $\omega$ be the partition \smash{$\omega_{\ell^\-}^{(n)}(\muS)$} of $n|\muS|$ defined
in Definition~\ref{defn:plethysticGreatestSignedWeight}.
Let~$\lambda$ be an~$(\ell^\-,\ell^\+)$-large partition. 
Let~$L$ be the maximum of
\begin{bulletlist}
\item $\LBound\bigl([\lambda^\-,\omega^\-]^\ellmb_\unLHDS, \kappa^\-\bigr)/n$,
\item $\LBound\bigl([\lambda^\+,\omega^\+ + (|\lambda^\+| - |\omega^\+|)]_\unlhd, \kappa^\+ \bigr)/n$
\item $\bigl( \omega_1^\+ + \omega_2^\+ - 2\lambda_1^\+ + 2|\lambda^\+| - 2|\omega^\+| 
\bigr)/n(\kappa^\+_1-\kappa^\+_2)$,
\item \smash{$\bigl( \max( \ell(\lambda^\+), \ell^\+ ) + |\omega^\-| - |\lambda^\-| - \omega^\-_{\ell^\-} \bigr)/
 n\kappa^\-_{\ell(\kappa^\-)}$}.
\item $\LPBound\bigl(n,\mu' : \lambda^\-, \kappa^\- : \max(\ell(\mus), \ell(\mu^\+), \ell(\kappa^\+)\bigr)$ 
\item 
$\LPBound\bigl(n, \mu : \lambda^\+, \kappa^\+ : a(\mus) + \ell(\kappa^\-)\bigr) 
+ |\lambda^\+| - |\omega^\+|$
\end{bulletlist} 
omitting the third if $\kappa^\+_1 = \kappa^\+_2$ and the fourth if $\kappa^\- = \varnothing$. Then
\[ \bigl\langle s_\nuSeq{M} \circ s_{\muS \opluss M(\kappa^\-, \kappa^\+)}, 
s_{\lambda \opluss nM(\kappa^\-, \kappa^\+)} \bigr\rangle \]
is constant for $M \ge L$, where if $|\kappa^\-|$ is even then $\nuSeq{M} = \nu$ for all $M$
and if $|\kappa^\-|$ is odd then $\nuSeq{M} = \nu$ if $M$ is even and $\nuSeq{M} = \nu'$ if $M$ is odd.
Moreover if $\lambda \notunlhddot \omega$ in the $\ell^\-$-twisted dominance order 
then the plethysm coefficient is $0$ for all $M \in \N_0$.
\end{theorem}

\begin{proof} We apply the Signed Weight Lemma (Lemma~\ref{lemma:SWL}).
For $M \in \N_0$ set 
\[ \PSeq{M} = \bigl[\lambda \oplus M(\kappa^\-, \kappa^\+), 
\omega_{\ell^\-}^{(n)}(\muS) \oplus M(\kappa^\-, \kappa^\+)\bigr]_\unlhddotS. \]
If $\lambda \notunlhddot \omega_{\ell^\-}^{(n)}(\muS)$ then 
Lemma~\ref{lemma:adjoinToLarge}
implies that
\[\lambda \oplus M\swtp{\kappa} \notunlhddot \omega_{\ell^\-}^{(n)}(\muS) \oplus M\swtp{\kappa} \] 
for all $M \in \N_0$
and hence, 
by  Proposition~\ref{prop:twistedWeightBoundInner},
$ \langle s_\nuSeq{M} \circ s_{\muS \opluss M(\kappa^\-, \kappa^\+)},$
$s_{\lambda \opluss nM(\kappa^\-, \kappa^\+)} \rangle = 0$ for all $M \in \N_0$. Thus 
all the plethysm coefficients are zero, as claimed in the final part of the statement.

We may therefore assume that \smash{$\lambda \unlhddot \omega_{\ell^\-}^{(n)}(\muS)$}.
If $\ell^- \not =0$ then
by the hypothesis that $\muS$ is $(\ell^\-+1,\ell^\+)$-weight large,
the partition \smash{$\omega_{\ell^\-}^{(n)}(\muS)$} is $(\ell^\-+1,\ell^\+)$-large, which is
the other hypothesis required in
 Corollary~\ref{cor:signedIntervalStable}. 
Hence, by this corollary,
$(\PSeq{M})_{M \in \N_0}$ is a stable partition system
with respect to the map $\pi \mapsto \pi \oplus \swtp{\kappa}$ and the twisted symmetric
functions $g_\pi = e_{\pi^\-}h_{\pi^\+}$.
We take the subsystem $(\QSeq{M})_{M \in \N_0}$ where
$\QSeq{M} = \PSeq{nM}$. Up to the factor $1/n$, the first four bounds in our hypotheses
are those in Corollary~\ref{cor:signedIntervalStable}. Therefore $\QSeq{M}$ is stable for
$M \ge L$.

We are now ready to verify the conditions in the Signed Weight Lemma (Lemma~\ref{lemma:SWL})
taking $\nu^{(M)}$ as already defined, $\muSSeq{M} = \muS \oplus M(\kappa^\-,\kappa^\+)$,
and $\QSeq{M}$ as our stable partition system.
Since $\lambda \oplus nM \swtp{\kappa} \in \QSeq{M}$,
this implies the theorem.

\subsubsection*{Condition \emph{(i)} in the Signed Weight Lemma}
By Lemma~\ref{lemma:stablePartitionSystemForMuVarying}
the stable partition system $\QSeq{M}$ satisfies condition (i) of the Signed Weight Lemma (Lemma~\ref{lemma:SWL})
for the plethysms $s_{\nuSeq{M}} \circ s_{\muS \opluss M(\kappa^\-, \kappa^\+)}$.

\subsubsection*{Condition \emph{(ii)} in the Signed Weight Lemma}
We must verify that
\begin{equation}
\label{eq:muPSSYTStable} \begin{split}&\bigl|\PSSYTw{\nuSeq{M}}{\muSSeq{M}}{\pi^\-}{\pi^\+}\bigr|\\
& \qquad\qquad = \bigl|\PSSYTw{\nuSeq{M+1}}{\muSSeq{M+1}}{\pi^\-+n\kappa^\-}{\pi^\+ + n\kappa^\+}\bigr|
\end{split} \end{equation}
for all $M \in \N_0$.
By hypothesis $\muS$ is $\bigl(\ell(\kappa^\-) + a(\mus), \ell(\kappa^\+) \bigr)$-large
and $\bigl( \ell(\kappa^\-), \ell(\mus) \bigr)$-large as required in Proposition~\ref{prop:plethysticSignedTableauInnerStable}.
The final two bounds on $M$ in the statement
are those required by Proposition~\ref{prop:plethysticSignedTableauInnerStable}.
Fix $M \in \N$ at least these bounds and let $\nurho = \nu^{(M)}$.
By this proposition
we have
\[ \begin{split} &\bigl|\PSSYTw{\nurho}{\muS \opluss M(\kappa^\-,\kappa^\+}{\pi^\-}{\pi^\+} \bigr| \\ & = 
\begin{cases}
   \bigl|\PSSYTw{\nurhod}{\muS \opluss (M\!+\!1)(\kappa^\-,\kappa^\+)}{\pi^\-+n\kappa^\-}{\pi^\++n\kappa^\+} \bigr| & \text{\!\!if $|\kappa^\-|$ is even} \\
   \bigl|\srPSSYTw{\nurhod}{\muS \opluss (M\!+\!1)(\kappa^\-,\kappa^\+)}{\pi^\-+n\kappa^\-}{\pi^\++n\kappa^\+} \bigr| & \text{\!\!if $|\kappa^\-|$ is odd} \end{cases}
    \end{split}
\]
for all $\pi \in \QSeq{M}$. 
If $|\kappa^\-|$ is even then $\nuSeq{M+1} = \nurhod = \nurho = \nuSeq{M}$ and we have~\eqref{eq:muPSSYTStable}.
Otherwise we
use the final part of Lemma~\ref{lemma:negativePositivePlethysticSpecialisation} 
to obtain
\[ \begin{split} & \bigl|\srPSSYTw{\nurho'}{\muS \oplus (M+1)(\kappa^\-,\kappa^\+}{\pi^\-+n\kappa^\-}{\pi^\++n\kappa^\+}\bigr|
\\ &\qquad = \bigl|\PSSYTw{\nurho'}{\muS \oplus (M+1)(\kappa^\-,\kappa^\+}{\pi^\-+n\kappa^\-}{\pi^\++n\kappa^\+}\bigr|, 
\end{split}\]
and since $\nurho'= \nuSeq{M+1}$ and $\nurho = \nuSeq{M}$, we again get~\eqref{eq:muPSSYTStable}.
Therefore the stable partition system $\QSeq{M}$ satisfies condition (ii) 
of the Signed Weight Lemma.
\end{proof}

\begin{example}\label{ex:dBPW}
We use Theorem~\ref{thm:muStableSharp}
to show that the plethysm coefficients 
\smash{$ \langle s_\nud \circ s_{\mu' \sqcups (1^M) + (1^{\ell(\mu')})},
s_{\lambda' \sqcups (1^{nM}) + (1^{n\ell(\mu')})} \rangle$}
in~\eqref{eq:dBPW} relevant to \cite[Theorem~1.1]{deBoeckPagetWildon}
are constant for all $M \ge 0$.
 (All we need concerning $\nud$ is that
it is a partition of $n$.)
Take $\kappa^\- = (1)$ and $\kappa^\+ = \varnothing$
in Theorem~\ref{thm:muStableSharp}
and  replace $\nu$ in the theorem with $\nud$
and $\mu$ in the theorem with $\mu' + (1^{\ell(\mu')})$.
Since $\ell^\- = 1$, the negative part of a partition $\alpha$
is simply $(\ell(\alpha))$ and the $1$-decomposition is
\begin{equation}\label{eq:oneDecomposition} \dec{(\ell(\alpha))}{(\alpha_1-1,\ldots, \alpha_k-1)}
\end{equation}
where $k$ is greatest such that $\alpha_k \ge 2$.
 Therefore  the negative
part of $\lambda' + (1^{n\ell(\mu')})$ is
$\max(\ell(\lambda'), n\ell(\mu'))$. Denote this quantity $P$.
Since the 
greatest signed tableau $t_1(\mu' + (1^{\ell(\mu')}))$ of shape~$\mu' + (1^{\ell(\mu')})$ defined
in Definition~\ref{defn:greatestSignedTableau}  has $\ell(\mu')$ entries of $-1$,
we have $\omega^\- = (n\ell(\mu'))$. Therefore 
the first 
bound in Theorem~\ref{thm:muStableSharp}
is 
\[ \L\bigl( [(P), n\ell(\mu')]^{(1)}_\unLHDS, (1)\bigr)
= \frac{ n\ell(\mu') - P - n\ell(\mu')}{n}  = -\frac{P}{n} < 0. \]
(Note that the exceptional case in Definition~\ref{defn:LBound} applies, giving
a smaller bound than that obtained by using $L_1$.)
Since $\kappa^\+ = \varnothing$, the second and third bounds vanish. 
Now observe that, by~\eqref{eq:oneDecomposition},
if $\ell(\lambda') < n\ell(\mu')$ then
the positive part of $\lambda' + (1^{n\ell(\mu')})$ 
is $\lambda'$, while if $\ell(\lambda') \ge n\ell(\mu')$ 
then the positive part has length at most $\ell(\lambda')$.
Therefore $\ell(\lambda^\+) \le \ell(\lambda')$ and the fourth bound
  is at most
\[ \frac{\max( \ell(\lambda'), 0) + n\ell(\mu') - P- n\ell(\mu')}{n}
\le \frac{\ell(\lambda') - P}{n} 
\le 0. \]
Since the positive part of $\mu' + (1^{\ell(\mu')})$ is $\mu'$,
the fifth bound is
\[\begin{split} &\mathrm{LP}\bigl(n, (\mu' + (1^{\ell(\mu')}))' : (P), (1) : \max(0, 
\ell(\mu'),0 ) \bigr) \\
&\qquad\qquad =
\frac{n\ell(\mu') - P - \ell(\mu') + \max(\ell(\mu'),\ell(\mu'),0)}{1-0} 
\le 0. \end{split} \]
Since $\kappa^\+ = \varnothing$, the sixth bound vanishes.
Hence, as claimed earlier in \S\ref{subsec:earlierWork},
the plethysm coefficient is constant for all $M \in \N_0$.
Finally, note that a semistandard signed tableau of shape $\mu' + (1^{\ell(\mu')} )$ can have
at most $\ell(\mu')$ entries of $-1$. 
Therefore if $\ell(\lambda') > n \ell(\mu')$, and so $P > n \ell(\mu')$,
 the
set $\PSSYTw{\nud}{\mu + (1^{\ell(\mu')})}{(\ell(\lambda'))}{\pi^\+}$
is empty for any partition $\pi^\+$. Correspondingly, since 
$\ell(\lambda') > n \ell(\mu')$ implies that
$\lambda \notunlhddot \omega$ in the $1$-twisted dominance order,
 it follows from the final part of
Theorem~\ref{thm:muStableSharp} that
the plethysm coefficient vanishes when $M=0$. Since this is its constant value,
it vanishes for all $M \in \N_0$.
\end{example}

We end this section with a generalization of the final part of the example above.
Observe that when $\mus = \varnothing$, the greatest signed weight $\omega_{\ell^\-}(\mu)$
is simply the $\ell^\-$-decomposition of $\mu$
and so it is immediate
from Definition~\ref{defn:plethysticGreatestSignedWeight} that the partition
\smash{$\omega^{(n)}_{\ell^\-}(\mu)$} has $\ell^\-$-decomposition $n\langle \mu^\-, \mu^\+ \rangle$.
Hence, by the definition of the $\ell$-twisted dominance order 
in Definition~\ref{defn:ellTwistedDominanceOrder}
we have $\lambda \unlhddot \omega^{(n)}_{\ell^\-}(\mu)$ if and only if
$(\lambda^\-,\lambda^\+) \unlhd n(\mu^\-, \mu^\+)$, where $\unlhd$ is the 
$\ell$-signed dominance order of Definition~\ref{defn:ellSignedDominanceOrder}.
Therefore, the final part of Theorem~\ref{thm:muStableSharp}
implies that, unless 
$(\lambda^\-, \lambda^\+) \unlhd n(\mu^\-, \mu^\+)$, 
\[ \langle s_\nu \circ s_{\mu \,\oplus\, M(\kappa^\-,\kappa^\+)}, s_{\lambda  \,\oplus\, nM(\kappa^\-,\kappa^\+)}
\rangle = 0 \]
 for all $M \in \N_0$. This justifies the remark after Theorem~\ref{thm:muStable} in the introduction.

\section{The positive case of Theorem~\ref{thm:muStable}}\label{sec:muStableSharpPositive}
In this section we state the case of Theorem~\ref{thm:muStableSharp} when 
$\kappa^\- = \varnothing$, and then the still more special case where $\mus = \varnothing$;
as we saw in the survey in \S\ref{subsec:earlierWork}, 
this special case implies many of the
stability results on plethysm coefficients in the literature.
Moreover, as we 
mentioned in Remark~\ref{remark:asymmetry},
by applying the $\omega$ involution, these special cases easily
imply the analogous special cases where $\kappa^\+ = \varnothing$.
The $\mathrm{L}$ and $\mathrm{LP}$ bounds
are defined in Definitions~\ref{defn:LBound} and~\ref{defn:LPBound}, respectively.
By  Definition~\ref{defn:signedWeightTableau},
$\wt(t)$ is the positive part of the 
signed weight of a tableau having only positive integer
entries. 

\begin{corollary}\label{cor:muStableSharpPositive}
Let $\nu$ be a partition of $n \in \N$. Let $\kappa$ be a partition
and let $\mu/\mus$ be a $\bigl( a(\mus), \ell(\kappa) \bigr)$-large 
skew partition.
Let~$\lambda$ be a partition of $n |\mu/\mus|$ with
$\ell(\lambda) \ge \ell(\kappa)$.
Let $t$ be the semistandard
tableau of shape $\mu/\mus$ having 
$1$, $2$, \ldots, $\mu_j' - \mu_{\star j}'$ as its entries
in column~$j$, for each $j \le a(\mu)$.
Set $\omega = n\wt(t)$. 
Then
$\langle s_\nu \circ s_{\mus + M\kappa}, s_{\lambda + nM\kappa} \rangle$ 
is constant for 
\[ M \ge \max \Bigl( \frac{\LBound\bigl( [\lambda, \omega]_\unlhd, \kappa\bigr)}{n},
\LPBound\bigl(n, \mu : \lambda, \kappa : a(\mus) \bigr) \Bigr) \]
If $\lambda \,\notunlhd\, \omega$ then the plethysm coefficient is $0$ for all $M \in \N_0$.
Moreover if $\eta \,\notunlhd \kappa$ then
$\langle s_\nu \circ s_{\muS + M\kappa}, s_{\lambda + nM\eta} \rangle = 0$ 
for all $M > L$ where $L$ is the minimum of
\[ \frac{\sum_{i=1}^k \omega_i - \sum_{i=1}^k \lambda_i}{n \bigl( \sum_{i=1}^k
\eta_i - \sum_{i=1}^k \kappa_i \bigr)} \]
taken over those $k$ such that the denominator is strictly positive.
\end{corollary}

\begin{proof}
The final part is immediate from Proposition~\ref{prop:muStableZero} applied with 
$\kappa^\- = \eta^\- = \varnothing$ and $\kappa^\+ = \kappa$, $\eta^\+ = \eta$.
For the main part
we apply Theorem~\ref{thm:muStableSharp} with $\kappa^\- = \varnothing$ and $\kappa^\+ = \kappa$.
Note that \smash{$n \wt(t) = \omega^{(n)}_0(\muS)$} is the partition~$\omega$ in Theorem~\ref{thm:muStableSharp}.
Since $\ell^- = 0$, the only largeness condition
in the theorem that has any force is that $\muS$ is $(a(\mus), \ell(\kappa))$-large.
Of the six bounds, the first is $0$,
the second simplifies to $\LBound([\lambda,\omega]_\unlhd, \kappa)$,
the third to $(\omega_1+\omega_2 - 2\lambda_1)/n(\kappa_1-\kappa_2)$ which
is either one of the bounds contributing to $\LBound([\lambda,\omega]_\unlhd, \kappa)$,
or ignored because $\kappa_1 = \kappa_2$, the fourth and fifth bounds are $0$ 
and the sixth simplifies to $\LPBound\bigl(n, \mu : \lambda, \kappa : a(\mus) \bigr)$.
\end{proof}

\begin{corollary}\label{cor:muStableSharpPositiveNonSkew}
 Let $\nu$ be a partition of $n \in \N$. Let $\kappa$ be a partition and let $\mu$
be a partition of $m$ and let $\lambda$ be a partition of $mn$.
Then
$\langle s_\nu \circ s_{\mu + M\kappa}, s_{\lambda + nM\kappa} \rangle$ 
is constant for $M$ at least the maximum of
\[ \frac{n\sum_{i=1}^r \mu_i - \sum_{i=1}^r \lambda_i - \mu_r + \mu_{r+1}}{\kappa_r-\kappa_{r+1}} \]
for $1 \le r \le \ell(\kappa)$, where any terms with zero denominator are ignored.
If $\lambda \notunlhd n\mu$ 
then the plethysm coefficient is $0$ for all $M \in \N_0$.
Moreover if $\eta \,\notunlhd \kappa$ then
$\langle s_\nu \circ s_{\mu + M\kappa}, s_{\lambda + nM\eta} \rangle = 0$ 
for all $M > L$ where $L$ is the minimum of
\[ \frac{n\sum_{i=1}^k \mu_i - \sum_{i=1}^k \lambda_i}{n \bigl( \sum_{i=1}^k
\eta_i - \sum_{i=1}^k \kappa_i \bigr)} \]
taken over those $k$ such that the denominator is strictly positive.

\end{corollary}

\begin{proof}
We apply Corollary~\ref{cor:muStableSharpPositive}. 
By Lemma~\ref{lemma:skewLargeImpliesWeightLargeImpliesLarge}, 
since $\mus = \varnothing$, the only largeness condition in this corollary holds
trivially.
Note that, again since $\mus = \varnothing$,
we have $\wt(t) = \mu$ and so the partition $\omega$ in the corollary is simply $n\mu$,
and so by part of the corollary,
if $\lambda \notunlhd n\mu$ then
 the plethysm coefficients are zero.
If $n=1$ then the plethysm coefficient is $\langle s_{\mu + M\kappa}, s_{\lambda + M\kappa}\rangle$, 
which is obviously constant.
When $n \ge 2$, the bound $M \ge \LPBound\bigl(n, \mu : \lambda, \kappa : 0 \bigr)$ 
(see Definition~\ref{defn:LPBound}) is equivalent to
\[ M(\kappa_k-\kappa_{k+1}) \ge n\sum_{i=1}^k \mu_i - \sum_{i=1}^k \lambda_i - \mu_k + \mu_{k+1} \]
for each $1 \le k \le \ell(\kappa)$. Using that $\omega = n\mu$, the bound we require, namely
that $M \ge \mathrm{L}([\lambda, \omega]_\unlhd, \kappa)/n$ is equivalent to
\[ M(\kappa_k-\kappa_{k+1}) \ge \frac{1}{n}\Bigl(
2\sum_{i=1}^{k-1} n\mu_i + n\mu_k +  n\mu_{k+1} - 2\sum_{i=1}^k \lambda_i \Bigr) \]
again for each $1 \le k \le \ell(\kappa)$. Fixing $k$,
the difference of the two right-hand sides is
\[ (n-2)\sum_{i=1}^{k-1}\mu_i + (n-2)\mu_k - \frac{n-2}{n} \sum_{i=1}^k \lambda_i 
= \frac{n-2}{n} \sum_{i=1}^k (n \mu_i - \lambda_i ) 
\]
which is non-negative because $n\mu \,\unrhd\, \lambda$. 
%
Therefore the hypotheses of this corollary imply that
$M \ge \mathrm{L}([\lambda, \omega]_\unlhd, \kappa)/n$, as required to apply
Corollary~\ref{cor:muStableSharpPositive}. The final claim of this corollary
is immediate from Corollary~\ref{cor:muStableSharpPositive}.
\end{proof}

We remark that if $\kappa = (1^R)$ then by one further specialization we obtain
that $\langle s_\nu \circ s_{\mu + (M^R)}, s_{\lambda + n(M^R)} \rangle$ is constant
for $M \ge n \sum_{i=1}^R \mu_i - \sum_{i=1}^R \lambda_i - \mu_R + \mu_{R+1}$.
This recovers Theorem~1.2 in \cite{deBoeckPagetWildon}.

\addtocontents{toc}{\smallskip}
\addtocontents{toc}{\textbf{Theorem~\ref{thm:nuStable}: outer stability}}

\section{Twisted weight bound for Theorem~\ref{thm:nuStable}}
\label{sec:twistedWeightBoundForStronglyMaximalWeight}
This section is the analogue of \S\ref{sec:twistedWeightBoundInner},
culminating in Corollary~\ref{cor:signedWeightBoundForStronglyMaximalSignedWeight}, 
the analogue of Corollary~\ref{cor:twistedWeightBoundInnerGrowing}, giving
an upper bound (in a sense made precise in the corollary) on the
constituents $s_\sigma$ of the plethysm $s_{\nu^{(M)}} \circ s_\muS$
such that $\sigma \unrhddot \lambda \oplus M\swtp{\kappa}$ in the $\ell(\kappa^\-)$-twisted
dominance order. 
Here, 
as in  Theorem~\ref{thm:nuStable},
$\nu^{(M)} = \nu + (M^R)$ if the strongly maximal signed
weight $\swtp{\kappa}$ has sign $+1$
and \smash{$\nu^{(M)} = \nu \sqcups (R^M)$} if the strongly maximal signed
weight $\swtp{\kappa}$ has sign $-1$.
We outline the strategy of the proof in \S\ref{subsec:overviewSignedWeightBoundOuter}
after the essential preliminaries in the following subsection.

\subsection{Adapted signed colexicographic order for a strongly maximal signed weight}
By Lemma~\ref{lemma:stronglyMaximalSemistandardSignedTableauFamilyIsUnique},
if $\swtp{\kappa}$ is a strongly maximal signed weight 
then there is a unique
semistandard signed tableau family of signed weight $\swtp{\kappa}$
of the same shape, size and type as $\swtp{\kappa}$.
By Definition~\ref{defn:stronglyMaximalSignedWeight}
the sign of $\swtp{\kappa}$ is the common sign of the tableaux in this family;
in the sense of Definition~\ref{defn:semistandardSignedTableauFamily},
the family has row-type if the common sign is $-1$ and column-type if the common sign is $+1$.
Recall that plethystic semistandard signed tableaux were
defined in Definition~\ref{defn:plethysticSemistandardSignedTableau};
in this definition the
inner tableaux are ordered by the signed colexicographic order (see
Definition~\ref{defn:signedColexicographicOrder}).

\begin{notation}\label{notation:stronglyMaximalTableauFromWeight}
Let  $\swtp{\kappa}$ 
be a strongly maximal signed weight of sign~$\epsilon$.
Let $\SMw{\kappa}$ be the unique semistandard signed tableau family of signed weight $\swtp{\kappa}$.
Let $\TMw{\kappa}$ be the unique
plethystic semistandard tableau of outer shape $\rho$ and inner shape
$\muS$ having as its entries the elements of $\SMw{\kappa}$,
where $\rho = (1^R)$ if $\epsilon = +1$ and $\rho = (R)$ if $\epsilon = -1$.
\end{notation}

As we mentioned in Remark~\ref{remark:otherOrder}, in this section we use the 
freedom in Lemma~\ref{lemma:negativePositivePlethysticSpecialisation}
to define plethystic semistandard signed tableaux using an order on their inner tableaux
adapted to the relevant strongly maximal semistandard signed tableau family.

\begin{definition}[Adapted colexicographic order]\label{defn:adaptedSignedColexicographicOrder}
Let $\swtp{\kappa}$ be a strongly maximal signed weight of sign $\epsilon$.
Let $\le$ be the signed colexicographic order if $\epsilon = -1$
and the sign-reversed colexicographic order if $\epsilon = +1$.
The $\swtp{\kappa}$-\emph{adapted colexicographic order}, denoted $\le_\kappa$,
is the total order on semistandard signed tableaux of shape $\muS$
defined by $s \le_\kappa t$ if $s \in \SMw{\kappa}$ and $t\not\in \SMw{\kappa}$; 
in the remaining cases where either both or neither
of $s$ and~$t$ are in $\SMw{\kappa}$, we set $s \le_\kappa t$ if and only
if $s \le t$. 
\end{definition}


Thus the
elements in $\SMw{\kappa}$ always come first in the adapted colexicographic order for 
$\swtp{\kappa}$, and thanks
to the choice of the signed colexicographic order when $\swtp{\kappa}$ has sign $-1$
(and so the elements of $\SMw{\kappa}$ are all negative)
and the sign-reversed colexicographic order when the $\swtp{\kappa}$ has sign $+1$
(and so the elements of $\SMw{\kappa}$ are all positive),
the next greatest elements are the remaining semistandard signed $\muS$-tableaux
of the same sign as those in $\SMw{\kappa}$.

\begin{remark}\label{remark:TAdapted}
The  plethystic semistandard signed tableau
$\TMw{\kappa}$ defined in Definition~\ref{notation:stronglyMaximalTableauFromWeight}
(in which the signed colexicographic order is used to order inner tableau entries) is 
$\swtp{\kappa}$-adapted; in fact, since all its inner
tableau entries
have the same sign, it is semistandard
for any of our orders on semistandard signed tableaux.
\end{remark}

These observations can easily be  verified in following two examples.

\begin{example}\label{ex:adaptedColexicographicOrder411}
Let $(\kappa^\-,\kappa^\+)  = \bigl( \varnothing, (4,1,1) \bigr)$ of shape $(2)$, size $3$
and sign $+1$.
This is the strongly maximal signed weight
seen in Example~\ref{ex:stronglyMaximalSemistandardSignedTableauFamilies},
for which the unique semistandard signed tableau family is
$\SMw{\kappa} = \left\{ \hskip1pt\young(11)\spy{0pt}{,\ts} \young(12)\spy{0pt}{,\ts} \young(13)\hskip1pt \right\}$ of column-type.
In the adapted colexicographic order for $\kappa$ we have
\[ \scalebox{0.95}{$\young(11)\spy{0pt}{\,$<_\kappa$\hspace*{1pt}} 
\young(12)\spy{0pt}{\,$<_\kappa$\hspace*{0.1pt}} \young(13)\spy{0pt}{\,$<_\kappa$\hspace*{0.1pt}}
\young(22)\spy{0pt}{\,$<_\kappa$\hspace*{0.1pt}} \young(23)\spy{0pt}{\,$<_\kappa$\hspace*{0.1pt}} \!\!\ldots\!\!
\spy{0pt}{\,$<_\kappa$\hspace*{0.1pt}}\young(\oM1)\spy{0pt}{\,$<_\kappa$\hspace*{0.1pt}} \young(\oM2)\spy{0pt}{\,$<_\kappa$\hspace*{0.1pt}} \!\!\ldots$}
\]
whereas in the sign-reversed colexicographic order, we have
\[ \scalebox{0.95}{$ \young(11)\spy{0pt}{\,$<$\hspace*{1pt}} 
\young(12)\spy{0pt}{\,$<$\hspace*{1pt}} \young(22)\spy{0pt}{\,$<$\hspace*{1pt}}
\young(13)\spy{0pt}{\,$<$\hspace*{1pt}} \young(23)\spy{0pt}{\,$<$\hspace*{1pt}} \!\!\ldots\!\!
\spy{0pt}{\,$<$\hspace*{1pt}}\young(\oM1)\spy{0pt}{\,$<$\hspace*{1pt}} \young(\oM2)\spy{0pt}{\,$<$\hspace*{1pt}}\!\! \ldots $}\ .
\]
\end{example}




\begin{example}\label{ex:singletonAdapted}
By Lemma~\ref{lemma:maximalAndStronglyMaximalSingletonSemistandardSignedTableauFamilies},
if $\swtp{\kappa}$ is the signed weight of a strongly $c^\+$-maximal 
singleton semistandard signed tableau family of shape $\muS$
then the family is $\{ t_{\ell(\kappa^\-)}(\muS) \}$; since by Remark~\ref{remark:greatestSignedTableauIsLeastInSignedColexicographicOrder},
 $t_{\ell(\kappa^\-)}(\muS)$
 is the unique least semistandard signed tableau in the signed colexicographic order if
$\epsilon = -1$ and in the sign-reversed colexicographic order if $\epsilon = +1$,
in this special case the adapted order agrees with the usual order.
\end{example}

\begin{definition}\label{defn:adaptedPlethysticSemistandardSignedTableau}
Given a strongly maximal signed weight $\swtp{\kappa}$, 
let  $\PSSYT_\kappa(\rho, \muS)$ 
be the set of all plethystic semistandard
of outer shape $\rho$ and inner shape $\muS$
with negative entries from $\bigl\{-1,\ldots,-\ell(\kappa^\-)\bigr\}$, defined
as in Definition~\ref{defn:plethysticSemistandardSignedTableau}, but
using the $\swtp{\kappa}$-adapted colexicographic order to order the inner $\muS$-tableaux.
We say that such plethystic semistandard signed tableaux are $\swtp{\kappa}$-\emph{adapted}.
We write $\PSSYT_\kappa(\rho, \muS)_{(\pi^\-,\pi^\+)}$ for the adapted plethystic
semistandard signed tableaux in $\PSSYT_\kappa(\rho, \muS)$ whose signed weight is $(\pi^\-,\pi^\+)$.
\end{definition}

This is the obvious extension of the notation in Definitions~\ref{defn:plethysticSignedTableau}
and~\ref{defn:signedWeightPlethystic}.


\subsection{Overview and running example}\label{subsec:overviewSignedWeightBoundOuter}
Fix a strongly maximal signed weight $\swtp{\kappa}$.
We shall substantially simplify the exposition in this section and from \S\ref{subsec:exceptionalColumnBound}
onwards by
stating and proving all results only in the case when 
$\swtp{\kappa}$ has sign $+1$.
The modifications for sign $-1$ are  routine and are given briefly
in \S\ref{subsec:negativeSignAndCombined} at the end of this section.

\begin{remark}\label{remark:boundedExceptionalColumns}
The underlying principle in this section and the next is
\emph{provided $M$ is sufficiently large,
every $\swtp{\kappa}$-adapted plethystic semistandard signed tableau in 
$\PSSYTwk{\kappa}{\nu + M(1^R)}{\muS}{\lambda^\- + M\kappa^\-}{\lambda^\+ + M\kappa^\+}$ 
has the elements of $\SMw{\kappa}$, which form the inner tableaux of the
plethystic semistandard signed
tableau $\TMw{\kappa}$, in the top $R$ positions of almost all its columns}.
In Definition~\ref{defn:exceptionalColumnAndRow} we say that such columns are `typical'.
In particular, as we make precise in Corollary~\ref{cor:boundOnExceptionalColumns},
the number of columns whose top $R$ positions contain semistandard signed tableaux whose
total signed weight is \emph{not}
dominated in the $\ell^\-$-signed dominance order by $\swtp{\kappa}$ is bounded \emph{independently}
of~$M$.
\end{remark}

We see this principle in the first running example begun below, proving
the special case of Theorem~\ref{thm:nuStable} that
$\langle s_{(2,1) + M(1,1,1)} \circ s_{(2)},  s_{(4,2) + M(4,1,1)}\rangle$
is ultimately constant.

\begin{example}\label{ex:stronglyMaximalSignedWeightSWLii411}
Let $\swtp{\kappa}  = \bigl(\varnothing, (4,1,1)\bigr)$ be the strongly maximal signed
weight of the column-type tableau family
$\left\{ \,\young(11)\spy{0pt}{,\ts} \young(12)\spy{0pt}{,\ts} \young(13)\, \right\}$
in Example~\ref{ex:adaptedColexicographicOrder411} of shape $(2)$, size $3$ and sign $+1$.
To apply the Signed
Weight Lemma (Lemma~\ref{lemma:SWL})
to prove that
$\langle s_{(2,1) + M(1,1,1)} \circ s_{(2)},  s_{(4,2) + M(4,1,1)}\rangle$
is ultimately constant,
it is natural to look for a stable partition system $\bigl(\PSeq{M})_{M \in \N_0}$
such that $(4,2) + M(4,1,1) \in \PSeq{M}$ for each $M \in \N_0$ and, for condition~(ii), such that
\[ |\PSSYTwk{(\varnothing, (4,1,1))}{(2,1)+M(1,1,1)}{(2)}{\varnothing}{\pi}|\]
 is ultimately constant for all $\pi \in \PSeq{M}$. Note
that here the first subscript refers to adapted plethystic semistandard
signed tableaux (in the sense of Definition~\ref{defn:adaptedPlethysticSemistandardSignedTableau}),
whose inner $(2)$-tableaux are ordered according to the $\bigl(\varnothing, (4,1,1)\bigr)$-sign-reversed colexicographic order. In this case however,
since there are no negative entries, the distinction between the sign-reversed colexicographic
order and the usual colexicographic order is irrelevant.

In the special case where $\pi = (4,2) + M(4,1,1)$ the $\bigl(\varnothing, (4,1,1)\bigr)$-adapted
plethystic semistandard signed tableaux 
in this set are, when $M=0$,
\[ \pyoung{1.2cm}{0.7cm}{ {{\young(11), \young(11)}, {\young(22)}} }  \spy{18pt}{,}\quad
   \pyoung{1.2cm}{0.7cm}{ {{\young(11), \young(12)}, {\young(12)}} }  
\]
and, when $M=1$,
\[ \scalebox{0.8}{$\pyoung{1.2cm}{0.7cm}{ {{\young(11), \young(11), \young(11)}, 
                                           {\young(12), \young(12)}, {\young(23)}} }  \spy{26pt}{,}
\pyoung{1.2cm}{0.7cm}{ {{\young(11), \young(11), \young(11)}, {\young(12), \young(13)}, {\young(22)}} }  \spy{26pt}{,}
\pyoung{1.2cm}{0.7cm}{ {{\young(11), \young(11), \young(11)}, {\young(12), \young(22)}, {\young(13)}} }  \spy{26pt}{,}
\pyoung{1.2cm}{0.7cm}{ {{\young(11), \young(11), \young(12)}, {\young(12), \young(12)}, {\young(13)}} } 
\spy{26pt}{.}$}  
\]
 There are four plethystic semistandard signed tableaux when $M=2$.
 They are obtained by inserting the 
tableau $\TMw{\kappa}$ shown in the margin as a new first column into each of the four plethystic
semistandard signed tableaux for $M=1$. Note that since $\young(13) <_\kappa \young(22)$ in the
$\bigl(\varnothing,(4,1,1)\bigr)$-adapted colexicographic order, this preserves the semistandard condition even when we insert into the second plethystic semistandard
tableau for $M=1$.
\marginpar{ \qquad\qquad
\raisebox{0.1cm}{\scalebox{0.8}{$\quad\ \pyoung{1.2cm}{0.7cm}{ {{\young(11)}, {\young(12)}, {\young(13)}} }$}}}
\end{example}

In the previous example we saw that 
inserting $\TMw{\kappa}$ as a new column of height $R$ in a $\swtp{\kappa}$-adapted
plethystic semistandard signed tableau
of weight $\lambda \oplus M\swtp{\kappa}$ 
give a bijection
establishing hypothesis (ii) in the Signed Weight Lemma (Lemma~\ref{lemma:SWL}).
But still it is not obvious how to choose $\PSeq{M}$, or that the same bijection
will work when $\lambda \oplus M\swtp{\kappa}$ is replaced with an arbitrary 
$\pi$ in the relevant $\PSeq{M}$. We continue the example to show
one difficulty, circumvented using the final result
of this section (see  Corollary~\ref{cor:signedWeightBoundForStronglyMaximalSignedWeight}).

\begin{example}\label{ex:stronglyMaximalSignedWeightSWLi411}
Since the greatest partition $(6)$ in the dominance order is obviously an upper bound for
the constituents of $s_{(2,1)} \circ s_{(2)}$, Example~\ref{ex:stronglyMaximalSignedWeightSWLii411}
suggests we might apply the Signed
Weight Lemma (Lemma~\ref{lemma:SWL}) with the stable
partition system
\[ \begin{split} 
\bigl[(4,2){}+{}& M(4,1,1), (6)+M(4,1,1)\bigr]_\unlhd \\
&\quad = \bigl\{(4,2) + M(4,1,1), (5,1) + M(4,1,1), (6) + M(4,1,1) \bigr\}. \end{split} \]
for $M \in \N_0$.
(Here $\unlhd$ is the usual dominance order: 
by Remark~\ref{remark:twistedDominanceOrderGeneralizesDominanceOrder}, this is the
$0$-twisted dominance order, so we have $\ell^\- = 0$ and the symmetric functions in the lemma are $h_\pi$ for $\pi \in \Par$.)
However condition (i) in the Signed Weight Lemma fails when $M=1$:
we have $(8,3,1) \in [(4,2)+(4,1,1), (6) + (4,1,1)]_\unlhd$ and since $(8,4) \unrhd (8,3,1)$ 
and $s_{(8,4)}$ is a constituent of $s_{(3,2,1)} \circ s_{(2)}$
--- for instance, this follows from the generalized Cayley--Sylvester formula~\eqref{eq:twoRow} in~\S\ref{subsec:twoRow}
--- 
we have
\[ s_{(8,4)} \in \supp(h_{(8,3,1)}) \meet \supp (s_{(3,2,1)} \circ s_{(2)} ). \]
But $(8,4) \not\in [(4,2)+(4,1,1), (6) + (4,1,1)]_\unlhd$ since $(8,4)$ and $(10,1,1)$ are incomparable.
It might seem that the problem is that our
chosen upper bounds $(6) + M(4,1,1)$ are too small to contain
all partitions in the support of $s_{(2,1)+M(1,1,1)} \circ s_{(2)}$. However, one can show using
Theorem 1.5 of \cite{deBoeckPagetWildon} that the 
maximal constituents of $s_{(2,1) + M(1,1,1)} \circ s_{(2)}$
are precisely the partitions
\begin{equation}
\label{eq:stronglyMaximalSignedWeightSWLi}
\bigl\{ (5,1) + (M-a) (4,1,1) + a(3,3) : 0 \le a \le M \bigr\} 
\end{equation}
and since $(4,1,1)$ and $(3,3)$ are incomparable in the dominance order, 
there is \emph{no} stable partition system of intervals 
\begin{equation}
\label{eq:unstablePartitionSystem}
[\lambda + (M-S)(4,1,1), \omega + (M-S)(4,1,1)]_\unlhd 
\end{equation}
with $\lambda$ and $\omega$ partitions of $6S$
that contains
all the maximal constituents of $s_{(2,1) + M(1,1,1)} \circ s_{(2)}$ 
for all~$M \ge S$, or even
for all $M$ sufficiently large.
(Beginning with partition of $6S$ gives ample freedom 
to avoid technical
issues to do with `largeness', in the sense of Definition~\ref{defn:large} and Definition~\ref{defn:weightLarge}, so this is not the problem.)
Perhaps surprisingly, \emph{we conclude that it is essential to use the lower bound as well}.
And because a plethystic semistandard signed tableau
of shape $(2,1) + M(1,1,1)$ and signed weight $\bigl( \varnothing,
(4,2) + M(4,1,1)\bigr)$ may have an exceptional column in the sense of Definition~\ref{defn:exceptionalColumnAndRow} (see Example~\ref{ex:exceptionalColumnsGeneralWeight411}),
the stable partition system we define has to start with partitions of $12$.
(This corresponds to taking $S=1$ in~\eqref{eq:unstablePartitionSystem}.)
Since $(4,2) + (4,1,1) = (8,3,1)$ 
we therefore consider the intervals
\begin{equation}\label{eq:Kintervals} \bigl[ (8,3,1) + N(4,1,1), (12) + N(4,1,1) \bigr]_\unlhd \end{equation}
for $N \in \N_0$. 
If $\pi \in \bigl[(8,3,1) + N(4,1,1), (12) + N(4,1,1)\bigr]_\unlhd$
and 
\[ \sigma \in \supp(h_\pi) \cap \supp (s_{(2,1) + (N+1)(1,1,1)} \circ s_{(2)}) \]
then,
by Lemma~\ref{lemma:twistedKostkaMatrix} applied to $\mathrm{supp}(h_\pi)$ we have
\begin{align}
\label{eq:sigma831atLeast} \sigma \,&\unrhd\, \pi \,\unrhd\, (8,3,1) + N(4,1,1)
\intertext{
and, by~\eqref{eq:stronglyMaximalSignedWeightSWLi},
we have}
 \sigma &\unlhd (9,2,1) + (N-a)(4,1,1) +a(3,3) \nonumber
\\ & \hspace*{1.5in} = (9-a,2+2a,1-a) + N(4,1,1) \label{eq:sigma831atMost} \end{align}
for some $a$ with $0 \le a \le N$.
If $a=0$ then, by~\eqref{eq:sigma831atMost},  $\sigma \,\unlhd\, (9,2,1) + N(4,1,1) \unlhd (12) + N(4,1,1)$. 
Similarly if $a=1$ then $\sigma \unlhd (8,4) + N(4,1,1) \unlhd (12) + N(4,1,1)$,
and, despite involving the problematic partition $(8,4)$, thanks to our choice in~\eqref{eq:Kintervals},
$\sigma$ is in the interval for all $N \in \N_0$.
Finally if $a \ge 2$ then we must have $N \ge 1$ and 
we get $\sigma_1\le 9-a + 4N$
which, using the \emph{lower bound} on the intervals in~\eqref{eq:Kintervals}
--- justified by~\eqref{eq:sigma831atLeast} obtained using Lemma~\ref{lemma:twistedKostkaMatrix}
--- 
that $\sigma \unrhd (8,3,1) + N(4,1,1)$,
contradicts that $\sigma_1 \ge 8 + 4N$.
 Therefore  condition~(i) in the Signed Weight Lemma (Lemma~\ref{lemma:SWL})
holds for all $N \in \N_0$. We note that this 
contradiction 
was obtained by comparing in the dominance order
just on the first part, and correspondingly, $(4,1,1)$ is a strongly $1$-maximal
signed weight.
\end{example}

We continue this example in Example~\ref{ex:exceptionalColumns411}.

\subsection{Exceptional columns and rows}\label{subsec:exceptionalColumnsAndRows}
We define the \emph{signed weight} of a subset $\mathcal{B}$ of the boxes
of a plethystic semistandard tableau to be the sum of the weights
of the inner tableaux in $\mathcal{B}$.
In the following definition we use the $\ell^\-$-signed dominance order on the set $\W_{\ell^\-} \times \W$
defined in Definition~\ref{defn:ellSignedDominanceOrder}
to compare $\swtp{\phi}$ and $\swtp{\kappa}$.
Adapted plethystic semistandard signed tableaux are defined in 
Definition~\ref{defn:adaptedPlethysticSemistandardSignedTableau}.

\begin{definition}\label{defn:exceptionalColumnAndRow} 
Let $(\kappa^\-, \kappa^\+)$ be a strongly
$c^\+$-maximal signed weight of shape $\muS$ and size $R$.
Let $T$ be a $\swtp{\kappa}$-adapted
plethystic semistandard signed tableau of inner shape $\muS$. 
When $\swtp{\kappa}$ has sign $+1$, we say that a column
of~$T$ of height at least $R$ whose top $R$ boxes
have signed weight $(\phi^\-, \phi^\+)$ is
\emph{small} if $\swtp{\phi} \lhd \swtp{\kappa}$,
\emph{typical} if the top $R$ boxes
in the column form the plethystic semistandard signed tableau $T_{(\kappa^\-,\kappa^\+)}$
and \emph{exceptional} if $\swtp{\phi} \notunlhd \swtp{\kappa}$.
In the latter case we say the column is
\begin{defnlist}\vspace*{1pt}
\item \emph{large-exceptional} if
$\ell(\phi^\+) > \ell(\kappa^\+)$;
\item \emph{negative-exceptional} if $|\phi^\-| < |\kappa^\-|$;
\item \emph{positive-exceptional} if $|\phi^\-| + \sum_{i=1}^{c^\+} \phi^\+_i < |\kappa^\-| + \sum_{i=1}^{c^\+} \kappa^\+_i$
.\vspace*{1pt}
\end{defnlist}
When $\swtp{\kappa}$ has sign $-1$ we 
make the analogous definitions replacing `column' with `row', now considering the leftmost $R$ boxes in the row.
\end{definition}

The relevant strongly maximal signed weight $\swtp{\kappa}$ in this definition
will always be clear from the context. 
We shall prove in  Lemma~\ref{lemma:columnIsExceptionalOrBounded}(i)
that a column is either small, typical or exceptional, and 
in Lemma~\ref{lemma:columnIsExceptionalOrBounded}(ii)
that an exceptional column is either large-exceptional, negative-exceptional or positive-exceptional.
Note the latter three cases are not mutually exclusive: in fact any combination of them may hold. 

\begin{remark}\label{remark:singletonExceptional}
If
 $R=1$ then, by Lemma~\ref{lemma:maximalAndStronglyMaximalSingletonSemistandardSignedTableauFamilies},
the unique strongly maximal semistandard
signed tableau family of shape $\muS$ is $\{t_{\ell^\-}(\muS)\}$. 
By Lemma~\ref{lemma:greatestSignedWeight},
$\{t_{\ell^\-}(\muS)\}$ has the greatest signed weight,
in
the $\ell^\-$-signed dominance order on all $\muS$-tableaux with
entries from $\{-1,\ldots,-\ell^\-\} \cup \N$. 
Therefore in the notation of Definition~\ref{defn:exceptionalColumnAndRow},
we always have $\swtp{\phi} \unlhd \swtp{\kappa}$ where $\swtp{\kappa} = 
\bigl(\omega_{\ell^\-}(\muS)^\-, \omega_{\ell^\-}(\muS)^\+ \bigr)$ is the signed
weight of $t_{\ell^\-}(\muS)$, and so there are no
exceptional columns or rows. It is instructive to see how the remaining results in this
\marginpar{
\scalebox{0.7}{\quad\, \pyoungAnnotated{1.2cm}{0.7cm}{ {{\young(11), \young(11), \young(11)}, 
                                           {\young(12), \young(12)}, {\young(23)}} }{{\!\!small}}
                                           } }%
\marginpar{ \scalebox{0.7}{\!\! \pyoungAnnotated{1.2cm}{0.7cm}{ {{\young(11), \young(11), \young(11)}, {\young(12), \young(13)}, {\young(22)}} }{{\quad positive-e}} }}%
\marginpar{ \scalebox{0.7}{\quad \pyoungAnnotated{1.2cm}{0.7cm}{ {{\young(11), \young(11), \young(11)}, {\young(12), \young(22)}, {\young(13)}} }{{typical}}   }}%
\marginpar{ \scalebox{0.7}{\quad \pyoungAnnotated{1.2cm}{0.7cm}{ {{\young(11), \young(11), \young(12)}, {\young(12), \young(12)}, {\young(13)}} }{{typical}} }}%
section specialize to easy corollaries of Lemma~\ref{lemma:greatestSignedWeight} in
this case: we summarise the situation in Remark~\ref{remark:singletonSimpler} at the 
end of this section.
\end{remark}

\begin{example}\label{ex:exceptionalColumns411}
Take the strongly $1$-maximal signed weight $\bigl(\varnothing, (4,1,1)\bigr)$. 
Of the four $(\varnothing, (4,1,1))$-adapted
plethystic semistandard signed tableaux in the set
 $\PSSYTwk{(\varnothing, (4,1,1))}{(3,2,1)}{(2)}{\varnothing}{(8,3,1}$
shown at the end of Example~\ref{ex:stronglyMaximalSignedWeightSWLii411}, 
and repeated in the margin for ease of reference, the 
first column of the first tableau 
has signed weight $\bigl( \varnothing, (3,2,1)\bigr) \unlhd (\varnothing, (4,1,1)\bigr)$
so is small. The
first column of the
second has signed weight $\bigl(\varnothing, (3,3)\bigr)$, which is incomparable with 
$\bigl(\varnothing, (4,1,1)\bigr)$; this column is 
therefore
exceptional and since
the sums on the left- and right-sides of~(c) are $3$ and~$4$ respectively
it is positive-exceptional. 
The final two
tableaux each have first column $T_{(\varnothing, (4,1,1))}$ of
signed weight $\bigl(\varnothing, (4,1,1) \bigr)$; these two columns are typical.
The boxes not in the column of height $3$ are not classified
by Definition~\ref{defn:exceptionalColumnAndRow}.
\end{example}

This example is continued in Example~\ref{ex:exceptionalColumnsGeneralWeight411}.
We now give a further example to show the full generality
of Definition~\ref{defn:exceptionalColumnAndRow}.

\begin{example}\label{ex:exceptionalColumns22and31}
Let $(\kappa^\-,\kappa^\+) = \bigl( (2,2), (3,1) \bigr)$.
By Example~\ref{ex:LawOkitaniSignedWeightsAreStronglyMaximal}(ii), 
this is the strongly $1$-maximal weight of the column-type tableau family 
$\mathcal{M}_{((2,2),(3,1))}$ of shape $(4)$, size $2$ and sign $+1$ 
shown below 
\[ \left\{\hskip1pt\young(\oM\tM11)\, ,\ \young(\oM\tM12) \hskip1pt\right\} \,. \]
The special
case $\nu = (2,1)$ and $\muS = (4)/\varnothing$ of Theorem~\ref{thm:nuStable} is 
that the plethysm coefficients
$\langle s_{(2,1) + (M, M)} \circ s_{(4)}, s_{\lambda \sqcups (2^M) + M(3,1)}
\rangle$ are ultimately constant for all partitions~$\lambda$ of $12$.
First we take $\lambda = (8,3,1)$ with $2$-decomposition $\dec{(3,2)}{(6,1)}$.
There are three $((2,2),(3,1))$-adapted
plethystic semistandard tableaux in 
$\PSSYTwAdapted{(2,1)}{(4)}{(3,2)}{(6,1)}{((2,2),(3,1))}$, namely
\[ \scalebox{0.7}{$\pyoungAnnotated{2.0cm}{0.7cm}{ {{\young(\oM\tM11), \young(\oM\tM11)}, 
{\young(\oM112)}}}{{\text{small}}}  \spy{20pt}{\scalebox{1.4}{,}}\
\pyoungAnnotated{2.0cm}{0.7cm}{ {{\young(\oM\tM11), \young(\oM\tM12)}, {\young(\oM111)}} }{{\text{negative-e}}} \spy{20pt}{\scalebox{1.4}{,}}\
\pyoungAnnotated{2.0cm}{0.7cm}{ {{\young(\oM\tM11), \young(\oM111)}, {\young(\oM\tM12)}} }{{\text{\!\!\!\!\!typical}}} \spy{20pt}{\scalebox{1.4}{.}}$}
\]
The first column of the first has signed weight $\bigl((2,1),(4,1)\bigr)$ which is dominated
in the signed dominance order
by $((2,2),(3,1))$. Therefore it is small.
The first column of the second has signed weight $\bigl((2,1),(5)\bigr)$
so is not small, but instead is negative-exceptional, being deficient in negative
entries. In the third the first column is typical. Since the second columns are singleton,
they are not classified by Definition~\ref{defn:exceptionalColumnAndRow}.
Growing by $\nu \mapsto \nu + (1,1)$ and $\lambda \mapsto \lambda \oplus \bigl((2,2), (3,1)\bigr)$
as in Theorem~\ref{thm:nuStable}, we find that
the four 
plethystic semistandard signed tableaux in the set 
$\PSSYTwAdapted{(3,2)}{(4)}{(5,4)}{(9,2)}{((2,2),(3,1)}$ are 

\smallskip
\centerline{\scalebox{0.7}{$\pyoungAnnotated{2.0cm}{0.7cm}{ {{\young(\oM\tM11), \young(\oM\tM11), \young(\oM\tM11)}, {\young(\oM\tM12), \young(\oM112)}} }{{\!\!\!\!\!\!\!\!\!\text{typical}, \text{small}}} \spy{20pt}{\scalebox{1.4}{,}}
\pyoungAnnotated{2.0cm}{0.7cm}{ {{\young(\oM\tM11), \young(\oM\tM11), \young(\oM\tM12)}, {\young(\oM\tM12), \young(\oM111)}} }{{\text{\!\!\!\!\!\!\!\!\!typical}, \text{\!\!\!negative-e}}}  \spy{20pt}{\scalebox{1.4}{,}}
\pyoungAnnotated{2.0cm}{0.7cm}{ {{\young(\oM\tM11), \young(\oM\tM11), \young(\oM\tM11)}, {\young(\oM\tM22), \young(\oM111)}} }{{\text{\!\!\!\!\!\!\!\!\!small}, \text{\!\!\!negative-e}}}   \spy{20pt}{\scalebox{1.4}{,}}
\pyoungAnnotated{2.0cm}{0.7cm}{ {{\young(\oM\tM11), \young(\oM\tM11), \young(\oM111)}, {\young(\oM\tM12), \young(\oM\tM12)}} }{{\text{\!\!\!\!\!\!typical}, \text{\!\!\!\!\!\!typical}}}  
\spy{20pt}{\scalebox{1.4}{.}}$}}
\smallskip

\noindent We remark that in fact
\[ |\PSSYTwAdapted{(2,1)+M(1,1)}{(4)}{(3+2M,2+2M)}{(6+3M,1+M)}{((2,2),(3,1)}| = 4 \]
for all $M \ge 1$; a bijective proof is given by insertion of the 
plethystic semistandard signed tableau shown in the margin
corresponding to $\mathcal{M}_{((2,2),(3,1))}$ as a new first column;
this is the $\mathcal{H}$ map in the proof of Theorem~\ref{thm:nuStableSharp}.
\marginpar{\qquad\qquad\scalebox{0.7}{\pyoung{2.0cm}{0.7cm}{ {{\young(\oM\tM11)}, {\young(\oM\tM12)}} }}}
(In this case, unlike Example~\ref{ex:stronglyMaximalSignedWeightSWLii411},
the $\bigl((2,2),(3,1)\bigr)$-adapted colexicographic order
defined in Definition~\ref{defn:adaptedSignedColexicographicOrder}
coincides with the usual sign-reversed colexicographic order of 
Definition~\ref{defn:signedColexicographicOrder}, and so using
\emph{either} order, the insertion map preserves semistandardness.)
Computation by computer algebra shows that the
constant value of the plethysm coefficient is $2$.

In this case there are no large-exceptional columns because the maximum positive
entry permitted by the signed weight $\bigl((3+2M,2+2M),(6+3M,1+M)\bigr)$,
namely $2$,
is also the length of the positive part of the strongly maximal weight, namely 
$\ell\bigl((3,1)\bigr) = 2$.

\enlargethispage{6pt}
We now extend $\mathcal{M}_{((2,2),(3,1))}$ to the strongly $1$- and $2$-maximal
tableau family still of shape $(4)$ and sign $+1$ but now of size $3$ 
and signed weight $\bigl((3,3), (3,3)\bigr)$ shown below.
\[ \left\{\hskip1pt\young(\oM\tM11)\, ,\ \young(\oM\tM12)\, ,\ \young(\oM\tM22) \hskip1pt\right\} \,. \]
Its weight appears in the bottom right column of Table 4.23 for $\ell^\- = 2$.
In order to show the large-exceptional case in an example of manageable size,
we choose to regard this weight as strongly $1$-maximal.
The three plethystic semistandard signed tableaux in 
$\PSSYTwAdapted{(3,3)}{(3,3)}{(2,1,1)}{(4)}{((4,4),(5,2,1))}$ are 
\[ \scalebox{0.7}{$\pyoungAnnotated{2.0cm}{0.7cm}{ {{\young(\oM\tM11), \young(\oM\tM11)}, 
{\young(\oM\tM22)}, {\young(\oM\tM13)}}}{{\text{small}}}  \spy{27pt}{\scalebox{1.4}{,}}\
\pyoungAnnotated{2.0cm}{0.7cm}{ {{\young(\oM\tM11), \young(\oM\tM11)}, 
{\young(\oM\tM12)}, {\young(\oM\tM23)}}}{{\text{small}}}
 \spy{27pt}{\scalebox{1.4}{,}}\
\pyoungAnnotated{2.0cm}{0.7cm}{ {{\young(\oM\tM11), \young(\oM\tM12)}, 
{\young(\oM\tM12)}, {\young(\oM\tM13)}}}{{\text{large-e}}}
\spy{27pt}{\scalebox{1.4}{.}}$}
\]
The final plethystic semistandard signed tableau has first column of signed weight $\bigl((3,3),(4,1,1)\bigr)$.
This signed weight is not dominated by $\bigl((3,3), (3,3)\bigr)$ so it is not small,
and since the negative parts agree it is not negative-exceptional.
Since $|(3,3)| + 4 \not< |(3,3)| + 3$ it is not positive-exceptional (this is the relevant
comparison because of our choice that $c^\+ = 1)$.
But since it has an entry of $3$, it is large-exceptional.
We remark that since the negative part of the signed weight $((4,4),(5,2,1))$
is $(4,4)$, every inner tableau has the form $\young(\oM\tM\cdot\cdot)$\hskip1pt;
this explains the absence of negative-exceptional columns in this case. 



\end{example}

\subsection{Exceptional column bound}\label{subsec:exceptionalColumnBound}
The aim of this subsection is to prove Lemma~\ref{lemma:boundOnExceptionalColumns}, 
making the underlying principle
in Remark~\ref{remark:boundedExceptionalColumns}  precise. As promised earlier, to simplify
the exposition, from now on we assume that the strongly maximal signed weight has sign $+1$.
See \S\ref{subsec:negativeSignAndCombined} for the modifications for sign $-1$.

In the next lemma it is important to note that while the strongly maximal
signed weights defined in Definition~\ref{defn:stronglyMaximalSignedWeight} 
are of tableau families having all elements of the same sign
(by (b) the sign is $+1$ for column-type and $-1$ for row-type),
the comparison in Definition~\ref{defn:maximalSignedWeight}
is with all families of the relevant shape and size, with negative
entries from the prescribed set 
$\bigl\{-1,\ldots, -\ell(\kappa^\-)\bigr\}$ --- see the italicised
end to the paragraph after Definition~\ref{defn:maximalSignedWeight}.
We remind the reader that, by Remark~\ref{remark:TAdapted},
$T_{(\kappa^\-,\kappa^\+)}$ is semistandard in the
$\swtp{\kappa}$-adapted colexicographic order.

\begin{lemma}\label{lemma:columnIsExceptionalOrBounded}
Let $(\kappa^\-, \kappa^\+)$ be a strongly
maximal signed weight of shape $\muS$, size $R$ and sign $+1$.
Let $T$ be a $\swtp{\kappa}$-adapted plethystic semistandard signed tableau of inner shape $\muS$. 
Let $\swtp{\phi}$ be the signed weight of the top $R$ boxes in a column of $T$.
\begin{thmlist}
\item If $\swtp{\phi} = \swtp{\kappa}$ then the top $R$ boxes
in the column form the plethystic semistandard signed tableau $T_{(\kappa^\-,\kappa^\+)}$.
\item If $\swtp{\phi} \notunlhd \swtp{\kappa}$,
then the column is large-exceptional, negative-exceptional or positive-exceptional.
\end{thmlist}
\end{lemma}

\begin{proof}
Part (i) follows from the uniqueness
of the plethystic semistandard signed tableau family corresponding to a 
strongly maximal signed weight, proved in Lemma~\ref{lemma:stronglyMaximalSemistandardSignedTableauFamilyIsUnique}.
For (ii) let $\swtp{\phi} \notunlhd \swtp{\kappa}$ and suppose that the column is not large-exceptional.
Thus $\ell(\phi^\+) \le \ell(\kappa^\+)$.
Let $\swtp{\psi}$ be a maximal signed weight in the dominance order on $\W_{\ell(\kappa^\-)} \times \W$ 
of a column-type
semistandard signed tableau family of shape $\muS$ and size $R$
such that $\swtp{\phi} \unlhd \swtp{\psi}$,
as in Definition~\ref{defn:maximalSignedWeight}.  By hypothesis, $\swtp{\psi} \not= \swtp{\kappa}$.
Since the column is not large-exceptional, by
Lemma~\ref{lemma:stronglyMaximalImpliesMaximal},
either~\eqref{eq:stronglyMaximalM} holds and
we have
\begin{align*}
|\phi^\-| \le |\psi^\-| < |\kappa^\-| \\
\intertext{and the column is negative-exceptional or~\eqref{eq:stronglyMaximalP} 
holds and} 
\textstyle |\phi^\-| + \sum_{i=1}^{c^\+} \phi^\+_i \le
|\psi^\-| + \sum_{i=1}^{c^\+} \psi^\+_i &< \textstyle  |\kappa^\-| + \sum_{i=1}^{c^\+} \kappa^\+_i 
\end{align*}
and the column is positive-exceptional.
\end{proof}

To state our bound on the number of exceptional columns we need the  statistics on partitions
in the following two definitions.

\newcounter{defn:BBoxes}
\setcounter{defn:BBoxes}{\value{theorem}}
\begin{definition}\label{defn:BBoxes}
Given a partition $\nu$ and $R \in \N$, we define
$B_R(\nu) = |\nu| - R\nu_R$.
\end{definition}

In the general setting of
\S\ref{subsec:negativeSignAndCombined} this statistic is $B_R^\+(\nu)$,
and $B_R^\-(\nu)$ is
the row version of Definition~\ref{defn:BBoxes}
needed when $(\kappa^\-,\kappa^\+)$ has sign $-1$.

Equivalently $B_R(\nu)$ is the number of boxes 
 $(i,j)$ of $[\nu]$
such that either $i > R$ or $\nu'_j < R$;
these are precisely the boxes 
not in the top $R$ positions of a column
of height at least $R$. 
We shall use many times below that 
\begin{equation}\label{eq:BRconstant}
B_R\bigl(\nu + M(1^R) \bigr)= B_R(\nu)
\end{equation} for all $M \in \N_0$, as can be seen from Figure~\ref{fig:BR}.

\begin{figure}[h!]

\begin{center}\begin{tikzpicture}[x=0.5cm, y=-0.5cm, line width=0.5pt]
\begin{scope}[xshift = 1cm]
\fill[color = lightgrey] (9,0)--(14,0)--(14,1)--(13,1)--(13,1.5)--
(11,1.5)--(11,3)--(9,3)--(9,0);
\end{scope}

\fill[color = lightgrey] (0,4)--(4,4)--(3,4)--(3,5)--(2,5)--(2,5.5)
--(1,5.5)--(1,7)--(0,7)--(0,4);
\draw (0,0)--(4.5,0); \draw (5.5,0)--(15,0); 
\draw (11,4)--(6,4); 
\draw (7,4)--(5.5,4);
\draw (4,4)--(4.5,4);
\draw[dashed, line width=0.25pt] (4,0)--(4,4);

\begin{scope}[xshift = 1cm]
\draw (9,0)--(14,0)--(14,1)--(13,1)--(13,1.5);
\draw (11,2.5)--(11,3)--(9,3)--(9,4)--(7,4);
\node at (11.75,2) {\rotatebox{90}{$\scriptstyle \ddots$}};
\end{scope}

\draw (4,4)--(3,4)--(3,5)--(2,5)--(2,5.5);
\draw (1,6.5)--(1,7)--(0,7)--(0,0);
\draw[dashed, line width=0.25pt] (0,4)--(3,4);
\node at (1.25,6) {\rotatebox{90}{$\scriptstyle \ddots$}};
\node at (5,4) {$\scriptstyle \ldots$};
\node at (5,0) {$\scriptstyle \ldots$};

\draw[<-] (-0.5,0)--(-0.5,1.5);
\draw[->] (-0.5,2.5)--(-0.5,4.2);
\node at (-0.5, 2) {$\scriptstyle R$};

\draw[<-] (0,-0.5)--(1.5,-0.5);
\draw[->] (2.5,-0.5)--(3.8,-0.5);
\node at (2, -0.5) {$\scriptstyle \nu_R$};

\draw[<-] (4,-0.5)--(7,-0.5);
\draw[->] (8,-0.5)--(11,-0.5);
\node at (7.5, -0.5) {$\scriptstyle M$};

\end{tikzpicture}
\end{center}

\caption{The partition $\nu + M(1^R)$ with the boxes
not in the top~$R$ positions of a column of height at least~$R$
shaded. These are the boxes counted by
$B_R\bigl(\nu + M(1^R)\bigr)$. The diagram
also shows that $B_R\bigl(\nu + M(1^R)\bigr) = B_R(\nu)$
and that we can visualize the boxes added by the summand $M(1^R)$ as lying
in columns $\nu_R+1$, \ldots, $\nu_R + M$.
\label{fig:BR}
}
\end{figure}

Note also that 
if $T$ is a plethystic semistandard signed tableau
of outer shape $\nu$ and inner shape $\muS$ then, by
Lemma~\ref{lemma:greatestSignedWeight}, the contribution from the
boxes counted by $B_R(\nu)$ in this tableau to the signed weight of $T$ is at most
\begin{equation}\label{eq:BRtoSignedWeight}
 B_R(\nu)  \bigl(\omega_{\ell^\-}(\tauS)^\-, \omega_{\ell^\-}(\tauS)^\+ \bigr)
\end{equation}
where $\bigl(\omega_\ell^\-(\muS), \omega_\ell^\+(\muS) \bigr)$ is the
greatest signed weight of Definition~\ref{defn:greatestSignedWeight}.

\begin{definition}\label{defn:AStatistics}
Given a skew partition $\muS$ and $\ell^\- \in \N_0$ and $c^\+ \in \N_0$, we define
\smash{$\AmuSm = |\omega_{\ell^\-}(\muS)^\-|$}
and \smash{$\AmuSp = |\omega_{\ell^\-}(\muS)^\-|$ +} $\sum_{i=1}^{c^\+} \omega_{\ell^\-}(\muS)^\+_i$.
\end{definition}

Note that $\AmuSm$ is the number of negative entries in the $\ell^\-$-negative
greatest tableau $t_{\ell^\-}(\mu/\mus)$
of signed weight
$\bigl(\omega_{\ell^\-}(\muS)^\-, \omega_{\ell^\-}(\muS)^\+ \bigr)$;
as we saw in Lemma~\ref{lemma:greatestSignedWeight}
this 
is the greatest signed weight in the $\ell^\-$-signed dominance order (see
Definition~\ref{defn:ellSignedDominanceOrder})
of a semistandard signed tableau of shape $\muS$. 
In particular, no semistandard signed tableau of shape~$\muS$ 
with entries from $\bigl\{-1,\ldots,-\ell^\- \bigr\}$ can
have more than $\AmuSm$ negative entries.
The set $ \PSSYTwk{\kappa}{\nu + M(1^R)}{\muS}{\pi^\-}{\pi^\+}$
in the following lemma
is defined in Definition~\ref{defn:adaptedPlethysticSemistandardSignedTableau}.

\begin{lemma}[Bound on exceptional columns]\label{lemma:boundOnExceptionalColumns}
Let $\nu$ be a partition. 
Let $(\kappa^\-, \kappa^\+)$ be a strongly
$c^\+$-maximal signed weight of shape $\muS$, size $R$ and sign $+1$.
Fix $\ell^\- = \ell(\kappa^\-)$.
Let $\swtp{\lambda}$ and $\swtp{\pi} \in \W_{\ell^\-} \times \W$ be signed weights.
Let \[ T \in \PSSYTwk{\kappa}{\nu + M(1^R)}{\muS}{\pi^\-}{\pi^\+}.\]
If $\swtp{\pi} \unrhd \swtp{\lambda} + M\swtp{\kappa}$
then $T$ has at most

\smallskip
\begin{thmlist}
\item \smash{$\sum_{i=\ell(\kappa^\+)+1}^{\ell(\lambda^\+)}
\lambda^\+_i$} large-exceptional columns;\\[-12pt]
\item $B_R(\nu)\AmuSm + \nu_R |\kappa^\-| 
- |\lambda^\-|$  
negative-exceptional columns;
\item $B_R(\nu)\AmuSp + \nu_R \bigl( |\kappa^\-| + \sum_{i=1}^{c^\+} \kappa^\+_i
\bigr) - \bigl(|\lambda^\-| +\sum_{i=1}^{c^\+} \lambda^\+_i\bigr)$ positive-exceptional columns
that each are neither large-exceptional nor negative-exceptional.
\end{thmlist}
\end{lemma}

\begin{proof}
Consider the integer entries of the inner $\muS$-tableaux in $T$. Exactly
\smash{$\sum_{i=\ell(\kappa^\+)+1}^{\ell(\lambda^\+)} \pi^\+_i$} of these entries
are strictly greater than $\ell(\kappa^\+)$.
Since $\swtp{\pi} \unrhd \swtp{\lambda} + M\swtp{\kappa}$, we have
\[ \sum_{i=\ell(\kappa^\+)+1}^{\ell(\lambda^\+)} \pi^\+_i
\le \sum_{i=\ell(\kappa^\+)+1}^{\ell(\lambda^\+)} \lambda^\+_i.\\[1pt]\] 
Therefore
there at most \smash{$\sum_{i=\ell(\kappa^\+)+1}^{\ell(\lambda^\+)} \lambda^\+_i$} such
entries. Now (i) follows since, 
by Definition~\ref{defn:exceptionalColumnAndRow}(a),
each large-exceptional column has at least one of them.
By Lemma~\ref{lemma:columnIsExceptionalOrBounded}, each 
remaining exceptional column is either positive-exceptional or negative-exceptional.

To prove (ii) and (iii),
we consider the integer entries of $T$ in the sets $\{-1,\ldots, -\ell^\-\}$
and $\{-1,\ldots, -\ell^\-,1,\ldots, c^\+ \}$, respectively,
and the inner $\muS$-tableaux in which they lie.
There are $B_R(\nu)$ such entries not lying in the top~$R$ boxes of a column having at least $R$ boxes.
By the remark immediately before this lemma,
the $\muS$-tableaux in these boxes have between them,
at most $B_R(\nu)\AmuSm$ entries in  $\{-1,\ldots, -\ell^\-\}$.
Moreover, by Lemma~\ref{lemma:greatestSignedWeight},
each such $\muS$-tableau has
signed weight bounded above (in the $\ell^\-$-signed dominance order in Definition~\ref{defn:ellSignedDominanceOrder}) 
by $\bigl(\omega_{\ell(\kappa^\-)}(\muS)^\-, \omega_{\ell(\kappa^\-)}(\muS)^\+\bigr)$, and so there are
at most $B_R(\nu)\AmuSp$ entries in
$\{-1,\ldots, -\ell^\-,$ $1,\ldots, c^\+\}$ in these $\muS$-tableaux.
Each remaining $\muS$-tableau entry lies in the top~$R$ boxes of a column of~$T$ having at least $R$ boxes.
As can be seen from Figure~\ref{fig:BR}, there are $\nu_R + M$ such columns. 

For (ii), suppose that $E^\-(T)$ of these columns
are negative-exceptional. In a non-negative-exceptional column whose top $R$ entries
have signed weight $(\phi^\-,\phi^\+)$, we have, by Definition~\ref{defn:exceptionalColumnAndRow}(b),
$|\phi^\-| = |\kappa^\-|$.
Hence the $\muS$-tableaux in
the top $R$ rows of the non-negative-exceptional columns have, 
between them, exactly
\smash{$\bigl(\nu_R + M - E^\-(T)\bigr)|\kappa^\-|$} 
entries in $\{-1,\ldots, -\ell^\-\}$.
By \eqref{eq:stronglyMaximalM} in Lemma~\ref{lemma:stronglyMaximalImpliesMaximal},
in each negative-exceptional column
there are at most  $|\kappa^\-| - 1$
entries in $\{-1,\ldots, -\ell^\-\}$.
Summing these bounds gives
\begin{align*} 
\sum_{i=1}^{\ell^\-} \pi^\-_i &\le B_R(\nu) \AmuSm
+ \bigl(\nu_R + M -E^\-(T)\bigr)|\kappa^\-| 
+ E^\-(T) \bigl( |\kappa^\-| - 1 \bigr) 
\\[-6pt]
&= B_R(\nu)\AmuSm + (\nu_R + M) |\kappa^\-| 
- E^\-(T).
\end{align*}
On the other hand, since $\swtp{\pi} \unrhd \swtp{\lambda} + M(\kappa^\-,\kappa^\+)$ we 
have 
\[ |\pi^\-| \ge |\lambda^\-| + M |\kappa^\-|. \]
Combining the two displayed inequalities and cancelling 
$M|\kappa^\-|$ 
we get
\[ E^\-(T) \le B_R(\nu)\AmuSm
 + \nu_R |\kappa^\-| - |\lambda^\-| 
 \]
as required.

For (iii), suppose there are $E^\+(T)$ exceptional columns of
$T$ of height at least~$R$ that
are not  large-exceptional and not negative-exceptional. 
By Lemma~\ref{lemma:columnIsExceptionalOrBounded}(ii) these columns
are positive-exceptional.
Let $(\phi^\-,\phi^\+)$ be the signed weight of such a column.
Using \eqref{eq:stronglyMaximalP} in Lemma~\ref{lemma:stronglyMaximalImpliesMaximal}
and Definition~\ref{defn:exceptionalColumnAndRow}(c),
we have $|\phi^\-| + \sum_{i=1}^{c^\+} \phi^\+_i
< |\kappa^\-| + \sum_{i=1}^{c^\+} \kappa^\+_i$.
The analogous inequalities are therefore
\begin{align*} 
|\pi^\-| + \sum_{i=1}^{c^\+} \pi^\+_i &\le
B_R(\nu)\AmuSp + (\nu_R + M) \bigl( |\kappa^\-| + \sum_{i=1}^{c^\+} \kappa^\+_i 
\bigr) - E^\+(T) \\
\intertext{and}
|\pi^\-| + \sum_{i=1}^{c^\+} \pi^\+_i &\ge |\lambda^\-| + \sum_{i=1}^{c^\+}\lambda^\+_i + 
M\bigl( |\kappa^\-| +  \sum_{i=1}^{c^\+} \kappa^\+_i \bigr).\end{align*}
Combining these two 
inequalities and cancelling $M \bigl( |\kappa^\-| + \sum_{i=1}^{c^\+} \kappa^\+_i \bigr)$
we get
\[ E^\+(T) \le B_R(\nu)\AmuSp+ \nu_R\bigl( |\kappa^\-| + \sum_{i=1}^{c^\+} \kappa^\+_i
\bigr) - |\lambda^\-| - \sum_{i=1}^{c^\+} \lambda^\+_i \]
again as required.
\end{proof}

Motivated by this result we make the following definition. The statistics
$B_R(\nu)$, $\AmuSm$ and $\AmuSp$ are defined in Definitions~\ref{defn:BBoxes} and~\ref{defn:AStatistics}.

\newcounter{defn:exceptionalColumnBound}
\setcounter{defn:exceptionalColumnBound}{\value{theorem}}
\begin{definition}\label{defn:exceptionalColumnBound}
Let $\nu$ be a non-empty partition.
Let $(\kappa^\-, \kappa^\+)$ be a strongly
$c^\+$-maximal signed weight of shape $\muS$, size $R$ and sign $+1$.
Let $\lambda$ be a partition of $|\nu||\muS|$. 
Fix $\ell^\- = \ell(\kappa^\-)$. Define
\begin{align*}
E^\-  &= \textstyle
B_R(\nu)\AmuSm  + \nu_R  |\kappa^\-| 
- |\lambda^\-|, 
\\
E^\+ &=  \textstyle
B_R(\nu)\AmuSp + \nu_R\bigl( |\kappa^\-| + \sum_{i=1}^{c^\+} \kappa^\+_i
\bigr) - \bigl( |\lambda^\-| +\sum_{i=1}^{c^\+} \lambda^\+_i\bigr) \\
\intertext{and if $R \ge 2$,}
&\hspace*{-0.195in}E_{c^\+\hskip-1pt,(\kappa^\-,\kappa^\+)}(\nu, \muS : \lambda) = 
\max\bigl( E^\- + E^\++
\textstyle \sum_{i=\ell(\kappa^\+)+1}^{\ell(\lambda^\+)} \lambda^\+_i,0\bigr).\end{align*}
If $R =1$ we instead define $E_{c^\+\hskip-1pt,(\kappa^\-,\kappa^\+)}(\nu, \muS : \lambda)=0$.
\end{definition}

Note that 
since  $R = (|\kappa^\-| + |\kappa^\+|)/|\muS|$, there is no need
to state $R$ explicitly in the definition of 
$E_{c^\+\hskip-1pt,(\kappa^\-,\kappa^\+)}(\nu, \muS : \lambda)$. 

\begin{corollary}\label{cor:boundOnExceptionalColumns}
Let $\nu$ be a partition. 
Let $\swtp{\kappa}$ be a strongly $c^\+$-maximal signed weight of shape $\muS$, size $R$ and sign $+1$.
Let $\lambda$ be an 
$\bigl(\ell(\kappa^\-),$ $\ell(\kappa^\+)\bigr)$-large partition.
Suppose that $\pi \unrhddot \lambda \oplus M\swtp{\kappa}$ in the $\ell(\kappa^\-)$-twisted
dominance order.
If $T \in \PSSYT_\kappa(\nu + M(1^R), \muS)_{(\pi^\-,\pi^\+)}$ then
$T$ has at most $E_{c^\+\hskip-1pt,(\kappa^\-,\kappa^\+)}(\nu, \muS : \lambda)$
exceptional columns.
\end{corollary}

\begin{proof}
By Lemma~\ref{lemma:adjoinToLarge} 
the $\ell(\kappa^\-)$-decomposition of $\lambda$ is 
$\decs{\lambda^\- + M\kappa^\-}{\lambda^\+ + M\kappa^\+}$.
Therefore $\pi \unrhddot \lambda \oplus M\swtp{\kappa}$
is equivalent, by the definition
of the $\ell(\kappa^\-)$-twisted dominance order in Definition~\ref{defn:ellTwistedDominanceOrder}, to
$\swtp{\pi} \unrhd ( \lambda^\- + M\kappa^\-, 
\lambda^\+ + M\kappa^\+)$.
If
$R \ge 2$
the corollary is now immediate from Lemma~\ref{lemma:boundOnExceptionalColumns}, given
the definition of $E_{c^\+\hskip-1pt,(\kappa^\-,\kappa^\+)}(\nu, \muS : \lambda)$ 
in Definition~\ref{defn:exceptionalColumnBound}. When $R=1$ it follows
from Remark~\ref{remark:singletonExceptional}. 
\end{proof}

See Example~\ref{ex:22and31bound}
for the bound in this corollary in the running example
of the `signed' case begun in Example~\ref{ex:exceptionalColumns22and31}.

\begin{example}\label{ex:exceptionalColumnsGeneralWeight411}
Continuing our running `unsigned' example
(see Examples~\ref{ex:stronglyMaximalSignedWeightSWLii411},
~\ref{ex:stronglyMaximalSignedWeightSWLi411} and \ref{ex:exceptionalColumns411}),
let $\swtp{\kappa}$ be the strongly $1$-maximal signed weight 
$\bigl( \varnothing, (4,1,1) \bigr)$ of shape $(2)$, size $3$ and sign $+1$.
(Note that this size can be computed directly from $\bigl( \varnothing, (4,1,1) \bigr)$
knowing the shape
using
the remark immediately after Definition~\ref{defn:exceptionalColumnBound}.)
Thus $\muS = (2)/\varnothing$ and $R =3$. 
We have $\ell^\- = 0$ and $c^\+ = 1$.
Generalizing
slightly, to show more clearly the effect of columns
of height at least $R$ in $\nu$, 
let $\nu = (2,1) + \KC(1,1,1)$. The relevant statistics are $B_3\bigl( (2,1) 
+ \KC(1,1,1) \bigr) = 3$,
and, since $\omega_{0}\bigl( (2) \bigr)
= (2)$ corresponding to the greatest tableau \raisebox{0.5pt}{$\young(11)$}\,, we have $\AmuSm = 0$ and
 $\AmuSp = 2$. 
The quantities $E^\-$ and $E^\+$ in Definition~\ref{defn:exceptionalColumnBound}
are 
\begin{align*}
E^\- &= 0, \\
E^\+ &= 3 \times 2 + 4\KC - 0 - \lambda_1 = 6 + 4\KC - \lambda_1
\end{align*}
and so $E_{1,(\varnothing,(4,1,1))}\bigl( (2,1), (2) \hspace*{-2pt}:\hspace*{-2pt} \lambda \bigr) = 0 + 6 +4\KC - \lambda_1 + \lambda_4 + \cdots $ 
for any partition $\lambda$ of $6 + 4\KC$.
Taking $\KC=0$ and $\lambda = (4,2)$ as earlier we have
$E_{1,(\varnothing,(4,1,1))}\bigl( (2,1), (2) \hspace*{-2pt}:\hspace*{-2pt} (4,2) \bigr) = 2$
and so, by Lemma~\ref{lemma:boundOnExceptionalColumns},
a $\bigl(\varnothing, (4,1,1) \bigr)$-adapted
plethystic semistandard signed tableaux lying in the set  
\[ \PSSYT_{(\varnothing,(4,1,1))}\bigl((2,1) + M(1,1,1), (2)\bigr)_{(\varnothing,\pi^\+)} \]
where $\pi^\+ \,\unrhd\, (4,2) + M(4,1,1)$ may have at most two exceptional
columns. For a general $\KC$, we note that
if $\lambda = (4,2) + \KC(4,1,1)$ then $\lambda_1 = 4\KC + 4$ and
so
\[ E_{1,(\varnothing,(4,1,1))}\bigl( (2,1) + \KC(1,1,1), (2) \hspace*{-2pt}:\hspace*{-2pt} (4,2) + \KC(4,1,1) \bigr) = 2 \]
giving the same bound; this is expected from the proof
of Lemma~\ref{lemma:boundOnExceptionalColumns}, because 
the contribution $(4,1,1)$ to $\lambda$ \emph{can only} come
from a typical column, equal to the plethystic semistandard
signed tableau shown in the margin.
\marginpar{
\raisebox{0.1cm}{\scalebox{0.8}{$\ \pyoung{1.2cm}{0.7cm}{ {{\young(11)}, {\young(12)}, {\young(13)}} }$}}}
Of course if~$\KC$ is large and $\lambda$ is not of this special
form then there may be many more exceptional columns.

We now show by a direct argument that 
if $\pi^\+ \,\unrhd (4,2) + M(4,1,1)$ and $T \in
\PSSYT_{(\varnothing,(4,1,1))}\bigl((2,1) + M(1,1,1), (2)\bigr)_{(\varnothing,\pi^\+)}$
then $T$ has at most one exceptional column. 
Thus as is usually the case, the bound from Corollary~\ref{cor:boundOnExceptionalColumns} is not optimal.
The key observation is that each typical column contain four $1$s as entries of
its inner $(2)$-tableaux, and since the maximum positive entry is $3$
and there are no negative entries,
a column that is not typical, i.e.~not equal to the plethystic
semistandard signed tableau $T_{(\varnothing,(4,1,1))}$ shown
earlier in the margin, has at most three $1$s. 
The three boxes in $T$ not in columns of length $3$ 
contribute at most five $1$s.
(This can  easily be seen from the tableaux in
Example~\ref{ex:stronglyMaximalSignedWeightSWLii411}.)
Therefore if there are $N$ non-typical columns,
$T$ has at most $4M - N + 5$ entries of $1$.
On the other hand,
since $\pi^\+ \unrhd (4,2) + M(4,1,1)$, we have $\pi^\+_1 \ge 4 + 4M$.
Therefore $4M-N+5 \ge 4 + 4M$ and so $N \le 1$.
Since exceptional columns are non-typical, this implies
there is at most one exceptional column, as claimed,
and moreover, this exceptional column is positive-exceptional.
Since the signed weight of the column
is not dominated by $(\varnothing, (4,1,1)\bigr)$,
it is necessarily equal to the plethystic semistandard
signed tableau $T_{(\varnothing, (3,3))}$ shown in the margin,
\marginpar{ \qquad\qquad \raisebox{0.1cm}{\scalebox{0.8}{$\ \pyoung{1.2cm}{0.7cm}{ {{\young(11)}, {\young(12)}, {\young(22)}} }$}}}
defined by the strongly $2$-maximal semistandard signed tableau family
of signed weight $\bigl( \varnothing, (3,3) \bigr)$.
We continue this line of argument in Example~\ref{ex:omegaBound411}.
\end{example}

%



We finish this example in Example~\ref{ex:omegaBound411} below.

\subsection{Signed weight bound in the $\ell^\-$-signed dominance order}
We now turn the bound on the number of exceptional columns in Lemma~\ref{lemma:boundOnExceptionalColumns}
into an upper bound on signed weights in the $\ell^\-$-signed dominance order in
Definition~\ref{defn:ellSignedDominanceOrder}.
We continue to simplify the exposition by assuming the strongly maximal signed
weight has sign $+1$. 
See \S\ref{subsec:negativeSignAndCombined} for the modifications for sign $-1$.
Recall from Definition~\ref{defn:greatestSignedWeight}
and Lemma~\ref{lemma:greatestSignedWeight} 
that $\bigl(\omega_{\ell^\-}(\muS)^\-, \omega_{\ell^\-}(\muS)^\+ \bigr)$ 
is the greatest signed weight in the $\ell^\-$-signed dominance order 
of a semistandard signed tableau of shape $\muS$.
The set $\PSSYTwk{\kappa}{\nu + M(1^R)}{\muS}{\pi^\-}{\pi^\+}$
is defined in Definition~\ref{defn:adaptedPlethysticSemistandardSignedTableau}.
Small columns were defined in Definition~\ref{defn:exceptionalColumnAndRow}
and the statistic $B_R(\nu)$ in Definition~\ref{defn:BBoxes}.
An example is given after the proposition.

\begin{proposition} 
\label{prop:signedWeightBoundForStronglyMaximalWeight}
Let $\nu$ be a partition
Let $(\kappa^\-, \kappa^\+)$ be a strongly
$c^\+$-maximal signed weight of shape $\muS$, size~$R$ and sign $+1$.
Fix $\ell^\- = \ell(\kappa^\-)$. Let $\muS$ be a skew partition.
Set 
$E = E_{c^\+\hskip-1pt,(\kappa^\-,\kappa^\+)}(\nu, \muS : \lambda)$.
Let $M \in \N_0$. 
Suppose that $T \in \PSSYTwk{\kappa}{\nu + M(1^R)}{\muS}{\pi^\-}{\pi^\+}$
where
 \[ \swtp{\pi} \,\unrhd\, \swtp{\lambda} + M\swtp{\kappa} \]
in the $\ell^\-$-signed dominance order.
Suppose that $T$ has $\nSmall$ small columns and that their 
top $R$ boxes have
signed weights $\swtp{\phi_1}, \ldots, \swtp{\phi_\nSmall}$.
If $M \ge - \nu_R + \nSmall + E$,
\begin{align*}  \swtp{\pi} \,&\unlhd\,  \bigl(B_R(\nu) + ER \bigr)
\bigl( \omega_{\ell^\-}(\muS)^\-, \omega_{\ell^\-}(\muS)^\+ \bigr) \\
&\qquad + \swtp{\phi_1} + \cdots + \swtp{\phi_\nSmall} + 
\bigl( \nu_R - \nSmall - E + M \bigr) \swtp{\kappa}. 
\intertext{and
if $M \ge E - \nu_R$, the weaker bound}
\swtp{\pi} \,&\unlhd\,  \bigl(B_R(\nu) \!+ \!ER \bigr) 
\bigl( \omega_{\ell^\-}(\muS)^\-, \omega_{\ell^\-}(\muS)^\+ \bigr)\! +\!
\bigl( \nu_R \hskip-0.5pt-\hskip-0.5pt E \hskip-0.5pt+ \hskip-0.5pt M \bigr) \swtp{\kappa}  \end{align*}
also holds.
\end{proposition}


\begin{proof}
The second bound on $\swtp{\pi}$ follows easily from the first since,
by the definition of small in Definition~\ref{defn:exceptionalColumnBound},
we have $\swtp{\phi_i} \unlhd \swtp{\kappa}$ for all $i$, and so
the first bound is maximized when $\nSmall=0$.
It remains to prove the first bound. 

If $R=1$ then 
by  Remark~\ref{remark:singletonExceptional},
$\swtp{\kappa} = \bigl(\omega_{\ell^\-}(\muS)^\-, \omega_{\ell^\-}(\muS)^\+ \bigr)$
and there are no exceptional columns.
The claim is therefore that 
\[ \begin{split} (\pi^\-,\pi^\+) \unlhd \bigl( B_1(\nu) + \nu_1 - \nSmall + M \bigr)
 &\bigl(\omega_{\ell^\-}(\muS)^\-, \omega_{\ell^\-}(\muS)^\+ \bigr)
\\ & \qquad + \swtp{\phi_1} + \cdots + \swtp{\phi_\nSmall}. \end{split}\]
By Lemma~\ref{lemma:greatestSignedWeight},
each box that is not the topmost box in one of the 
$\nSmall$ non-small columns of~$T$
contributes at most  $\bigl(\omega_{\ell^\-}(\muS)^\-, \omega_{\ell^\-}(\muS)^\+ \bigr)$ to the signed weight of $T$. 
There are $B_1\bigl( \nu + (M) \bigr) + \nu_1 -\nSmall + M$ such
boxes. Since by~\eqref{eq:BRconstant}, we have
$ B_1\bigl(\nu + (M) \bigr) = B_1(\nu)$,
the first bound now follows.

\begin{figure}[b]
\begin{center}\begin{tikzpicture}[x=0.5cm, y=-0.5cm, line width=0.5pt]
\begin{scope}[xshift = 1cm]
\fill[color = lightgrey] (9,0)--(14,0)--(14,1)--(13,1)--(13,1.5)--
(11,1.5)--(11,3)--(9,3)--(9,0);
\end{scope}

\begin{scope}[xshift = -2cm]
\fill[pattern = north west lines] (10,0)--(12.5,0)--(12.5,4)--(10,4)--cycle;
\fill[color = grey] (12.5,0)--(15,0)--(15,4)--(12.5,4)--cycle;

\draw[<-] (10.1,-0.5)--(10.85,-0.5);
\draw[->] (11.65,-0.5)--(12.4,-0.5);
\node at (11.25, -0.5) {$\scriptstyle \nSmall$};
\begin{scope}[yshift = -0.1cm]{
\fill[color = white] (10.2,1.1)--(12.2,1.1)--(12.2,1.9)--(10.2,1.9)--cycle;
\node at (11.25, 1.5) {\small small};
\fill[color = white] (13.0,1.1)--(14.5,1.1)--(14.5,1.9)--(13.0,1.9)--cycle;
\node at (13.75, 1.55) {\small exc.};
}
\end{scope}

\draw[<-] (12.6,-0.5)--(13.35,-0.5);
\draw[->] (14.05,-0.5)--(15,-0.5);
\node at (13.75, -0.5) {$\scriptstyle e$};

\end{scope}

\node at (3, 1.775) {\small typical};

\fill[color = lightgrey] (0,4)--(4,4)--(3,4)--(3,5)--(2,5)--(2,5.5)
--(1,5.5)--(1,7)--(0,7)--(0,4);
\draw (0,0)--(4.5,0); \draw (5.5,0)--(15,0); 
\draw (11,4)--(6,4); 
\draw (7,4)--(5.5,4);
\draw (4,4)--(4.5,4);
\draw[dashed, line width=0.25pt] (6,0)--(6,4);
\draw[dashed, line width=0.25pt] (8.5,0)--(8.5,4);
\draw[dashed, line width=0.25pt] (11,0)--(11,4);

\begin{scope}[xshift = 1cm]
\draw (9,0)--(14,0)--(14,1)--(13,1)--(13,1.5);
\draw (11,2.5)--(11,3)--(9,3)--(9,4)--(7,4);
\node at (11.75,2) {\rotatebox{90}{$\scriptstyle \ddots$}};
\end{scope}

\draw (4,4)--(3,4)--(3,5)--(2,5)--(2,5.5);
\draw (1,6.5)--(1,7)--(0,7)--(0,0);
\draw[dashed, line width=0.25pt] (0,4)--(3,4);
\node at (1.25,6) {\rotatebox{90}{$\scriptstyle \ddots$}};
\node at (5,4) {$\scriptstyle \ldots$};
\node at (5,0) {$\scriptstyle \ldots$};

\draw[<-] (-0.5,0)--(-0.5,1.5);
\draw[->] (-0.5,2.5)--(-0.5,4.2);
\node at (-0.5, 2) {$\scriptstyle R$};

\draw[<-] (0,-0.5)--(0.9,-0.5);
\draw[->] (5,-0.5)--(5.9,-0.5);
\node at (3, -0.5) {$\scriptstyle \nu_R + M - e - \nSmall$};

\end{tikzpicture}
\end{center}
\caption{A plethystic semistandard signed tableau of outer
shape $\nu + M(1^R)$ showing the contributions to the signed
weight identified in the proof of Proposition~\ref{prop:signedWeightBoundForStronglyMaximalWeight}.
Each of the $B_R(\nu) + eR$ shaded boxes 
has an inner $\muS$-tableau whose contribution is
bounded by $\bigl( \omega_{\ell^{\scriptstyle -}}(\muS)^{\scriptstyle -}, 
\omega_{\ell^{\scriptstyle -}}(\muS)^{\scriptstyle +} \bigr)$.
The hatched small columns
contribute $(\phi_j^{\scriptstyle -},
\phi_j^{\scriptstyle +})$ for $1 \le j \le \nSmall$.
The white boxes are in typical columns, the top $R$ boxes in
each contributing $(\kappa^{\scriptstyle -},
\kappa^{\scriptstyle +})$. (It is possible
that some of the~$e$ exceptional columns
appear to the left of the $\nSmall$ small columns.)
Note that as~$M$ varies, all 
but a constant number
of boxes are in typical columns and so 
their signed weight (per column) is bounded by the stronger bound
$(\kappa^{\scriptstyle -}, \kappa^{\scriptstyle +})$
rather than the weaker bound (per box) from
$\bigl( \omega_{\ell^{\scriptstyle -}}(\muS)^{\scriptstyle -}, 
\omega_{\ell^{\scriptstyle -}}(\muS)^{\scriptstyle +} \bigr)$.
\label{fig:signedWeightBoundForStronglyMaximalSignedWeight}}
\end{figure}

Now suppose that $R \ge 2$.
Suppose that $T$ has $e$ exceptional columns,
so $T$ has $eR$ boxes in the top $R$ rows of exceptional columns.
By~\eqref{eq:BRconstant} and \eqref{eq:BRtoSignedWeight} and
Lemma~\ref{lemma:greatestSignedWeight}, these boxes, together with
the $B_R\bigl( \nu + M(1^R) \bigr) = B_R(\nu)$ boxes
not in the top $R$ positions of a column of length~$R$,
contribute at most
\begin{equation}\label{eq:exceptionalBoxes} \bigl(B_R(\nu) + eR \bigr)
\bigl( \omega_{\ell^\-}(\muS)^\-, \omega_{\ell^\-}(\muS)^\+ \bigr) \end{equation}
to the signed weight of $T$.
Each of the $\nu_R\hskip1pt-\hskip1pt \nSmall\hskip1pt - \hskip1pt e \hskip1pt+\hskip1ptM$ columns
of height at least $R$ that is both non-exceptional  
and non-small is typical (see Lemma~\ref{lemma:columnIsExceptionalOrBounded})
and so contributes $\swtp{\kappa}$ to the signed weight of $T$.
There are also $\nSmall$ small columns which contribute $\swtp{\phi_1} + \cdots + \swtp{\phi_\nSmall}$.
Taken together these columns therefore contribute
\begin{equation}\label{eq:nonexceptionalBoxes} \bigl( \nu_R + M - e - \nSmall \bigr) \swtp{\kappa}
+ \swtp{\phi_1} + \cdots + \swtp{\phi_\nSmall} 
\end{equation}
to the signed weight of $T$.
This is shown diagrammatically in Figure~\ref{fig:signedWeightBoundForStronglyMaximalSignedWeight}.

The sum of~\eqref{eq:exceptionalBoxes} and~\eqref{eq:nonexceptionalBoxes}
is an upper bound on the signed weight of~$T$.
Since 
 $(\kappa^\-,\kappa^\+)$ is the signed weight of a tableau of outer shape $(1^R)$
and inner shape $\muS$, it follows,
again by Lemma~\ref{lemma:greatestSignedWeight}, that
\[ \swtp{\kappa} \unlhd R\bigl( \omega_{\ell^\-}(\muS)^\-, \omega_{\ell^\-}(\muS)^\+ \bigr). \]
We conclude that (for fixed $\nSmall$), 
this upper bound 
is maximized in the $\ell(\kappa^\-)$-signed dominance order
when $e$ is as large as possible. By Lemma~\ref{lemma:boundOnExceptionalColumns}
we have $e \le  E_{c^\+\hskip-1pt,(\kappa^\-,\kappa^\+)}(\nu, \muS : \lambda)$.
Therefore we obtain an upper bound on $\swtp{\pi}$ by substituting $E$
for $e$ in the sum of \eqref{eq:exceptionalBoxes} and~\eqref{eq:nonexceptionalBoxes}.
(Note that by hypothesis $\nu_R  - \nSmall - E + M \ge 0$ so the right-hand side is
a valid signed weight, having non-negative entries.)
This proves the first bound.
\end{proof}


\begin{example}\label{ex:22and31bound} 
As in Example~\ref{ex:exceptionalColumns22and31} we
take the strongly $1$-maximal signed weight $\bigl( (2,2), (3,1)\bigr)$
of shape $(4)$, size $2$ and sign $+1$, and take $\nu = (2,1)$.
We have $B_R(\nu) = B_2\bigl((2,1)\bigr) = 1$.
Since $\bigl( \omega_{\ell^\-}\bigl((4)\bigr)^\-, \omega_{\ell^\-}\bigl((4)\bigr)^\+ \bigr)
= \bigl((1,1), (2)\bigr)$,
and from Definition~\ref{defn:AStatistics} we have
$A^\-\bigl((4)\bigr) = 2$, $A^\+\bigl((4)\bigr) = 4$ 
and hence Definition~\ref{defn:exceptionalColumnBound} gives
\begin{align*} E^\- &= 1 . 2 + 1 . 4 - |\lambda^\-| = 6 -|\lambda^\-|, \\ 
E^\+ &= 1 . 4 + 1 . (4+3) -
\bigl( |\lambda^\-| + \lambda^\+_1 \bigr) =
11 - |\lambda^\-| - \lambda^\+_1.
\end{align*}
Therefore if $\lambda$ is a partition of $12$ having exactly $k$ parts of size at least $2$, 
\begin{align*} E_{1,((2,2),(3,1))} &\bigl((2,1), (4) : \lambda \bigr) \\ &= 
17 - 2|\lambda^\-| - (\lambda_1 - 2) + \lambda^\+_3 + \cdots 
\\ &= 17 - 2|\lambda^\-| - (\lambda_1 - 2) + (\lambda_3-2) + \cdots + (\lambda_k - 2),
\end{align*}
or $0$ is this is negative.
We denote this quantity by $E$ as usual.

The plethystic semistandard signed tableaux relevant to the plethysm coefficients
$\langle s_{(2,1)+M(1,1)} \circ s_{(4)}, s_{\lambda \oplus M((2,2), (3,1)} \rangle$
for $M \in \N_0$ lie 
in the set $\PSSYTwk{((2,2), (3,1))}{(2,1) + M(1^3)}{(4)}{\pi^\-}{\pi^\+}$.
Since $\nu_R = \nu_2 = 1$,
the weaker bound from Proposition~\ref{prop:signedWeightBoundForStronglyMaximalWeight}
is 
that if $M \ge E-1$ and $(\pi^\-,\pi^\+) \unrhd (\lambda^\-,\lambda^\+) + M\bigl((2,2),(3,1)\bigr)$
then
\[
 (\pi^\-,\pi^\+) \unlhd (1 + 2E) \bigl( (1,1), (2) \bigr) + (1-E + M) \bigl( (2,2), (3,1) \bigr)
\]
in the $2$-signed dominance order.
In the first case in the earlier Example~\ref{ex:exceptionalColumns22and31}
when $\lambda = (8,3,1)$ with $2^\-$-decomposition $\dec{(3,2)}{(6,1)}$ we have
Then $E^\- = 1$, $E^\+ = 0$ and $E = 1$ and so the upper bound
in the $2$-signed dominance order is that
if $(\pi^\-,\pi^\+) \unrhd \bigl((3,2), (6,1)\bigr) + M\bigl((2,2), (3,1)\bigr)$ then
\begin{align*}
(\pi^\-,\pi^\+) &\unlhd 3\bigl( (1,1), (2)\bigr) + M \bigl((2,2), (3,1) \bigr) \\
&= \bigl( (3,3), (6) \bigr) + M\bigl((2,2), (3,1)\bigr) 
\end{align*}
for $M \ge 0$.
If instead
$\lambda = (6,3,3)$ with $2^\-$-decomposition $\dec{(3,3)}{(4,1,1)}$ then
$E^\- = 0$, $E^\+ = 1$ but because the proof of Lemma~\ref{lemma:boundOnExceptionalColumns}
allows, as it must in general, one exceptional column for each 
integer entry in $\{3,4,\ldots \}$, of which there are $\lambda^\+_3 = 1$, we have 
$E = 0+1+1 = 2$ and the upper bound
in the $2$-signed dominance order is that
if $(\pi^\-,\pi^\+) \unrhd \bigl((3,3), (4,1,1)\bigr) + M\bigl((2,2), (3,1)\bigr)$ then
\begin{align*}
(\pi^\-,\pi^\+) &\unlhd (1 + 2.2)\bigl( (1,1), (2)\bigr) + (1-2+M) \bigl((2,2), (3,1) \bigr) \\
&= \bigl( (3,3), (7, -1) \bigr) + M\bigl((2,2), (3,1)\bigr)
\end{align*}
for $M \ge 1$, where just for this inequality,
 to facilitate comparison, we allow a negative entry in what
would normally be a signed weight.
Note in each case that the upper bound is conditional on the lower bound, as we saw
is necessary in Example~\ref{ex:stronglyMaximalSignedWeightSWLi411}.
\end{example}

We conclude this example in Example~\ref{ex:22and31final}
in which small columns are also relevant.

\subsection{Signed weight bound in the $\ell$-twisted dominance order}
We are now almost ready to prove Corollary~\ref{cor:signedWeightBoundForStronglyMaximalSignedWeight};
it is the analogue of 
Proposition~\ref{prop:twistedWeightBoundInner} and Corollary~\ref{cor:twistedWeightBoundInnerGrowing}.
First though we must address the technical point that to apply Corollary~\ref{cor:signedIntervalStable} when $\ell^\- \not= 0$,
we require an $(\ell(\kappa^\-)+1, \ell(\kappa^\+)\bigr)$-large partition as the upper bound.

\begin{lemma}\label{lemma:omegaKappaEllDecomposition}
Let $\swtp{\kappa}$ be a strongly maximal signed weight. 
Fix $\ell^\- = \ell(\kappa^\-)$. Let $\muS$ be a skew partition. Given
any $\FC \in \N$ and $\GD \in \N_0$, the pair
\[ \FC\decs{\kappa^\-}{\kappa^\+}  + \GD\dec{\omega_{\ell^\-}(\muS)^\-}{\omega_{\ell^\-}(\muS)^\+} \]
is the $\ell^\-$-decomposition of the $(\ell^\-+1, \ell(\kappa^\+)\bigr)$-large partition
\[ \kappa \oplus (\FC-1)\swtp{\kappa} \oplus \GD\bigl( \omega_{\ell^\-}(\muS)^\-,\omega_{\ell^\-}(\muS)^\-
\bigr)
\]
where $\kappa$ is the partition with $\ell^\-$-decomposition $\decs{\kappa^\-}{\kappa^\+}$.
\end{lemma}

\begin{proof}
Proposition~\ref{prop:stronglyMaximalSignedWeightsAreEllDecompositions} states that
$\decs{\kappa^\-}{\kappa^\+}$ is the $\ell^\-$-decomposition of 
an $\bigl(\ell^\- + 1, \ell(\kappa^\-)\bigr)$-large partition and so~$\kappa$ is well-defined. By Remark~\ref{remark:ellDecompositionLarge},~$\kappa$
is $\bigl(\ell(\kappa^\-),\ell(\kappa^\+)\bigr)$-large. 
Since the parts of $\dec{\omega_{\ell^\-}(\muS)^\-}{\omega_{\ell^\-}(\muS)^\+}$
are partitions by Lemma~\ref{lemma:greatestSignedWeight}, the
same holds for
\[ \FC\dec{\kappa^\-}{\kappa^\+} + \GD\dec{\omega_{\ell^\-}(\muS)^\-}{\omega_{\ell^\-}(\muS)^\+}. \]
By Lemma~\ref{lemma:adjoinToLarge} and the following
remark, this pair is the $\ell^\-$-decomposition
of the partition 
$\kappa \oplus (\GD-1)\swtp{\kappa} \oplus \GD\bigl( \omega_{\ell^\-}(\muS)^\-,\omega_{\ell^\-}(\muS)^\-
\bigr)$. This partition is $\bigl(\ell^\- + 1, \ell(\kappa^\+)\bigr)$-large because $\kappa$ is.
\end{proof}

Since, by part of Remark~\ref{remark:ellDecompositionLarge} we have $\varnothing \oplus (\kappa^\-,\kappa^\+) = \kappa$, one could replace $\kappa \oplus (\FC-1)\decs{\kappa^\-}{\kappa^\+}$
in the statement of this lemma with $\varnothing \oplus \FC\decs{\kappa^\-}{\kappa^\+}$;
we prefer the lemma as stated, since then it does not depend on the technical point
that we join first in the adjoin operation defined in~\eqref{eq:oplus}.

As some motivation for the definition of $\omega$ in
the following corollary,
 we recall from Example~\ref{ex:stronglyMaximalSignedWeightSWLi411}
that to get suitable upper bounds
in a stable partition system when $\lambda = (4,1,1)$
and $\swtp{\kappa} = \bigl(\varnothing, (4,1,1) \bigr)$ we
had to begin with $\lambda \oplus \swtp{\kappa}$, rather than~$\lambda$,
because there could be a single exceptional column.
The partition $\kappa$
is well defined by Proposition~\ref{prop:stronglyMaximalSignedWeightsAreEllDecompositions}.
%

\newcounter{savedthm}\setcounter{savedthm}{\value{theorem}}

\begin{corollary}[Outer Twisted Weight Bound]
\label{cor:signedWeightBoundForStronglyMaximalSignedWeight}
Let $\kappa$ be a strongly $c^\+$-maximal weight of shape $\muS$, size $R$ and sign $+1$.
Fix $\ell^\- = \ell(\kappa^\-)$.
Let $\nu$ be a partition and set $\nu^{(M)} = \nu + (M^R)$ for $M \in \N_0$.
Let~$\lambda$ be an $\bigl(\ell^\-,\ell(\kappa^\+)\bigr)$-large partition
of $|\muS||\nu|$.
Let $\kappa$ be the unique partition with $\ell^\-$-decomposition 
$\decs{\kappa^\-}{\kappa^\+}$.
Set $E = E_{c^\+\hskip-1pt,(\kappa^\-,\kappa^\+)}(\nu, \muS : \lambda)$.
Define  
\[ \omega = \begin{cases} \kappa \oplus (B_R(\nu) + ER) \bigl( \omega_{\ell^\-}(\muS)^\-,\omega_{\ell^\-}(\muS)^\+
\bigr) \\ 
\kappa \oplus (B_R(\nu) + ER) \bigl( \omega_{\ell^\-}(\muS)^\-,\omega_{\ell^\-}(\muS)^\+
\bigr) \oplus (\nu_R - E - 1) \swtp{\kappa} \end{cases} \] 
choosing the case according to whether $E \ge \nu_R$ or $E < \nu_R$.
Then $\omega$ is an $\bigl(\ell^\-+1,\ell(\kappa^\+)\bigr)$-large partition
of 
\[ \begin{cases}  |\lambda| + (E-\nu_R + 1)Rm & \text{if $E \ge \nu_R$} \\
 |\lambda| & \text{if $E < \nu_R$.} \end{cases} \]
Suppose that $\sigma \unrhddot \lambda \oplus M\bigl(\kappa^\-,\kappa^\+\bigr)$
in the $\ell^\-$-twisted dominance order.
If~$s_\sigma$ is a constituent of the plethysm $s_{\nu^{(M)}} \circ s_\muS$
then
\[ \sigma \unlhddot  \begin{cases} \omega \oplus \bigl(-(E-\nu_R + 1) + M\bigr) \swtp{\kappa}
& \text{if $E \ge \nu_R$} \\
\omega \oplus M\swtp{\kappa} & \text{if $E < \nu_R$} \end{cases} \]
for all $M \in \N_0$ such that $M > E - \nu_R$.
\end{corollary}

\begin{proof}
Let $m = |\muS|$ and let $n = |\nu|$. By Definition~\ref{defn:BBoxes} we have
\[ B_R(\nu) + ER = \begin{cases} |\nu| + (E - \nu_R)R & \text{if $E \ge \nu_R$} \\
 |\nu| - (\nu_R - E)R & \text{if $E < \nu_R$.} \end{cases} \]
Hence, using that $|\kappa^\-| + |\kappa^\+| = Rm$ and $|\omega_{\ell^\-}(\muS)^\-| + |\omega_{\ell^\-}(\muS)^\+|
= m$, we obtain
\[ |\omega| = \begin{cases} Rm + (|\nu|  + (E-\nu_R)R)m  & \text{if $E \ge \nu_R$}
\\ Rm - (\nu_R - E)Rm + mn + (\nu_R - E - 1)Rm = mn  & \text{if $E < \nu_R$}\end{cases} \]
which, since $|\lambda| = |\nu|Rm$, shows that the size of $\omega$ is as claimed. By Lemma~\ref{lemma:omegaKappaEllDecomposition},
$\omega$ is $\bigl(\ell(\kappa^\-)+1,\ell(\kappa^\+)\bigr)$-large.
Note 
that, by Lemma~\ref{lemma:omegaKappaEllDecomposition},
\begin{equation}
\label{eq:omegaDecomposition}\begin{split} \decs{\omega^\-}{\omega^\+} = \kappa &+
\KK\decs{\kappa^\-}{\kappa^\+} 
\\ &\quad + (B_R(\nu) + ER)\dec{\omega_{\ell^\-}(\muS)^\-}{\omega_{\ell^\-}(\muS)^\+} \end{split}
\end{equation}
where $\KK = 0$ if $E \ge \nu_R$ and $\KK = \nu_R - E - 1$ otherwise.
We now follow part of the proof of
Proposition~\ref{prop:twistedWeightBoundInner}.
By Lemma~\ref{lemma:twistedKostkaMatrix}, $s_\sigma$ is a direct summand of $e_{\sigma^\-}h_{\sigma^\+}$.
Hence, by Proposition~\ref{prop:plethysticSignedKostkaNumbers}
we have
\[ |\PSSYT\bigl(\nu + M(1^R), \muS\bigr)_{(\sigma^\-,\sigma^\+)}\bigr| = 
\langle e_{\sigma^\-}h_{\sigma^\+}, s_{\nu +M(1^R)} \circ s_\muS \rangle \ge 1. \]
By hypothesis $\sigma \unrhddot \lambda \oplus M\bigl(\kappa^\-,\kappa^\+\bigr)$.
Hence, by the definition of the $\ell^\-$-twisted dominance
order in Definition~\ref{defn:ellTwistedDominanceOrder},
the largeness assumption on $\lambda$ and Lemma~\ref{lemma:adjoinToLarge}, we have
\[ \decs{\sigma^\-}{\sigma^\+} \unrhd \decs{\lambda^\-}{\lambda^\+} + M\swtp{\kappa}. \]
Suppose that $E \ge \nu_R$.
Then applying the weaker second bound 
in Proposition~\ref{prop:signedWeightBoundForStronglyMaximalWeight} 
to the hypothesis $\PSSYTwk{\kappa}{\nu + M(1^R)}{\muS}{\sigma^\-}{\sigma^\+} \not=\varnothing$
we obtain
\begin{align*} \swtp{\sigma} &\unlhd (\kappa^\-,\kappa^\+) +  \bigl(B_R(\nu) + ER \bigr)
\bigl(\omega_{\ell^\-}(\muS)^\-,\omega_{\ell^\-}(\muS)^\+\bigr)  \\
&\hspace*{2in} + \bigl( \nu_R - E + M-1\bigr) \swtp{\kappa} \\
&= (\omega^\-,\omega^\+) + \bigl( \nu_R - E + M-1\bigr) \swtp{\kappa}
\end{align*}
for $M > E - \nu_R$.
By~\eqref{eq:omegaDecomposition} this is equivalent to
\[ \sigma \,\unlhddot\, \omega \oplus \bigl( \nu_R - E + M-1\bigr) \swtp{\kappa}, \]
as required. The proof in the remaining case $E < \nu_R$ is entirely analogous,
now using Proposition~\ref{prop:signedWeightBoundForStronglyMaximalWeight} to get
\begin{align*} \swtp{\sigma} &\unlhd (\nu_R - E)\swtp{\kappa} + \bigl(B_R(\nu) + ER \bigr)
\dec{\omega_{\ell^\-}(\muS)^\-}{\omega_{\ell^\-}(\muS)^\+} \\ &
\hspace*{3.4in}  + M\swtp{\kappa} \end{align*}
where again, since we are in the case $\nu_R -E \ge 1$,
the first two summands are the $\ell^\-$-decomposition of the partition $\omega$.
\end{proof}
%

\begin{remark}\label{remark:singletonSimpler}
If $R=1$ then $E=0$ and, as we show in \S\ref{subsec:singletonStronglyMaximalSignedWeights},
the conclusion of
Corollary~\ref{cor:signedWeightBoundForStronglyMaximalSignedWeight}
reduces to \smash{$\sigma \unlhd \omega^{(n)}_{\ell^\-}(\muS)$}
where \smash{$\omega^{(n)}_{\ell^\-}(\muS)$}
is the partition
with $\ell^\-$-decomposition
\smash{$n\langle\omega_{\ell^\-}(\muS)^\-, \omega_{\ell^\-}(\muS)^\+\rangle$}
 defined in Definition~\ref{defn:plethysticGreatestSignedWeight}.
The special case of the corollary is therefore logically
equivalent to Proposition~\ref{prop:twistedWeightBoundInner}. 
This should be expected, because as seen in Example~\ref{ex:singletonAdapted}
using Lemma~\ref{lemma:maximalAndStronglyMaximalSingletonSemistandardSignedTableauFamilies}, 
the unique strongly maximal signed weight of shape $\muS$
and size $1$ is the signed weight of the greatest
semistandard signed tableau $t_{\ell^\-}(\muS)$, namely
\smash{$(\omega_{\ell^\-}(\muS)^\-, \omega_{\ell^\-}(\muS)^\+)$}.
\end{remark}

The upper bound in Corollary~\ref{cor:signedWeightBoundForStronglyMaximalSignedWeight}
is usually far from tight.

\begin{example}\label{ex:omegaBound411}
We finish our first `unsigned' running example 
(see Examples~\ref{ex:stronglyMaximalSignedWeightSWLii411}, \ref{ex:stronglyMaximalSignedWeightSWLi411},
\ref{ex:exceptionalColumns411} and~\ref{ex:exceptionalColumnsGeneralWeight411}),
with the strongly $1$-maximal signed weight
$\bigl(\varnothing, (4,1,1)\bigr)$ of shape $(2)$ and size $3$,
so $\muS = (2) / \varnothing$.
As in
Example~\ref{ex:exceptionalColumnsGeneralWeight411}
we take $\nu = (2,1) + \KC(1,1,1)$ and $\lambda = (4,2) + \KC(4,1,1)$.
We saw in this example that $B_R(\nu) = B_3\bigl((2,1) + \KC(1,1,1)\bigr) = 3$
and $E_{1,(\varnothing,(4,1,1))}\bigl( (2,1)+\KC(1,1,1), (2) 
; (4,2) + \KC(4,1,1) \bigr) = 2$ for all $\KC \in \N_0$.
Since $\kappa = \kappa^\+ = (4,1,1)$ we have  $\ell(\kappa^\-) = 0$.
Since
 $\omega_0\bigl((2)\bigr)^\- = \varnothing$ and $\omega_0\bigl((2)\bigr)^\+ = (2)$,
the partition~$\omega$ in 
Corollary~\ref{cor:signedWeightBoundForStronglyMaximalSignedWeight} is
therefore
\[ \omega =
\begin{cases} (4,1,1) + (3 + 2. 3)(2) \\
 (4,1,1) + (3 + 2. 3)(2)
+ (\KC-3)(4,1,1)\end{cases} \]
choosing the case according to whether $2 \ge \KC$ or $2 < \KC$.
Remembering that the $0$-twisted dominance order is simply the usual dominance order,
the conclusion of the corollary is that if $\sigma \unrhd (4,2) + 
\KC(4,1,1) + M(4,1,1)$ 
and $\PSSYT_\kappa\bigl((2,1) + (\KC+M)(1,1,1) , (2) \bigr)_{(\varnothing,\sigma)} \not= \varnothing$ then
\begin{align}\label{eq:upperBound411} 
\sigma &\unlhd \begin{cases} (4,1,1) + 9(2) + (\KC-3+M)(4,1,1) \\
(4,1,1) + 9(2) + (\KC-3)(4,1,1) + M(4,1,1) \end{cases} \nonumber \\
&= (22,1,1) + (\KC+M-3)(4,1,1) \end{align}
for all $M > 2 - \KC$. (The unification of the two
cases is expected for the same reason mentioned in Example~\ref{ex:exceptionalColumnsGeneralWeight411} 
that columns of signed weight $\bigl( \varnothing, (4,1,1)\bigr)$ are typical.)
We verify this bound directly when $\KC =0$. Recall from after Definition~\ref{defn:signedWeightTableau}
that $\wt(t)$ is the positive part of the 
signed weight of a tableau having only positive integer
entries. We saw in Example~\ref{ex:exceptionalColumnsGeneralWeight411} 
that if $\pi^\+ \unrhd (4,2) + M(4,1,1)$ then
there is at
most one non-typical column in any $T \in \PSSYT_\kappa\bigl((2,1) + (M,M,M) , (2) \bigr)_{(\varnothing,
\pi^\+)}$, and this column is large-exceptional. Hence $T$ has the form
\[ \begin{tikzpicture}[x=1.2cm,y=-0.75cm]
 \pyoungInner{ {{\young(11)}, {\young(12)}, {\young(13)}} }
 \draw (2,1)--(2.5,1); \draw (3,1)--(3.5,1);
 \draw (2,4)--(2.5,4); \draw (3,4)--(3.5,4);
 \node at (2.75,2.5) {$\ldots$};
\begin{scope}[xshift=3cm,yshift=0cm]
 \pyoungInner{ {{\young(11),\young(11)}, {\young(12),\young(12)}, {\young(13),$t$}} }
\end{scope} 
\begin{scope}[xshift=5.4cm,yshift=0cm]
 \pyoungInner{ {{\young(11),$u$}, {$v$}} }
\end{scope} 
   \end{tikzpicture}
\]
where, by counting $1$s as in Example~\ref{ex:exceptionalColumnsGeneralWeight411}, we require
$4M+1 + \wt(t)_1 + \wt(u)_1 + \wt(v)_1 \ge 4M+4$.
Similarly, by counting entries in $\{1,2\}$ and $\{1,2,3\}$, we obtain
the necessary and sufficient condition
$\wt(t) + \wt(u) + \wt(v) \unrhd (3, 2, 1)$. If there 
is an exceptional column then $t = \young(22)$ and $u = \young(11)$ and \emph{either}
$v = \young(12)$ \emph{or}
$v = \young(13)$.
If $v = \young(12)$ then $T$ has weight $(M-1)(4,1,1) + (3,3) + (5,1) = (16,6,2) + (M-3)(4,1,1)$
and if $v = \young(13)$ then, very similarly $T$ has weight $(16,5,3) + (M-3)(4,1,1)$.
Otherwise there are two cases:
\begin{itemize}
\item[(a)] $t = \young(13)$ and either $u = \young(12)$\,, $v = \young(12)$ or $u = \young(11)$\,, $v = \young(22)$\,;\smallskip
\item[(b)] $t = \young(23)$ and $u = \young(11)$ and $v = \young(12)$\,;
\end{itemize}
in which, once again, 
$T$ has weight $(M-1)(4,1,1) + (4,1,1) + (4,2) = (M-1)(4,1,1) + (3,2,1) + (5,1) = (16,5,3) + (M-3)(4,1,1)$.
Thus, as expected from the remark before this example,
the upper bound~\eqref{eq:upperBound411} is easily met.
\end{example}

See Example~\ref{ex:411bound} 
for the stable plethysm from the example above.
For a further `signed' example of Corollary~\ref{cor:signedWeightBoundForStronglyMaximalSignedWeight},
used in the context of the proof of Theorem~\ref{thm:nuStable},
see Example~\ref{ex:22and31final}, which continues
the running example in Examples~\ref{ex:exceptionalColumns22and31} and~\ref{ex:22and31bound}.

\subsection{Results for both signs}\label{subsec:negativeSignAndCombined}
We now give the analogous definitions and a combined final result applicable
to strongly maximal signed weights of either sign.
We have already defined exceptional rows in 
Definition~\ref{defn:exceptionalColumnAndRow}.
In the following definition $B^\+_R(\nu)$ is the same as $B_R(\nu)$ 
defined in Definition~\ref{defn:BBoxes}.

\setcounter{theoremBoth}{\value{defn:BBoxes}}
\begin{definitionBoth}\label{defnBoth:BBoxes}
\textbf{(Both signs.)}
Given a partition $\nu$, $R \in \N$ and a sign $\pm 1$, 
let $B^\+_R(\nu)$ be the number of boxes $(i,j)$ of $[\nu]$
such that either $i > R$ or $\nu_j' < R$ and
let $B^\-_R(\nu)$ be the number of boxes $(i,j)$ of $[\nu]$
such that either $j > R$ or $\nu_i < R$.
\end{definitionBoth}

In the following definition
$E_{c^\+\hskip-1pt,(\kappa^\-,\kappa^\+)}(\nu, \muS : \lambda)$
is as already defined in Definition~\ref{defn:exceptionalColumnBound}
if $\swtp{\kappa}$ has sign $+1$.
For ease of reference we recall
from Definition~\ref{defn:AStatistics}
that we have defined \smash{$\AmuSm = |\omega_{\ell^\-}(\muS)^\-|$} and 
\smash{$\AmuSp = |\omega_{\ell^\-}(\muS)^\-|$ + $\sum_{i=1}^{c^\+} \omega_{\ell^\-}(\muS)^\+_i$}.

\setcounter{theoremBoth}{\value{defn:exceptionalColumnBound}}
\begin{definitionBoth}\label{defnBoth:exceptionalColumnBoundBoth}
\textbf{(Both signs.)}
Let $\nu$ be a partition.
Let $(\kappa^\-, \kappa^\+)$ be a strongly
$c^\+$-maximal signed weight of shape $\muS$ and size $R$
where $R \ge 2$.
Let $\lambda$ be a partition of $|\nu||\muS|$. 
Fix $\ell^\- = \ell(\kappa^\-)$.
Set $\nu_R^\+ = \nu_R$ and $\nu_R^\- = \nu'_R$. Define
\begin{align*}
E^\-  &= \textstyle
B^\pm_R(\nu)\AmuSm  + \nu_R^\pm  |\kappa^\-| 
- |\lambda^\-|, 
\\
E^\+ &=  \textstyle
B^\pm_R(\nu)\AmuSp + \nu_R^\pm \bigl( |\kappa^\-| + \sum_{i=1}^{c^\+} \kappa^\+_i
\bigr) -|\lambda^\-| -\sum_{i=1}^{c^\+} \lambda^\+_i
\end{align*}
where the sign in the four appearances of $\pm$ is the sign of $\swtp{\kappa}$.
Define 
\[ \textstyle E_{c^\+\hskip-1pt,(\kappa^\-,\kappa^\+)}(\nu, \muS : \lambda) = \max\bigl( E^\- + E^\++
\sum_{i=\ell(\kappa^\+)+1}^{\ell(\lambda^\+)} \lambda^\+_i, 0\bigr).\]
If $R =1$ we instead define $E_{c^\+\hskip-1pt,(\kappa^\-,\kappa^\+)}(\nu, \muS : \lambda)=0$.
\end{definitionBoth}

Again we remind the reader that the partition $\kappa$ in the following corollary
is well-defined by Proposition~\ref{prop:stronglyMaximalSignedWeightsAreEllDecompositions}.

\setcounter{theoremBoth}{\value{savedthm}}
\begin{corollaryBoth}[Outer Twisted Weight Bound]\label{cor:twistedWeightBoundForStronglyMaximalWeightCombined}
\textbf{\emph{(Both signs.)}}
 Let $\nu$ be a partition.
Let $\swtp{\kappa}$ be a strongly $c^\+$-maximal signed weight of shape $\muS$ and size~$R$.
Fix $\ell^\- = \ell(\kappa^\-)$.
Let $\lambda$ be an $\bigl(\ell^\-,\ell(\kappa^\+)\bigr)$-large partition
of $|\nu||\muS|$.
Set $\nu_R^\+ = \nu_R$ and $\nu_R^\- = \nu'_R$. 
Let 
\begin{align*}
\nu^{(M)} = &\begin{cases} \nu + M(1^R) & \text{if $\swtp{\kappa}$ has sign $+1$} \\
    \nu \sqcup (R^M) & \text{if $\swtp{\kappa}$ has sign $-1$.}\end{cases}\end{align*}
Set $ E = E_{c^\+\hskip-1pt,(\kappa^\-,\kappa^\+)}(\nu, \muS : \lambda)$.
Suppose that $\pi \unrhddot \lambda \oplus M\swtp{\kappa}$ in the $\ell^\-$-twisted
dominance order and that 
\smash{$T \in \PSSYT_\kappa(\nu^{(M)}, \muS)_{(\pi^\-,\pi^\+)}$}.
Throughout  $\pm$ is the sign of $\swtp{\kappa}$.
\begin{thmlist}
\item If $\swtp{\kappa}$ has sign $+1$ then
$T$ has at most $E$ 
exceptional columns and if $\swtp{\kappa}$ has sign $-1$ then~$T$ 
has at most $E$ 
exceptional rows.
\item Let $M \ge -\nu_R^\pm + E$. If $T$ has
 $\nSmall$ small columns  whose
top $R$ boxes have signed weights $\swtp{\phi_1}, \ldots, \swtp{\phi_\nSmall}$ 
\emph{(}when $\swtp{\kappa}$ has sign $+1$\emph{)}
or $\nSmall$ small rows whose
leftmost $R$ boxes have
signed weights $\swtp{\phi_1}, \ldots, \swtp{\phi_\nSmall}$ \emph{(}when $\swtp{\kappa}$ has sign $-1$\emph{)}
then
\begin{align*} \swtp{\pi} \,&\unlhd\, \bigl(B^\pm_R(\nu) + ER \bigr)
\bigl( \omega_{\ell^\-}(\muS)^\-, \omega_{\ell^\-}(\muS)^\+ \bigr) \\
&\quad + \swtp{\phi_1} + \cdots + \swtp{\phi_\nSmall} + 
\bigl( \nu^\pm_R - \nSmall - E + M \bigr) \swtp{\kappa}
\intertext{in the $\ell^\-$-signed dominance order 
for $M \ge -\nu^\pm_R + \nSmall + E$ and the weaker bound}
\swtp{\pi} &\unlhd \bigl(B^\pm_R(\nu) + ER \bigr)\bigl( \omega_{\ell^\-}(\muS)^\-, \omega_{\ell^\-}(\muS)^\+ \bigr) \\ &\hspace*{1.94in} +
\bigl( \nu^\pm_R - E + M \bigr) \swtp{\kappa} \end{align*}
for $M \ge - \nu^\pm_R +E$ also holds.
\item Let $\kappa$ be the unique partition with $\ell^\-$-decomposition 
$\decs{\kappa^\-}{\kappa^\+}$.
Define
\[ \omega = \begin{cases} \kappa \oplus (B^\pm_R(\nu) + ER) \bigl( \omega_{\ell^\-}(\muS)^\-,\omega_{\ell^\-}(\muS)^\+
\bigr) \\ 
\kappa \oplus (B^\pm_R(\nu) + ER) \bigl( \omega_{\ell^\-}(\muS)^\-,\omega_{\ell^\-}(\muS)^\+
\bigr) \oplus (\nu_R^\pm - E - 1) \swtp{\kappa} \end{cases} \] 
choosing the case according to whether $E \ge \nu_R^\pm$ or $E < \nu_R^\pm$.
Then~$\omega$ is an $\bigl(\ell^\-+1, \ell(\kappa^\+)\bigr)$-large partition
of  $|\lambda| + SR|\muS|$ where
\[ S = \begin{cases} E - \nu_R^\pm + 1 & \text{if $E \ge \nu_R^\pm$} \\
0 & \text{if $E< \nu_R^\pm$.} \end{cases} \]
Suppose that $\sigma \unrhddot \lambda \oplus M\bigl(\kappa^\-,\kappa^\+\bigr)$
in the $\ell^\-$-twisted dominance order.
If $s_\sigma$ is a constituent of the plethysm $s_{\nu^{(M)}} \circ s_\muS$
then
\[ \sigma \unlhddot 
\begin{cases} \omega \oplus{}  \bigl(-(E - \nu_R^\pm + 1) + M\bigr) \swtp{\kappa}
& \text{if $E \ge \nu_R^\pm$} \\
\omega \oplus{}  M\swtp{\kappa} & \text{if $E < \nu_R^\pm$} \end{cases} \]
for all $M \in \N_0$ such that $M > E - \nu_R^\pm$. 
\end{thmlist}
\end{corollaryBoth}

\begin{proof}
In (iii) the size of $\omega$ and that $\omega$ is $\bigl(\ell^\-+1, \ell(\kappa^\+)\bigr)$-large
follow exactly 
as in Corollary~\ref{cor:signedWeightBoundForStronglyMaximalSignedWeight}; note that  the
sign of $(\kappa^\-,\kappa^\+)$ is irrelevant to this first part of the proof.
Part (i) is proved in Corollary~\ref{cor:boundOnExceptionalColumns}  when $\swtp{\kappa}$ has sign $+1$;
the proof is precisely analogous for sign $-1$, using the modified definitions above and the
obvious modifications of Lemma~\ref{lemma:boundOnExceptionalColumns} and
Remark~\ref{remark:singletonExceptional}. 
When the sign is $+1$,
(ii) is the stronger bound in Proposition~\ref{prop:signedWeightBoundForStronglyMaximalWeight},
and (iii) is Corollary~\ref{cor:signedWeightBoundForStronglyMaximalSignedWeight}.
Again in all three cases the proof is precisely analogous for sign $-1$.
\end{proof}

\section{Proof of Theorem~\ref{thm:nuStable}}\label{sec:nuStable}
In this section we prove Theorem~\ref{thm:nuStable} using the Signed Weight Lemma (Lemma~\ref{lemma:SWL}).
We begin with the second part of the theorem where the stable multiplicity is zero in
\S\ref{subsec:nuStableZero}.
In \S\ref{subsec:stablePartitionSystemNu} we construct a suitable stable partition system.
Then in \S\ref{subsec:boxMoves} we prove a final preliminary lemma on the length of signed weights,
closely analogous to a well known result on the length of vectors in the Type A root system.
Then finally in \S\ref{subsec:nuStableProof} we prove Theorem~\ref{thm:nuStableSharp}
which restates Theorem~\ref{thm:nuStable} with an explicit bound.

\subsection{The vanishing case of Theorem~\ref{thm:nuStable}}\label{subsec:nuStableZero}
Recall from Definition~\ref{defn:LZBound}
that $\LZBound\bigl( [\lambda, \omega]_\unlhddotS, \swtp{\kappa}, \swtp{\eta}\bigr)$
is defined whenever $\swtp{\eta} \notunlhd \swtp{\kappa}$, and so in particular,
whenever $\swtp{\eta} \unrhd \swtp{\kappa}$. (Here $\unrhd$ is the signed
dominance order of Definition~\ref{defn:ellSignedDominanceOrder}).
See Definition~\ref{defnBoth:exceptionalColumnBoundBoth}
in~\S\ref{subsec:negativeSignAndCombined}
for the definition of $E_{c^\+\hskip-1pt,(\kappa^\-,\kappa^\+)}(\nu, \muS : \lambda)$ in
its `both signs' version.

\begin{proposition}\label{prop:nuStableZero}
Let $\nu$ be a partition.
Let $\swtp{\kappa}$ be a strongly maximal signed weight of size~$R$, shape $\muS$ and 
sign $\epsilon$.
Fix $\ell^\- = \ell(\kappa^\-)$.
Let $\eta^\-$ and $\eta^\+$ be partitions
with $\ell(\eta^\-) \le \ell^\-$ and $|\eta^\-| + |\eta^\+| = |\kappa^\-| + |\kappa^\+|$. 
Let $\ell^\+ = \max\bigl( \ell(\kappa^\+), \ell(\eta^\+) \bigr)$.
Let~$\lambda$ be an $(\ell^\-,
\ell^\+)$-large partition of $|\nu||\muS|$.
Set $\nu^{(M)} = \nu + (M^R)$
if~$\swtp{\kappa}$ has sign $+1$ and $\nu^{(M)} = \nu \sqcups (R^M)$ if $\swtp{\kappa}$ has sign $-1$. 
Set $\nu_R^\+ = \nu_R$ and $\nu_R^\- = \nu_R'$ and let $\pm$ be the sign of $\swtp{\kappa}$.
Set $ E = E_{c^\+,(\kappa^\-,\kappa^\+)}(\nu, \muS : \lambda)$.
If $\swtp{\kappa} \lhd \swtp{\eta}$ then 
\[ \bigl\langle s_{\nu^{(M)}} \circ s_\muS, 
s_{\lambda \opluss M\swtp{\eta}} \bigr \rangle = 0 \]
for all 
\[ M > \begin{cases} \LZBound\bigl([\lambda \oplus (E\!-\! \nu_R^\pm\! +\!1) \swtp{\eta}, 
\omega]_\unlhddotS, \swtp{\kappa},
\swtp{\eta}\bigr) \!+\! (E\!-\!\nu_R^\pm \!+\!1) \\
\LZBound\bigl([\lambda, \omega]_\unlhddotS, \swtp{\kappa}, \swtp{\eta}\bigr) \end{cases}
 \]
choosing the case according to whether $E \ge \nu^\pm_R$ or $E < \nu^\pm_R$,
where $\omega$ is the relevant partition for the two cases taken from
Corollary~\ref{cor:twistedWeightBoundForStronglyMaximalWeightCombined}.
\end{proposition}

\begin{proof}
By Corollary~\ref{cor:signedWeightBoundForStronglyMaximalSignedWeight}(iii), if $s_\sigma$ is a
constituent of $s_{\nu^{(M)}} \circ s_\muS$ such that $\sigma \unrhddot \lambda \oplus M\swtp{\kappa}$
then 
\[ \sigma \unlhddot 
\begin{cases} \omega \oplus{}  \bigl(-(E - \nu_R^\pm + 1) + M\bigr) \swtp{\kappa}
& \text{if $E \ge \nu_R^\pm$} \\
\omega \oplus{}  M\swtp{\kappa} & \text{if $E < \nu_R^\pm$} \end{cases} \]
for all $M \in \N_0$ such that $M > - \nu_R^\pm + E$. 
Since $\swtp{\eta} \rhd \swtp{\kappa}$ we may apply this result taking 
$\sigma = \lambda \oplus M\swtp{\eta}$.
 In the second case, when $E < \nu_R^\pm$, the proposition then follows
by an argument very closely analogous to the proof of Proposition~\ref{prop:muStableZero}; the analogue of~\eqref{eq:sigmaSignedWeightBound} is that
if $\sigma \decMap \decs{\sigma^\-}{\sigma^\+}$ then
\begin{align*}
 (\sigma^\-, \sigma^\+) 
&\unlhd 
(\kappa^\-, \kappa^\+) + (B_R^\pm(\nu) + ER) (\omega_{\ell^\-}(\muS)^\-,
\omega_{\ell^\-}(\muS)^\+) \\ &\hspace*{1.4in} + (\nu_R^\pm - E -1)(\kappa^\-,\kappa^\+)
  + M (\kappa^\-, \kappa^\+) \end{align*}
and so, substituting $(\lambda^\-,\lambda^\+) + M(\eta^\-,\eta^\+)$
for $(\sigma^\-,\sigma^\+)$, as it justified by
Lemma~\ref{lemma:adjoinToLarge} since $(\eta^\-, \eta^\+) \unrhd (\kappa^\-,\kappa^\+)$,
the analogue of~\eqref{eq:signedEtaBound} is
\[ 
\begin{split}
(\lambda^\-, \lambda^\+) + M(\eta^\-, \eta^\+)\, \unlhd\, &
(B_R^\pm(\nu) + ER) (\omega_{\ell^\-}(\muS)^\-,
\omega_{\ell^\-}(\muS)^\+)
\\ & \qquad\qquad\quad + (\nu_R^\pm - E) (\kappa^\-,\kappa^\+)
 + M(\kappa^\- ,\kappa^\+) 
 \end{split}
 \]
where since $E < \nu_R^\pm$, we add $M(\kappa^\-, \kappa^\+)$ to a 
well-defined signed weight. The application of the inequalities
is then exactly as before, giving
a contradiction whenever $M >
\LZBound\bigl([\lambda, \omega]_\unlhddotS, \swtp{\kappa}, \swtp{\eta}\bigr)$.
In the first case, when $E \ge \nu_R^\pm$, we note
that because of the shift in $M$, we now require
\[ - (E-\nu_R^\pm +1) + M\ge 
\LZBound\bigl([\lambda \oplus (E- \nu_R^\pm +1) \swtp{\eta}, \omega]_\unlhddotS, \swtp{\kappa},
\swtp{\eta}\bigr)\] 
where $\omega$ is now defined by the first case in 
Corollary~\ref{cor:signedWeightBoundForStronglyMaximalSignedWeight}.
The argument is otherwise the same.
\end{proof}

\subsection{Stable partition system for Theorem~\ref{thm:nuStable}}\label{subsec:stablePartitionSystemNu}
The following lemma is the analogue of Lemma~\ref{lemma:stablePartitionSystemForMuVarying}.
Again we remind the 
reader that the statistic
 statistic $E_{c^\+\hskip-1pt,(\kappa^\-,\kappa^\+)}(\nu, \muS : \lambda)$ is defined in 
Definition~\ref{defnBoth:exceptionalColumnBoundBoth}. Here we also use
$B_R^\pm(\nu)$ from Definition~\ref{defnBoth:BBoxes}; 
each definition is given in its `both signs' versions in~\S\ref{subsec:negativeSignAndCombined}.
The partition $\kappa$ in the following lemma is well-defined by
Proposition~\ref{prop:stronglyMaximalSignedWeightsAreEllDecompositions}.

\begin{lemma}\label{lemma:stablePartitionSystemForNuVarying}
Let $\nu$ be a partition.
Let $\swtp{\kappa}$ be a strongly $c^\+$-maximal signed weight of shape $\muS$ and size $R$.
Fix $\ell^\- = \ell(\kappa^\-)$.
Let $\lambda$ be an $\bigl(\ell^\-+1,\ell(\kappa^\+)\bigr)$-large partition of $|\nu||\muS|$.
Let $\nu_R^\pm = \nu_R$ if $\swtp{\kappa}$ has sign $+1$ and let $\nu_R^\pm = \nu'_R$ if $\swtp{\kappa}$
has sign $-1$.
Define
\[ \nu^{(M)} = \begin{cases} \nu + M(1^R) & \text{if $\swtp{\kappa}$ has sign $+1$} \\
    \nu \sqcup (R^M) & \text{if $\swtp{\kappa}$ has sign $-1$.}\end{cases} \]
Set $E = E_{c^\+\hskip-1pt,(\kappa^\-,\kappa^\+)}(\nu, \muS : \lambda)$.
Let $\kappa$ be the unique partition with $\ell^\-$-decomposition 
$\decs{\kappa^\-}{\kappa^\+}$.
Define 
\[ \omega = \begin{cases} \kappa \oplus (B^\pm_R(\nu) + ER) \bigl( \omega_{\ell^\-}(\muS)^\-,\omega_{\ell^\-}(\muS)^\+
\bigr) \\ 
\kappa \oplus (B^\pm_R(\nu) + ER) \bigl( \omega_{\ell^\-}(\muS)^\-,\omega_{\ell^\-}(\muS)^\+
\bigr) \oplus (\nu_R^\pm - E - 1) \swtp{\kappa} \end{cases} \]
choosing the case according to whether $E \ge \nu_R^\pm$ or $E < \nu_R^\pm$. 
If $E \ge \nu_R^\pm$ then set $S = -\nu_R^\pm + E + 1$ 
and set $\PSeq{M} = \varnothing$
for $M < S$ and 
\begin{align*} \PSeq{M} &= 
\bigl[\lambda \oplus M\swtp{\kappa}, \omega \oplus ( M-\Shift)\swtp{\kappa} \bigr]_\unlhddotS
\intertext{if $M \ge \Shift$. 
If $E < \nu_R^\pm$ then set $S=0$ and set}
\PSeq{M} &= \bigl[\lambda \oplus M\swtp{\kappa}, \omega \oplus M\swtp{\kappa} \bigr]_\unlhddotS
\end{align*}
for all $M \in \N_0$. 
Then $(\PSeq{M})_{M \in \N_0}$ is a stable partition system with respect
to the map $\pi \mapsto \pi \oplus \swtp{\kappa}$ and the 
twisted symmetric functions $g_\pi = e_{\pi^\-}h_{\pi^\+}$,
stable for $M \ge \Shift + K$ where $K$ 
is the maximum of
\begin{bulletlist}
\item $\LBound\bigl([\lambda^\- + \Shift\kappa^\-,\omega^\-]^{\scriptscriptstyle
\ell^\-}_\unLHDS, \kappa^\-\bigr)$,
\item $\LBound\bigl([\lambda^\+ + \Shift\kappa^\+,\omega^\++(|\lambda^\+| + \Shift|\kappa^\+|
-|\omega^\+|)]_\unlhd, \kappa^\+\bigr)$,
\item $\bigl( \omega_1^\+ + \omega_2^\+ - 2\lambda_1^\+ -2\Shift\kappa_1^\+ + 2|\lambda^\+| +2\Shift|\kappa^\+| 
- 2|\omega^\+|  \bigr)/(\kappa^\+_1-\kappa^\+_2)$,
\item $\bigl( \max( \ell(\lambda^\+), \ell(\kappa^\+) ) + |\omega^\-| - |\lambda^\-|
- \Shift|\kappa^\-| - \omega^\-_{\ell^\-} \bigr)/
\kappa^\-_{\ell^\-}$.
\end{bulletlist}
and zero.
Moreover, if \smash{$\pi \in \PSeq{M}$} and~$s_\sigma$ is a summand of $e_{\pi^\-}h_{\pi^\+}$ appearing
in the plethysm \smash{$s_{\nu^{(M)}}   \circ s_\muS$} then
$\sigma \in \PSeq{M}$.
\end{lemma}

\begin{proof}
By Corollary~\ref{cor:twistedWeightBoundForStronglyMaximalWeightCombined}(iii) in its
`both signs' version in \S\ref{subsec:negativeSignAndCombined}, $\omega$
is an $\bigl(\ell^\-+1,\ell(\kappa^\+)\bigr)$-large partition,
of size $|\lambda| + SR|\muS|$. 
If $E < \nu_R^\pm$ then it is immediate from 
Corollary~\ref{cor:signedIntervalStable} applied  with $\lambda$ and~$\omega$
that the partition system $(\PSeq{M})_{M\in \N_0}$ is stable, and since $\Shift = 0$, the bounds above defining $K$
are exactly the bounds
defined in the statement of this lemma.
If $E \ge \nu_R^\pm$ then $S = E- \nu_R^\pm + 1$ and  we instead
apply the corollary to the partitions $\lambda \oplus S\swtp{\kappa}$ and $\omega$.
Since, by Lemma~\ref{lemma:adjoinToLarge} we have $\bigl(\lambda \oplus \Shift\swtp{\kappa}\bigr)^\- = \lambda^\-
\oplus \Shift\kappa^\-$, and so on, the result from Corollary~\ref{cor:signedIntervalStable}
is that the partition system $(\PSeq{N+S})_{N\in \N_0}$ is stable for $N \ge K$, where again~$K$
is as defined in the statement of this lemma.
Since $K \ge 0$, it follows that 
$(\PSeq{M})_{M\in\N_0}$ is stable for $M \ge S + K$.
(The reader may easily check that Definition~\ref{defn:stablePartitionSystem} permits any finite number of the
sets $\PSeq{M}$ to be empty.)
For the `moreover' part of the result,
first note that by  Lemma~\ref{lemma:twistedKostkaMatrix},
$\sigma \unrhddot \pi$. By 
Corollary~\ref{cor:twistedWeightBoundForStronglyMaximalWeightCombined}(iii)
we have $\sigma \unlhddot \omega \oplus ( M -\Shift) \swtp{\kappa}$.
Hence $\sigma \in \PSeq{M}$.
This completes the proof.
\end{proof}

\subsection{Box moving bound}\label{subsec:boxMoves}

Let $\V_{\ell^\-} \times \V$ be the abelian group generated by the set $\W_{\ell^\-} \times \W$ of
signed weights. Identifying $\bigl( (\alpha^\-_1,\ldots,\alpha^\-_{\ell^\-}), (\alpha_1^\+,\alpha_2^\+,
\ldots) \bigr)$ $\in \V_{\ell^\-} \times \V$ with $\bigl( \alpha_1^\-, \ldots, \alpha^\-_{\ell^\-}, \alpha^\+_1,\alpha^\+_2, \ldots)$
as in the definition of the $\ell^\-$-signed dominance
order on $\W_{\ell^\-} \times \W$ (see Definition~\ref{defn:ellSignedDominanceOrder})
we define  $\epsilon^{(j)} \in \V_{\ell^\-} \times \V$ for each $j \in \N$ by
\[ \epsilon^{(j)}_i = \begin{cases} 1 & \text{if $i=j$} \\ -1 & \text{if $i=j+1$} \\ 0 & \text{otherwise.}
\end{cases} \]
For example if $\ell^\- = 2$ then $\epsilon^{(1)} = \bigl( (1,-1), \varnothing \bigr)$,
$\epsilon^{(2)} = \bigl( (0,1), (-1) \bigr)$, $\epsilon^{(3)} = \bigl( (0,0), (1,-1) \bigr)$, and so on.
Observe that the $\epsilon^{(j)}$ for $j \in \N$ form a $\Z$-basis for the 
subgroup of $\V_{\ell^\-} \times \V$ of elements of
sum~$0$. We say that $(\gamma^\-,\gamma^\+) \in \V_{\ell^\-} \times \V$ is \emph{root-positive}
if, under this identification,
 it is a linear combination of the $\epsilon^{(j)}$ with non-negative coefficients.
Given a root-positive element  $(\beta^\-,\beta^\+) \in \V_{\ell^\-} \times \V$
expressed uniquely in the $\epsilon^{(j)}$ basis as
\[ (\beta^\-,\beta^\+) = \sum_{j \in \N} b_j \epsilon^{(j)}, \]
we define the \emph{root-length} of $(\beta^\-,\beta^\+)$, denoted
\smash{$||(\beta^\-,\beta^\+)||$},
by 
\begin{equation} ||(\beta^\-,\beta^\+)|| = \sum_j b_j. \label{eq:rootLength} \end{equation}
Note the sum in~\eqref{eq:rootLength} is a 
finite sum of non-negative integers.
In practice, it is more often helpful 
to think of each $\epsilon^{(j)}$ as defining a single box move
as in the remark following Example~\ref{ex:62upset}.

\begin{lemma}\label{lemma:normRespectsSignedDominance}
Fix $\ell^\- \in \N$. Let $\pi$, $\sigma$ and $\tau$ be partitions
such that $\pi \unlhddot \sigma \unlhddot \tau$. Then
$\tau - \sigma$ is root-positive and $||\swtp{\tau}-\swtp{\sigma}|| \le ||\swtp{\tau}- \swtp{\pi}||$.
\end{lemma}

\begin{proof}
By Definition~\ref{defn:ellTwistedDominanceOrder} we have $\swtp{\sigma} \unlhd \swtp{\tau}$
in the $\ell^\-$-signed dominance order on $\W_{\ell^\-} \times \W$. 
By Definition~\ref{defn:ellSignedDominanceOrder} this is the usual
dominance order on concatenated weights. Hence by a standard
result on the dominance order familiar from the Type~A root system, 
which also follows from the remark above about single box moves, 
$(\tau^\-- \sigma^\-, \tau^\+-\sigma^\+)$ is root-positive. 
Let
\begin{align*}
(\sigma^\- - \pi^\-, \sigma^\+-\pi^\+) &= \textstyle \sum_{j \in \N} b_j \epsilon^{(j)} \\
(\tau^\- - \sigma^\-, \tau^\+-\sigma^\+) &= \textstyle \sum_{j \in \N} c_j \epsilon^{(j)}\end{align*}
where $b_j, c_j \ge 0$ for each $j$.
Now $(\tau^\- - \pi^\-, \tau^\+-\pi^\+) = \sum_{j \in \N} (b_j + c_j) \epsilon^{(j)}$ 
and since $b_j + c_j \ge c_j$ the remaining claim follows.
\end{proof}


The $\ell^\-$-signed dominance order 
on the set $\W_{\ell^-} \!\times \W$
of signed weights used in the following lemma
is defined in Definition~\ref{defn:ellSignedDominanceOrder}.

\begin{lemma}\label{lemma:positiveLowerBound}
Let $(\phi_i^\-,\phi_i^\+)$ for $1 \le i \le \nSmall$
and $(\kappa^\-,\kappa^\+)$ be signed weights of the same size such that
$\swtp{\phi_i} \lhd \swtp{\kappa}$ in the $\ell^\-$-signed dominance order
for each $i$. Then $\nSmall\swtp{\kappa} - \sum_{i=1}^\nSmall (\phi_i^\-,\phi_i^\+) $
is root-positive and
\[ \bigl|\bigl| \nSmall(\kappa^\-,\kappa^\+) - \sum_{i=1}^\nSmall (\phi_i^\-,\phi_i^\+)  \bigr|\bigr|
\ge \nSmall. \]
\end{lemma}

\begin{proof}
By Lemma~\ref{lemma:normRespectsSignedDominance}, $\swtp{\kappa} -\swtp{\phi}_i$ is root-positive. 
Write $(\kappa^\- - \phi_i^\-,\kappa^\+-\phi_i^\+) = \sum_{j \in \N} b_{ij} \epsilon^{(j)}$
where $b_{ij} \in \N_0$ for all $i$ and $j$. For each $i$, at least one coefficient
$b_{ij}$ is non-zero, and so
\[ \nSmall\swtp{\kappa} - \sum_{i=1}^\nSmall (\phi_i^\-,\phi_i^\+)  = \sum_{j \in \N}  \sum_{i = 1}^\nSmall b_{ij} 
\epsilon^{(j)} \] 
where the  sum of all the coefficients is at least $\nSmall$.
\end{proof}

\subsection{Small columns and rows}
Let $\swtp{\kappa}$ be a strongly $c^\+$-maximal signed weight
of shape $\muS$
and size $R$. We use `/' to distinguish the cases when $\swtp{\kappa}$ has sign $+1$/$-1$.
Recall from Definition~\ref{defn:exceptionalColumnAndRow}
that in each  non-exceptional column/row of a plethystic semistandard signed tableau $T$ 
of inner shape $\muS$
\emph{either}
the top/leftmost $R$ boxes
in the column/row form the plethystic semistandard signed tableau $T_{(\kappa^\-,\kappa^\+)}$,
\emph{or} the column/row is small having signed weight $\swtp{\phi}$ such that $\swtp{\phi} \lhd \swtp{\kappa}$
 in the $\ell(\kappa^\-)$-signed dominance order
on set $\W_{\ell^\-} \!\!\times \W$
defined in Definition~\ref{defn:ellSignedDominanceOrder}.
The bound $E_{c^\+\hskip-1pt,(\kappa^\-,\kappa^\+)}(\nu, \muS : \lambda)$
in the following lemma 
is defined in Definition~\ref{defnBoth:exceptionalColumnBoundBoth} and the statistic $B_R^\pm(\nu)$ is defined in Definition~\ref{defnBoth:BBoxes},
each in their `both signs' version
in~\S\ref{subsec:negativeSignAndCombined}.

\begin{lemma}\label{lemma:smallColumnAndRowBound}
Let $\nu$ be a partition. 
Let $(\kappa^\-,\kappa^\+)$ be a strongly $c^\+$-maximal signed weight of shape $\muS$ 
and size $R$.
Fix $\ell^\- = \ell(\kappa^\-)$. Let $\lambda$ be a
 $\bigl(\ell^\-, \ell(\kappa^\+)\bigr)$-large partition of $|\muS||\nu|$.
Let $E = E_{c^\+\hskip-1pt,(\kappa^\-,\kappa^\+)}(\nu, \muS : \lambda)$.
Let $\nu_R^\pm = \nu_R$ if $\swtp{\kappa}$ has sign $+1$ and let $\nu_R^\pm = \nu'_R$ if $\swtp{\kappa}$
has sign~$-1$.
Let $M \in \N_0$ and set
\[ \nu^{(M)} = \begin{cases} \nu + M(1^R) & \text{if $\swtp{\kappa}$ has sign $+1$} \\
    \nu \sqcup (R^M) & \text{if $\swtp{\kappa}$ has sign $-1$.}\end{cases} \]
Let $T \in \PSSYT_\kappa(\nu^{(M)}, \muS)_{(\pi^\-,\pi^\+)}$ be a $\swtp{\kappa}$-adapted
plethystic
semistandard signed tableau
where $\swtp{\pi} \,\unrhd\, \swtp{\lambda} + M\swtp{\kappa}$
in the $\ell^\-$-signed dominance order on $\W_{\ell^\-} \!\!\times \W$.
Set 
 \[ D\! =\! \bigl|\bigl|  \bigl(B^\pm_R(\nu) + ER\bigr)
 \bigl( \omega_{\ell^\-}(\muS)^\-\!, \omega_{\ell^\-}(\muS)^\+ \bigr) 
 + (\nu_R^\pm - E) \swtp{\kappa}
- \swtp{\lambda} \bigr|\bigr| \]
where in this equation the root-length is of a root-positive element.
If $M \in \N_0$ and $M \ge D + E - \nu_R^\pm$ then 
 $T$ has at most $D$ small columns/rows.
\end{lemma}

\begin{proof}
For ease of notation set $\swtp{\omega} =  \bigl( \omega_{\ell^\-}(\muS)^\-, \omega_{\ell^\-}(\muS)^\+ \bigr)$.
Suppose that $T$ has exactly $\nSmall$ small columns/rows. 
The bound
in the `both signs' version of
Corollary~\ref{cor:twistedWeightBoundForStronglyMaximalWeightCombined}(ii)
states that, in the $\ell^\-$-signed dominance order, $(\pi^\-,\pi^\+)
\unlhd (\sigma^\-,\sigma^\+)$ where
\[
 \swtp{\sigma} =\bigl(B^\pm_R(\nu) + ER\bigr)\swtp{\omega} + 
\sum_{i=1}^\nSmall (\phi^\-_i,\phi^\+_i) + (\nu_R^\pm - E + M - \nSmall) \swtp{\kappa} \]
for $M \ge -\nu_R^\pm + d + E$
and the weaker bound in the same result is that $\swtp{\pi}\unlhd \swtp{\tau}$ where
\[ \swtp{\tau} = \bigl(B^\pm_R(\nu) + ER\bigr)\swtp{\omega} +
(\nu_R^\pm - E + M ) \swtp{\kappa}\]
for $M \ge E - \nu_R^\pm$.
Applying Lemma~\ref{lemma:normRespectsSignedDominance} to $\swtp{\pi} \unlhd
\swtp{\sigma} \swtp{\tau}$ 
 we obtain
\begin{equation}\label{eq:Dupper} \bigl|\bigl|\swtp{\tau} - \swtp{\pi} \bigr|\bigr|
\ge \bigl| \bigl| \nSmall \swtp{\kappa} - \sum_{i=1}^\nSmall \phi_i \bigr|\bigr|
\ge \nSmall .\end{equation}
where the final inequality uses
Lemma~\ref{lemma:positiveLowerBound}.
On the other hand, applying Lemma~\ref{lemma:normRespectsSignedDominance}
to the chain of inequalities
\[ \swtp{\lambda} + M\swtp{\kappa} \unlhd \swtp{\pi} \unlhd
\swtp{\tau} \]
valid for $M \in \N_0$ with $M \ge - \nu_R^\pm + d + E$ we get
\begin{align}
\bigl|\bigl|  &\swtp{\tau} - \swtp{\pi} \bigr|\bigr| \nonumber \\ 
&\le   \bigl|\bigl| \bigl(B^\pm_R(\nu) + ER\bigr)\swtp{\omega} +
(\nu_R^\pm - E + M ) \swtp{\kappa} \nonumber \\
& \hspace*{2.5in} - \bigl( \swtp{\lambda} + M\swtp{\kappa} \bigr) \bigr|\bigr|\nonumber
\\
&=  \bigl|\bigl|  \bigl(B^\pm_R(\nu) + ER\bigr)\swtp{\omega} + (\nu_R^\pm - E) \swtp{\kappa}
- \swtp{\lambda} \bigr|\bigr|. \label{eq:Dlower}
\end{align}
By the same lemma, the right-hand side is the root-length of a root-positive element.
Combining~\eqref{eq:Dupper} and~\eqref{eq:Dlower} we obtain
the required bound $d \le D$, valid for all 
$M \in \N_0$ with $M \ge D + E - \nu_R^\pm$.
\end{proof}

\begin{example}\label{ex:22and31final}
Continuing Examples~\ref{ex:exceptionalColumns22and31} and~\ref{ex:22and31bound}
we take the strongly $1$-maximal signed weight $\bigl( (2,2), (3,1)\bigr)$,
defined by the
maximal plethystic semistandard tableau family
$\bigl\{\, \young(\oM\tM11)\,,\ \young(\oM\tM12)\, \bigr\}$ of 
size $2$ and sign $+1$,
and consider the plethysm coefficients
$\langle s_{(2,1)+M(1,1)} \circ s_{(4)}, s_{\lambda \oplus M((2,2), (3,1))} \rangle$
for $M \in \N_0$. The partition $\kappa$ with $2$-decomposition
$\dec{(2,2)}{(3,1)}$ is $(5,3)$. 
Let $E = E_{1,((2,2),(3,1))}\bigl( (2,1), (4) : \lambda)$.
Note that $B_R(\nu) = B_2\bigl((2,1)\bigr) = 1$.
If $E \ge \nu_2 = 1$ then the first
case for $\omega$ in Corollary~\ref{cor:signedWeightBoundForStronglyMaximalSignedWeight} and
 Lemma~\ref{lemma:stablePartitionSystemForNuVarying}
 applies and 
 so the partition $\omega$ 
 is
\[ \omega = (5,3) \oplus  (1 + 2E) \bigl((1,1),(2) \bigr) = (7+4E, 3, 2^{1+2E}). 
\]
If instead $E = 0$ then the second case for $\omega$ in these results applies,
but now $\nu_R^\pm - E - 1  = 1 - 0 - 1 = 0$, and so 
the second case defines exactly the same partition.
Similarly, in either case, the statistic~$S$ in
 Lemma~\ref{lemma:stablePartitionSystemForNuVarying} is $-\nu_2 + E+ 1 = -1 + E + 1 =  E$
 and so the stable partition system from this lemma is
\[ \PSeq{M} = 
\bigl[\lambda \oplus M\bigl((2,2), (3,1)\bigr), \omega \oplus (M - E) \bigl((2,2), (3,1)\bigr)\bigr]_\unlhddotS \]
for $M \ge E$. As a small check, note that $\lambda$ is a partition of $12$
and $\omega$ is a partition of $12 + 8E$, and so the sizes of the partitions
defining the interval for the $2$-twisted dominance order agree, as they must.
In the remainder of this example we take $\lambda = (8,3,1)$, the first of the cases
in the earlier examples; note that $\lambda$ 
is $\bigl(\ell(\kappa^\-)+1, \ell(\kappa^\+) \bigr) =
(3,2)$-large as required by
Lemma~\ref{lemma:stablePartitionSystemForNuVarying}.
We show all the ideas in the proof of Theorem~\ref{thm:nuStableSharp} 
by checking the conditions for the Signed Weight Lemma (Lemma~\ref{lemma:SWL}).

\subsubsection*{The stable partition system for $\lambda = (8,3,1)$ explicitly}
We saw in Example~\ref{ex:22and31bound}
that $E=1$, and so
$\omega = (11,3,2^3) = (8,2,2) \oplus \bigl((2,2), (3,1)\bigr)$.
It is routine to check using the definition of the $2$-twisted
dominance order in Definition~\ref{defn:ellTwistedDominanceOrder} that
the partition system $(\PSeq{M})_{M \in \N_0}$ 
is
\[ \PSeq{M} = \bigl\{ (8,3,1), (9,2,1), (7,3,2), (8,2,2) \bigr\} \oplus M\bigl((2,2), (3,1)\bigr) \]
for $M \ge 1$,
where the notation indicates that $M\bigl((2,2), (3,1)\bigr)$ is adjoined to all
four partitions in the given set. 
The bounds in Lemma~\ref{lemma:stablePartitionSystemForNuVarying}
are $-2$, $-1$, $-3$ and $-1$ and so $K = \max(-2,-1,-3,-1,0) = 0$.
The shift $S$ in this lemma is $-\nu_R^\pm +E + 1 = -1 + 1 + 1 = 1$.
Therefore, Lemma~\ref{lemma:stablePartitionSystemForNuVarying} states that
$(\PSeq{M})_{M \in \N_0}$ is stable for $M \ge 1$.
Here the bound is tight, since, by the definition
in Lemma~\ref{lemma:stablePartitionSystemForNuVarying}, we have
$\PSeq{0} = \varnothing$.

\subsubsection*{Condition \emph{(ii)} in the Signed Weight Lemma for $\lambda = (8,3,1)$}
We start with~(ii) because the calculations are helpful for (i).
We saw in Example~\ref{ex:exceptionalColumns22and31}
that 
\[ \bigl|\PSSYTwAdapted{(2,1)+M(1,1)}{(4)}{(3+2M,2+2M)}{(6+3M,1+M)}{((2,2),(3,1)}\bigr| = 4 \]
for all $M \ge 1$, giving condition (ii) in the Signed Weight Lemma
(Lemma~\ref{lemma:SWL}) for the partitions obtained by adjoining to $(8,3,1) \decMap \dec{(3,2)}{(6,1)}$.
It is routine to check by similar arguments
using the $2$-decompositions $\dec{(3,2)}{(7)}$, $\dec{(3,3)}{(5,1)}$ and $\dec{(3,3)}{(6)}$
of the three larger partitions in $\PSeq{0}$ that
the corresponding sets of plethystic semistandard signed tableaux for these 
partitions have sizes $1$, $1$ and $0$ for all $M \ge 0$.
However, rather than use this ad-hoc argument, we take the opportunity to motivate the
relevant part of the proof
of Theorem~\ref{thm:nuStableSharp}.
Let $\decs{\pi^\-}{\pi^\+}$ be the $2$-decomposition of one of the four partitions
in $\PSeq{M}$. Then the map defined
by inserting the 
plethystic semistandard signed tableau \marginpar{\qquad\ \scalebox{0.9}{\pyoung{2.0cm}{0.7cm}{ {{\young(\oM\tM11)}, {\young(\oM\tM12)}} }}}
shown in the margin as a new typical first column
 into a plethystic semistandard tableau
in $\PSSYT_{((2,2),(3,1))}\bigl((2,1)+M(1,1),(4)\bigr)_{(\pi^\-,\pi^\+)}$
is surjective, and so bijective, if and only if
every plethystic semistandard signed tableaux in
\smash{$\PSSYT_{((2,2),(3,1))}\bigl((2,1)+(M+1)(2,2),$} \smash{$(4)\bigr)_{(\pi^\-+(3,1),\pi^\++(2))}$}
has at least one typical column, i.e.~one equal to the tableau in the margin. Since $E=1$, each such~$T$ has at most one exceptional column.
By Lemma~\ref{lemma:smallColumnAndRowBound}, with $(\lambda^\-,\lambda^\+) = \bigl((3,2), (6,1)\bigr)$,
$T$ has at most
\begin{align*}
 \bigl|\bigl| (1 + 2.1) \bigl( (1,1), (2) \bigr) +{}& (1-1) \bigl( (2,2), (3,1) \bigr) 
- \bigl((3,2), (6,1) \bigr) \bigr| \bigr|
\\ &= \bigl|\bigl| \bigl( (3,3), (6) \bigr) - \bigl( (3,2), (6,1) \bigr) \bigr| \bigr|
\\ &= \bigl|\bigl| \bigl( (0,1), (0,-1) \bigr) \bigr|\bigr| 
\\ &= \bigl|\bigl| \epsilon^{(2)} + \epsilon^{(3)} \bigr|\bigr| 
\\ &= 2
\end{align*}
small columns. 
By Lemma~\ref{lemma:columnIsExceptionalOrBounded},
every column that is not small or exceptional is typical. Since $E= 1$, there is at most
one exceptional column, and so
the insertion map is surjective for $M \ge 3$. In fact, as seen
in Example~\ref{ex:exceptionalColumns22and31} when $(\pi^\-,\pi^\+) = (8,3,1) + M\bigl((1,1),(2)\bigr)$,
and as follows from the ad-hoc calculation earlier in this paragraph, the insertion map
is surjective for all $M \ge 1$.

As a further illustration we remark that if instead $\lambda = (6,3,3)$ then $E = 2$,
as seen in the second part of Example~\ref{ex:22and31bound}, and the
bound on the number of small columns from Lemma~\ref{lemma:smallColumnAndRowBound}
becomes $\bigl|\bigl| (1+2.2)\bigl((1,1), (2)\bigr) + (1-2)\bigl((2,2), (3,1)\bigr) - \bigl((3,3),(4,1,1)\bigr)\bigr|\bigr|
 = ||(0,0,3,-2,-1)|| = 4$. 

\subsubsection*{Condition \emph{(i)} in the Signed Weight Lemma for $\lambda = (8,3,1)$}
As promised by the final
claim in Lemma~\ref{lemma:stablePartitionSystemForNuVarying},
 if $\sigma \unrhddot (8,3,1) \oplus M\bigl((2,2), (3,1)\bigr)$ 
and $s_\sigma$ appears in $s_{(2,1) + M(1,1)} \circ s_{(4)}$ then
$\sigma$ is one of the four partitions in $\PSeq{M}$;
in fact, since there are no plethystic semistandard signed tableaux of
signed weight 
$\bigl((3,3), (6) \bigr) + M\bigl((1,1),(2)\bigr)$
equal to the $2$-decomposition of the upper bound $(8,2,2) \oplus M\bigl((2,2), (3,1)\bigr)$,
only the first three partitions listed in $\PSeq{M}$ appear.

\subsubsection*{Conclusion}
Using computer algebra one may obtain the constant values of
\smash{$\langle s_{(2,1) + M(1,1)} \circ s_{(4)}, s_{\sigma \oplus M((2,2),(3,1))} \rangle$}
for $\sigma \in \PSeq{0}$; they 
are $2$, $1$,~$1$ and~$0$, attained for $M \ge 1$ when $\sigma = \lambda = (8,3,1)$
 and $M \ge 0$ when $\sigma = (9,2,1), (7,3,2)$ or $(8,2,2)$.
We shall see below in Example~\ref{ex:22and31postTheorem} that
the bound from Theorem~\ref{thm:nuStableSharp} is $M \ge 2$ 
for $(8,3,1)$.

%
%
%
%
%


\end{example}

We mention that since the insertion map inserts a new column of 
height~$R$, it is a new first column if and only if $\ell(\nu) \le R$,
and otherwise it must become a new column $\nu_{R+1}+1$.
This is the main feature of the general positive sign case
not seen in the previous example; in the negative sign case
we instead insert a new row, and a similar remark applies.


\subsection{Proof of Theorem~\ref{thm:nuStable}}\label{subsec:nuStableProof}
The `moreover' part of Theorem~\ref{thm:nuStable} has already been proved
in Proposition~\ref{prop:nuStableZero}.
The next theorem proves the main part of Theorem~\ref{thm:nuStable}  with an explicit stability bound.
Note that by Remark~\ref{remark:becomesLarge}  
there is no loss
of generality in the `largeness' hypotheses in the theorem. 
The greatest signed weight $\bigl( \omega_{\ell^\-}(\muS)^\-,\omega_{\ell^\-}(\muS)^\+)$
is  defined in Definition~\ref{defn:greatestSignedWeight}
and strongly $c^\+$-maximal signed weights are defined in 
Definition~\ref{defn:stronglyMaximalSignedWeight}.
The $\mathrm{L}$ bounds
are defined in Definition~\ref{defn:LBound}. 
(Remark~\ref{remark:intervalNotation} explains the small difference in notation for the intervals
in the first two bounds.)
The statistics $B_R(\nu)$ and
$E_{c^\+\hskip-1pt,(\kappa^\-,\kappa^\+)}(\nu, \muS : \lambda)$ are defined in 
Definition~\ref{defnBoth:exceptionalColumnBoundBoth} in its `both signs' version in~\S\ref{subsec:negativeSignAndCombined}. The `shift' $S$ below
was first seen in Example~\ref{ex:stronglyMaximalSignedWeightSWLi411} and then
in the definition of $\omega$ in its continuation in Example~\ref{ex:omegaBound411}.

\begin{theorem}\label{thm:nuStableSharp}
Let $\swtp{\kappa}$ be a strongly $c^\+$-maximal signed weight of a semistandard tableau family
of shape $\muS$ and size $R$. Fix $\ell^\- = \ell(\kappa^\-)$.
Let~$\nu$ be a partition and let~$\lambda$ be
an $\bigl(\ell^\-,\ell(\kappa^\+)\bigr)$-large partition of $|\nu||\muS|$.
Set $\nu_R^\+ = \nu_R$ and $\nu_R^\- = \nu'_R$ and let $\pm$ denote the sign of $\kappa$. 
Set $\nu^{(M)} = \nu \sqcups (R^M)$ if $\swtp{\kappa}$ has sign $-1$ and $\nu^{(M)} = \nu + (M^R)$
if $\swtp{\kappa}$ has sign $+1$. Set  $E = E_{c^\+\hskip-1pt,(\kappa^\-,\kappa^\+)}(\nu, \muS : \lambda)$
and 
\[ \Shift = \begin{cases} E - \nu_R^\pm + 1 & \text{if $E \ge \nu_R^\pm$} \\
0 & \text{if $E < \nu_R^\pm$.} \end{cases}\]
Set
\[ D = \bigl|\bigl|  \bigl(B^\pm_R(\nu) + ER\bigr) \bigl( \omega_{\ell^\-}(\muS)^\-, \omega_{\ell^\-}(\muS)^\+ \bigr) 
+ (\nu_R^\pm - E) \swtp{\kappa}
- \swtp{\lambda} \bigr|\bigr|. \]
Let $\kappa$ be the unique partition with $\ell^\-$-decomposition 
$\langle\kappa^\-,\kappa^\+\rangle$.
Define  
\[ \omega = \begin{cases} \kappa \oplus (B^\pm_R(\nu) + ER) \bigl( \omega_{\ell^\-}(\muS)^\-,\omega_{\ell^\-}(\muS)^\+
\bigr) \\ 
\kappa \oplus (B^\pm_R(\nu) + ER) \bigl( \omega_{\ell^\-}(\muS)^\-,\omega_{\ell^\-}(\muS)^\+
\bigr) \oplus (\nu_R^\pm - E - 1) \swtp{\kappa} \end{cases} \]
choosing the case according to whether $E \ge \nu_R^\pm$ or $E < \nu_R^\pm$.
Let $L$ be the maximum of
\begin{bulletlist}
\item $\Shift + \LBound\bigl([\lambda^\- + \Shift\kappa^\-,\omega^\-]^{\scriptscriptstyle \ell^-}_\unLHDS, \kappa^\-\bigr)$,
\item $\Shift + \LBound\bigl([\lambda^\+ + \Shift\kappa^\+,\omega^\++(|\lambda^\+| + \Shift|\kappa^\+|
-|\omega^\+|)]_\unlhd, \kappa^\+\bigr)$,
\item $\Shift + \bigl( \omega_1^\+ + \omega_2^\+ - 2\lambda_1^\+ -2\Shift\kappa_1^\+ + 2|\lambda^\+| +2\Shift|\kappa^\+| 
- 2|\omega^\+|  \bigr)/(\kappa^\+_1-\kappa^\+_2)$,
\item \smash{$\Shift + \bigl( \max( \ell(\lambda^\+), \ell(\kappa^\+) ) + |\omega^\-| - |\lambda^\-|
- \Shift|\kappa^\-| - \omega^\-_{\ell^\-} \bigr)/
\kappa^\-_{\ell(\kappa^\-)}$}
\item \smash{$D + E- \nu^\pm_R + \nu^\pm_{R+1}$},
\end{bulletlist}
omitting the third if $\kappa_1^\+ = \kappa_2^\+$ and the fourth if $\kappa^\- = \varnothing$.
Then
\[ \bigl\langle s_{\nu^{(M)}} \circ s_\muS, 
s_{\lambda \opluss M\swtp{\kappa}} \bigr \rangle \]
is constant for $M \ge L$. Moreover if $\lambda \oplus S\swtp{\kappa} \notunlhddot\, \omega$ in the 
$\ell^\-$-twisted
dominance order then the plethysm
coefficient is $0$ for all $M \in \N_0$.
\end{theorem}

\begin{proof} 
We apply the Signed Weight Lemma (Lemma~\ref{lemma:SWL}) to the stable partition system 
\[ \PSeq{M} = \bigl[\lambda \oplus M\swtp{\kappa}, \omega \oplus (M-S)\swtp{\kappa} \bigr]_\unlhddotS\]
defined in Lemma~\ref{lemma:stablePartitionSystemForNuVarying}.
The intervals are, as ever, for the $\ell^\-$-twisted dominance order.
By hypothesis $\lambda$ is $\bigl(\ell^\-,\ell(\kappa^\+)\bigr)$-large.
By Corollary~\ref{cor:twistedWeightBoundForStronglyMaximalWeightCombined},
$\omega$ is an $\bigl(\ell(\kappa^\-)+1, \ell(\kappa^\+)\bigr)$-large partition
of  $|\lambda| + SR|\muS|$. Again by this corollary, if $\sigma$ is a partition of $|\lambda| + MR|\muS|$,
such that $\sigma \unrhddot \lambda \oplus M\swtp{\kappa}$ such that
$s_\sigma$ is a constituent of the plethysm $s_{\nu^{(M)}} \circ s_\muS$ then
$\sigma \unlhddot \omega \oplus (M-S)\swtp{\kappa}$. 
But by Lemma~\ref{lemma:adjoinToLarge}
and the `if' direction of Lemma~\ref{lemma:signedDominancePreservedByOplusWhenLarge},
if $\lambda \oplus S\swtp{\kappa} \,\notunlhddot\, \omega$ then, writing
\[ \lambda \oplus M\swtp{\kappa} = \lambda \oplus S\swtp{\kappa} \oplus (M-S)\swtp{\kappa},
 \] 
we have 
\[ \lambda \oplus M\swtp{\kappa} \hskip0.5pt\notunlhddot\, \omega \oplus (M-S)\swtp{\kappa}. \]
Hence  if $\lambda \,\oplus\, S\swtp{\kappa} \,\,\hbox{$\notunlhddot$}\, \omega$ 
then $\langle s_{\lambda \,\oplus\, (\kappa^\-,\kappa^\+)}, s_{\nu^{(M)}} \circ s_\muS \rangle = 0$ for all $M  \ge S$. This proves the final claim in the theorem. Moreover,
we may now assume that, for all $M \ge S$, the twisted interval $\PSeq{M}$ is non-empty.

\subsubsection*{Condition \emph{(i)} in the Signed Weight Lemma}
By Lemma~\ref{lemma:stablePartitionSystemForNuVarying}
the stable partition system $\PSeq{M}$ satisfies condition (i) of the Signed Weight Lemma (Lemma~\ref{lemma:SWL})
for the plethysms $s_{\nu^{(M)}} \circ s_\muS$.

\subsubsection*{Condition \emph{(ii)} in the Signed Weight Lemma}
Let $M \in \N_0$ and let $\pi \in \PSeq{M}$.
Recall that $\PYT(\nu, \muS)$ denotes the
set of plethystic signed tableaux of shape $\nu$ having entries from the set $\YT(\muS)$
of signed tableaux of shape $\muS$.
Let $\rho = (1^R)$ if $\swtp{\kappa}$ has sign $+1$ and let $\rho = (R)$ if $\swtp{\kappa}$ has sign $-1$.
Recall from Definition~\ref{notation:stronglyMaximalTableauFromWeight}
that $T_{\swtp{\kappa}}$ is the unique plethystic
semistandard signed tableau 
of size $R$, outer shape $\rho$, inner shape $\muS$ and signed weight $\swtp{\kappa}$.
By Remark~\ref{remark:TAdapted}, it remains semistandard in the 
$\swtp{\kappa}$-adapted colexicographic order.
 Define 
\[ \Hmap : \PSSYT_\kappa(\nu^{(M)}, \muS) \rightarrow \PYT(\nu^{(M+1)},\muS) \]
on $T \in \PSSYT_\kappa(\nu^{(M)}, \muS)$
by inserting $T_{\swtp{\kappa}}$ as a new column immediately
after column $\nu_R$ of $T$ when $\swtp{\kappa}$ has sign $+1$
and as a new row immediately after row $\nu'_R$ of $T$ when $\swtp{\kappa}$
has sign $-1$. Since $T_{\swtp{\kappa}}$ has semistandard entries, all $\muS$-tableau entries
in the image are semistandard.

Suppose that $M \ge D + E - \nu^\pm_R + \nu^\pm_{R+1}$. (The reason for adding $\nu^\pm_{R+1}$ to
the bound from Lemma~\ref{lemma:smallColumnAndRowBound} will be seen shortly.)
By Corollary~\ref{cor:twistedWeightBoundForStronglyMaximalWeightCombined}(i),~$T$ has at most~$E$ exceptional columns/rows.
By Lemma~\ref{lemma:smallColumnAndRowBound}, using that $\lambda \oplus S(\kappa^\-,\kappa^\+)
\unlhddot \omega$,
the bound $D$ is well-defined (i.e.~we take the
root-length of a root-positive element) and $T$ has at most~$D$
small columns/rows. By Definition~\ref{defn:exceptionalColumnAndRow},
a column/row is either exceptional, typical or small.
Since there are $M + \nu^\pm_R$ columns/rows of $T$ of height at least $R$,
there are at least $M + \nu^\pm_R - D - E$ typical columns, in which 
the top/leftmost $R$ entries form the plethystic semistandard signed tableau $T_\swtp{\kappa}$.
(This  requires our use of the $\swtp{\kappa}$-adapted colexicographic order
to order the inner $\muS$-tableau entries of $T$:
see Figure~\ref{fig:signedWeightBoundForStronglyMaximalSignedWeight}.)
Therefore if $M + \nu^\pm_R - D -E \ge \nu^\pm_{R+1}$, the map~$\Hmap$ inserts
$T_{\swtp{\kappa}}$ as a new column/row immediately to the right/below
an identical column/row. (Note that this
condition implies $M \ge - \nu_R^\pm + D + E$, 
and since $D$ is an upper bound for the number of small columns,
the hypothesis on $M$ in Corollary~\ref{cor:twistedWeightBoundForStronglyMaximalWeightCombined}
is satisfied.)
Hence $\Hmap$ is a well-defined bijection
for $M \ge D + E - \nu^\pm_R + \nu^\pm_{R+1}$.
\end{proof}

\begin{example}\label{ex:22and31postTheorem}
In the final part
of the running example in Examples~\ref{ex:exceptionalColumns22and31}, \ref{ex:22and31bound}
and~\ref{ex:22and31final} using the strongly $1$-maximal signed
weight $\bigl((2,2),(3,1)\bigr)$ we saw
that $\langle s_{(2,1) + M(1,1)} \circ s_{(4)}, s_{\lambda \oplus M((2,2), (3,1))} \rangle$
is ultimately constant for each $\lambda \in \bigl\{(8,3,1), (9,2,1), (7,3,2), (8,2,2)\bigr\}$.
To illustrate Theorem~\ref{thm:nuStableSharp}
we find an explicit bound for $(8,3,1)$.
In this context we have $R=2$, $B_2\bigl((2,1)) = 1$,
$E=1$ and $\bigl(\omega_{2}\bigl((4))^\+, \omega_{2}\bigl((4))^\- \bigr) = ((1,1),(2))$
and we saw that the bound $D$ on the number of small columns is $2$.
The fifth bound in Theorem~\ref{thm:nuStableSharp}
is therefore $1 + 2 - 1  + 0 = 2$. We saw earlier that the other bounds are respectively $-2, -1,
-3$ and $-1$, and so the overall bound is $2$.
We also saw that when
$\lambda = (8,3,1)$ the constant value is attained for $M=1$,
so in this case the bound from Theorem~\ref{thm:nuStableSharp}
is not sharp. We remark that if instead $\lambda = (8,2,2) \decMap \dec{(3,3)}{(6)}$ then $E=0$, $S = 0$
and $\omega = (5,3) \oplus \bigl((1,1), (2,2)\bigr) = (7,3,2) \decMap \dec{(3,3)}{(5,1)}$ and
so we are in the final case of the theorem where $\lambda \notunlhddot \omega$ in the $2^\-$-twisted
dominance order (in fact $\lambda \unrhddot \omega$), and so the plethysm coefficient is $0$ for all $M \in \N_0$.
\end{example}

\subsection{Extending Theorem~\ref{thm:nuStableSharp} to the case
$\nu = \varnothing$}
Unless $\ell^\- = 0$,
this  theorem cannot be applied directly when $\nu = \varnothing$
because of the hypothesis that~$\lambda$, which has size $|\muS||\nu|$,
is $(\ell^\-,\ell(\kappa^\+)$-large. Since this is a case to which we would
want our methods to apply, we show how this restriction may easily be circumvented.

\begin{corollary}\label{cor:nuStableEmpty}
Let $\swtp{\kappa}$ be a strongly $c^\+$-maximal signed weight of shape 
$\muS$ and size $R$. Fix $\ell^\- = \ell(\kappa^\-)$
and let $\kappa^{(N)}$ be the unique
partition with $\ell^\-$-decomposition $N\langle \kappa^\-,\kappa^\+\rangle$.
If $\swtp{\kappa}$ has sign $-1$ then
\[ \langle s_{(R^M)} \circ s_\muS, s_{\kappa^{(N)}} \rangle = 1 \]
for all $N \in \N_0$ and if $\swtp{\kappa}$ has sign $+1$ then
the same holds replacing $(R^N)$ with $(N^R)$.
\end{corollary}

\begin{proof}
We apply Theorem~\ref{thm:nuStableSharp} with 
$\swtp{\kappa}$ and $\nu = (R)$ if $\swtp{\kappa}$ has sign $-1$
and $\nu = (1^R)$ if $\swtp{\kappa}$ has sign $+1$
and $\lambda = \kappa$.
Note that, by Remark~\ref{remark:ellDecompositionLarge},
$\kappa$ is $\bigl(\ell(\kappa^\-)+1, \ell(\kappa^\+)\bigr)$-large,
as required in this theorem.
In the notation of the theorem,
$\nu^{(M)} = (R^{M+1})$ if $\swtp{\kappa}$ has sign $-1$
and $\nu^{(M)} = ((M+1)^R)$ if $\swtp{\kappa}$ has sign $+1$.
Moreover $\lambda^{(M)} = \kappa \oplus M(\kappa^\-,\kappa^\+) = 
\kappa^{(M-1)}$. 
Therefore the theorem states that
the plethysm coefficients in the corollary are
constant for all $M$ at least the bound in the theorem.
Since $\kappa = \lambda$, and $B_R^\pm(\nu) = 0$
we have $E^\- = E^\+ = 0$ and hence $E=0$
in Definition~\ref{defnBoth:exceptionalColumnBoundBoth}.
Therefore the case $E < \nu_R^\pm$ applies and $S=0$. Again using
that $B_R^\pm(\nu) = 0$, we have $\omega = \kappa$.
It is now easily seen that the first two
in Theorem~\ref{thm:nuStableSharp} are $0$.
The third, using
that $S=0$,
is $(\kappa_1^\+ + \kappa_2^\+ -2\kappa_1^\+ + 2|\kappa^\+| - 2|\kappa^\+|)/
(\kappa_1^\+-\kappa_2^\+) = -1$, unless $\kappa_1^\+ = \kappa_2^\+$, in which
case this bound is disregard.
The fourth is $(\ell(\kappa^\+) - \kappa^\-_{\ell^\-}) / \kappa^\-_{\ell^\-}$
which is non-positive because $\kappa^-_{\ell^\-} = \kappa'_{\ell^\-} \ge
\kappa'_{\ell^\-+1} = \ell(\kappa^\+)$.
The fifth is $0$ since $D$ is the root-length of the zero weight. 
Therefore the constant value is attained for $M=0$,
proving the corollary for $N \in \N$. When $\kappa$ has sign $-1$,
since $(\kappa^\-,\kappa^\+)$ is a strongly maximal signed weight
$T_{(\kappa^\-,\kappa^\+)}$ is the
unique element of the set 
$\PSSYT\bigl( (R), \muS \bigr)_{(\kappa^\-,\kappa^\+)}\bigr)$,
and so, Proposition~\ref{prop:plethysticSignedKostkaNumbers}
and Lemma~\ref{lemma:twistedKostkaMatrix},
using again that $(\kappa^\-,\kappa^\+)$ is strongly maximal, we have
\[ \langle s_{(R)} \circ s_\muS, s_\kappa \rangle
= \langle s_{(R)} \circ s_\muS, e_{\kappa^\-}h_{\kappa^\+}\rangle
=  \bigl|\PSSYT\bigl( (R), \muS \bigr)_{(\kappa^\-,\kappa^\+)}\bigr| = 1. \]
Therefore the constant value is $1$.
Finally,
since
$s_\varnothing \circ s_\muS = 1$ (the unit element in
the ring of symmetric functions),
the plethysm coefficient is constant for all $N \in \N_0$.
\end{proof}

For example, we saw in
Example~\ref{ex:LawOkitaniSignedWeightsAreStronglyMaximal}(i)
that $\bigl((1^d), (m-d)\bigr)$ is a strongly $1$-maximal signed
weight of shape $(m)$ and sign $(-1)^d$. 
The unique partition $\kappa^{(M)}$ with $d$-decomposition
$M\dec{(1^d)}{(m-d)}$ is $(d^M) + (M(m-d))$ and so Corollary~\ref{cor:nuStableEmpty} implies that if $d$ is odd then
\[ \langle s_{(1^M)} \circ s_{(m)}, s_{(d^M) + M(m-d)} \rangle
= 1 \]
for all $M \in \N_0$ and the same holds replacing $(1^M)$ with $(M)$
if $d$ is even. The analogous stability result, which we believe
is even less obvious, obtained from 
the case $R = 2$ of 
Example~\ref{ex:LawOkitaniSignedWeightsAreStronglyMaximal}(ii)
is that if $d$ is odd then
\[ \langle s_{(2^M)} \circ s_{(m)}, s_{(d^{2M}) + M(2m-2d-1,1)} \rangle
= 1 \] 
for all $M \in \N_0$, and the same holds replacing $(2^M)$ with $(M,M)$ 
if $d$ is even.
For an example of the corollary in the case of skew partitions see
Example~\ref{ex:nuStableSingletonSkewPartition}.


\section{Applications of Theorem~\ref{thm:nuStable}}
\label{sec:nuStableSharpApplications}

\subsection{Theorem~\ref{thm:nuStable} for singleton strongly maximal signed weights}\label{subsec:singletonStronglyMaximalSignedWeights}
We saw in Lemma~\ref{lemma:maximalAndStronglyMaximalSingletonSemistandardSignedTableauFamilies}
that the signed weight 
$(\omega_{\ell^\-}(\muS)^\-, \omega_{\ell^\-}(\muS)^\+)$ 
of the greatest semistandard signed tableau $t_{\ell^\-}(\muS)$ in Definitions~\ref{defn:greatestSignedTableau} and~\ref{defn:greatestSignedWeight}
is a strongly $c^\+$-maximal signed weight where $c^\+ = \ell(\omega_{\ell^\-}(\muS)^\+)$
is the greatest positive entry appearing in $t_{\ell^\-}(\muS)$.
In this subsection we give the special case of 
Theorem~\ref{thm:nuStableSharp} for such strongly maximal weights, which we call \emph{singleton}.
The $\LBound$ bounds below are defined in Definition~\ref{defn:LBound};
see Remark~\ref{remark:intervalNotation} for the reason
for the difference in the notation for intervals
in the first two bounds below.
Recall that $a(\lambda)$ denotes the first part of a partition $\lambda$.

\begin{corollary}\label{cor:nuStableSingleton}
Let $\nu$ be a  partition of $n$ and let
$\muS$ be a skew partition.
Fix $\ell^\- \in \N_0$ and set
$(\kappa^\-,\kappa^\+) =
\bigl(\omega_{\ell^\-}(\muS)^\-,$ $\omega_{\ell^\-}(\muS)^\+ \bigr)$.
Let~$\lambda$ be an $(\ell^\-,\ell(\kappa^\+))$-large partition of $|\nu||\muS|$.
Let $\nu^{(M)} = \nu \sqcup (1^M)$ if $|\kappa^\-|$ is odd
and $\nu^{(M)} = \nu + (M)$ if $|\kappa^\-|$ is even.
Then
\[ \langle s_{\nu^{(M)}} \circ s_\muS, s_{\lambda\, \oplus\,
M(\kappa^\-,\kappa^\+) } \rangle \]
is constant for all $M \ge L$ where $L$ is the maximum
of 
\begin{bulletlist}
\item $\LBound\bigl([\lambda^\-, n\kappa^\-]^\ellmb_\unLHDS, \kappa^\- \bigr)$,
\item $\LBound\bigl([\lambda^\+, n\kappa^\+ + ( |\lambda^\+| - n |\kappa^\+| ) ]_{\unlhd}, \kappa^\+ \bigr)$,
\item $\bigl(n\kappa^\+_1 + n\kappa^\+_2 - 2\lambda^\+_1 + 2|\lambda^\+| - 2n|\kappa^\+|\bigr)/ (\kappa^\+_1 - \kappa^\+_2)$,
\item $\bigl(\max( \ell(\lambda^\+), \ell(\kappa^\+) ) + n |\kappa^\-| - |\lambda^\-| - n\kappa^\-_{\ell(\kappa^\-)}
\bigr) / \kappa^\-_{\ell(\kappa^\-)} $,
\item $\bigl| \bigl| n(\kappa^\- , \kappa^\+) - (\lambda^\-, \lambda^\+) \bigr|\bigr| - 
\nu^\pm_1 + \nu^\pm_2$,
\end{bulletlist}
omitting the third if $\kappa^\+_1 = \kappa^\+_2$ and the fourth if $\ell^\- = 0$ and so $\kappa^\- = \varnothing$.
Moreover if $\lambda \notunlhddot \omega^{(n)}(\muS)$ 
then the plethysm coefficient
is $0$ for all $M \in \N_0$.
\end{corollary}

\begin{proof}
By Lemma~\ref{lemma:maximalAndStronglyMaximalSingletonSemistandardSignedTableauFamilies},
$\bigl(\kappa^\-, \kappa^\+\bigr)$
 is the strongly $\ell\bigl(\omega_{\ell^\-}(\muS)^\+\bigr)$-maximal 
 signed weight of the singleton tableau
family $\{ t_{\ell^\-}(\muS) \}$ of shape $\muS$. 
(Note this holds even if $t_{\ell^\-}(\muS)$ has only negative entries,
in which case the positive part of the signed weight is $\varnothing$.)
Since $t_{\ell^\-}(\muS)$
has $|\kappa^\-|$ negative entries, its sign
is $(-1)^{|\kappa^\-|}$.
Since the tableau family has size $R = 1$, Definition~\ref{defn:exceptionalColumnBound}
states that $E = 0$. Therefore the case $E < \nu_R^\pm$ of
Theorem~\ref{thm:nuStableSharp} applies. 
The partition $\kappa$ in Theorem~\ref{thm:nuStableSharp}
is the unique partition with 
$\ell^\-$-decomposition $\dec{\omega_{\ell^\-}(\muS)^\-}{\omega_{\ell^\-}(\muS)^\+}$. 
(This $\ell^\-$-decomposition is 
well-defined by Lemma~\ref{lemma:ellDecompositionGreatestSignedWeight}.)
We have
\[ B_R^\pm (\nu) = n - a(\nu^\pm). \]
The upper bound partition $\omega$ in Theorem~\ref{thm:nuStableSharp} is 
\[ 
\kappa \oplus \bigl( n - a(\nu^\pm) \bigr)
(\kappa^\-,\kappa^\+)
\oplus \bigl( a(\nu^\pm) - 0 - 1 \bigr)
(\kappa^\-, \kappa^\+) = 
\kappa \oplus (n-1)(\kappa^\-, \kappa^\+) \]
with $\ell^\-$-decomposition $n\decs{\kappa^\-}{\kappa^\+} =
n\dec{\omega_{\ell^\-}(\muS)^\-}{\omega_{\ell^\-}(\muS)^\+}$. 
Thus in all cases
$\omega$ is the partition \smash{$\omega^{(n)}_{\ell^\-}(\muS)$}
defined in Definition~\ref{defn:plethysticGreatestSignedWeight}.
Similarly, we have
\begin{align*} D &= \bigl| \bigl| \bigl(n - a(\nu^\pm) \bigr) (\kappa^\-, \kappa^\+) + \bigl( a(\nu^\pm) - 0)
\bigr) (\kappa^\-, \kappa^\+) - (\lambda^\-, \lambda^\+) \bigr| \bigr| \\ &=
\bigl| \bigl| n(\kappa^\- , \kappa^\+) - (\lambda^\-, \lambda^\+) \bigr|\bigr|. \end{align*}
Since we are in the case $E < \nu_R^\pm$, we have $S=0$.
It is now very easily seen 
that the 
five bounds in Theorem~\ref{thm:nuStableSharp}  simplify as claimed.
The corollary, including the final claim that
if $\lambda \notunlhddot \omega^{(n)}(\muS)$ 
then the plethysm coefficient
is $0$ for all $M \in \N_0$, now follows from this theorem.
\end{proof}


Note that the condition for the plethysm coefficient to vanish
is the same as the one in Remark~\ref{remark:singletonSimpler}.
The following example
illustrates the skew partition case of
Corollary~\ref{cor:nuStableSingleton}.

\begin{example}\label{ex:nuStableSingletonSkewPartition}
Take $\ell^\- = 2$ and $\muS = (4,2)/(1)$.
The tableau $t_2\bigl( (4,2)/(1) \bigr)$ is as shown in the margin and correspondingly
\marginpar{\qquad\raisebox{-3pt}{\ \ \young(:\oM\tM1,\oM\tM)}}
\[  \bigl( \omega_2\bigl( (4,2)/(1) \bigr)^\-, \omega_2\bigl( (4,2)/(1) \bigr)^\+ \bigr)
= \bigl( (2,2), (1) \bigr). \]
By Lemma~\ref{lemma:maximalAndStronglyMaximalSingletonSemistandardSignedTableauFamilies},
this is a strongly $1$-maximal signed weight of shape $(4,2)/(1)$, size $1$ and sign $+1$.
Let $\nu$ be a partition of $n$ and let $\lambda$ be a $(2,1)$-large partition of $5n$.
By Corollary~\ref{cor:nuStableSingleton} the plethysm coefficients
\[ \langle s_{\nu + (M)} \circ s_{(4,2)/(1)}, s_{\lambda \,\oplus\, M((2,2), (1))}\rangle \]
are constant for $M$ at least the bound in this corollary, and the constant value is $0$
unless $\lambda \unlhddot \omega^{(n)}\bigl((4,2)/(1)\bigr) = (n+2,2^{2n-1}) \decMap \dec{(2n,2n)}{(n)}$.
Note that $s_{(4,2)/(1)} = s_{(4,1)} + s_{(3,2)}$ is not a single Schur function, so, as in Example~\ref{ex:needMaximalGood},
this result needs the generality of skew partitions.
Taking $\nu = (1,1)$, the table below shows
values for the inner product for varying partitions~$\lambda$ 
of $12$ (shown decreasing
in the $2$-twisted dominance order),
together with the bound from the corollary.

\smallskip
\begin{center}
\begin{tabular}{cccccccc} \toprule 
$\lambda$ & $0$ & $1$ & $2$ & $3$ & $4$ & $5$ & bound \\ \midrule
$(4,2,1^4)$ & $0$ & $0$ & $0$ & $0$ & $0$ & 0 & 0 \\
$(3,3,2,2)$ & $1$ & $1$ & $1$ & $1$ & $1$ & $1$ & 1 \\
$(4,4,2)$ & $1$ & $12$ & $19$ & $22$ & $22$ & $22$ & 5 \\
$(8,1,1)$ & $1$ & $9$ & $17$ & $17$ & $17$ & $17$ & 5 \\
$(8,2)$ & $0$ & $7$ & $15$ & $16$ & $16$ & $16$ & 6 \\ \bottomrule
\end{tabular}
\end{center}


\smallskip
\noindent
In the first case $(4,2,1,1,1) \decMap \dec{(5,2)}{(2)}$ is incomparable
with $(4,2,2,2) \decMap \dec{(4,4)}{(2)}$ in the $2$-twisted dominance order,
and so the constant multiplicity is $0$.
In each remaining case,
 the fifth bound, $\bigl|\bigl| 2( (2,2), (1) ) - (\lambda^\-,\lambda^\+)\bigr|\bigr|$
is the largest.
For instance
$s_{(1+M,1)} \circ (s_{(4,1)} + s_{(3,2)}), s_{(4+M,4,2^{2M+1})} \rangle = 22$
for all \hbox{$M \ge 3$}. Similar calculations
by computer algebra using the bound from Corollary~\ref{cor:nuStableSingleton} show
that
\begin{align*}
\langle s_{(1+M,1)} \circ s_{(4,1)}, s_{(4+M,4,2^{2M+1})} \rangle &= 0 \quad \text{for all $M \ge 0$}\\
\langle s_{(1+M,1)} \circ s_{(3,2)}, s_{(4+M,4,2^{2M+1})} \rangle &= 7 \quad \text{for all $M \ge 3$;}
\end{align*}
this illustrates the failure of the plethysm product to be 
distributive over addition in its second component.
\end{example}

\begin{remark}\label{remark:nuStableSingletonPartitionCase}
If $\mus = \varnothing$ then, by~\eqref{eq:greatestSignedWeightPartitionCase},
we have $\bigl( \omega_{\ell^\-}(\mu)^\- , \omega_{\ell^\-}(\mu)^\+) =
(\mu^\-, \mu^\+)$ and we may replace $\kappa^\-$ with $\mu^\-$ and $\kappa^\+$ with $\mu^\+$ in all the expressions in Corollary~\ref{cor:nuStableSingleton}.
 
\end{remark}

 We use this remark in \S\ref{subsec:hookExplicitBounds} and 
 \S\ref{subsec:LawOkitaniExplicitBounds}
  below.
Combining Corollary~\ref{cor:nuStableEmpty}
with Remark~\ref{remark:nuStableSingletonPartitionCase}, it follows that, for
any fixed $\ell^\-$-decomposition $(\mu^\-,\mu^\+)$, if $|\mu^\-|$ is odd then
\[ \langle s_{(1^M)} \circ s_\mu, s_{M(\mu^\-,\mu^\+)} \rangle = 1 \]
for all $M \in \N_0$; if $|\mu^\-|$ is even then the same
holds replacing $(1^M)$ with~$M$. 

The following corollary is
the case $\kappa^\- = \varnothing$. For ease of reference
we recall that the greatest tableau $t_{\ell^\-}(\muS)$
is defined in Definition~\ref{defn:greatestSignedTableau}, the $\LBound$ bound
in Definition~\ref{defn:LBound}
and the root-length $||\alpha||$ in~\eqref{eq:rootLength}.

\begin{corollary}\label{cor:nuStableUnsignedSingleton}
Let $\nu$ be a partition of $n$, let
$\muS$ be a skew partition, and let~$\lambda$ be a partition of $|\nu||\muS|$. 
Let $\kappa$ be the positive part of the signed weight of the greatest tableau $t_0(\muS)$.
Then
\[ \langle s_{\nu + (M)} \circ s_\muS, s_{\lambda + M\kappa} \rangle \]
is constant for all $M \in \N_0$ such that $M \ge L$, where $L$ is the maximum
of $\LBound\bigl( [\lambda, n\kappa]_\unlhd, \kappa\bigr)$
and $||n\kappa - \lambda|| - \nu_1 + \nu_2$.
\end{corollary}

\begin{proof}
Apply Corollary~\ref{cor:nuStableSingleton}
taking $\ell^\- = 0$. Thus
$\bigl(\omega_0(\muS)^\-, \omega_0(\muS)^\+ \bigr)
= (\varnothing, \kappa)$
and  the first bound is ignored.
The second is $\LBound\bigl( [\lambda, n\kappa]_\unlhd, \kappa\bigr)$,
the third is the case $k=1$ in Definition~\ref{defn:LBound}, and so is implied by the second.
The fourth bound is again one that should be ignored, 
and the fifth becomes $||n\kappa - \lambda|| - \nu_1 + \nu_2$. 
\end{proof}

\subsection{Explicit bounds for hook stability}
\label{subsec:hookExplicitBounds}
By Lemma~\ref{lemma:maximalAndStronglyMaximalSingletonSemistandardSignedTableauFamilies},
if $1 \le d \le m$ then
$\bigl((d),(m-d)\bigr)$ is a strongly $1$-maximal signed weight of shape $(m-d+1,1^{d-1})$,
corresponding to the singleton tableau family 
$\bigl\{t_1\bigl((m-d+1,1^{d-1})\bigr)\bigr\}$. For instance 
$t_1\bigl((3,1,1,1)\bigr)$ is as shown in the margin.
\marginpar{\young(\oM11,\oM,\oM,\oM)}

\begin{proposition}\label{prop:hookBound}
Let $\nu$ be a partition of $n \in \N$ and let $1 \le d \le m$.
Let $\nu^{(M)} = \nu + (M)$ if $d$ is even and let
$\nu^{(M)} = \nu \sqcup (1^M)$ if $d$ is odd. If $\lambda$ is a partition of $mn$
with $1$-decomposition $\dec{(\ell(\lambda))}{\lambda^\+}$ then
\[ \langle s_{\nu^{(M)}} \circ s_{(m-d+1,1^{d-1})}, s_{\lambda \,\sqcup\, (1^{dM}) + M(m-d) } \rangle \]
is constant for all $M \ge L$ where $L$ is the maximum of
\begin{bulletlist}
\item $\bigl(|\lambda^\+| - 2\lambda^\+_1)/(m-d)$,
\item $\bigl( 2 |\lambda^\+| - 2\lambda^\+_1 - n(m-d) \bigr)/(m-d)$,
\item $\bigl| \bigl| \bigl( (nd), (n(m-d))  \bigr) - \bigl((\ell(\lambda), \lambda^\+\bigr) \bigr|\bigr|
-\nu_1 + \nu_2$.
\end{bulletlist}
Moreover if $\lambda \notunlhddot (1^{nd}) + (n(m-d))$ in the $1$-twisted dominance order
then the plethysm coefficient is $0$ for all $M \in \N_0$.
\end{proposition}

\begin{proof}
We take the singleton strongly maximal weight $(\kappa^\-,\kappa^\+) = \bigl((d),(m-d)\bigr)$ 
in Corollary~\ref{cor:nuStableSingleton},
together with  $\muS = (m-d+1,1^{d-1})$ and $\ell^\- = 1$.
Since $\lambda$ is non-empty, it is $(1,1)$-large, as required in this corollary.
By Definition~\ref{defn:ellDecomposition},
 the $1$-decomposition of the partition $\lambda$
is $\dec{(\ell(\lambda))}{(\lambda_1-1,\ldots,\lambda_{b}-1)}$,
where $b$ is maximal such that $\lambda_b \ge 2$.
Hence, by Definition~\ref{defn:LBound}, the first bound in Theorem~\ref{thm:nuStableSharp}
is $(nd - \ell(\lambda) - nd)/d$, which is non-positive.
(Note the case where
$\ell(\lambda^\-) \le \ell^\-$ applies since $\lambda^\-$ has at most one part.) 
Similarly since $\ell(\kappa^\+) = 1$ 
the second bound is $(n(m-d) + |\lambda^\+| - n(m-d) - 2\lambda^\+_1)/(m-d)$
which simplifies to the first bound above. 
(Again this ignores a potentially stronger bound if $\ell(\lambda^\+) \le 1$.)
The third and fifth bounds in the corollary simplify to the final two bounds
above. The fourth bound is 
$\bigl( \max(\ell(\lambda^\+), 1) + nd - \ell(\lambda) - nd\bigr)/d
= \bigl( \max(\ell(\lambda^\+), 1) - \ell(\lambda)\bigr)/d$
which, since $\lambda^\+$ has length $b \le \ell(\lambda)$, is non-positive.
Since the partition $\omega^{(n)}\bigl((m-d+1,1^{d-1})$ with $1$-decomposition
$n\decs{\kappa^\-}{\kappa^\+} = n\dec{(d)}{(m-d)}$ is $(1^{nd}) + \bigl(n(m-d)\bigr)$,
the result now follows from Corollary~\ref{cor:nuStableSingleton}.
\end{proof}

For example, taking $\nu = (2,1)$, $m=4$ and $d=2$ we find that
\[ \langle s_{(2+M,1)} \circ s_{(3,1)}, s_{\lambda \sqcups (1^{2M}) + (2M)} \rangle\]
is ultimately constant, and zero unless $\lambda \unlhddot (7,1^5)$ in
the $1$-twisted dominance order. (This is equivalent to the condition
$\ell(\lambda) \le 6$.) The case
$\lambda = (4,3,3,2)$, for which the sequence
of plethysm coefficients is $2$, $16$, $31$, $33$, $33$, \ldots 
is illustrative. Here $\lambda^\+ = (3,2,2,1)$ and so
the bounds from Proposition~\ref{prop:hookBound}
are $(8 - 6)/(4-2) = 1$, $\bigl(2 \times 8 - 2 \times 3 - 3(4-2)\bigr)/(4-2) = 4/2 = 2$
and finally $11 - 2 + 1 = 10$, since
\[ \bigl|\bigl| \bigl((6), (6)\bigr) - \bigl((4), (3,2,2,1)\bigr) \bigr|\bigr|
= \bigl|\bigl| \bigl( (2), (3,-2-2,-1) \bigr)\bigr|\bigr| = 2 + 5 +3 + 1 = 11. \]
Therefore $\langle s_{(2+M,1)} \circ s_{(3,1)}, s_{(4+2M,3,3,2,1^{2M})}\rangle = 33$
for all $M \ge 10$.

\subsection{Explicit bounds for Law--Okitani stability}
\label{subsec:LawOkitaniExplicitBounds}

Using Corollary~\ref{cor:nuStableSingleton} and Remark~\ref{remark:nuStableSingletonPartitionCase}, and
very similar arguments to the proof of Proposition~\ref{prop:hookBound},
we can give the first explicit bounds for the stability
result discussed in \S\ref{subsec:earlierWork} due to Law and Okitani
\cite{LawOkitani}, and a sufficient condition for
the stable plethysm coefficient to be zero.
We exclude the case $d=0$
because it is a special case of 
Corollary~\ref{cor:nuStableUnsignedSingleton},
and the case $d=m$ because it reduces to the case $d=0$ by applying the $\omega$ involution.

\begin{proposition}\label{prop:LawOkitaniBound}
Let $\nu$ be a partition of $n$ and let $1 \le d < m$.
Let $\nu^{(M)} = \nu + (M)$ if $d$ is even and let
$\nu^{(M)} = \nu \sqcup (1^M)$ if $d$ is odd. If $\lambda$ is a partition of $mn$
such that $a(\lambda) \ge d$ having
 $d$-decomposition $\decs{\lambda^\-}{\lambda^\+}$,
 then
\[ \langle s_{\nu^{(M)}} \circ s_{(m)}, s_{\lambda \,\sqcup\, (d^M) + M(m-d)} \rangle \]
is constant for all $M \ge L$ where $L$ is the maximum of
\begin{bulletlist}
\item $n(d-1) - |\lambda^\-|$,
\item $\bigl(|\lambda^\+| - 2\lambda^\+_1\bigr)/(m-d)$,
\item $\bigl( 2 |\lambda^\+| - 2\lambda^\+_1 - n(m-d) \bigr)/(m-d)$,
\item $\max(1, \ell(\lambda^\+)) + n(d-1) - |\lambda^\-|$,
\item $\bigl| \bigl| \bigl( (n^d) , ( n(m-d))  \bigr) - (\lambda^\-, \lambda^\+) \bigr|\bigr|
-\nu_1 + \nu_2$.
\end{bulletlist}
Moreover if $\lambda \notunlhddot (d^n) + (n(m-d))$ in the $d$-twisted dominance order
then the plethysm coefficient is $0$ for all $M \in \N_0$.
\end{proposition}

\begin{proof}
By Example~\ref{ex:LawOkitaniSignedWeightsAreStronglyMaximal}(i),
$\bigl((1^d), (m-d)\bigr)$ is a strongly $1$-maximal signed weight
of shape $(m)$, size $1$ and sign $(-1)^d$. 
We take this as $(\kappa^\-,\kappa^\+)$ in Corollary~\ref{cor:nuStableSingleton},
together with  $\muS = (m)$ and $\ell^\- = d$. The condition that $\lambda$
is $(\ell^\-,\ell(\kappa^\+))$-large is equivalent to $a(\lambda) \ge d$.
Observe that  $\kappa^\-_k -\kappa^\-_{k+1}$ is non-zero
only when $k = d$. Hence, by Definition~\ref{defn:LBound}, the first bound in Theorem~\ref{thm:nuStableSharp}
is $(nd - |\lambda^\-| - n)/1$, which
simplify to the first bound above. (Note that the case where
$\ell(\lambda^\-) \le \ell^\-$ applies.) Similarly since $\ell(\kappa^\+) = 1$ 
the second bound is $(n(m-d) + |\lambda^\+| - n(m-d) - 2\lambda^\+_1)/(m-d)$
which again simplifies as shown. 
(Again, as in the earlier proof
of Proposition~\ref{prop:hookBound},
this ignores a potentially stronger bound if $\ell(\lambda^\+) \le 1$.)
The third, fourth and fifth bounds are routine specializations 
of the bounds in the corollary. Since the partition 
$\omega^{(n)}\bigl((m))$ with $d$-decomposition
$n\decs{\kappa^\-}{\kappa^\+} = n\decs{(1^d)}{(m-d)}$ is $(d^n) + n(m-d)$,
the result now follows from Corollary~\ref{cor:nuStableSingleton}.
\end{proof}

\begin{example}\label{ex:LawOkitaniExplicitBounds}
We take $m=4$, $d=3$ and $\nu = (2,1)$.
The table below shows
values of $\langle s_{(2,1^{M+1})} 
\circ s_{(4)}, s_{\lambda \sqcups (3^M) + (M)}\rangle$
for small values of $M$ for varying partitions $\lambda$ of $12$,
shown decreasing in the $2$-twisted dominance order,
together with the bound from the corollary.

\smallskip
\begin{center}
\begin{tabular}{ccccccccc} \toprule 
$\lambda$ & $0$ & $1$ & $2$ & $3$ & $4$ & $5$ & $6$ & bound \\ \midrule
$(5,3,3,1)$ & $0$ & $0$ & $0$ & $0$ & $0$ & 0 & 0& 0 \\
$(6,3,3)$ & $0$ & $0$ & $0$ & $0$ & $0$ & $0$ & 0 &0 \\
$(7,3,2)$ & $1$ & $1$ & $1$ & $1$ & $1$ & $1$ & 1 & 0 \\
$(7,4,1)$ & $1$ & $2$ & $2$ & $2$ & $2$ & $2$ & 2 & 3 \\
$(7,5)$ & $1$ & $4$ & $5$ & $6$ & $6$ & $6$ & 6   & 7 \\
$(6,6)$ & $0$ & $2$ & $5$ & $6$ & $7$ & $7$ & 7  & 8\\ \bottomrule
\end{tabular}
\end{center}


\smallskip
\noindent
In the first case $(5,3,3,1) \decMap \dec{(4,3,3)}{(2)}$ is 
greater than 
the upper bound $(6,3,3) \decMap \dec{(3,3,3)}{(3)}$ in the $3$-twisted dominance order,
and so the constant multiplicity is $0$.
For $(7,5)$ the constant multiplicity is indeed $6$, as can be
checked by using computer algebra to compute the next three values,
or using the generalized Cayley--Sylvester formula in~\eqref{eq:twoRow}.
Thus $\langle s_{(2,1^{M+1})} \circ s_{(4)}, s_{(7+M,5,3^M)} \rangle = 6$
for $M \ge 3$.
It is worth noting that we can obtain further information about the
\emph{same} plethysm $s_{(2,1^{M+1})} \circ s_{(4)}$ by instead taking
$d=1$ in Proposition~\ref{prop:LawOkitaniBound}, now using the 
strongly $1$-maximal signed weight $\bigl((1), (3)\bigr)$.
For instance the proposition
implies that $\langle s_{(2,1^{M+1})} \circ s_{(4)}, s_{(7+3M,5,1^M)} \rangle = 6$
for $M \ge 4$; in fact the constant value is attained for $M \ge 3$.
Similarly
$\langle s_{(2,1^{M+1})} \circ s_{(4)}, s_{(6+3M,6,1^M)} \rangle = 8$
for $M \ge 5$; now the constant value is attained for $M \ge 4$.
These results and bounds may be verified using the Magma code
mentioned in the introduction.
\end{example}




\subsection{The positive non-skew case of Theorem~\ref{thm:nuStable}}\label{subsec:nuStableSharpPositive}
In this section we specialize Theorem~\ref{thm:nuStableSharp} in two ways at once by assuming that 
$\kappa^\- = \varnothing$ and $\mus = \varnothing$. 
(Taken separately, these specializations do not lead
to simplifications significant enough to be worth recording.)
We begin by giving the special case of Definition~\ref{defn:maximalSignedWeight}
and Definition~\ref{defn:stronglyMaximalSignedWeight} since the latter
simplifies greatly in this case.
Recall that $\max \T$ denotes the maximum integer entry of a family of tableaux with integer entries.

\begin{definition}\label{defn:stronglyMaximalSignedWeightPositive}
Let $\mu$ be a non-empty partition and let $R \in \N$.
A family $\SM$ of $R$ distinct semistandard $\mu$-tableaux with entries from $\N$ 
of weight $\kappa$ is
\emph{maximal} if $\kappa$ is maximal in the dominance order  amongst
all such families. It is \emph{strongly} $c$-maximal if whenever
$\phi$ is the weight of a maximal family $\T$ such that $\max \T \le \max \SM$ then
either $\T = \SM$ or
\smash{$\sum_{i=1}^c \phi_i < \sum_{i=1}^c \kappa_i$}.
\end{definition}

It is clear that $\kappa$ is a strongly $c$-maximal weight if and only if
$(\varnothing, \kappa)$ is a strongly $c$-maximal signed weight in the sense 
of Definition~\ref{defn:stronglyMaximalSignedWeight}. We give
examples in \S\ref{subsec:nuStableSharpPositiveNonSkewExamples}.
In the following corollary,
the $\mathrm{L}$ bound is defined in Definition~\ref{defn:LBound}.
The unsigned analogue of the~$\LZBound$ bound in Definition~\ref{defn:LZBound} is 
defined by specializing this definition: given
partitions $\lambda$ and $\omega$ of the same size, and partitions $\eta$ and $\kappa$,
we define $\LZBound\bigl([\lambda, \omega]_\unlhddotS, \kappa, \eta\bigr)$ to to be the minimum
of the quantities
\begin{bulletlist}
\item $\displaystyle \frac{\sum_{i=1}^k \omega_i - \sum_{i=1}^k \lambda_i}{ \sum_{i=1}^k
\eta_i - \sum_{i=1}^k \kappa_i }$.\\[1pt]
\end{bulletlist}
taking those $k$ for which the denominator is strictly positive. In the corollary
$\kappa \lhd \eta$ so the minimum is well-defined.

\begin{corollary}\label{cor:nuStableSharpPositiveNonSkew}
Let $\nu$ be a partition of $n$, let $\mu$ be a partition of $m$ and 
let~$\lambda$ be a partition of $mn$.
Let $\kappa$ be a strongly $c$-maximal weight of shape $\mu$ and size $R$.
Set $E = 0$ if $R=1$ and otherwise set $E$ to the maximum of 
\[ B_R(\nu) \sum_{i=1}^c \mu_i + \nu_R\sum_{i=1}^c \kappa_i - \sum_{i=1}^c \lambda_i 
+ \sum_{i=\ell(\kappa)+1}^{\ell(\lambda)} \lambda_i \]
and zero.
Set $D = \bigl|\bigl|  (B_R(\nu) + ER)\mu + (\nu_R - E)\kappa - \lambda \bigr|\bigr|$.
Let $L$ be the maximum of
\[ \begin{cases}
E-\nu_R + \LBound\bigl( [\lambda + (E-\nu_R)\kappa, (B_R(\nu) + ER)\mu ]_\unlhd, \kappa \bigr) &
\text{if $E \ge \nu_R$} \\
 \LBound\bigl( [\lambda,  (B_R(\nu) + ER)\mu + (\nu_R-E)\kappa ]_\unlhd, \kappa \bigr) &\text{if $E < \nu_R$}
\end{cases} \]
and $D+E- \nu_R + \nu_{R+1}$. Then
$\langle s_{\nu + (M^R)} \circ s_\mu, s_{\lambda + M\kappa} \rangle$
is constant for $M \ge L$.
Moreover if $\eta$ is a partition of $MR$ such that $\kappa \lhd \eta$ then
$\langle s_{\nu + (M^R)} \circ s_\mu, s_{\lambda + M\eta} \rangle = 0$
for all
\[ M > \begin{cases}
E-\nu_R + \LZBound\bigl( [\lambda + (E-\nu_R)\eta, (B_R(\nu) + ER)\mu]_\unlhddotS, \kappa, \eta \bigr) &
\text{if $E \ge \nu_R$} \\
 \LZBound\bigl( [\lambda,  (B_R(\nu) + ER)\mu + (\nu_R-E)\kappa]_\unlhddotS, \kappa, \eta 
 \bigr) &\text{if $E < \nu_R$.}
\end{cases} \]
\end{corollary}

\begin{proof}
We apply Theorem~\ref{thm:nuStableSharp} with $\kappa^\- = \varnothing$ and $\kappa^\+ = \kappa$.
Thus $\ell^\- = 0$ and $\bigl(\omega_{\ell^\-}(\muS)^\-, \omega_{\ell^\-}(\muS)^\+)  = (\varnothing,
\mu)$
by~\eqref{eq:greatestSignedWeightPartitionCase}
and the following remark. The sign of $(\varnothing, \kappa^\+)$ is $+1$
so the definitions given in the main part of \S\ref{sec:twistedWeightBoundForStronglyMaximalWeight} apply
and $\nu^\+ = \nu$ in the statement of the theorem.
It is easily seen from 
Definition~\ref{defn:exceptionalColumnBound} 
that $E_{c,(\varnothing,\kappa)}(\nu, \mu : \lambda)$ is $E$
as stated in the corollary and similarly from Theorem~\ref{thm:nuStableSharp}
using that $\omega_0(\mu)^\- = \varnothing$ and
\[ \omega = \begin{cases} (B_R(\nu) + ER)\mu + \kappa & \text{if $E \ge \nu_R$} \\
(B_R(\nu) + ER)\mu + (\nu_R -E)\kappa & \text{if $E < \nu_R$} \end{cases} \]
that $D$ is as stated.
Since $|\lambda| + SR|\mu| =  
|\lambda^\+| + S|\kappa^\+| = |\omega| = |\omega^\+|$,
where the second equality uses
Corollary~\ref{cor:twistedWeightBoundForStronglyMaximalWeightCombined}(iii) and
that $\LBound\bigl( [\varnothing, \varnothing]_\unlhd, \varnothing \bigr) = 0$,
the bounds defining~$L$ in this theorem simplify to
$S$, $S + \LBound\bigl([\lambda + S\kappa, \omega]_\unlhd, \kappa\bigr)$,
$S + (\omega_1 + \omega_2- 2\lambda_1 - 2S\kappa_1)/(\kappa_1-\kappa_2)$,
$0$, and $D+E - \nu_R + \nu_{R+1}$, respectively.
If $\kappa_1 = \kappa_2$ then the third quantity should be disregarded;
otherwise it is one of the lower bounds appearing in
Definition~\ref{defn:LBound} defining 
$L\bigl( [\lambda + S\kappa, \omega], \kappa\bigr)$.
Finally we slightly simplify $\LBound\bigl( [\lambda + S\kappa, \omega], \kappa\bigr)$
using the lemma that $\LBound\bigl( [\alpha + \kappa, \beta + \kappa]_\unlhd,\kappa\bigr)
= L\bigl( [\alpha,\beta]_\unlhd,\kappa \bigr) - 1$ to show that $L$ is as
claimed in the statement of the corollary. 
The proof of the `moreover' part is very similar, using the bounds
from Proposition~\ref{prop:nuStableZero} with the same specializations
and the same simplification of the $\mathrm{LZ}$ bound.
\end{proof}

We note that, again using Theorem~\ref{thm:nuStableSharp},
the plethysm coefficient is zero
if $\lambda + S\swtp{\kappa} \notunlhd \omega$,
where $S = E - \nu_R + 1$ if $E \ge \nu_R$ and $S=0$ otherwise,
and $\omega$ is as defined in the proof above.

\subsection{Examples of Corollary~\ref{cor:nuStableSharpPositiveNonSkew}}
\label{subsec:nuStableSharpPositiveNonSkewExamples}
In Examples~\ref{ex:maximalSemistandardSignedTableauFamiliesMixedSign}
and~\ref{ex:stronglyMaximalSemistandardSignedTableauFamilies},
we saw that $(4,1,1)$ and $(3,3)$ are the two strongly maximal
weights of shape $(2)$ and size $3$
and that the corresponding semistandard tableau families are
\[  \left\{ \hskip1pt \young(11)\spy{0pt}{,\ts} \young(12)\spy{0pt}{,\ts} \young(13)\hskip1pt  \right\}  \!, \;
   \left\{ \hskip1pt \young(11)\spy{0pt}{,\ts} \young(12)\spy{0pt}{,\ts} \young(22)\hskip1pt  \right\}  \]
respectively. 

\begin{example}\label{ex:411bound}
In the running example using the strongly $1$-maximal weight $(4,1,1)$ of shape $(2)$,
size $3$ and sign $+1$ completed
in Example~\ref{ex:omegaBound411} we saw that
if $\nu = (2,1) + C(1,1,1)$ and $\lambda = (4,2) + C(4,1,1)$ then $E=2$;
this can now be computed more simply using the
formula in Corollary~\ref{cor:nuStableSharpPositiveNonSkew}. The quantity $D$
in this corollary is 
\[ \bigl|\bigl| (3+2.3)(2) + (C-2)(4,1,1) - \bigl((4,2) +C(4,1,1) \bigr) \bigr|\bigr| 
= \bigl|\bigl| (6,-4,-2) \bigr|\bigr| = 8, \]
independent of the value of $C$. Exploiting similar cancellation we have,
for $C = 0$ or $C = 1$,
\[ \begin{split} 
\LBound\bigl( \bigl[ (4,2) +C(4,1,1) + (2-{}&{}C)(4,1,1), (3 + 2.3)(2) \bigr]_\unlhd, (4,1,1)\bigr) \\
&\ = \LBound\bigl( \bigl[ (12,4,2), (18) \bigr]_\unlhd, (4,1,1) \bigr) = 0.\end{split} \]
Now taking $C=0$, Corollary~\ref{cor:nuStableSharpPositiveNonSkew} implies that
\[ \langle s_{(2,1) + M(1,1,1)} \circ s_{(2)}, s_{(4,2) + M(4,1,1)} \rangle \] 
is constant
for $M \ge 10$; the two bounds are respectively $2$ and $10$.
In fact it follows from the enumeration of plethystic semistandard tableaux
in the running example that the plethysm coefficient is constant for $M \ge 2$;
the constant value is $2$.
\end{example}

It is routine to give a similar example using the strongly $2$-maximal weight 
$(3,3)$. This gives a special case of Proposition~\ref{prop:BOR} below.
The proof of this proposition
is a good example of how stronger bounds than the generic bounds in our main
theorems can be obtained by ad-hoc reasoning.
Note that the assumption $\ell(\mu) \le \ell$ is without loss of generality,
since this condition is necessary for there to be a semistandard $\mu$-tableau with entries
from $\{1,\ldots, \ell\}$.

\begin{proposition}\label{prop:BOR}
Fix $\ell \in \N$ and let $\mu$ be a  partition with $\ell(\mu) \le \ell$. Let $R$ be
the number of semistandard tableaux of shape $\mu$ with entries from $\{1,\ldots,\ell\}$.
Set $q = R|\mu|/\ell$. Then for any partitions $\nu$ and $\lambda$ with $\ell(\nu) < R$
and $\ell(\lambda) \le \ell$,
\[ \langle s_{\nu + M(1^R)} \circ s_\mu, s_{\lambda + M(q^\ell)} \rangle \]
is constant for $M \in \N_0$. 
\end{proposition}

\begin{proof}
Let $\mathcal{T}$ be the strongly maximal signed tableau family  consisting
of all semistandard $\mu$-tableaux with entries from $\{1,\ldots, \ell\}$.
(Thus there are no negative entries.)
By Lemma~\ref{lemma:allTableauFamily} its weight is the
strongly $\ell$-maximal weight
$(q^\ell)$ of shape $\mu$ and size $R$.
Since $\ell(\nu) < R$ we have $B_R(\nu) = |\nu|$ and since $\ell(\lambda) \le \ell$,
the exceptional column bound $E$ in the corollary is $|\nu||\mu| - |\lambda| = 0$.
The bounds in Corollary~\ref{cor:nuStableSharpPositiveNonSkew}
 are therefore 
 \[ \LBound\bigl( [\lambda, |\nu|\mu]_{\unlhd}, (q^\ell)
\bigr) = (|\nu||\mu| - |\lambda| - |\nu|\mu_\ell)/q  - |\nu|\mu_\ell /q < 0 \] 
and 
$\bigl|\bigl|\, |\nu|\mu - \lambda\, \bigr|\bigr|$. Inspection of the proof of
Theorem~\ref{thm:nuStableSharp} shows that the second bound is needed to ensure 
that the insertion map $\mathcal{H}$, defined by inserting a new column of height $R$ into a plethystic
semistandard tableau with entries from the tableau family $\mathcal{T}$, is surjective.
But since $\mathcal{T}$ contains \emph{all} tableaux of shape $\mu$, \emph{any} column
of height $R$ in a plethystic semistandard signed tableau having $\mu$-tableau entries
from $\{1,\ldots, \ell\}$ is of this special form. Therefore the plethysm coefficient is immediately
constant.
\end{proof}

\begin{example}\label{ex:BOR}
Take $\mu = (2,1)$ and $\ell = 3$ and the strongly maximal semistandard tableau family of all 
$(2,1)$-tableaux with entries from $\{1,2,3\}$ relevant to the famous eightfold way adjoint
representation of $\mathrm{SU}_3(\C)$ (see \cite[page 179]{FultonHarrisReps}), 
containing the $8$ tableaux shown below
\[ \young(11,2)\, , \ \young(12,2)\, , \ \young(11,3)\, , \ \young(12,3)\, , 
\ \young(13,2)\, , \ \young(13,3)\, , \ \young(22,3)\, , \ \young(23,3)\, . \]
The corresponding $3$-strongly maximal weight is $(8,8,8)$. Applying Proposition~\ref{prop:BOR} 
we find that
\[ \langle s_{\nu + M(1^8)} \circ s_{(2,1)}, s_{\lambda + M(8,8,8)} \rangle \]
is constant for $M \in \N_0$,
whenever $\nu$ and $\lambda$ are partitions with $\ell(\nu) < 8$ and $\ell(\lambda) \le 3$.
For example, taking $\nu = (2)$ and $\lambda = (3,2,1)$, 
the stable value of the plethysm coefficient is $1$.
\end{example}

Many further examples of non-obvious stability results can be given
using the strongly maximal weights found in 
Example~\ref{ex:maximalSemistandardSignedTableauFamiliesShape2} and 
Table 4.23 in~\S\ref{subsec:stronglyMaximalSignedWeightsTable}.

\tocless{\section*{Acknowledgements}}

The second author thanks Martin Forsberg Conde and Eoghan McDowell
for very helpful conversations when he visited OIST, Japan
in November 2022 sponsored by the Theoretical Sciences Visiting Programme. He thanks OIST for its support. Both authors thank {\'A}lvaro Guti{\'e}rrez and Stacey Law
for helpful comments on an earlier version of this paper. The second author was supported by
the Heilbronn Institute for Mathematical Research.

\bigskip
\addtocontents{toc}{\smallskip}
\renewcommand{\MR}[1]{\relax}
\def\cprime{$'$} \def\Dbar{\leavevmode\lower.6ex\hbox to 0pt{\hskip-.23ex
  \accent"16\hss}D} \def\cprime{$'$}
\providecommand{\bysame}{\leavevmode\hbox to3em{\hrulefill}\thinspace}
\providecommand{\MR}{\relax\ifhmode\unskip\space\fi MR }
\providecommand{\MRhref}[2]{%
  \href{http://www.ams.org/mathscinet-getitem?mr=#1}{#2}
}
\providecommand{\href}[2]{#2}
\def\cprime{$'$} \def\Dbar{\leavevmode\lower.6ex\hbox to 0pt{\hskip-.23ex
  \accent"16\hss}D} \def\cprime{$'$}
\providecommand{\bysame}{\leavevmode\hbox to3em{\hrulefill}\thinspace}
\providecommand{\MR}{\relax\ifhmode\unskip\space\fi MR }
\providecommand{\MRhref}[2]{%
  \href{http://www.ams.org/mathscinet-getitem?mr=#1}{#2}
}
\providecommand{\href}[2]{#2}

\end{document}